%% file: HomogeneousUlrich.tex
\documentclass[11pt]{amsart}
\usepackage{hyperref}
\usepackage{amsfonts}
\usepackage{amsmath}
\usepackage{amssymb}
\usepackage{amscd}
\usepackage{graphicx}
\usepackage{latexsym}
\usepackage{color}
\usepackage{epsfig}
\usepackage{amsopn}
\usepackage{xypic}
\usepackage{mathrsfs}
\usepackage{array}

\usepackage{booktabs}
\usepackage{longtable}
\theoremstyle{plain}
\newtheorem{lemma}{Lemma}[section]
\newtheorem*{theorem*}{Theorem}
\newtheorem*{lemma*}{Lemma}
\newtheorem*{proposition*}{Proposition}
\newtheorem*{conjecture*}{Conjecture}
\newtheorem*{corollary*}{Corollary}
\newtheorem*{problem*}{Problem}
\newtheorem{theorem}[lemma]{Theorem}
\newtheorem{conjecture}[lemma]{Conjecture}
\newtheorem{corollary}[lemma]{Corollary}
\newtheorem{proposition}[lemma]{Proposition}

\newtheorem{problem}[lemma]{Problem}

\theoremstyle{definition}
\newtheorem{definition}[lemma]{Definition}
\newtheorem{example}[lemma]{Example}
\newtheorem{remark}[lemma]{Remark}

\newtheorem{warning}[lemma]{Warning}
\newtheorem{observation}[lemma]{Observation}

\newcommand{\Z}{\mathbb{Z}}

\newcommand{\N}{\mathbb{N}}

\newcommand{\C}{\mathbb{C}}

\newcommand{\OO}{\mathcal{O}}
\newcommand{\te}{\otimes}
\newcommand{\sm}{\setminus}

\newcommand{\cO}{\mathcal{O}}

\newcommand{\PP}{\mathbb{P}}

\DeclareMathOperator{\rk}{rk}

\DeclareMathOperator{\hook}{hook}
\DeclareMathOperator{\init}{init}

\begin{document}

\date{\today}
\author[I. Coskun]{Izzet Coskun}
\address{Department of Mathematics, Statistics and CS \\University of Illinois at Chicago, Chicago, IL 60607}
\email{coskun@math.uic.edu}
\author[L. Costa]{Laura Costa}
\address{Facultat de Matem\`{a}tiques, Departament d'Algebra i Geometria, Gran Via de les Corts Catalanes 585, 08007 Barcelona, SPAIN}
\email{costa@ub.edu}
\author[J. Huizenga]{Jack Huizenga}
\address{Department of Mathematics, The Pennsylvania State University, University Park, PA 16802}
\email{huizenga@psu.edu}
\author[R. M. Mir\'{o}-Roig]{Rosa Maria Mir\'{o}-Roig}
\address{Facultat de Matem\`{a}tiques, Departament d'Algebra i Geometria, Gran Via de les Corts Catalanes 585, 08007 Barcelona, SPAIN}
\email{miro@ub.edu}
\author[M. Woolf]{Matthew Woolf}
\address{Department of Mathematics, Statistics and CS \\University of Illinois at Chicago, Chicago, IL 60607}
\email{woolf@math.uic.edu}
\thanks{During the preparation of this article the first author was partially supported by the NSF CAREER grant DMS-0950951535, the second author was partially supported by MTM2013-45075-P, the third author was partially supported by a National Science Foundation Mathematical Sciences Postdoctoral Research Fellowship, and the fourth author was partially supported by MTM2013-45075-P}

\title{Ulrich Schur bundles on Flag varieties}

\begin{abstract}
In this paper, we study equivariant vector bundles on partial flag varieties arising from Schur functors.  We show that a partial flag variety with three or more steps does not admit an Ulrich bundle of this form with respect to the minimal ample class.  We classify Ulrich bundles of this form on two-step flag varieties $F(1,n-1;n)$, $F(2, n-1; n)$, $F(2, n-2; n)$, $F(k, k+1; n)$ and $F(k, k+2; n)$. We give a conjectural description of the two-step flag varieties which admit such Ulrich bundles.
\end{abstract}
\maketitle

\setcounter{tocdepth}{1}
\tableofcontents

\section{Introduction}\label{sec-intro}
Let $V$ be an $n$-dimensional vector space and  $0< k_1 < \cdots < k_r <n$ an increasing sequence of integers. For convenience, set $k_0 = 0$ and $k_{r+1} =n$. The $r$-step partial flag variety $F(k_1, \dots, k_r; n)$  parameterizes partial flags $$W_1 \subset W_2 \subset \cdots \subset W_r \subset V,$$ where $W_i$ is a $k_i$-dimensional subspace of $V$ for $1 \leq i \leq r$. The variety $F(k_1, \dots, k_r; n)$ is endowed with a collection of tautological subbundles
$$0=T_0 \subset T_1 \subset T_2 \subset \cdots \subset T_r  \subset T_{r+1}= \underline{V} = V \otimes \OO_{F(k_1, \dots, k_r; n)}, $$ where $T_i$ is a vector bundle of rank $k_i$ and $\underline{V}$ denotes the trivial bundle of rank $n$. Let $U_i = T_i/ T_{i-1}$. In this paper, we study the equivariant vector bundles $E_\lambda$ on $F(k_1, \dots, k_r; n)$ defined by tensor products of Schur functors of these tautological bundles $$E_\lambda = \mathbb{S}^{\lambda_1} U_1^* \otimes \mathbb{S}^{\lambda_2} U_2^* \otimes \cdots \otimes \mathbb{S}^{\lambda_{r+1}} U_{r+1}^*,$$ where the partition $\lambda = (\lambda_1|\cdots |\lambda_{r+1})$ is the concatenation of partitions $\lambda_i$ of length $k_i-k_{i-1}$.  We call the bundles $E_\lambda$ on the flag variety \emph{Schur bundles}.

The question we study in this paper is to determine which Schur bundles on partial flag varieties are Ulrich bundles with respect to the minimal ample class.  The Picard group of $F(k_1, \dots, k_r; n)$ is generated by the classes of Schubert divisors. The sum of the Schubert divisors corresponds to a line bundle $\cO(1)$ that defines a projectively normal embedding of $F(k_1, \dots, k_r; n)$ \cite{Ramanathan}. Unless we specify otherwise, we will always consider the partial flag varieties in this embedding.

Let $X \subset \PP^m$ be an arithmetically Cohen-Macaulay variety of dimension $d$.  A vector bundle $\mathcal{E}$ of rank $r$ on $X$ is Ulrich if $H^i(X, \mathcal{E}(-i)) = 0$ for $i>0$ and $H^j(X, \mathcal{E}(-j-1))=0$ for $j < d$ (see \cite{BaHU}, \cite{BHU}, \cite{ESW}). These are the bundles whose Hilbert polynomials have $d$  zeros at the first $d$ negative integers. Their pushforwards to $\PP^m$ via the inclusion of $X$ has a minimal free resolution where all the syzygies are linear and the $i$-th Betti number is $\deg(X) r {m-d \choose i}$ \cite[Proposition 2.1]{ESW}.

Ulrich bundles have important applications in liaison theory, singularity theory and Boij-S\"{o}derberg theory. For example, if $X$ admits an Ulrich bundle, then its  cone of cohomology tables coincides with that of $\PP^n$ \cite{EisenbudSchreyer}.   In view of their importance, Eisenbud, Schreyer and Weyman \cite{ESW}  formulated the problem of determining which varieties admit Ulrich bundles. Many authors have considered this problem. We refer the reader to \cite{CasanellasHartshorne}, \cite{CKM}, \cite{CostaMiroRoig1}, \cite{CMP}, \cite{ESW}, \cite{Faenzi}, \cite{Miro}, \cite{MP1}, \cite{MP2} for references and further results.  For example, Ulrich bundles exist on curves, linear determinantal varieties, hypersurfaces and more generally on complete intersections (see \cite{BaHU}, \cite{BHU}, \cite{ESW}, \cite{KP}).

The second and fourth authors have initiated the study of homogeneous Ulrich bundles on homogeneous varieties. In \cite{CostaMiroRoig1}, they  classified all Schur bundles which are Ulrich on Grassmannians under the Pl\"{u}cker embedding. In this paper, we extend this study to flag varieties. Our first main theorem is the following.

\begin{theorem*}[Theorem \ref{thm-higherStep}]
If $r\geq 3$, then no flag variety $F(k_1, \dots, k_r; n)$ admits a Schur bundle $E_\lambda$ which is Ulrich with respect to $\OO(1)$.
\end{theorem*}

In view of Theorem \ref{thm-higherStep}, we concentrate on two-step flag varieties. We refer the reader to \S \ref{sec-overview2step} for a detailed description of our results for two-step flag varieties. We may summarize our results as follows.

\begin{theorem*}[Theorems \ref{thm-1n1}, \ref{thm-beta1intro}, \ref{thm-beta2intro}, \ref{thm-1n2intro}, \ref{thm-2n2intro}]
We classify all the Schur bundles $E_\lambda$ which are Ulrich with respect to $\cO(1)$ on
$$ F(1,n-1;n), \quad F(1, n-2; n),  \quad F(2, n-2; n), \quad F(k, k+1; n), \quad \mbox{and} \quad F(k, k+2; n).$$
\end{theorem*}

The Borel-Weil-Bott Theorem computes the cohomology of Schur bundles on flag varieties. Let $N$ denote the dimension of $F(k_1, \dots, k_r; n)$. In order to determine whether a bundle $E_\lambda$ is Ulrich, we need to check that the cohomology of  $N$ consecutive twists of $E_\lambda$ vanishes. Using the Borel-Weil-Bott Theorem, this reduces to an intricate combinatorial problem. In \S \ref{sec-prelim}, we explain the algebraic geometry and representation theory background needed to turn the problem into a combinatorial problem. Then for the rest of the paper we study the combinatorial problem, which may be stated as follows.

\begin{problem}
Let $$P= (a_1^1, \dots, a_{l_1}^1 | a_1^2, \dots, a^2_{l_2}| \cdots | a_1^{r+1}, \dots, a_{l_{r+1}}^{r+1})$$ be a strictly decreasing sequence of integers divided into $r+1$ subsequences. Let $P(t)$ denote the sequence $$(a_1^1 - tr , \dots, a_{l_1}^1 -tr | a^2_1-t(r-1) , \dots, a^2_{l_2} - t (r-1) | \cdots | a_1^{r+1} , \dots, a^{r+1}_{l_{r+1}})$$ obtained by subtracting $t (r-i+1)$ from $a^i_j$. Classify the sequences $P$ for which the sequences $P(t)$ have exactly two equal entries for $1 \leq t \leq \sum_{1\leq i<j \leq r+1} l_i l_j$.
\end{problem}

Theorem \ref{thm-higherStep} is equivalent to the statement that such sequences do not exist if $r \geq 3$. The combinatorial problem is most interesting when $r=2$. In this case, there are infinite families of examples, which makes their classification challenging. We conjecture that there do not exist such sequences with both $l_1$ and $l_2$ at least three. In terms of geometry, this conjecture translates to the following.

\begin{conjecture}\label{conj-main}
The two-step flag variety $F(k_1, k_2; n)$ does not admit a Schur bundle which is Ulrich with respect to $\OO(1)$ if $k_1 \geq 3$ and $k_2-k_1 \geq 3$.
\end{conjecture}

If Conjecture \ref{conj-main} is true, then our results give a complete classification of all Schur bundles which are Ulrich with respect to $\OO(1)$  on partial flag varieties. In particular, such bundles are very rare. Hence, in order to construct Ulrich bundles on partial flag varieties, one has to look elsewhere.

One may also consider Schur bundles which are Ulrich with respect to other embeddings of $F(k_1, \dots, k_r; n)$. Our constructions all scale. In particular, if $E_\lambda$ is Ulrich with respect to $\OO(1)$, then multiplying the integers in $\lambda$ by $m$ gives a bundle which is Ulrich with respect to $\OO(m)$.  Our combinatorial techniques can in principle be applied to other polarizations, though we do not discuss any such results here.

\subsection*{The organization of the paper} In \S \ref{sec-prelim}, we explain the necessary background from algebraic geometry and representation theory needed to turn the classification of Schur bundles which are Ulrich into a combinatorial problem. In \S \ref{sec-comb}, we explain the combinatorial problem. In \S \ref{sec-higherStep}, we solve the combinatorial problem for flag varieties of at least three steps and show that they do not admit a Schur bundle which is Ulrich for $\cO(1)$. In \S \ref{sec-overview2step}, we explain our results for two-step flag varieties and classify Ulrich bundles on flag varieties of the form $F(1,n-1;n)$. In sections \S \ref{sec-sumset}, \S \ref{sec-beta2}, \S \ref{sec-1n2} and \S \ref{sec-2n2}, we classify Ulrich bundles on flag varieties of the form  $F(k, k+1; n)$, $F(k, k+2; n)$, $F(2, n-1; n)$ and $F(2, n-2; n)$, respectively.

\section{Algebraic geometry background}\label{sec-prelim}

\subsection*{Partial flag varieties}
Let $0< k_1 < \cdots < k_r <n$ be an increasing sequence of integers. Set  $k_0 = 0$ and $k_{r+1} = n$. The $r$-step partial flag variety $F(k_1, \dots, k_r; n)$  parameterizes partial flags $W_1 \subset W_2 \subset \cdots \subset W_r \subset V,$ where $\dim W_i = k_i$.
  Given any set of indices $1 \leq i_1 < i_2 < \cdots < i_s \leq r$, there is a natural projection $$\pi_{i_1, \dots, i_s} : F(k_1, \dots, k_r; n) \rightarrow F(k_{i_1}, \dots, k_{i_s}; n).$$  In particular, the partial flag variety can be realized as an iterated Grassmannian bundle
$$F(k_1, \dots, k_r; n) \rightarrow F(k_2, \dots, k_r; n) \rightarrow \cdots \rightarrow G(k_r; n).$$ From this description, it is immediate to read that
$$N:=\dim(F(k_1, \dots, k_r; n))= \sum_{i=1}^r k_i (k_{i+1} - k_i).$$
The partial flag variety has $r$ projections to Grassmannians $$\pi_i : F(k_1, \dots, k_r; n) \rightarrow G(k_i; n).$$ The Picard group of $F(k_1, \dots, k_r; n)$ is generated by $L_i= \pi_i^* \OO_{G(k_i; n)}(1)$. A line bundle $L$ is ample on $F(k_1, \dots, k_r; n)$ if and only if $L = L_1^{\otimes a_1} \otimes \cdots \otimes L_r^{\te a_r}$ with $a_i > 0$ for every $1 \leq i \leq r$.

\subsection*{The degree of partial flag varieties} Let $X \subset \PP^n$ be an equivariant embedding of a projective homogeneous variety. The Weyl dimension formula implies that the Hilbert polynomial of $X$ factors as a product of linear polynomials. Consequently, one obtains a formula for the degree of $X$ \cite{GrossWallach}.
Let $\psi$ denote the dominant weight defining the embedding of $X$ in $\PP^n$. Let $\rho$ be half the sum of the positive roots and let $N$ be the dimension of $X$. Then the degree of $X$ is given by
\begin{equation}\label{eq-degree}
N! \prod_{\alpha} \frac{\langle \psi, \check{\alpha} \rangle }{\langle \rho, \check{\alpha} \rangle },
\end{equation}
where the product is over all positive roots $\alpha$ with $ \langle \psi, \check{\alpha} \rangle >0$ and $\check{\alpha}$ denotes the coroot  \cite{GrossWallach}.

Now we specialize to the case of partial flag varieties. Let $a_1, \dots, a_r$ be a sequence of positive integers. Set $$b_{ij} = \sum_{k=i}^j a_k.$$ We would like to calculate the degree of $F(k_1, \dots, k_r; n)$ under the embedding defined by $$L_{a_1, \dots, a_r} = L_1^{\otimes a_1} \otimes \cdots \otimes L_r^{\otimes a_r}.$$ The ample line bundle $L_{a_1, \dots, a_r}$ corresponds to the dominant weight
$$ \psi_{a_1, \dots, a_r} = b_{1r} e_1 + \cdots + b_{1r}e_{k_1} + b_{2r} e_{k_1 + 1} + \cdots + b_{2r} e_{k_2} + \cdots + b_{rr} e_{k_r}.$$
We have that  $$\rho = (n-1) e_1 + (n-2) e_2 + \cdots + e_{n-1}.$$ The coroots that have nonzero intersection with $\psi_{a_1, \dots, a_r}$ are precisely $e_i - e_j$, where $k_{s-1} < i \leq k_s$ for some $1 \leq s \leq r$ and $k_t < j \leq k_{t+1}$ for $s<t \leq r$. For such $e_i - e_j$, the numerator in Equation (\ref{eq-degree}) is $b_{st}$ and the denominator is $j-i$. We thus conclude the following proposition.

\begin{proposition}
The degree of the partial flag variety $F(k_1, \dots, k_r; n)$ embedded by the line bundle $L_{a_1, \dots, a_r}$ is
\begin{equation}\label{eq:partialflagdegree}
N! \frac{\prod_{1 \leq i \leq j \leq r} b_{ij}^{(k_i - k_{i-1})(k_{j+1} - k_j)}}{\prod_{s=1}^r \prod_{k_{s-1} < i \leq k_s} \frac{(n-i)!}{(k_s -i)!}}.
\end{equation}
\end{proposition}

\begin{remark}
We record a few special cases of the formula.
The degree of $G(k,n)$ under the $a$th power of the Pl\"{u}cker embedding is
 $$(k(n-k))! a^{k(n-k)} \prod_{1 \leq i \leq k} \frac{(k-i)!}{(n-i)!}.$$
The degree of $F(k_1, \dots, k_r; n)$ under the line bundle $L_{1,\dots,1}$ is
$$N !  \frac{\prod_{1 \leq i \leq j \leq r} (j-i+1)^{(k_i - k_{i-1})(k_{j+1} - k_j)}}{\prod_{s=1}^r \prod_{k_{s-1} < i \leq k_s} \frac{(n-i)!}{(k_s -i)!}}.$$
In particular, specializing to the case $F(1, n-1; n)$, we obtain
$$(2n-3)! \frac{2}{(n-1)!(n-2)!}.$$
Specializing to the case $F(1, n-2, n)$, we obtain
$$(3n-7)! \frac{4}{(n-1)!(n-2)!(n-3)!}.$$
\end{remark}

\subsection*{Borel-Weil-Bott Theorem and dimension of cohomology groups}
The partial flag variety $F(k_1, \dots, k_r; n)$ is endowed with a collection of tautological subbundles
$$0=T_0 \subset T_1 \subset T_2 \subset \cdots \subset T_{r+1} = \underline{V}, $$  where $T_i$ is a vector bundle of rank $k_i$ and $\underline{V}$ is the trivial bundle of rank $n$. Let $U_i = T_i/ T_{i-1}$.

Let $E_\lambda$ be a Schur bundle on $F(k_1, \dots, k_r; n)$.  The partition $\lambda$ is a concatenation $\lambda = (\lambda_1|\cdots|\lambda_{r+1})$, where $\lambda_i$ has length $k_i - k_{i-1}$, and
$$E_\lambda = \mathbb{S}^{\lambda_1} U_1^* \otimes \mathbb{S}^{\lambda_2} U_2^* \otimes \cdots \otimes \mathbb{S}^{\lambda_{r+1}} U_{r+1}^* ,$$ where $\mathbb{S}^{\lambda_i}$ denotes the Schur functor of type $\lambda_i$.  Such bundles can be characterized as the pushforwards of line bundles on complete flag varieties.

Since $U_1, \dots, U_{r+1}$ give a filtration of the trivial bundle $\underline{V}$, the tensor product of their determinants is trivial by the Whitney formula. Hence, adding a constant integer to all the entries of $\lambda$ corresponds to tensoring the bundle $E_{\lambda}$ with a power of the trivial line bundle and does not change its isomorphism class.
 We will say two partitions $\lambda,\lambda'$ are equivalent if they differ by adding a constant to all the entries; equivalent partitions give isomorphic bundles.  There is also no harm in allowing negative integers in $\lambda$, since adding a large constant will make all entries positive.

For studying Ulrich bundles, it is important to know the effect of tensoring $E_\lambda$  by an ample line bundle $L_{a_1,\ldots,a_r}$.  By the Littlewood-Richardson rule, adding a constant sequence to $\lambda_i$ corresponds to tensoring the bundle by the determinant of $U_i$. Since the determinant of $U_i$ is $L_i \otimes L_{i-1}^{-1}$, we have
$$E_\lambda \te L_{a_1,\ldots,a_r} \cong E_{\lambda + \psi_{a_1,\ldots,a_r}},$$ where $  \psi_{a_1,\ldots,a_r}$ is the dominant weight corresponding to $L_{a_1,\ldots,a_r}$.

Let $\rho = (n-1, n-2, \dots, 1, 0)$. We say that $\lambda$ is {\em singular} if $\lambda + \rho$ has repeated entries. Otherwise, we say that $\lambda$ is {\em regular}. Let $s$ be the permutation in $\mathfrak{S}_n$ that lists the entries in $\lambda + \rho$ in decreasing order. Let $q(\lambda)$ denote the minimal length of the permutation $s$. The famous Borel-Weil-Bott Theorem computes the cohomology of line bundles on the complete flag variety, which allows us to calculate the cohomology of the Schur bundle $E_{\lambda}$ \cite{Weyman}.

\begin{theorem}[Borel-Weil-Bott]\label{thm-BWB}
Let $E_{\lambda}$ be a Schur bundle on $F(k_1, \dots, k_r; n)$.
\begin{enumerate}
\item If $\lambda$ is singular, then $H^i(E_{\lambda})=0$ for every $i$.
\item If $\lambda$ is regular, then $H^i(E_{\lambda}) = 0$ for $i \not= q(\lambda)$ and $H^{q(\lambda)}(E_{\lambda}) =  \mathbb{S}^{s(\lambda+\rho) - \rho} V^*$.
\end{enumerate}
\end{theorem}

We recall how to compute the dimension of a representation $\mathbb{S}^{\mu} V$ of $GL(n)$. Let $Y(\mu)$ be the Young diagram corresponding to the partition $\mu$. The hook-length $\hook(i,j)$ is the number of boxes in the same row to the right of the box $(i,j)$ and in the same column below the box $(i,j)$ (including the box itself).
The dimension of the representation $\mathbb{S}^{\mu} V$ is given by
$$\prod_{(i,j) \in Y(\mu)} \frac{n+j -i}{\hook(i,j)},$$ which is always nonzero.
In particular, the rank of the vector bundle $E_{\lambda}$ is given by
$$\prod_{s=1}^{r+1} \prod_{(i,j) \in Y(\lambda_s)} \frac{k_s-k_{s-1} +j -i}{\hook (i,j)}. $$

\section{Ulrich bundles and the combinatorial problem}\label{sec-comb}

In this section we recall basic facts on Ulrich bundles and phrase the condition that a Schur bundle $E_\lambda$ on the flag variety $F(k_1,\ldots,k_r;n)$ is Ulrich combinatorially.

\subsection{Ulrich bundles in general}Let $(X, \OO_X(1))$ be a polarized projective variety. A vector bundle $E$ on $X$ is {\em arithmetically Cohen Macaulay} if it is locally Cohen-Macaulay and $H_*^i (X, E) =0$ for every $1 \leq i \leq \dim(X) -1$. A vector bundle on $X$  is {\em initialized} if $H^0(X, E) \not= 0$, but $H^0(X, E(-1)) = 0$. If $E$ is ACM, then there exists a twist of $E$ which is initialized.

\begin{definition}
A vector bundle $E$ is an {\em Ulrich bundle} if $E$ is ACM and the initialized twist $E_{\init}$ of $E$ satisfies $h^0(X, E_{\init}) = \rk(E) \deg(X)$.
\end{definition}

\begin{remark}\label{rem-UlrichCrit}For our purposes, there is another useful criterion for determining when a bundle is Ulrich.  Let $X$ be a smooth, $N$-dimensional arithmetically Cohen-Macaulay projective variety and suppose $E$ is an initialized  bundle on $X$.   Then, by \cite[Proposition 2.1, Corollary 2.2]{ESW}, $E$ is Ulrich if and only if $H^i(E(-t))=0$ for all $i$ and $t\in [N] = \{1,\ldots,N\}$ (see \cite[Remark 2.6]{CostaMiroRoig1}).  It is well-known that flag varieties are arithmetically Cohen-Macaulay when embedded by any ample complete linear system (see \cite[Theorem 5]{Ramanathan}).
\end{remark}

\begin{remark}
Since for an initialized Ulrich bundle $h^0(E) = \rk(E) \deg(X)$, whenever there is a Schur bundle $E_\lambda$ which is Ulrich we obtain interesting combinatorial identities relating two hook-length formulas to the degree of the flag variety.  In the other direction, one could potentially find Ulrich bundles by searching for such identities; however, this seems difficult.
\end{remark}

\subsection{The combinatorial problem} Let $E_\lambda$ be a Schur bundle on $F(k_1,\ldots,k_r;n)$ corresponding to the concatenated sequence $\lambda = (\lambda_1|\cdots|\lambda_{r+1})$ of partitions.  The Borel-Weil-Bott Theorem \ref{thm-BWB} allows us to rephrase the condition that $E_\lambda$ is an initialized Ulrich bundle with respect to $\OO(1)$ purely in terms of the combinatorics of $\lambda$.  Recall $\lambda_i$ is a partition of length $k_i-k_{i-1}$, so $\lambda$ has total length $n$.

First, we write $P = \lambda +\rho$, where $\rho = (n-1,n-2,\ldots,1,0)$. For concreteness, we explicitly write out the partition $P$ as  $$P = (a_1^1, \dots, a_{l_1}^1 | a_1^2, \dots, a^2_{l_2}| \cdots | a_1^{r+1}, \dots, a_{l_{r+1}}^{r+1});$$ here $l_i  = k_i-k_{i-1}$.  Observe that $H^0(E_\lambda)\neq 0$ if and only if the sequence of entries of $P$ is strictly decreasing.

For any integer $t\geq 0$, we define a partition $$P(t) = (a_1^1 - tr,\ldots, a_{l_1}^1 - tr | a_1^2 - t(r-1),\ldots , a_{l_2}^2-t(r-1)|\cdots|a_1^{r+1},\ldots,a_{l_{r+1}}^{r+1})$$ by subtracting $t(r+1-i)$ from every entry in the $i$th part, and analogously define partitions $\lambda(t)$.  It is also convenient to view the entries of $P(t)$ as being functions of $t$ and define $$a_k^i(t) = a_k^i-t(r+1-i),$$ so that $$P(t) = (a_1^1(t),\ldots,a_{l_1}^1(t)|\cdots|a_1^{r+1}(t),\ldots,a_{l_{r+1}}^{r+1}(t)).$$ Then $E_{\lambda}(-t) = E_{\lambda(t)}$ has $H^i(E_{\lambda}(-t))=0$ for all $i$ precisely when there is a repeated entry in $P(t)$.  Therefore $E_\lambda$ is an initialized Ulrich bundle if and only if $P(t)$ has a repeated entry for all integers $t\in [N]$, where $N = \dim F(k_1,\ldots,k_r;n)$ is the dimension of the flag variety.

We now make the central combinatorial definitions of the paper.

\begin{definition}
Let the partition $$P = (a_1^1,\ldots,a_{l_1}^1|a_1^2,\ldots,a_{l_2}^2|\cdots|a_1^{r+1},\ldots, a_{l_{r+1}}^{r+1})$$ consist of a decreasing sequence of integers arranged into $r+1$ blocks, with $l_i>0$ for all $i$.  The \emph{type} of the partition is $(l_1,\ldots,l_{r+1})$.  The \emph{dimension} of the partition is $N = \sum_{k<h} l_kl_h$.

The partition $P$ is \emph{Ulrich}  if for all $t\in [N]=\{1,\ldots,N\}$ the partition $P(t)$ has a repeated entry.

Two partitions are \emph{equivalent} if they differ by adding a constant to all the entries.  When we count Ulrich partitions or talk about uniqueness, it is understood that we always do so modulo equivalence.
\end{definition}

We make the correspondence between the algebro-geometric and combinatorial problems explicit.

\begin{proposition}
There is a bijective correspondence $\lambda \mapsto \lambda + \rho = P$ between equivalence classes of partitions $\lambda$ such that the Schur bundle $E_\lambda$ on $F(k_1,\ldots,k_r;n)$ is Ulrich with respect to $\OO(1)$ and equivalence classes of Ulrich partitions $P$ of type $(k_1,k_2-k_1,\ldots,k_r-k_{r-1},n-k_r)$.
\end{proposition}

The main question which we study from this point forward is to determine the possible types of Ulrich partitions, and, when there exists an Ulrich partition of a given type, to classify all such partitions up to equivalence.

\begin{remark}
The dimension $N$ of a partition $P$ is precisely the number of pairs of entries coming from different blocks.  For there to be a repeated entry in $P(t)$ at every time $t\in [N]$, it is therefore necessary and sufficient that every pair of entries meets at some time $t\in [N]$ and that no two pairs of entries meet at the same time.
\end{remark}

\begin{remark}
There is no harm in considering lists $P$ where some of the entries are negative since the notion of Ulrich partition is invariant under equivalence.  It will at times also be convenient to change the definition of $P(t)$ slightly by subtracting a different consecutive range of integers from each part.  For example, in the two-step flag variety case where $P$ has $3$ parts it will be convenient to subtract $1,0,-1$ from the three parts, so that the middle block is constant with time.  We will make our choice of normalization clear in each section as needed.
\end{remark}

We now provide a simple example illustrating the combinatorial definitions and the conversion to the algebro-geometric language.

\begin{example}\label{ex-combToAG}
Consider the partition $P = (5|3,-1,-2,-4|{-5})= (a|b_1,b_2,b_3,b_4|c)$ of type $(1,4,1)$.  For convenience, we normalize $P(t)$ by subtracting $1,0,-1$ at each time.  The dimension is $N=9$, and this partition is Ulrich.  To visualize the Ulrich condition, we draw a \emph{time evolution diagram} showing the positions of the entries of $P(t)$ in Figure \ref{fig-141}.  Along the first row we place the entries $a,b,b,b,b,c$ in the corresponding positions; we provide numerical positions above some of the entries for reference.  For each time $t\in [N]$ we similarly place the entries of $P(t)$ in the corresponding row.  When two entries intersect, we put a box around them to highlight the intersection.  The Ulrich condition means that for every time $t\in [N]$ there is exactly one box in the corresponding row.  We also display $P(N+1)$ at time $N+1$, which is relevant to the computation of the \emph{dual partition} $P^*$; see \S \ref{ssec-symmetryDuality}.

\begin{figure}[htbp]
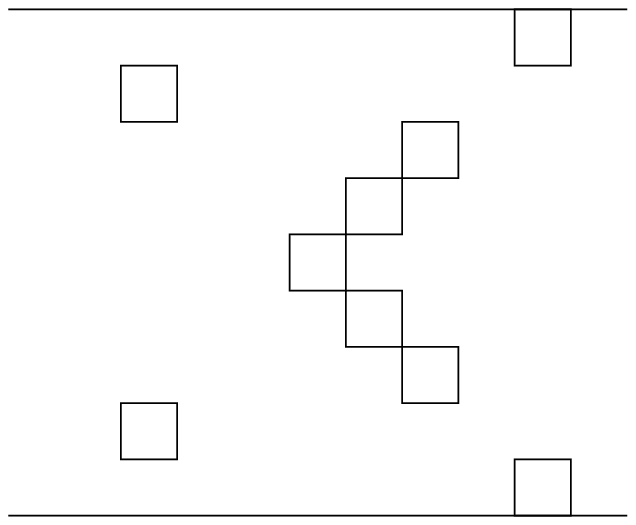
\caption{Time evolution diagram of the Ulrich partition $(5|3,-1,-2,-4|{-5})$ of type $(1,4,1)$. See Example \ref{ex-combToAG}.}\label{fig-141}
\end{figure}

Replacing $P$ by the equivalent partition $(11|9,5,4,2|1)$ and putting $\lambda = P - \rho$, we have $\lambda = (6|5,2,2,1|1) = (\lambda_1|\lambda_2|\lambda_3)$.  The corresponding initialized Ulrich bundle on $F(1,5;6)$ is $$E_\lambda = \mathbb{S}^{\lambda_1} U_1^* \otimes \mathbb{S}^{\lambda_2} U_2^* \otimes \mathbb{S}^{\lambda_{3}} U_3^*.$$ A straightforward computation with the hook-length formula shows that $\rk(E_\lambda) = 70$ and $h^0(E_\lambda) = 17640$.  This is consistent with the calculation $\deg F(1,5;6) = 252$ for the embedding of $F(1,5;6)$ by $\OO(1)$.
\end{example}

\subsection{Symmetry and duality}\label{ssec-symmetryDuality} There are certain symmetries satisfied by Ulrich partitions which we will exploit throughout the paper.  Suppose $P = (a_1^1,\ldots,a_{l_1}^1|\cdots | a_1^{r+1},\ldots,a_{l_{r+1}}^{r+1})$ is a partition of type $(l_1,\ldots,l_{r+1})$.  Multiplying all the entries of $P$ by $-1$ and reversing the order in which they are written, we obtain the \emph{symmetric} partition $$P^s=(-a_{l_{r+1}}^{r+1},\ldots,-a_{1}^{r+1}|\cdots| {-a_{l_1}^1},\ldots,-a_1^1)$$  of the reversed type $(l_{r+1},\ldots,l_1)$.  It is clear that if $P$ is Ulrich then $P^s$ is also Ulrich.  The time evolution diagram of $P^s$ is flipped about a vertical axis.  It is often convenient to add a constant to all the entries of the symmetric partition to keep some entries fixed; there is no harm in this since we only care about Ulrich partitions up to equivalence.

There is another slightly less trivial symmetry which we call \emph{duality}. Starting from an Ulrich partition $P$ we define the partition $$P^* = (a_1^{r+1}(N+1),\ldots,a_{l_{r+1}}^{r+1}(N+1)|\cdots | a_1^1(N+1),\ldots , a_{l_1}^1(N+1)).$$ The integers which appear in $P^*$ are precisely the integers in $P(N+1)$; we have just reordered them to be decreasing.  In the time evolution diagram for $P$, $P^*$ can be read off from the final row below the horizontal line.  Then $P^*$ is Ulrich of type $(l_{r+1},\ldots, l_1)$; its time evolution diagram is obtained from the diagram for $P$ by flipping about a horizontal axis $t = (N+1)/2$.

Both the symmetries we have described reverse the type of the Ulrich partition, so change the type unless it is symmetric.  The partition $(P^s)^*$ has the same type as $P$, and we call this partition the \emph{symmetric dual}.  In particular, if there is a unique Ulrich partition of a given type then it is equal to its symmetric dual.

\begin{example}
The symmetric partition obtained from $(5|3,-1,-2,-4|{-5})$  is $(5|4,2,1,-3|{-5})$.  Both of these partitions are self-dual.
\end{example}

\begin{remark}
Suppose the flag variety $F(k_1,\ldots,k_r;n)$ carries an Ulrich Schur bundle $E_{\lambda}$, and let $P = \lambda +\rho$.  Then $E_{P^s-\rho}$ and $E_{P^* -\rho}$ are Ulrich bundles on the dual flag variety $F(n-k_r,\ldots,n-k_1;n)$, while $E_{(P^{s})^*-\rho}$ is an Ulrich bundle on the original flag variety.
\end{remark}

\section{Ulrich bundles on flag varieties with three or more steps}\label{sec-higherStep}
In this section, we prove the following theorem.

\begin{theorem}\label{thm-higherStep}
A flag variety of three or more steps does not have a Schur bundle $E_\lambda$ which is  Ulrich with respect to the line bundle $\cO(1)$.
\end{theorem}

As a special case, we obtain the following corollary.

\begin{corollary}
The complete flag variety $F(1, 2, 3, \dots, n-1; n)$ does not admit a Schur bundle which is Ulrich for $n \geq 4$.
\end{corollary}

We will split the proof into three parts. We first analyze flag varieties with at least five steps. The three and four step flag varieties require separate (and more intricate) arguments. We begin by deriving simple congruences among the entries of an Ulrich partition.

\begin{lemma}\label{lem-congruence}
Let $(a^1_i|a^2_i|\cdots|a^{s}_i)$ be an Ulrich partition. Suppose $1 \leq j \leq k \leq s$. Then all the entries $a^j_i$ in the $j$-block are congruent to all the entries $a^k_i$ in the $k$-block modulo $k-j$.
\end{lemma}
\begin{proof}
We can normalize the evolution of the Ulrich partition so that we keep all the entries $a^j_i$ in the $j$-block  fixed and add $k-j$ to all the entries $a^k_i$ in the $k$-block. This preserves the congruence classes of the $a^j_i$ and $a^k_i$ modulo $k-j$.  In order for them to collide at some time, they must all have the same congruence class.
\end{proof}

 \subsection{Flag varieties with at least five steps}
\begin{proposition}\label{prop-fiveStep}
There are no Ulrich partitions with $6$ or more blocks.
\end{proposition}

\begin{proof}
Suppose there is an Ulrich partition $(a^1_1,\ldots,a^1_{l_1}|\cdots|a^s_1,\ldots,a^s_{l_s})$ with $s \geq 6$. Since $s \geq 6$, we have that  $\max\{k-1,s-k\} \geq 3$ for every $1 \leq k \leq s$. By Lemma \ref{lem-congruence}, we conclude that the entries $a^k_i$ in the $k$-block all have the same congruence class modulo 6. In particular, they all have the same parity.

At time $t=1$, two adjacent blocks must meet; say $a_{l_k}^k(1) = a_{1}^{k+1}(1)$.  This forces $a_{l_k}^k$ and $a_1^{k+1}$ to have opposite parities.
If $k-j$ is even, then the entries $a^k_i(t)$ and $a^j_h(t)$ in the $k$- and $j$-blocks have the same parities at all times $t$. On the other hand, if $k-j$ is odd, then the entries in the $k$- and $j$-blocks have the same parity at odd times and opposite parities at even times $t$. Consequently, at even times the intersections must occur between entries in nonadjacent blocks.

Suppose that at time $t=2$, $a^j_{l_j}$ and $a^{j+2}_1$ collide.  Normalize the evolution of the sequence so that $1,0,-1$ is subtracted from the $j$-, $(j+1)$-, and $(j+2)$-blocks, respectively.
Then, up to switching their order, at times $t=1,3$, the collisions must be $a^j_{l_j}a^{j+1}_1$ and $a^{j+1}_1a^{j+2}_1$ depending on the sequence at time $t=0$. We observe that the possible positions of the entries at time $t=0$ are
$$ \cdots \times \times \  a^j_{l_j} \ a^{j+1}_1 \ \times \times \ a^{j+2}_1 \ \cdots \quad \mbox{or} \quad  \cdots \times \times \  a^j_{l_j} \  \times \times \ a^{j+1}_1 \  a^{j+2}_1 \cdots,$$ where we denote the positions unoccupied by an entry by $\times$.  In particular, after normalizing the positions, the triple $(a_{l_j}^j,a_1^{j+1},a_1^{j+2})$ either equals  $(1,0,-3)$ or $(3,0,-1)$.  Furthermore, there is a unique entry in the $(j+1)$-part: $l_{j+1}=1$.

At time $t=4$, the collision again must be between entries from nonadjacent blocks. For $k \not= j$, if $a^k_{l_k}(4) = a^{k+2}_1(4)$, then  $a^k_{l_k}$ passes through the entries in the $(k+1)$-block without intersecting them, contradicting that the partition is Ulrich.  Hence, the only possible collisions at time $t=4$ are either $a^j_{l_j} a^{j+2}_2 $  or $a^j_{l_j-1} a^{j+2}_1$. However, these are both impossible since they would give a separation of only four between either $a^j_{l_j}$ and $a^{j}_{l_j-1}$ or $a^{j+2}_1$ and $a^{j+2}_2$. This contradicts the fact that entries in each block are congruent modulo 6.

We conclude that there cannot be any Ulrich partitions with $s \geq 6$.
\end{proof}

\subsection{Three-step flag varieties} In this subsection, we analyze Ulrich bundles on three-step flag varieties.

\begin{proposition}\label{prop-threeStep}
There are no Ulrich partitions with $4$ blocks.
\end{proposition}
\begin{proof}
Let $(a_1,\dots, a_j|b_1, \dots, b_k|c_1, \dots, c_l|d_1, \dots, d_m)$ be an Ulrich partition for a three-step flag variety. By Lemma \ref{lem-congruence}, the parities of $a_i(t)$ and $c_h(t)$ (respectively, $b_i(t)$ and $d_h(t)$) are equal at all times.  On the other hand, by the same argument as in the proof of Proposition \ref{prop-fiveStep}, entries in adjacent blocks have the same parity at odd times $t$ and opposite parities at even times $t$. Consequently, when $t$ is even, the only collisions  can be  $ac$ or  $bd$ collisions. Since the last intersection at time $t = N$ has to be between entries from adjacent blocks, we conclude that  the dimension $N$ is odd.

By symmetry (see \S \ref{ssec-symmetryDuality}), we can assume that at time  $t=2$, the collision is between $a_j$ and $c_1$.
We normalize the evolution of the partition so that we subtract $2,1,0,-1$, respectively, from the $A$-, $B$-, $C$-, and $D$-blocks. Consequently,  the first three collisions are
$$a_jb_1,\  a_j c_1, \ b_1 c_1 \quad \mbox{or} \quad b_1 c_1,\  a_j  c_1, \ a_j  b_1.$$ The corresponding sequences at time $t=0$ look like
$$\cdots a_j \  b_1 \times \times \  c_1 \ \times \times   \cdots \quad \mbox{respectively} \quad \cdots a_j \ \times \times  \ b_1 \ c_1 \ \times \times  \cdots $$
In either case, we conclude that the length $k$ of the  $B$-block is one since by time $t=3$, $a_j$ must have passed through the entire $B$-block. From now on we denote the unique entry in the $B$-block by $b$.

Similarly, $t=N-1$ is an even time.  Hence, the only possible collisions are $ac$ or  $bd$ collisions. If the collision is between $b$ and $d_m$, then the last three collisions at times $N-2, N-1$ and $N$ are
$$ c_1  d_m,\  b  d_m, \ b  c_1 \quad \mbox{or} \quad b   c_1, \ b  d_m, \  d_m  c_1.$$ In particular, we conclude that the length of the $C$-block is also one. Denote the lengths of the $A$- and $D$-blocks by $j$ and $m$, respectively. The dimension of the flag variety is $$N=2j + jm + 2m + 1.$$ Since the collisions at even times are  $ac$ or $bd$ collisions, the dimension can be at most one more than twice the total number $j+m$ of possible $ac$ and $bd$ collisions.  We, therefore,  have the inequality
$$2j + jm + 2m \leq 2(j+m).$$ This implies that $jm \leq 0,$ which is absurd.

We conclude that at time $t= N-1$, the intersection must be between $a_1$ and $c_l$. Hence, the last three intersections at times $N-2,N-1,N$ are one of
 $$a_1  b,\  a_1  c_l, \  b  c_l \quad \mbox{or} \quad b   c_l,\  a_1  c_l, \ a_1  b.$$ Normalize the position of the sequence so that $b=y$   and $d_1=-y$. Then $c_1$ equals either $y-1$ or $y -3$. Furthermore, there cannot be any $c$'s in positions $-y+1$, $-y+2$ or $-y+3$ since  $d_1$ is at these positions at times $1,2,3$ when there are other collisions. In either case, $d_1$ intersects $c_1$ at a time $t_0 \geq 2y -3$. On the other hand, $b$ intersects $c_l$ at some time $t_1< 2y-3$. Hence, the intersection $bc_l$ happens before the intersection $c_1d_1$. This is a contradiction, since $bc_l$ is one of the last three intersections and $c_1d_1$ is not. We conclude that there are no Ulrich partitions for three step flag varieties with respect to $\cO(1)$.
\end{proof}

\subsection{Four-step flag varieties} Finally, we analyze Ulrich bundles on four-step flag varieties. This is the most intricate case.

\begin{proposition} \label{prop-fourStep}
There are no Ulrich partitions with $5$ blocks.
\end{proposition}
\begin{proof}

Let $(a_1, \dots, a_j|b_1, \dots, b_k|c_1, \dots, c_l|d_1, \dots, d_m|e_1, \dots, e_p)$ be an Ulrich partition for a four-step flag variety. We can normalize the evolution so that at each time we subtract $2,1,0,-1,-2$ from the entries in the $A$-, $B$, $C$, $D$, $E$-blocks, respectively. By parity, at time $t=2$ we must have an $ac$, $bd$ or $ce$ collision. Suppose we have a $bd$ collision. As before, the length $l$ of the $C$-block is $1$, so we denote its entry by $c$.  Then the first three collisions are one of
$$b_k  c,  \ \ b_k  d_1,  \ \  c  d_1 \quad \mbox{or} \quad c  d_1, \ \  b_k   d_1, \ \ b_k  c$$ corresponding to the following sequences at time $t=0$
$$ \cdots b_k \ c \ \times \ \times \ d_1 \cdots \quad \mbox{or} \quad \cdots b_k \ \times\ \times \ c \ d_1 \cdots.$$
By Lemma \ref{lem-congruence},  the entries in the $B$-  (respectively, $D$-) block are the same modulo 6, so the distances between consecutive entries are at least 6. Similarly, the consecutive entries in the $A$- (respectively, $E$-) blocks are at least 12 apart. At time $t=4$, the collision cannot be a $bd$ collision since then the distance between two consecutive entries in either the $B$- or $D$- blocks would  be only 4 instead of 6. Hence, the collision at time $t=4$ must be either $a_jc$ or $ce_1$. However, neither of these are possible. In the first case, $a_j$ never collides with $d_1$ and in the latter case $e_1$ never collides with $b_k$. We conclude that the first collision cannot be a $bd$ collision.

The collision at time $t=2$ must be an $ac$ or $ce$ collision. Without loss of generality, we may assume it is $ac$. Therefore, there is only one $b$ and $b$ collides with $c_1$ at time $1$ or $3$.  As in the  case of three-step flag varieties, the dimension $N$  of the flag variety must be odd.   Hence, at time $N-1$ we must have an $ac$, $bd$ or $ce$ collision. By symmetry, we see that the intersection at time $t=N-1$ cannot be a $bd$ collision, hence must be an $ac$ or $ce$ collision. If it were an $ac$ collision, then $b$ would have to intersect $c_l$ at time $N$ or $N-2$. Since the time it takes for $d_1$ to intersect $c_1$ is longer than the time it takes $b$ to intersect $c_l$ (as in the proof of Proposition \ref{prop-threeStep}), we conclude that the intersection at time $t = N-1$ must be $c_1e_p$ and there can be only one entry in the $D$-block.

We have so far concluded that the sequence looks like $(a_1, \dots, a_j|b|c_1, \dots, c_l|d|e_1, \dots, e_p)$. The first three intersections are one of
$$\mbox{Case I:} \  a_j   b ,\ \  a_j   c_1, \ \ b   c_1 \quad \mbox{or} \quad \mbox{Case II:} \ b  c_1, \ \  a_j  c_1, \ \  a_j  b$$ and the last three intersections are one of
$$\mbox{Case A:} \ d  e_p, \ \  c_1  e_p, \ \ c_1  d \quad \mbox{or} \quad \mbox{Case B:} \ c_1  d, \ \ c_1  e_p, \ \  d   e_p.$$

For the rest of the argument, we normalize positions so that $b=y$ and $d=-y$. Then at time $y$ we have the $bd$ collision at position $0$.

We first show that case IB is impossible.  We compute $c_1 = y-3$, and the equality $$c_1 =c_1(N-2) = d(N-2)=d+N-2$$ gives $N = 2y-1$.  Similarly, from $e_p(N) = d(N)$ we find $e_p = -3y+1$.  As $a_j = y+1$, we conclude the intersection $a_je_p$ happens at time $y$.  This contradicts that $b$ and $d$ already collide at time $y$.  Symmetrically, case IIA is also impossible since the intersection $a_1e_1$ happens at time $y$.

We conclude that we must be in Case IA or Case IIB. These cases are the same under duality (see \S\ref{ssec-symmetryDuality}). Consequently, it suffices to analyze Case IA.

In Case IA, we first prove that at time $t=0$ the sequence must look like
\begin{equation}\tag{$\ast$} \cdots a_j \ b \ \times  \times \ c_1 \ \times \times \times  \  c_2  \ \times \times \times \dots  \times \times \times \times \times \ d \ \times  \times  \times  \times  \times \cdots ,\end{equation}
where $\times$  denotes a position unoccupied by any entries $a,b,c,d,e$. At time $t=1,2,3$, the collisions are $a_j b, a_j c_1$ and $b c_1$. Since at these times $d$ is at positions $-y+1, -y+2, -y+3$, there cannot be any entries of the $C$-block in these positions.  Furthermore, there cannot be any entries of the  $C$-block in position $y-5$ because $a_j$ is at position $y-5$ at time $t=3$ while the collision $bc_1$ occurs. By Lemma \ref{lem-congruence}, the distance between the entries in the $C$-block are even. Hence, there cannot be any entries of the $C$-block  at positions $y-2\alpha$ for any integer $\alpha$. At time $t=4$, a collision of the form $ac$, $bd$, or $ce$ must occur. Since $e_1$ has not passed through $d$ and $b(4)>a_j(4)>d(4)$, the only possible collision  is $a_jc_2$. Hence,  $c_2=y-7$.

As in case IB, we now compute that the dimension of the flag variety is $N=2y-3$. At times $N-2$, $N-1$, and $N$, the collisions are specified and do not involve $b$. At these times $b$ is at positions $-y+5, -y+4, -y+3$. Hence, we conclude that there cannot be any entries of the $C$-block at positions $-y+1, \ldots, -y+5$. Since $e_p(N-2) = d(N-2)$ we have $e_p = -3y+5$.  The $c_2d$ collision happens at time $2y-7$, and $e_p(2y-7) = y-9$, so there cannot be a $c$ at position $y-9$ either. Combining all these claims, we recover the claimed shape of the sequence at time $t=0$.

Next consider the position of $e_1$.  Since there are intersections for times $1\leq t \leq 4$, we find $e_1\leq -y-5$.  To complete the claim that the entries of the partition are as in $(\ast)$, we only need to show that in fact $e_1< -y-5$.  So suppose $e_1 = -y-5$.  Since $c_2 = y-7$, the intersection  $c_2e_1$ occurs at time $y-1$.  However, since $a_j = y+1$ and $e_p = -3y+5$, the intersection $a_j e_p$ also occurs at time $y-1$.  This contradiction shows $e_1< -y-5$, and establishes the pattern $(\ast)$ for the entries.

We finally obtain a contradiction by showing there is no possible collision at time $t=5$.  The collision cannot be a $de$ collision since $e_1<-y-5$. Since nothing has passed through the $d$ before time $5$, the time $t=5$ intersection cannot involve an $e$.

Since there are no $c$ entries at position $-y+5$, the collision cannot be of the form $cd$. Since there are no $c$ entries at position $y-5$, the collision cannot be of the form $bc$. Since $b$ has not passed though $c_2$, the collision cannot be of the form $bd$. Since there are no $c$ entries in position $y-9$, the collision cannot be $a_j c$. Since the entries in the $A$-block are at least 12 apart, the collision cannot be $a_{j-1}  c_1$. We conclude that the collision is not of the form $ac$.  For the same reason, the collision cannot be of the form $a b$.

This only leaves the possibility of the collision being between $a_1$ and $d$. This implies $y=7$, so $c_2=0$ and there is a triple intersection by $b$, $c_2$, and $d$ at time $7$. This concludes the proof that there cannot be an Ulrich partition for four-step flag varieties.
\end{proof}

Theorem \ref{thm-higherStep} follows immediately from Propositions \ref{prop-fiveStep}, \ref{prop-threeStep} and \ref{prop-fourStep}.

\section{Results on two-step flag varieties}\label{sec-overview2step}

The rest of the paper is entirely combinatorial and will focus on the classification of Ulrich partitions with three parts,  corresponding to   Ulrich bundles $E_\lambda$ on two-step flag varieties with respect to $\OO(1)$.  In this section, we state our classification results, and outline the plan for the remainder of the paper.

From now on, $F(p,q;v)$ denotes a two-step flag variety.
Let $(\alpha,\beta,\gamma)= (p, q-p, v-q)$ be the type of a partition.  We present two different flavors of classification results.  We classify Ulrich partitions where $\beta \leq 2$ but $\alpha$ and $\gamma$ are arbitrary.  We then also classify Ulrich partitions where $\alpha \leq 2$ and $\gamma \leq 2$ but $\beta$ is arbitrary. We begin by analyzing the simplest case of $(1,n,1)$ partitions, which correspond to Schur bundles on $F(1, n; n+1)$.

\begin{theorem}\label{thm-1n1}
There are $2^n$ Ulrich partitions of type $(1,n,1)$ for any $n\geq 1$ (up to equivalence).
\end{theorem}

\begin{proof}Indeed, suppose $(a|b_1,\ldots,b_n|c)$ is Ulrich.  Since $a$ and $c$ have the same parity, we may normalize $a = y+1$, $c = -y-1$ for some integer $y\geq 0$.  At time $y+1$, we have the $ac$ intersection.  Before time $y+1$, every intersection is of type $ab$ or $bc$.  Thus, for every position $p\in [y]$, exactly one of $p$ or $-p$ is in $B=\{b_1,\ldots,b_n\}$.  We conclude $y = n$, and if $B'\subset [n]$ is any subset then there is a unique Ulrich partition $(n+1|B|{-n-1})$ with $B\cap [n] = B'$.  Thus equivalence classes of Ulrich partitions of type $(1,n,1)$ are in bijective correspondence with subsets of $[n]$.
\end{proof}

\begin{example}
The Ulrich partition $(5|3,-1,-2,-4|{-5})$ corresponding to $B' = \{3\}\subset [4]$ in the previous construction was already studied in Example \ref{ex-combToAG}.
\end{example}

We next describe the small $\beta$ case, where $\beta\leq 2$.  There are new one- and two-parameter infinite families of examples; let $m_1,m_2\geq 0$ be nonnegative integers and define $$k_i = \sum_{j=0}^{m_i} 4^j = \frac{4^{m_i+1}-1}{3}.$$ These sums of powers of $4$ will somewhat surprisingly appear in many examples.  We prove the next result in \S \ref{sec-sumset} by recasting it as a problem about sumsets.

\begin{theorem}\label{thm-beta1intro}
For every $m_1\geq 0$, there is a unique Ulrich partition of type $(2,1,k_1)$. Up to symmetry, any Ulrich partition of type $(\alpha,1,\gamma)$ is of this form.
\end{theorem}

The $\beta = 2$ case is more difficult, and occupies \S \ref{sec-beta2}.

 \begin{theorem}\label{thm-beta2intro}
There are Ulrich partitions of the following types:
\begin{enumerate}
\item\label{case222} $(2,2,2)$,
\item\label{case223} $(2,2,3)$,
\item\label{case12k} $(1,2,k_1)$ for any $m_1\geq 0$, and
\item\label{case12kk} $(1,2,k_1+k_2)$ for any $m_1,m_2\geq 0$.
\end{enumerate}
Unless the partition is of type $(1,2,2)$ or $(2,2,2)$, it is unique up to taking the symmetric dual.

Up to symmetry, any Ulrich partition of type $(\alpha,2,\gamma)$ is one of the above examples.
 \end{theorem}

We next turn our attention to cases where $\alpha$ and $\gamma$ are small and $\beta$ is arbitrary.  We first classify all partitions of type $(1,n,2)$ in \S\ref{sec-1n2}.

\begin{example}\label{example-1n2}
There is a fundamental example of this type, given by the partition  $$(a|b_1,\ldots,b_n|c_1,c_2) = (2n+1|n-1,n-2,\ldots,1,0|{-1},-2n-3).$$
Then
\begin{itemize}
\item for times $t\in [1,n]$, the $c_1$ entry meets the $B$-block.
\item At time $t = n+1$, $a$ meets $c_1$.
\item For times $t\in [n+2,2n+1]$, $a$ meets the $B$-block.
\item When $t=2n+2$, $a$ meets $c_2$, and
\item for $t\in [2n+3,3n+2]$, $c_2$ meets the $B$-block.
\end{itemize}
Therefore the partition is Ulrich.
In Figure \ref{fig-132} we present the time evolution diagram of the Ulrich partition $(7|2,1,0|{-1},-9)$ of type $(1,3,2)$.
\begin{figure}[htbp]
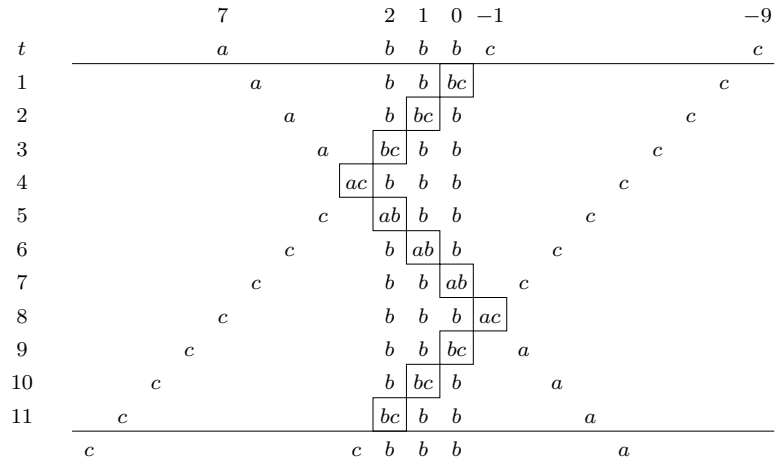
\caption{Time evolution diagram of the Ulrich partition $(7|2,1,0|{-1},-9)$ of type $(1,3,2)$.}\label{fig-132}
\end{figure}
\end{example}

\begin{theorem}\label{thm-1n2intro}
There is a bijective correspondence between Ulrich partitions of type $(1,n,2)$ and decompositions $$n = 2km+m-1 \qquad (k\geq 0, m>0)$$ of the integer $n$.
\end{theorem}

 The fundamental example of a type $(1,n,2)$ partition corresponds to the trivial decomposition of $n$ given by $k=0$ and $m=n+1$.
For example, when $n=4$, $k=2$, and $m=1$ we will obtain a partition of type $(1,4,2)$ different from the fundamental  example.

Finally, we classify all partitions of type $(2,n,2)$ in \S \ref{sec-2n2}.

\begin{theorem}\label{thm-2n2intro}
If there is an Ulrich partition of type $(2,n,2)$, then $n$ is even.  If $n$ is even, there are exactly two partitions of this type, and they are symmetric to each other.
\end{theorem}

\subsection{Conjectures}
Our classification of Ulrich partitions for two-step flag varieties remains incomplete.  However, we conjecture that we have found all the examples:
\begin{conjecture}\label{conj-fullList}
If $(\alpha,\beta,\gamma)$ is the type of an Ulrich partition, then up to symmetry this type is one of
 $$(1,n,1),\quad
(1,n,2),\quad
(2,2n,2),\quad
(2,2,3),\quad
(2,1,k_1),\quad
(1,2,k_1),\quad\mbox{or}\quad
(1,2,k_1+k_2)$$
where $n$ denotes a positive integer and $k_i$ is a number of the form $(4^{m_i+1}-1)/3$ for some $m_i\geq 0$.
\end{conjecture}

Our partial classification implies that Conjecture \ref{conj-fullList} is equivalent to the following a priori weaker statement.

\begin{conjecture}\label{conj-simple}
There are no Ulrich partitions of type $(\alpha,\beta,\gamma)$ if $\beta\geq 3$ and $\gamma \geq 3$.
\end{conjecture}

Indeed, suppose Conjecture \ref{conj-simple} is true.  If $\beta\leq 2$, then Theorems \ref{thm-beta1intro} and \ref{thm-beta2intro} completely classify Ulrich partitions of type $(\alpha,\beta,\gamma)$.  So suppose $\beta \geq 3$.  If $\alpha \leq 2$ and $\gamma\leq 2$, then Theorems \ref{thm-1n1}, \ref{thm-1n2intro}, and \ref{thm-2n2intro} completely classify Ulrich partitions of type $(\alpha,\beta,\gamma)$.  If instead either $\alpha \geq 3$ or $\gamma\geq 3$, then by symmetry we may assume $\gamma\geq 3$ and therefore there are no Ulrich partitions of type $(\alpha,\beta,\gamma)$.  Therefore Conjecture \ref{conj-simple} implies Conjecture \ref{conj-fullList}.
 \begin{remark}
 As evidence for Conjecture \ref{conj-simple}, we performed a brute-force computer search to verify that if there is an Ulrich partition of type $(\alpha,\beta,\gamma)$ with $\beta\geq 3$ and $\gamma \geq 3$, then $\alpha + \beta +\gamma \geq 13$.

 We note that Conjecture \ref{conj-simple} is open even in the special case $\alpha=1$.
 \end{remark}

\section{Sumsets and Ulrich partitions of type $(\alpha,1,\gamma)$}\label{sec-sumset}

In this section we produce a new class of examples of Ulrich partitions and classify all the types $(\alpha,\beta,\gamma)$ of an Ulrich partition where $\beta = 1$.

\begin{theorem}\label{thm-sumset}
For every $m\geq 0$, there is a unique Ulrich partition of type $(2,1,k)$, where $k = (4^{m+1}-1)/3$. Up to symmetry, any Ulrich partition of type $(\alpha,1,\gamma)$ is of this form.
\end{theorem}
\begin{proof}
We first rephrase our combinatorial problem to make it more tractable in this case.  Given a partition $$(a_1,\ldots,a_\alpha|b|c_1,\ldots,c_\gamma)$$ we first normalize $b=0$ (so the $c_j$ are all negative).  The entry $a_i$ meets $b=0$ at time $t= a_i$ and meets $c_j$ at time $\frac{1}{2}(a_i-c_j)$, while $c_j$ meets $b$ at time $t=-c_j$.  In order for the $a$'s to meet the $c$'s at integral times, they must all have the same parity.  Furthermore, they must all be odd, for otherwise at time $t=1$ no two entries would meet.

Let $A =\{a_1,\ldots,a_\alpha\}$ and $C = \{-c_1,\ldots,-c_\gamma\}$.  Then the Ulrich condition says precisely that the sumset $A+C:= \{a+c:a\in A,c\in C\}$ has size $\alpha\gamma$ and the set $[N]$ of times is a disjoint union $$ [N] = A\amalg C \amalg \tfrac{1}{2}(A+C).$$

Equivalently, by subtracting $1$ from every element of $A$ and $C$ and putting $N' = N-1$, we see that the existence of an Ulrich partition of type $(\alpha,1,\gamma)$ is equivalent to the existence of subsets of even numbers $A',C' \subset [0,N']$ of sizes $\alpha$ and $\gamma$ such that $[0,N']$ is a disjoint union $$[0,N'] = A'\amalg C' \amalg \tfrac{1}{2}(A'+C').$$ The numerics force $A'+C'$ to have size $\alpha\gamma$---no repetition can occur in the sumset.  This decomposition is the problem we will study for the rest of the proof.

We now claim that if $(\alpha,1,\gamma)$ is the type of an Ulrich partition then at least one of $\alpha,\gamma$ is $2$.  Suppose $A',C'\subset [0,N']$ correspond to such an Ulrich partition, and put $S' = \frac{1}{2}(A'+C')$.  Since $A'$ and $C'$ are disjoint, we have $0\notin S'$, and therefore either $0\in A'$ or $0\in C'$.  By symmetry we may assume $0\in A'$.  We must have $1\in S'$ since $A',C'$ consist of even numbers, and this forces $2\in C'$.

The number $N'$ must be even since $A'$ and $C'$ consist of even numbers no larger than $N'$.  Then $N'$ must actually be in $A'$ or $C'$.  If $N'\in C'$ then since $N'-1$ must be in $S'$ we have to have $N'-2\in A'$.  This is a contradiction since $0+N' = 2+(N'-2)$ gives a repetition in the sumset $A'+C'$.  We conclude $N'\in A'$ and similarly $N'-2\in C'$.

We next show by induction that if $k$ is an integer with $k\geq 0$ and $4k+2<N'$ then $4k+2\in C'$.  We already know $2\in C'$.  Suppose $4k+2\in C'$ for all $k\in [0,k_0)$.  Then $4k+2 \notin A'$ for all $k\in [0,k_0)$.  Furthermore, $4k+4\notin A'$ for all $k\in [0,k_0)$, for if $4k+4\in A'$ then $(4k+4)+(N'-2) = N'+(4k+2)$ is a repetition in the sumset.  Thus $A'$ contains no numbers in the interval $(0,4k_0]$. The odd number $2k_0+1$ must be in $S'$, but since every element of $C'$ is $\geq 2$ while the only nonzero elements of $A'$ are at least $4k_0+2$ the only possibility is that $4k_0+2\in C'$.  By induction, $C'$ must contain all numbers in $(0,N')$ which are $2 \pmod{4}$.  This argument also shows that $A' = \{0,N'\}$, i.e. $\alpha=2$.

Suppose $N'$ is not divisible by $4$.  Reversing the argument in the previous paragraph, we see that $C'$ must consist precisely of all the even numbers in $(0,N')$.  In particular, since $N'>4$ we have $4\in C'$, so $2\in S'$; this contradicts that $2\in C'$.  Therefore $4$ divides $N'$.

At this point we conclude that the sets $A',C'\subset [0,N']$ must have a recursive structure.  Let $$N'' = N'/4 \qquad A'' = \{0,N''\} \qquad C''=\{c:4c\in C'\} \qquad C''' = \{c\in [0,N']:c\equiv 2 \pmod 4\},$$ so that $C' = 4C'' \cup C'''$.  Then $$C''' \cup \tfrac{1}{2}(A'+C''') = \{d\in [0,N']:d\not\equiv 0 \pmod 4\};$$ furthermore, this union is disjoint and there is no repetition in the sumset.  It then follows that the sets $A',C'$ correspond to an Ulrich partition if and only if $$[0,N''] = A'' \amalg C'' \amalg \tfrac{1}{2}(A''+C'').$$ That is, either $A'',C''$ correspond to an Ulrich partition or $C''$  is empty, $N'' = 1$, and $A'' = \{0,1\}$, in which case we originally had $A' = \{0,4\}$ and $C' = \{2\}$.

Since $C''$ always has fewer elements than $C'$, we conclude by induction.  We may assume $C''$ has $(4^m-1)/3$ elements for some $m\geq 1$.  Then $$N'' +1=|A''|+|C''|+|A''||C''|=2+\frac{(4^{m}-1)}{3}+\frac{2(4^m-1)}{3}=4^m+1,$$ so $N'' = 4^m$ and $N' = 4^{m+1}$.  From this we similarly conclude $|C'| = (4^{m+1}-1)/3$, completing the proof.
\end{proof}

\begin{example}
By studying the proof of Theorem \ref{thm-sumset}, it is straightforward to recover the Ulrich partitions from the sumset construction.  As an example, let us find the Ulrich partition of type $(2,1,5)$, corresponding to $m=1$.  Then $N = 17$, so $N' = 16$, $A' = \{0,16\}$, and $C' = \{2,6,8,10,14\}$.  Correspondingly, the Ulrich partition is $(17,1|0|{-3},-7,-9,-11,-15)$.  We give the time evolution diagram of this partition in Figure \ref{fig-215}.
 \begin{figure}[htbp]
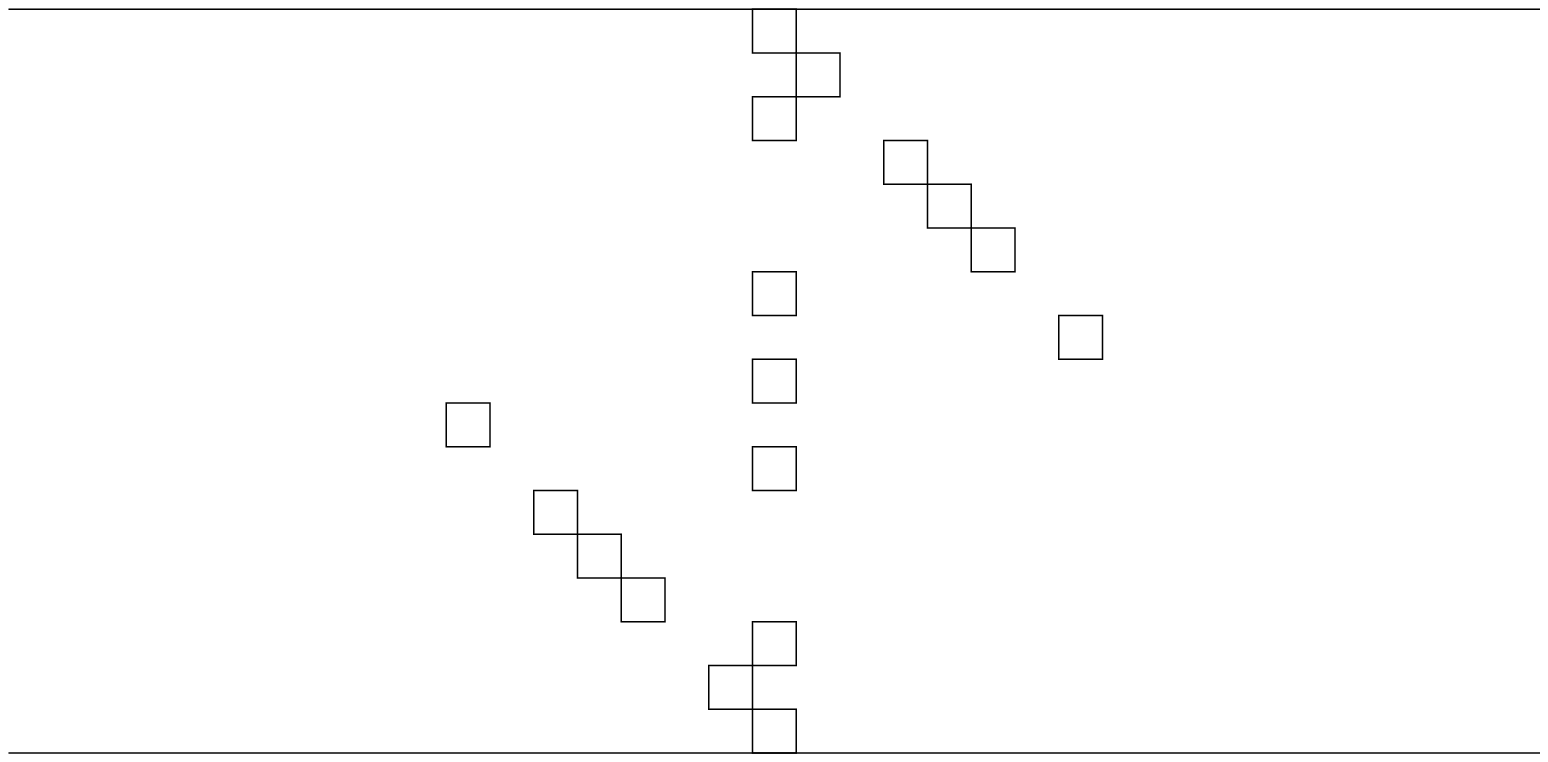
\caption{Time evolution diagram of the Ulrich partition of type $(2,1,5)$.}\label{fig-215}
\end{figure}
\end{example}

\section{Ulrich partitions of type $(\alpha,2,\gamma)$}\label{sec-beta2}

In this section we completely classify Ulrich partitions of type $(\alpha,2,\gamma)$.  There are many examples of such partitions, including a two-parameter infinite family and several sporadic examples.  The existence of these examples complicates the classification.

\subsection{The greedy algorithm} Most of our arguments on 3-part Ulrich partitions will involve studying the following algorithm for the construction of Ulrich partitions.  We normalize the evolution of partitions by subtracting $1,0,-1$, respectively, from the three blocks.  Fix a finite set $B\subset \Z$, which we view as being the list of entries in the $B$-block of a hypothetical Ulrich partition.  Let $A\subset \Z$ and $C\subset \Z$ be two finite sets such that $a>b>c$ holds for any $a\in A$, $b\in B$, $c\in C$.  We write $$A(t) = A -t \qquad B(t) = B \qquad C(t) = C+t,$$ where the right hand side is interpreted as a sumset, so that e.g. $A(t)$ gives the set of positions of the $A$-entries of $(A|B|C)$ at time $t$.

\begin{definition}
The partition $(A|B|C)$ is \emph{pre-Ulrich} if
\begin{enumerate}
\item all elements of $A$ and $C$ have the same parity, and
\item the set $T\subset \N_{>0}$ of times $t$ where two of $A(t), B(t),C(t)$ have a  common entry has size $N$, the dimension of $(A|B|C)$.
\end{enumerate}
\end{definition}

\begin{remark} Write $(A'|B'|C')\subset (A|B|C)$ if $A'\subset A$, $B'\subset B$, and $C'\subset C$.  Then if $(A|B|C)$ is pre-Ulrich, $(A'|B'|C')$ is pre-Ulrich.  Clearly Ulrich partitions are pre-Ulrich; the Ulrich condition requires that in addition $T = [N]$ is a consecutive range of integers.\end{remark}

Given any pre-Ulrich triple $(A|B|C)$, we introduce two operations for enlarging the $A$ or $C$ part of the triple.

\begin{definition}
Let $(A|B|C)$ be a pre-Ulrich partition.  Let $t_0\in \N_{>0}$ be the smallest time not in $T$.
\begin{enumerate}
\item Put $a_0(t_0) := \max(B(t_0) \cup C(t_0))$ and $a_0 := a_0(t_0)+t_0$.  The triple $(A\cup \{a_0\} |B|C)$ is \emph{obtained by adding a new $a$}.
 \item Put $c_0(t_0)=\min(A(t_0)\cup B(t_0))$ and $c_0 = c_0(t_0)-t_0$.  The triple $(A|B|C\cup \{c_0\})$ is \emph{obtained by adding a new $c$}.
\end{enumerate}
\end{definition}

\begin{warning}
It is not generally the case that the triple obtained from a pre-Ulrich triple by adding a new $a$ or new $c$ is again pre-Ulrich, since the new entry may meet old entries at times after $t_0$ that already have intersections.  See the proof of Lemma \ref{lem-BblockRestrict} for an example of this.
\end{warning}

The next result shows that an Ulrich partition $(A|B|C)$ can always be obtained from  $(\emptyset|B|\emptyset)$ by greedily adding $a$'s and $c$'s.

\begin{proposition}
If $(A|B|C)$ is an Ulrich triple, then it can be obtained by repeatedly adding new $a$'s and $c$'s to the triple $(\emptyset|B|\emptyset)$.  Furthermore, the sequence of $a$'s and $c$'s which are added is uniquely determined.
\end{proposition}
\begin{proof}
The sequence of $a$'s and $c$'s to add to $(\emptyset|B|\emptyset)$ is readily determined by the pattern of intersections of the Ulrich partition. Let $(A'|B|C') \subset (A|B|C)$ be any triple.  Since $(A|B|C)$ is Ulrich, $(A'|B|C')$ is pre-Ulrich.  Let $t_0\in \N_{>0}$ be the smallest time where $(A'|B|C')$ does not have an intersection, and consider the intersection occurring in $(A|B|C)$ at time $t_0$.

First suppose the intersection at time $t_0$ occurs between the $A$- and $B$- blocks, between an entry $a_0\in A$ and $b_0\in B$, so that $a_0(t_0)= b_0(t_0)$.  Then $a_0\notin A'$, since otherwise $(A'|B|C')$ has an intersection at time $t_0$.  No intersections between $a_0$ and an entry of $B$ or $C'$ have occurred before time $t_0$ since all these times already have intersections from $(A'|B|C')$.  In order for the intersections between $a_0$ and the entries of $B$, $C'$ to occur at time $t_0$ or later, we must have $a_0(t_0) \geq b(t_0)$ and $a_0(t_0) \geq c(t_0)$ for all $b\in B$ and $c\in C$, and we conclude $a_0(t_0) = \max(B(t_0)\cup C(t_0))$.  The triple obtained from $(A'|B|C')$ by adding a new $a$ is precisely $(A'\cup \{a_0\}|B|C')\subset (A|B|C)$.

If instead the intersection at time $t_0$ occurs between the $B$- and $C$-blocks, a symmetric argument shows that the triple $(A',B,C'\cup \{c_0\})$ obtained by adding a new $c$ is contained in $(A,B,C)$.

Finally suppose the intersection at time $t_0$ occurs between the $A$- and $C$-blocks, and let $a_0\in A$ and $c_0\in C$ be such that $a_0(t_0) = c_0(t_0)$.  By the choice of $t_0$, it can't be the case that both $a_0\in A'$ and $c_0\in C'$; we claim that in fact exactly one of $a_0\in A'$ or $c_0\in C'$ holds.  Suppose $a_0\notin A'$.  As before, this implies that $a_0(t_0) \geq b(t_0)$ and $a_0(t_0) \geq c(t_0)$ for all $b\in B$ and $c\in C'$, so $a_0(t_0) = \max(B(t_0)\cup C(t_0))$.  In fact, $a_0(t_0)>b(t_0)$ for all $b\in B$ since $a_0$ has not yet met the $B$-block at time $t_0$; fix some $b_0\in B$.  Since $c_0(t_0)=a_0(t_0) >b_0(t_0)$, the intersection between $c_0$ and $b_0$ must have occurred at a time before $t_0$.  This implies $c_0\in C'$.  Thus if $a_0\notin A'$, we conclude that the triple obtained from $(A'|B|C')$ by adding a new $a$ is $(A'\cup \{a_0\}|B|C')\subset (A|B|C)$; similarly, if $c_0\notin C'$, then the triple obtained from $(A'|B|C')$ by adding a new $c$ is $(A'|B|C'\cup \{c_0\})\subset (A|B|C)$.

Starting from $(\emptyset|B|\emptyset)$, we can now construct a chain
$$(\emptyset|B|\emptyset) \subset (A_1|B|C_1)\subset \cdots \subset (A_n|B|C_n) = (A|B|C)$$ of pre-Ulrich partitions where each triple differs from the previous one by adding an $a$ or a $c$.  Uniqueness is clear.
\end{proof}

\subsection{Duality and the trapezoid rule} The greedy algorithm will allow us to determine the structure of an Ulrich partition $(A|B|C)$ for early times $t$.  Recall that the dual $$(A|B|C)^* := ( A^* |B^*|C^*) := (C(N+1),B(N+1),A(N+1))$$ is also an Ulrich partition (see \S\ref{ssec-symmetryDuality}); this fact will allow us to determine the structure of an Ulrich partition at late (i.e. close to $N$) times.  The trapezoid rule is a simple observation which will allow us to combine information about a partition and its dual to say something about middle (i.e. close to $N/2$) times.  This is our principal tool for classifying Ulrich partitions of type $(\alpha,\beta,\gamma)$ with $\beta = 2$.

\begin{observation}[Trapezoid rule]\label{obsTrapezoid}
Let $(A|B|C)$ be an Ulrich partition and $(A^*|B^*|C^*)$ its dual.  If there exist $a\in A$, $a^*\in A^*$, $c\in C$, and $c^*\in C^*$ such that $a^*-a = c-c^*$, then $c(N+1) = a^*$ and $a(N+1) = c^*$.  In particular, $N+1=a^*-c=a-c^*$.
\end{observation}
\begin{proof}
The assumptions give that $c':=a^*-(N+1)\in C$ and $a':=c^*+(N+1)\in A$.  Then $a$ meets $c'$ at time $\frac{1}{2}(a-c')=\frac{1}{2}(a-a^*+(N+1))$ and $a'$ meets $c$ at time $\frac{1}{2}(a'-c)=\frac{1}{2}(c^*-c+(N+1))$. The hypothesis is that these times are equal; thus since $(A|B|C)$ is Ulrich we must have $a=a'$ and $c=c'$.
\end{proof}

The name of the trapezoid rule comes from the following geometric interpretation using time evolution diagrams.  Suppose we can find $a,a^*,c,c^*$ as in the trapezoid rule.  View $a$ and $c$ as being entries of $(A|B|C)$ at time $0$, and view $a^*$ and $c^*$ as being entries of $(C|B|A)$ at time $N+1$.  The plane trapezoid with vertices at $(a,0)$, $(c,0)$, $(c^*,N+1)$, and $(a^*,N+1)$ is horizontally symmetric by assumption.  The conclusion of the trapezoid rule is that this trapezoid has diagonals which meet orthogonally.  See Figure \ref{fig-trapezoid}.

 \begin{figure}[htbp]
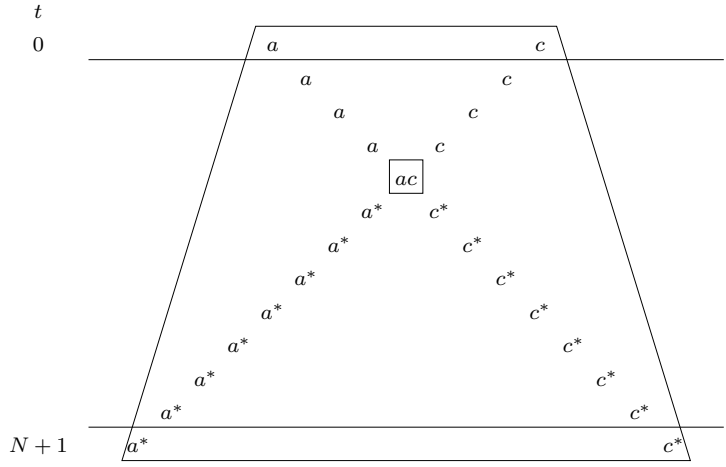
\caption{Graphical depiction of the trapezoid rule.  If $(A|B|C)$ is Ulrich and $a\in A,$ $a^*\in A^*,$ $c\in C,$ and $c^*\in C^*$ can be chosen such that the trapezoid displayed above is horizontally symmetric, then the diagonals meet orthogonally.  Equivalently, $N+1 = a^*-c = a-c^*$.}\label{fig-trapezoid}
\end{figure}

  See \S\ref{diff5} for the first simple application of this rule; more intricate applications will be discussed throughout the rest of the section.  We also point out one trivial rule which excludes many combinations of configurations of intersections for the early and late times.

\begin{observation}[Rectangle rule]\label{obsRectangle} Let $(A|B|C)$ be Ulrich, and let $(A^*|B^*|C^*)$ be its dual.  It is not possible that $A$ and $A^*$ share (at least) two entries.
\end{observation}
\begin{proof}
The hypotheses imply that there are two entries in $A$ with the same gap between them as two entries in $C$; there will necessarily be a multiple intersection at some time.
\end{proof}

Of course, by symmetry, an analogous version of the rectangle rule holds for $C$ and $C^*$.

\subsection{Restricting the $B$-block}  Suppose $(A|b_1,b_2|C)$ is an Ulrich partition.  Our first result restricts the gap $b_1-b_2$ in the $B$-block.

\begin{lemma}\label{lem-BblockRestrict}
If $(A|b_1,b_2|C)$ is an Ulrich partition then $$b_1-b_2\in \{1,3,5\}.$$
\end{lemma}
\begin{proof}
First let us show that $b_1-b_2$ is odd.  By way of contradiction, suppose $b_1=k$ and $b_2=-k$.  We consider what happens when we add new $a$'s and $c$'s to the pre-Ulrich triple $(\emptyset|k,-k|\emptyset)$.  Without loss of generality, we first add a new $a$, giving the triple $(k+1|k,-k|\emptyset)$.  This triple has intersections at times $1$ and $2k+1$; in particular there is no intersection at time $2$.  If we add a new $a$ to this triple we get $(k+2,k+1|k,-k|\emptyset)$, while if we add a new $c$ we get $(k+1|k,-k|{-k-2})$.  Neither triple is pre-Ulrich, since they both violate the parity requirement.

Next suppose that $b_1-b_2$ is odd and $\geq 7$.  We consider adding $a$'s and $c$'s to $(\emptyset|k,-k-1|\emptyset)$.  Without loss of generality, we first add an $a$ at time $1$ to get $(k+1|k,-k-1|\emptyset)$.  The triple is no longer pre-Ulrich if we add an $a$ at time $2$, so we add a $c$ at time $2$ and get $(k+1|k,-k-1|{-k-3})$.  At times $3$ and $4$ (which do not have intersections yet since $k\geq 3$) the parity condition implies we must add an $a$ and then another $c$, yielding $$(k+3,k+1|k,-k-1|{-k-3},-k-5).$$ This partition is no longer pre-Ulrich since both the $A$- and $C$-blocks have entries which are $2$ apart; there is a multiple intersection at time $k+3$.
\end{proof}

The classification varies considerably according to the value of $b_1-b_2$.  We consider each case separately in the next several subsections.

\subsection{Classification when $b_1-b_2=5$}\label{diff5}  This is the easiest case to classify.  For this subsection we fix $B = \{5,0\}$, let $(A|B|C)$ be an Ulrich partition, and assume that the intersection at time $t=1$ occurs between the $A$- and $B$-blocks, so $6\in A$.  After these choices, we prove there is a unique such partition.

\begin{example}[The (2,2,1) partition]\label{ex-221}
The partition $(8,6|5,0|{-2})$ is Ulrich.  It is obtained from $(\emptyset|5,0|\emptyset)$ by adding an $a$, then a $c$, then an $a$.  See Figure \ref{fig-221}.
 \begin{figure}[htbp]
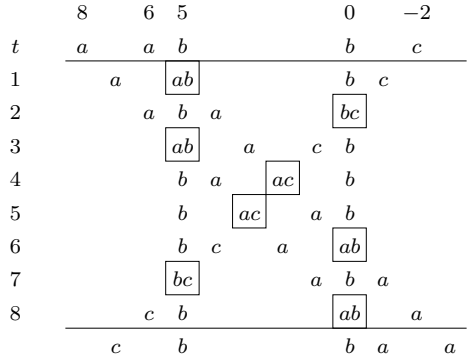
\caption{Time evolution diagram of the Ulrich partition $(8,6|5,0|{-2})$ of type $(2,2,1)$.}\label{fig-221}
\end{figure}
\end{example}

\begin{lemma}\label{diff5contain}
The partition $(A|B|C)$ contains $(8,6|5,0|{-2})$.
\end{lemma}
\begin{proof}
Consider the sequence of $a$'s and $c$'s which must be added to $(\emptyset|5,0|\emptyset)$ to produce $(A|B|C)$.  By assumption, at time $t=1$ an $a$ must be added, giving $(6|5,0|\emptyset)$.  By parity, at time $t=2$ a $c$ is added, giving $(6|5,0|{-2})$.  Again by parity, at time $t=3$ an $a$ is added, yielding $(8,6|5,0|{-2})$ as required.
\end{proof}

The next proposition is now our first application of duality and the trapezoid rule.

\begin{proposition}
The partition $(A|B|C)$ equals $(8,6|5,0|{-2})$.
\end{proposition}
\begin{proof}
By Lemma \ref{diff5contain}, we have $(8,6|5,0|{-2})\subset (A|5,0|C)$.  The dual partition $(A^*|B^*|C^*)$ is Ulrich.  Thus, by Lemma \ref{diff5contain}, it contains either $(8,6|5,0|{-2})$ or the symmetric partition $(7|5,0|{-1},-3)$, according to whether the time $t=1$ intersection occurs between the first two or last two blocks.  The first possibility is ruled out by the rectangle rule (Observation \ref{obsRectangle}).  Likewise, if it contains $(7|5,0|{-1},-3)$, then since $7-6 = -2-(-3)$ the trapezoid rule (Observation \ref{obsTrapezoid}) shows $N=8$.  Since $(8,6|5,0|{-2})$  has dimension $8$, this implies $(A|B|C) = (8,6|5,0|{-2})$.
\end{proof}

\subsection{Classification when $b_1-b_2=3$}
The classification here is substantially more complicated than when $b_1-b_2=5$. We normalize $B=\{3,0\}$, and assume the first intersection occurs between the $A$- and $B$-blocks, so $4\in A$.  There are three such examples of Ulrich partitions.

\begin{example}[The (1,2,1) partition]\label{ex121}
The partition $(4|3,0|{-2})$ is Ulrich.  It is obtained from $(\emptyset|B|\emptyset)$ by adding an $a$ and then a $c$.  This is an example of a partition obtained from Theorem \ref{thm-1n1}.
\end{example}

\begin{example}[The (2,2,2) partition]\label{ex222} The partition $(12,4|3,0|{-2},-8)$ is Ulrich.  It is obtained from $(\emptyset|3,0|\emptyset)$ by adding $a,c,c,a$. See Figure \ref{fig-222}.
 \begin{figure}[htbp]
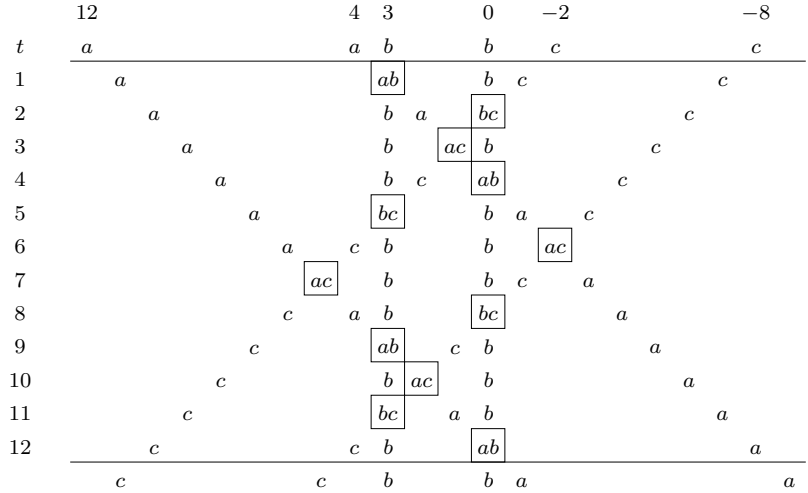
\caption{Time evolution diagram of the Ulrich partition $(12,4|3,0|{-2},-8)$ of type $(2,2,2)$.}\label{fig-222}
\end{figure}
\end{example}

\begin{example}[The (3,2,2) partition]\label{ex322}
The partition $(16,10,4|3,0|{-2},{-12})$ is Ulrich.  It is obtained from $(\emptyset|3,0|\emptyset)$ by adding in order $a,c,a,c,a$.  See Figure \ref{fig-322}.
 \begin{figure}[htbp]
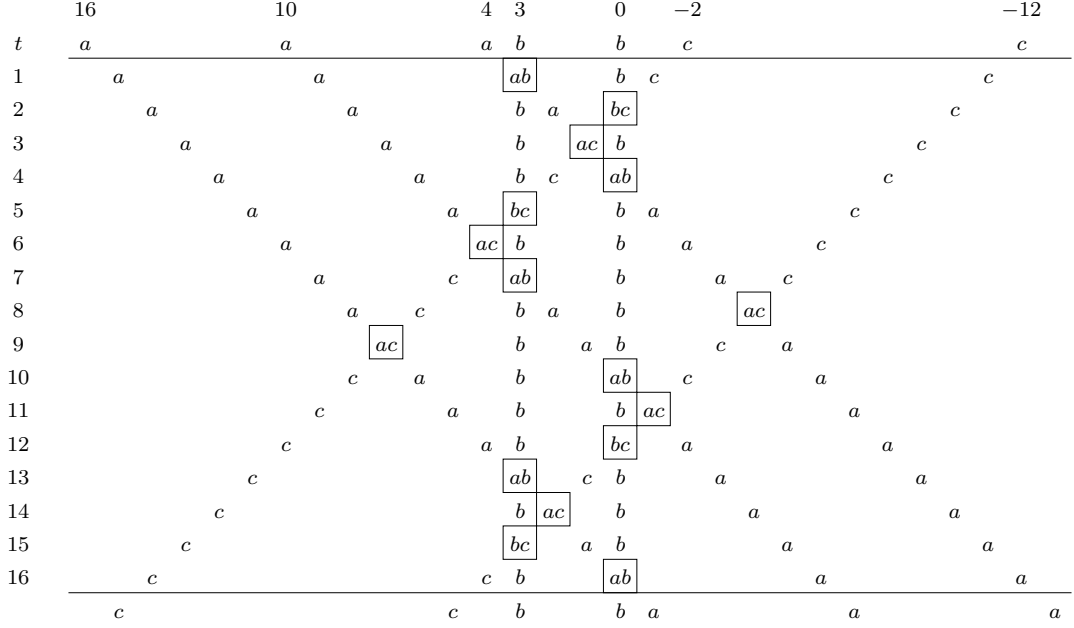
\caption{Time evolution diagram of the Ulrich partition $(16,10,4|3,0|{-2},-12)$ of type $(3,2,2)$.}\label{fig-322}
\end{figure}
\end{example}

By comparison with the $b_1-b_2=5$ case, there is less forced structure to the early intersections for an Ulrich triple in this case.  We let $\sigma\in \{a,c\}^*$ be the string of $a$'s and $c$'s which must be added to $(\emptyset|B|\emptyset)$ to yield the Ulrich triple $(A|B|C)$.  We write $|\sigma|$ for the length of $\sigma$.

\begin{lemma}\label{lemInitialSegment}
If $\sigma \neq ac$ (i.e. if $(A|B|C)$ is not the type $(1,2,1)$ partition of Example \ref{ex121}), then $\sigma$ begins with one of the strings
\begin{enumerate}
\item $acaaa$, so that $(16,14,10,4|3,0|{-2})\subset (A|B|C)$,
\item $acaca$, so that $(16,10,4|3,0|{-2},-12)\subset (A|B|C)$,
\item $acca$, so that $(12,4|3,0|{-2},-8)\subset (A|B|C)$, or
\item $acccc$, so that $(4|3,0|{-2},-8,-10,-14)\subset (A|B|C)$.
\end{enumerate}
\end{lemma}
\begin{proof}
By assumption the first letter of $\sigma$ is $a$, and by parity the second letter must be $c$.  Observe that the pre-Ulrich partitions $(10,4|3,0|{-2})$ and $(4|3,0|{-2},-8)$ corresponding to the strings $aca$ and $acc$ are both not Ulrich, so $|\sigma|\geq 4$.  To prove the lemma, we must show that unless $\sigma$ begins with $acca$ then $|\sigma|\geq 5$ and the fifth letter in $\sigma$ is determined by the first four.

Suppose $\sigma$ begins with $acaa$.  This gives $(14,10,4|3,0|{-2})\subset (A|B|C)$.  The first time where $(14,10,4|3,0|{-2})$ has no intersection is time $t=9$, and this triple is not Ulrich so $|\sigma|\geq 5$.  Adding a $c$ would yield $(14,10,4|3,0|{-2},-14)$, but this is not pre-Ulrich since there is a multiple intersection at time $t=14$.  Thus $\sigma$ must begin with $acaaa$.

If instead $\sigma$ begins with $acac$, then $(10,4|3,0|{-2},-12)\subset (A|B|C)$, and this triple is not Ulrich so $|\sigma|\geq 5$.  Adding a $c$ would yield $(10,4|3,0|{-2},-12,-14)$, which is not pre-Ulrich since there is a multiple intersection at time $t=12$.  So, $\sigma$ begins with $acaca$.

Finally suppose $\sigma$ begins with $accc$, so $(4|3,0|{-2},-8,-10)\subset (A|B|C)$.  This is not Ulrich, and adding an $a$ gives $(16,4|3,0|{-2},-8,-10)$.  Considering time $13$ shows this is not pre-Ulrich, so $\sigma$ begins with $acccc$.
\end{proof}

We now treat each case of Lemma \ref{lemInitialSegment} in further detail.

\begin{proposition}\label{propaccc}
Case (4) of Lemma \ref{lemInitialSegment} never actually arises: if $(A|B|C)$ is an Ulrich partition normalized as in this subsection, then the word $\sigma$ cannot begin with $acccc$.
\end{proposition}
\begin{proof}
If $\sigma$ begins with $acccc$ then $(4|3,0|{-2},-8,-10,-14)\subset (A|B|C)$, so $N \geq 17$ since there is an intersection at time $17$ in this subpartition.  We first study the dual Ulrich partition $(A^*|B^*|C^*)$.  By Lemma \ref{lemInitialSegment}, this partition contains one of the following 8 partitions, of which the first 7 are easily ruled out.
\begin{enumerate}
\item $(16,14,10,4|3,0|{-2})$.  In this case the equality $4-4=-2-(-2)$ and the trapezoid rule implies $N=5$, which is absurd.

\item $(5|3,0|{-1},-7,-11,-13)$.  Here the equality $5-4=-10-(-11)$ and the trapezoid rule gives $N=14$, a contradiction.

\item $(16,10,4|3,0|{-2},-12)$.  The same logic as in (1) applies.

\item $(15,5|3,0|{-1},-7,-13)$.  This time the equality $15-4 = -2-(-13)$ and the trapezoid rule gives $N = 16$.

\item $(12,4|3,0|{-2},-8)$.  Same as (1).
\item $(11,5|3,0|{-1},-9)$.  The equality $5-4=-8-(-9)$ and the trapezoid rule gives $N=12$.
\item $(4|3,0|{-2},-8,-10,-14)$.  Same as (1).
\item $(17,13,11,5|3,0|{-1}).$
\end{enumerate}

We conclude that $(A^*|B^*|C^*)$ must contain $(17,13,11,5|3,0|{-1})$.  Note that the sequence of $a$'s and $c$'s which must be added to $(\emptyset,B,\emptyset)$ to arrive at this subpartition is the sequence $caaaa$.

As a first observation, we find $|A|\geq 2$, for if $|A|=1$ then $N=4$.  Thus there is some $k\geq 4$ such that $\sigma$ begins with $ac^ka$.  By way of contradiction, let $k\geq 4$ be minimal such that there is an Ulrich partition $(A|B|C)$ such that the corresponding word $\sigma$ begins with $ac^ka$.  It follows from minimality and symmetry that the word $\sigma^*$ corresponding to the dual partition $(A^*|B^*|C^*)$ at least begins with $ca^k$.

We now compute the partition $(4|3,0|C_k)$ ($k\geq 1$) obtained from $(\emptyset|B|\emptyset)$ by adding the letters of the word $ac^k$.  We label the elements of $C_k$ in decreasing order as $C_k = \{c_1,\ldots,c_k\}$, so $C_{k+1} = C_k \cup \{c_{k+1}\}$.  We have $c_1 = -2$.  Let $T_k\subset \N_{>0}$ be the set of times which are \emph{not} intersection times for the partition $(4|3,0|C_k)$, so $T_1 = [6,\infty]$.  We let $t_{k+1}=\min T_k$.  For $k\geq 1$, when a $c$ is added to $(4|3,0|C_k)$ it is added at time $t_{k+1}$ and meets the $A$-entry at that time, so $c_{k+1}(t_{k+1}) = 4-t_{k+1}$ and $c_{k+1} = 4-2t_{k+1}$.  Computing the sequence of sets $C_k$ is therefore equivalent to computing the sequence of times $\{t_k\}_k$.

The computation of the sequence of times $\{t_k\}_k$ when new $c$'s are added is best explained in terms of a sieve.  Adding a new $c$ at time $t_{k+1}$ means we include $c_{k+1} = 4-2t_{k+1}$ in $C_{k+1}$.  Then $c_{k+1}(2t_{k+1}-4) = 0$ and $c_{k+1}(2t_{k+1}-1) = 3$, so $c_{k+1}$ meets the $B$-block at times $2t_{k+1}-4$ and $2t_{k+1}-1$.  Therefore $T_{k+1} = T_{k} \setminus \{t_{k+1},2t_{k+1}-4,2t_{k+1}-1\}$.

We now make this computation explicit.  To get started, we have $t_2 = 6$, which sieves out times $8$ and $11$.  We must then include $t_3 = 7$, which sieves times $10$ and $13$, etc.  We include the result of this sieve computation for small times below; $\times$'s denote times which are sieved out.
$$\begin{array}{ccccccccccccccccccccccc}6 & 7 & 8 & 9 & 10 & 11 & 12 & 13 & 14 & 15 & 16 & 17 & 18 & 19 & 20 & 21 & 22 & 23 & 24 & 25 & 26\\
t_2 & t_3 &\times& t_4 & \times & \times &t_5& \times &\times &t_6&t_7&\times&t_8&t_9&\times&t_{10}&t_{11}&\times&t_{12}&t_{13}&\times
\\
\\
27 & 28 & 29 & 30 & 31 & 32 & 33 & 34 & 35 & 36 & 37 & 38 & 39 & 40 & 41 & 42 & 43 & 44 & 45 & 46 & 47 \\
t_{14}&\times&\times&t_{15}&\times&\times&t_{16} &\times&\times & t_{17} &\times&\times& t_{18} & \times & \times& t_{19} & \times & \times & t_{20} & \times& \times
\\
\\
48 &49 & 50 & 51 & 52 & 53 & 54 & 55 & 56 & 57 & 58 & 59 & \cdots & 95 & 96 & 97 & 98 & 99 & 100 \\
t_{21} & \times& \times & t_{22} & t_{23} & \times & t_{24} & t_{25} & \times & t_{26} & t_{27} & \times & \cdots & \times & t_{52} & t_{53} & \times & t_{54} & \times
\end{array}$$
The result of the computation is easy to describe.  First, every time $t\geq 6$ with $t\equiv 0 \pmod{3}$ appears in $\{t_k\}$.  No times with $t \equiv 2 \pmod{3}$ appear.  The times congruent to $1 \pmod{3}$ may or may not appear, but the pattern with which they appear is simple.  The first time $(7)$ congruent to $1$ appears, the next 2 times (10,13) do not, the next 4 times (16,19,22,25) do appear, the next 8 do not, the next 16 do, etc.  This description is straightforward to prove, and completely specifies the sequence $\{t_k\}$.

Now suppose that an $a$, call it $a_2$, is added to $( 4|3,0|C_k)$.  Unless $k$ is very special, it turns out that the resulting partition is not pre-Ulrich.  We have $a_2(t_{k+1}) = c_1(t_{k+1})$, so $a_2 = 2t_{k+1}-2$.  Since $\{c_1,c_2,c_3\} = \{-2,-8,-10\}$, we find that $a_2$ will also meet the $C$-block at times $t_{k+1}+3$ and $t_{k+1}+4$.  For the resulting partition to be pre-Ulrich, it is therefore necessary that these times are not sieved out.  Additionally, if $t_{k+2} = t_{k+1} + 1$ then no intersection with $a_2$ will occur at time $t_{k+2}$, so it is necessary to add another $a$ or $c$ at this time; it must be a $c$ that is added, for otherwise we will have two $a$'s which are $2$ apart from one another, a contradiction since $c_2$ and $c_3$ are already $2$ apart from one another.  But then the $c$ which is added at time $t_{k+2}$, call it $c_{k+1}$, will satisfy $c_{k+1} = 4-2t_{k+2} = 2-2t_{k+1}$.  Then $a_2$, $0\in B$, and $c_{k+1}$ all coincide at time $2t_{k+1}-2$.  We finally conclude that the times $t_{k+1}+1$ must be sieved out and $t_{k+1}+3$ and $t_{k+1}+4$ cannot be sieved out in order for the partition to have a chance of being pre-Ulrich.  Inspecting the sieve, the only way for this to occur is for $t_{k+1}$ to be one of the times immediately before the $t_k$'s start occurring in pairs, e.g. $t_5 = 12$, $t_{21} = 48$, or $t_{85} = 192$.  A straightforward computation shows that these times are precisely the numbers $t_{k+1}$ of the form $3\cdot 4^m$ with $m\geq 1$, and $k+1 = \sum_{i=0}^m 4^i$.

Finally, suppose $t_{k+1} = 3\cdot 4^m$ and $a_2 = 2t_{k+1}-2$. We use the trapezoid rule to show that $(a_2,4|3,0|C_k)$ is not contained in any Ulrich partition $(A|B|C)$.  Observe that a $c$ was added at time $\frac{1}{2}t_{k+1}+1$; this time is the last time that was added in the block of pairs of times preceding $t_{k+1}$.  Thus $4-2(\frac{1}{2}t_{k+1}+1)=2-t_{k+1}\in C$.  Since the dual partition $(A^*|B^*|C^*)$ has corresponding word $\sigma^*$ beginning with $ca^k$ we symmetrically have $-(2-t_{k+1})+3=t_{k+1}+1\in A^*$ and $-1\in C^*$.  The equality $$(t_{k+1}+1)-(2t_{k+1}-2)=(2-t_{k+1})-(-1)$$ and the trapezoid rule prove $a_2(N+1) = -1$, from which it follows that $|A|=2$.  We also have $N=a_2 = 2t_{k+1}-2$.  At time $N$ we have $a_2(N) = 0\in B$, so at time $N-1$ the smallest entry $c_\gamma \in C$ has $c_\gamma(N-1) = 3$.  Then $c_\gamma = 3-(N-1)= 4-2(t_{k+1}-1)$.  This means that $c_\gamma$ had to be added at time $t_{k+1}-1$, contradicting that this time was sieved out since it is congruent to $2 \pmod{3}$.
\end{proof}

The next result has a proof that uses the exact same technique as the proof of Proposition \ref{propaccc}, so we omit it.

\begin{proposition}\label{propacaa}
Case (1) of Lemma \ref{lemInitialSegment} never actually arises: if $(A|B|C)$ is an Ulrich partition normalized as in this subsection, then the word $\sigma$ cannot begin with $acaaa$.
\end{proposition}

Fortunately, the positive classification results are easier.

\begin{proposition}\label{prop322}
If $\sigma$ begins with $acaca$ then $\sigma = acaca$ and $(A|B|C)$ is the $(3,2,2)$ partition of Example \ref{ex322}.
\end{proposition}
\begin{proof}
Since $\sigma$ begins in $acaca$, $(A|B|C)$ contains $(16,10,4|3,0|{-2},-12)$, and therefore $N\geq 16$ with equality iff $\sigma = acaca$.  By duality, the partition $(A^*|B^*|C^*)$ is Ulrich.  By Lemma \ref{lemInitialSegment} and Propositions \ref{propaccc} and \ref{propacaa} it contains one of the following partitions: \begin{enumerate}
\item $(16,10,4|3,0|{-2},-12)$; this is impossible by the rectangle rule.

\item $(15,5|3,0|{-1},-7,-13)$; in this case, the equality $15-10=-2-(-7)$ and the trapezoid rule give $N=16$, so in fact $\sigma = acaca$.
\item $(12,4|3,0|{-2},-8)$; the equality $4-4=-2-(-2)$ and the trapezoid rule give $N = 5$, a contradiction.
\item $(11,5|3,0|{-1},-9)$; the equality $11-4=-2-(-9)$ and the trapezoid rule again gives $N=12$.
\end{enumerate}
Thus the only possibility is that $\sigma = acaca$.
\end{proof}

Our final result in this subsection completes the classification when $B = \{3,0\}$.

\begin{proposition}\label{prop222}
If $\sigma$ begins with $acca$ then $\sigma = acca$ and $(A|B|C)$ is the $(2,2,2)$ partition of Example \ref{ex222}.
\end{proposition}
\begin{proof}
The Ulrich partition $(A|B|C)$ contains $(12,4|3,0|{-2},-8)$, so $N\geq 12$ with equality iff $\sigma = acca$.  As in the proof of Proposition \ref{prop322}, the dual $(A^*|B^*|C^*)$ contains  one of the following partitions:
\begin{enumerate}
\item $(16,10,4|3,0|{-2},-12)$; this is impossible by the trapezoid rule applied to the equality $4-4 = -2 - (-2)$.

\item $(15,5|3,0|{-1},-7,-13)$; the equality $5-12= -8-(-1)$ and the trapezoid rule gives $N=12$ so $\sigma = acca$ (in fact this is also a contradiction, since the dual partition has the wrong structure.)

\item $(12,4|3,0|{-2},-8)$; this contradicts the rectangle rule.
\item $(11,5|3,0|{-1},-9)$; the equality $11-4=-2-(-9)$ and the trapezoid rule give $N=12$, so $\sigma = acca$.
\end{enumerate}
We conclude that $\sigma = acca$.
\end{proof}

\subsection{Classification when $b_1-b_2=1$}
In this final case we normalize $B = \{1,0\}$.  We first classify an infinite family of examples which serve as primitive building blocks for a further family of examples.

\begin{proposition}\label{propn21}
Let $(2|1,0|C_k)$ be the partition obtained from $(2|1,0|\emptyset)$ by adding $k$ $c$'s.  This partition is Ulrich if and only if $k$ is a number of the form $k = \sum_{i=0}^m 4^i = \frac{1}{3}(4^{m+1}-1)$ for some $m\geq 0$.
\end{proposition}
\begin{proof}
Write $C_k = \{c_1,\ldots,c_k\}$ with increasing entries.  We prove the result by induction on $k$; clearly $(2|1,0|C_1)=(2|1,0|{-4})$ is Ulrich.  Fix some $m_0\geq 0$ and let $k_0 = \frac{1}{3}(4^{m_0+1}-1)$, and assume $(2|1,0|C_{k_0})$ is Ulrich.   The dimension of the flag variety corresponding to the type $(1,2,k_0)$ is $4^{m_0+1}+1$, so the set of times where $(2|1,0|C_{k_0})$ has an intersection is $[1,4^{m_0+1}+1]$.

At time $4^{m_0+1}+2$ the $a$ entry is at position $-4^{m_0+1}$, so adding a new $c$ to $(2|1,0|C_{k_0})$ gives $c_{k_0+1}(4^{m_0+1}+2) = -4^{m_0+1}$, i.e. $c_{k_0+1} = -2\cdot 4^{m_0+1}-2.$  This entry intersects the $B$-block at times $2\cdot 4^{m_0+1}+2$ and $2\cdot 4^{m_0+1}+3$.  Continuing inductively, if $\ell$ $c$'s are added to $(2|1,0|C_{k_0})$ and $\ell\leq 4^{m_0+1}$, then $c_{k_0+\ell}(4^{m_0+1}+1+\ell) = -4^{m_0+1}+1-\ell$ so $c_{k_0+\ell} = -2\cdot 4^{m_0+1} - 2\ell$, which meets the $B$-block at times $2\cdot 4^{m_0+1}+2\ell$ and $2\cdot 4^{m_0+1}+1+2\ell$.  Then if $\ell <4^{m_0+1}$, the partition $(2|1,0|C_{k_0+\ell})$ is not Ulrich since there is no intersection at time $2\cdot 4^{m_0+1}+1$.  When $\ell = 4^{m_0+1}$, there are intersections between some $c_{k_0+\ell}$ and the $a$ for all $t\in [4^{m_0+1}+2,2\cdot 4^{m_0+1}+1]$ and between some $c_{k_0+\ell}$ and the $B$-block for all $t\in [2\cdot 4^{m_0+1}+2,4^{m_0+2}+1]$.  Therefore $(2|1,0|C_{k_0})$ is Ulrich.
\end{proof}

\begin{remark}
Analyzing the proof of Proposition \ref{propn21}, the sets $C_k$ (equivalently, the numbers $c_k$) which appear are easily computable.  Here are the first several terms:
$$-4, \qquad -10,-12,-14,-16, \qquad -34,-36,\ldots,-64, \qquad -130,-132,\ldots,-256, \qquad  \ldots$$
A decreasing block of $4^k$ even integers ending in $4^{k+1}$ is followed by a gap of size $4^{k+1} + 2$.  In Figure \ref{fig-125} we give the time evolution diagram for the partition $(2|1,0|C_5)$ of type $(1,2,5)$.
 \begin{figure}[htbp]
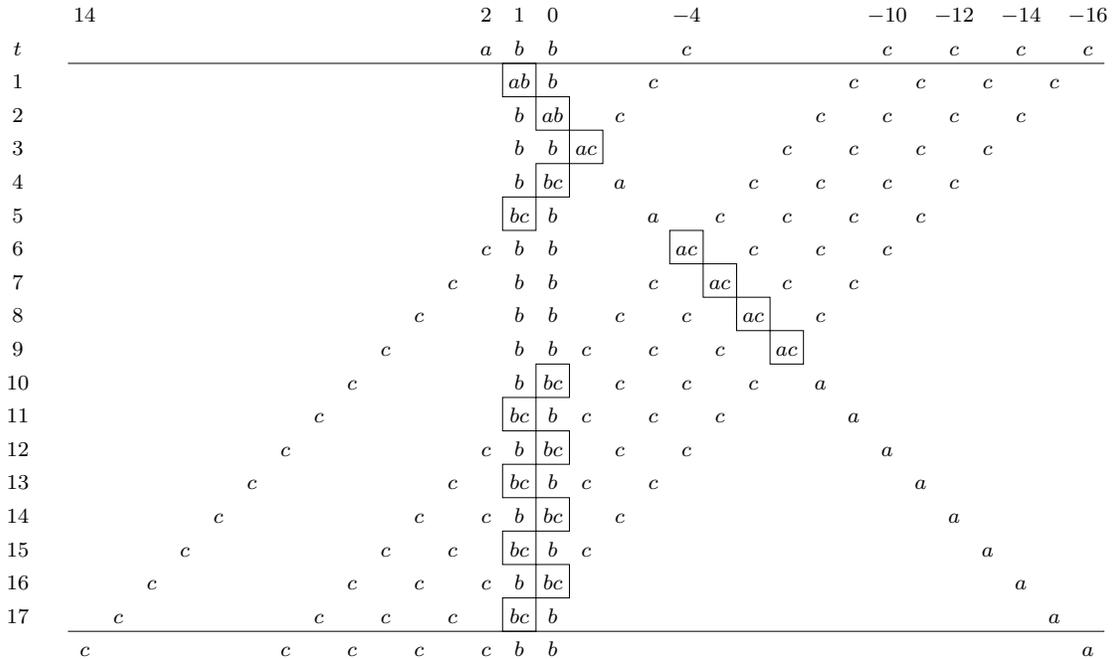
\caption{Time evolution diagram of the Ulrich partition $(2|1,0|C_5)$ of type $(1,2,5)$.}\label{fig-125}
\end{figure}
\end{remark}

A more careful analysis proves the following result.

\begin{lemma}\label{lem-stupid}
Suppose $(A|1,0|C)$ is an Ulrich partition and that the word $\sigma\in \{a,c\}^*$ generating it from $(\emptyset|B|\emptyset)$ begins with $acc$.  Then $|A|=1$ and $2\in A$, so $(A|B|C)$ is one of the examples of Proposition \ref{propn21}.
\end{lemma}
\begin{proof}
We will be brief, as the proof is very similar to other arguments in this section.  Consider the partition obtained from $(2|1,0|C_k)$ by adding an $a$, with $k\geq 2$.  It is easy to show that unless $k$ is of the form $k = (4^{m+1}-1)/3$ (so $m\geq 1$ and the original partition is Ulrich) then the resulting partition is not even pre-Ulrich.  On the other hand, if $k$ is of this form, then adding a new $a_2$ at time $t_0 = 4^{m+1}+2$ will give $a_2 = 2\cdot 4^{m+1}$.  One shows that we will be forced to add new $c$'s at positions $-3\cdot 4^{m+1}$ and $-4^{m+2}$.  Then the $c$ at position $-3\cdot 4^{m+1}$ meets $0\in B$ at the same time as $a_2$ meets the $c$ at position $-4^{m+2}$, a contradiction.
\end{proof}

Proposition \ref{propn21} and Lemma \ref{lem-stupid} allow us to focus on classifying Ulrich partitions $(A|1,0|C)$ with $2\in A$ and $|A|\geq 2$, since they completely classify Ulrich partitions where $A = \{2\}$.  The next lemma explains the importance of the examples of Proposition \ref{propn21} to more general Ulrich partitions.

\begin{lemma}\label{lemNoMorecs}
Let $(A|1,0|C) $ be an Ulrich partition, and suppose $2\in A$.  Let $c_1 = \max C$, so that $c_1$ meets the $B$-block at time $t_0:=-c_1$.  Let $A'\subset A$ be those $a$'s which have met the $B$-block before time $t_0$.  Then the partition $(A'|1,0|c_1)\subset (A|1,0|C)$ is Ulrich of dimension $t_0+1$, so the dual $(A'|1,0|c_1)^*$ is one of the Ulrich partitions of Proposition \ref{propn21}.
\end{lemma}
\begin{proof}
Write $A = \{a_\alpha,\ldots, a_1\}$ and $C = \{c_1,\ldots, c_\gamma\}$ in decreasing order, and consider how $(A|1,0|C)$ is built from $(\emptyset|1,0|\emptyset)$ by adding $a$'s and $c$'s; let $\sigma\in \{a,c\}^*$ be the corresponding word.  Since the time $t=1$ intersection is between $A$ and $B$, $\sigma$ begins with an $a$.  We can then write $\sigma = a^k c \sigma'$ for some $k\geq 1$ and some word $\sigma'$.  Then $A$ contains the first $k$ even integers $\{2,\ldots,2k\}$ and the intersections at times $t\in [1,2k-1]$ occur between the $A$- and $B$-blocks.  At time $2k$ the entry $a_1$ is at position $2-2k$, so $c_1(2k) = -2k+2$ and $c_1 = -4k+2$.

Suppose $c_2$ meets $a_1$ at time $t_1$.  We claim $t_1>t_0$.  Indeed, first notice that $c_1$ meets $a_i$ $(1\leq i\leq k)$ at time $2k-1+i$, so $t_0 \geq 3k$.  Since $c_2$ meets $a_1$ at time $t_1$, it meets $a_i$ $(1\leq i\leq k)$ at time $t_1-1+i$.  But then assuming $3k\leq t_1 \leq t_0=4k-2$, we find that $c_2$ meets some $a_i$ at the same time as $c_1$ meets $0\in B$.  We conclude $t_1>t_0$.

Therefore, if $\sigma$ contains at least $2$ $c$'s, then the second $c$ is added after time $t_0$.  It follows that if $\sigma' = a^k c a^\ell$, where $\ell$ is the additional number of $\ell$'s which are added before time $t_0$, then $\sigma'$ is an initial segment of $\sigma$, the corresponding subpartition is $(A'|1,0|c_1)$, and this partition is Ulrich.  The last intersection in this partition occurs between $c_1$ and $1\in B$, so its dual is $(2|1,0|A'^*)$ and Proposition \ref{propn21} applies.
\end{proof}

\begin{example}[A two-parameter family of Ulrich partitions]\label{ex2param}
Let $m_1,m_2\geq 0$, and let $k_i = \frac{1}{3}(4^{m_i+1}-1)$.  Let $(2|1,0|C_{k_1})$ be the Ulrich partition of Proposition \ref{propn21}, and let $(A_{k_2}|1,0|{-1})$ be the partition symmetric to $(2|1,0|C_{k_2})$.  There is an Ulrich partition $(A|B|C)$ uniquely specified by the requirements that the type is $(k_1+k_2,2,1)$ and \begin{align*}(2|1,0|C_{k_1})^* &\subset (A|B|C) \\ (A_{k_2}|1,0|{-1})^* &\subset (A|B|C)^*.\end{align*} The dimension of the type $(k_1+k_2,2,1)$ is $N = 4^{m_1+1}+4^{m_2+1}$.
 Let $t_0 = 4^{m_1+1}$ be as in Lemma \ref{lemNoMorecs}.  For times $t\in [1,t_0]$, the pattern of intersections in $(A|B|C)$ is the same as that of $(2|1,0|C_{k_1})^*$.  Applying Lemma \ref{lemNoMorecs} to the dual $(A|B|C)^*$, the corresponding time is $t_0^* = 4^{m_2+1}$, and the pattern of intersections in $(A|B|C)$ for times $t\in [t_0+1,N]$ is the same as that of $(A_{k_2}|1,0|{-1})$ for times $t\in [1,t_0^*]$.  Thus every time $t\in [1,N]$ has an intersection.

Observe that if $m_1\neq m_2$ then the partitions corresponding to $(m_1,m_2)$ and $(m_2,m_1)$ are distinct but related by the symmetric dual.  The partition corresponding to $(m_1,m_1)$ is its own symmetric dual.

For example, consider the case $m_1=0$ and $m_2= 1$.
Then we compute \begin{align*} N&= 20\\C_{k_1} &= \{-4\} \\ A_{k_2} &= \{17,15,13,11,5\}, \\(2|1,0|C_{k_1})^* &= (2|1,0|{-4})\\ (A_{k_2}|1,0|{-1})^*&= (17|1,0|{-1},-3,-5,-7,-13)\\
(A|B|C) &=  (20,18,16,14,8,2|1,0|{-4}).
\end{align*}
See Figure \ref{fig-621} for the time evolution diagram.  Note that the examples of Proposition \ref{propn21} can be regarded as (duals of) the degenerate case where $m_2 = -1$.
 \begin{figure}[htbp]
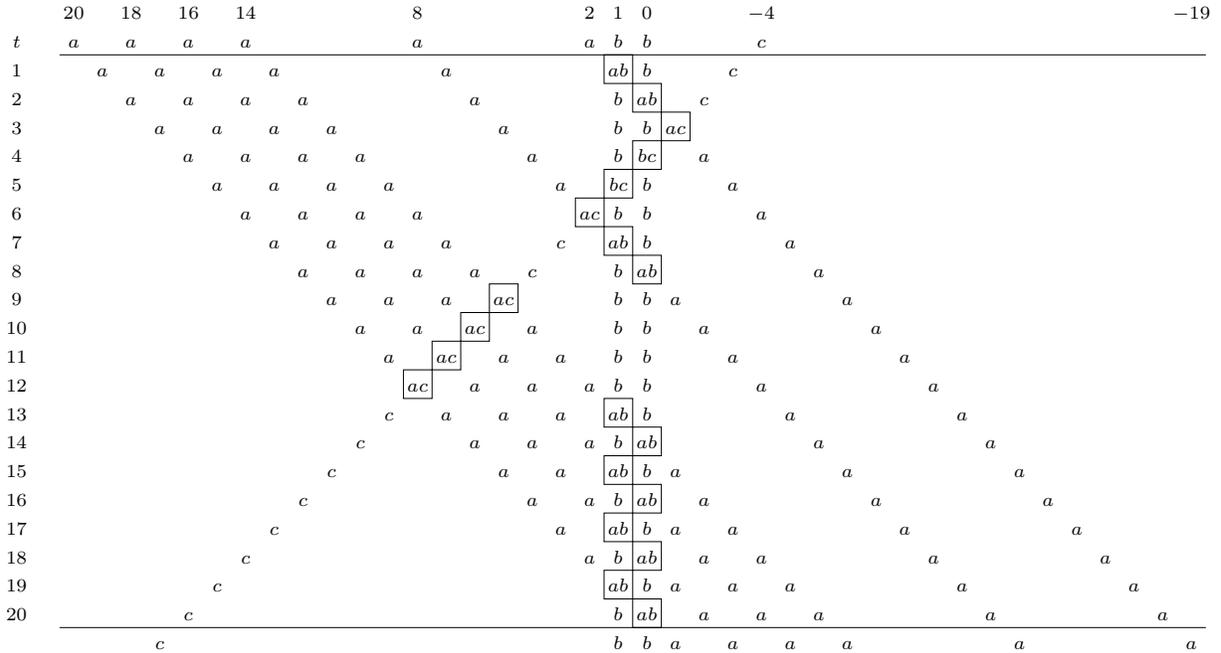
\caption{Time evolution diagram of the partition $(20,18,16,14,8,2|1,0|{-4})$ of type $(6,2,1)$.}\label{fig-621}
\end{figure}
\end{example}

\begin{theorem}
If $(A|1,0|C)$ is an Ulrich partition with $2\in A$ and $|A|\geq 2$ then it is either the dual of one of the examples from Proposition \ref{propn21}  or there exist $m_1,m_2\geq 0$ such that $(A|1,0|C)$ is the partition corresponding to $(m_1,m_2)$ in Example \ref{ex2param}.
\end{theorem}
\begin{proof}Applying Lemma \ref{lemNoMorecs} to $(A|B|C)$ and its dual, we find that there are $m_1,m_2\geq 0$ and $k_i = \frac{1}{3}(4^{m_i+1}-1)$ such that $(2|1,0|C_{k_1})^*\subset (A|B|C)$ and either $(2|1,0|C_{k_2})^*\subset (A|B|C)^*$ or $(A_{k_2}|1,0|{-1})^*\subset (A|B|C)^*$, according to whether $2\in A^*$ or $-1\in C^*$.  Without loss of generality, assume $k_1\leq k_2$.

If $2\in A^*$ then $(2|1,0|C_{k_2})^*\subset (A|B|C)^*$. Computing the dual of the Ulrich partition $(2|1,0|C_{k_2})$ of dimension $4^{m_2+1}+1$, we find $-4^{m_2+1}\in C^*$ and $4^{m_2+1}-2\in A^*$ (since $-4\in C_{k_2}$).  Using the containment $(2|1,0|C_{k_1})^*\subset (A|B|C)$, we also have that $-4^{m_1+1}\in C$ and $4^{m_1+1}-2\in A$.  Then the equality $$(4^{m_2+1}-2) - (4^{m_1+1}-2) = (-4^{m_1+1})-(-4^{m_2+1})$$ and the trapezoid rule give that $N = 4^{m_1+1}+4^{m_2+1}-1$ and $|C|=1$.  This is impossible: the dimension of the type $(k_1+k_2,2,1)$ is $4^{m_1+1}+4^{m_2+1}$, so no type $(k_3,2,1)$ can have dimension $4^{m_1+1}+4^{m_2+1}-1$ by considering congruences mod $3$.  Therefore it must be the case that $-1\in C^*$.

We now know that $(2|1,0|C_{k_1})^*\subset (A|B|C)$ and $(A_{k_2}|1,0|{-1})^*\subset (A|B|C)^*$.  Assume both of these containments are proper, since otherwise we are in the case of Proposition \ref{propn21}.  Think of building $(A|B|C)$ from $(\emptyset|B|\emptyset)$ by adding $a$'s and $c$'s.  After the word corresponding to $(2|1,0|C_{k_1})^*$ has been added, we must either add an $a$ or a $c$.

\emph{Case 1:} Suppose the next letter which is added is an $a$.  We claim that the trapezoid rule implies that $(A|B|C)$ is the partition of Example \ref{ex2param} corresponding to the integers $m_1,m_2$.  The Ulrich subpartition $(2|1,0|C_{k_1})^*$ has dimension $4^{m_1+1}+1$, so the new $a$ is added at time $4^{m_1+1}+2$.  The element of $C$ arising from the inclusion $(2|1,0|C_{k_1})^*\subset (A|B|C)$ is $-4^{m_1+1}\in C$, so the new $a$ is at position $2$ at time $4^{m_1+1}+2$, and thus $4^{m_1+1}+4\in A$.  On the other hand, the inclusion $(A_{k_2}|1,0|{-1})^* \subset (A|B|C)^*$ gives $4^{m_2+1}+1 \in A^*$ and $5-(4^{m_2+1}+2)=-4^{m_2+1}+3\in C^*$.  We have $$(4^{m_2+1}+1)-(4^{m_1+1}+4)=(-4^{m_1+1})-(-4^{m_2+1}+3)$$ so by the trapezoid rule $N = 4^{m_1+1}+4^{m_2+1}$ and $|C|=1$.  Therefore $(A|B|C)$ is the Example  \ref{ex2param}.

\emph{Case 2:} Suppose the next letter which is added is a $c$.  We claim that this is impossible: there are no such Ulrich partitions.  We focus solely on the Ulrich partition $(A|B|C) := (2|1,0|C_{k_1})^*,$ as the contradiction arises from this initial segment and not from ``global'' considerations given by the trapezoid rule.  If $m_1 = 0$ then Lemma \ref{lem-stupid} gives $|A| = 1$, a contradiction, so we assume $m_1\geq 1$. Let $(A|B|C\cup \{c_2\})$ be the partition obtained by adding $c_2$ at time $t_0:=4^{m_1+1} + 2$.  Writing $A = \{a_1,\ldots,a_{k_1}\}$ in increasing order, we have $c_2(t_0) =a_1 ( t_0) = -4^{m_1+1}$ since  $a_1 = 2$. For $1\leq i\leq 4^{m_1}$ we have $a_i = 2i$, so $c_2$ meets the $A$-block for all times $t\in [t_0,t_0+4^{m_1}-1]$.

At time $t_0+4^{m_1}$ there is no intersection yet.  It is not possible to add a new $c$ at this time.  If we were to add some $c_3$ at time $t_0+4^{m_1}$ then this would provide intersections between $c_3$ and the $A$-block for the next $4^{m_1}$ times.  However, $a_{4^{m_1}+1}  = a_{4^{m_1}}+4^{m_1}+2$ meets $c_2$ at time $t_0+4^{m_1}+2\cdot 4^{m_1-1}$, which is a time excluded by the intersection of $c_3$ with the $A$-block.  Thus, at time $t_0+4^{m_1}$ we must add a new $a$, call it $a_{k_1+1}$.  It has $$a_{k_1+1}(t_0+4^{m_1}) = c_1(t_0+4^{m_1})$$ so $$a_{k_1+1} = -4^{m_1+1}+2(t_0+4^{m_1})=t_0+2\cdot 4^{m_1}+2.$$

By the same argument as in the last paragraph, no new $c$'s can be added before  the time $t_1$ when $c_2$ and $a_{k_1}$ meet.  At each time in $[t_0,t_1]$ where $c_2$ does not meet the $A$-block, a new $a$ \emph{must} be added.  This statement amounts to the claim that $a_{k_1+1}$ does not meet the $B$-block before time $t_1$.  Since $a_{k_1}(t_0)=-4$ and $c_2(t_0) = -4^{m_1+1}$, we have $t_1 =t_0+2\cdot 4^{m_1}-2$.  Thus $a_{k_1+1}$ meets the $B$-block at times $t_1+3$ and $t_1+4$.

Finally, we obtain our contradiction at time $t_1+1$.  No intersection has been scheduled yet.  We cannot add some $c_3$ at this time, since it would still meet the $A$-block at time $t_1+3$ when $a_{k_1+1}$ meets $1\in B$.  Thus we must add an $a$, call it $a'$, at time $t_1+1$.  We have $a'(t_1+1)=c_1(t_1+1)$, so $$a'(t_0) = a'(t_1+1)+(t_1+1-t_0)= -4^{m_1+1}+2(t_1+1)-t_0= 4^{m_1+1}$$ while $c_2(t_0) = -4^{m_1+1}$.  Thus, at time $t_0+4^{m_1+1}$, all three of $a'$, $0\in B$, and $c_2$ coincide.  Therefore, the partition $(A|B|C\cup \{c_2\})$ cannot be extended to an Ulrich partition by adding $a$'s and $c$'s.
\end{proof}

\section{Ulrich partitions of type $(2,n,1)$}\label{sec-1n2}
In this section, we classify Ulrich partitions of type $(2,n,1)$.  Throughout the section, we consider Ulrich partitions of the form $P = (a_1,a_2|b_1,\ldots,b_n|c)$, and normalize the evolution of the partition to subtract $1,0,-1$ from the blocks, as in \S\ref{sec-beta2}.  We further normalize the positions $$a_2 = y, \quad a_1 = y+2m,\quad \mbox{and}\quad  c=-y,$$ so that the intersection $a_2c$ happens at time $y$ and the gap between the $a$-entries is $2m$.  In our classification we will view $m$ as being a fixed parameter, similarly to how $b_1-b_2$ was a fixed parameter in the classification of Ulrich partitions of type $(\alpha,2,\gamma)$.

\begin{definition}
For any $m\geq 2$, the {\em fundamental pattern $F_{m}$ of type $m$} is the partition of type $(2, m-1,1)$ given by  $$F_{m}=(3m, m | m-1, m-2, \dots, 2, 1 | {-m}).$$ The fundamental pattern is Ulrich by Example \ref{example-1n2}.
\end{definition}

\begin{definition}
The {\em elongation} of a partition $$P=(a_1, a_2|b_1, \dots, b_n|c) = (y+2m,y|b_1,\ldots,b_n|{-y})$$ of type $(2,n,1)$ is the partition $E(P)$ of type $(2, n + 2m, 1)$ obtained by adding two contiguous blocks of length $m$ at the beginning and end of the $b$ sequence and shifting the $a$ and $c$ entries as follows: $$(y + 5m, y + 3m | y + 3m-1, \dots, y + 2m, b_1, \dots, b_n,- y-m, \dots, -y-2m+1 |{-y- 3m}).$$ The $k$th elongation of $P$ is defined inductively by $E^k(P) := E(E^{k-1}(P))$ and $E^0(P)=P$; it has type $(2,n+2mk,1)$.
\end{definition}

\begin{example}
The fundamental pattern of type 2 is the partition $F_2= (6,2|1|{-2})$. Its first and second elongations are $$E(F_2)= (12, 8|7, 6, 1, -4,-5|{-8}) \quad E^2(F_2) = (18, 14|13,12, 7,6,1,-4,-5, -10,-11  |{ -14}).$$
The fundamental pattern of type 3 is the partition $F_3 = ( 9,3|2,1|{-3})$. Its first elongation is  $$E(F_3)= (18, 12|11,10,9,2, 1, -6,-7,-8|{-12}).$$ See Figure \ref{fig-251} for the time evolution diagram of $E(F_2)$. \begin{figure}[htbp]
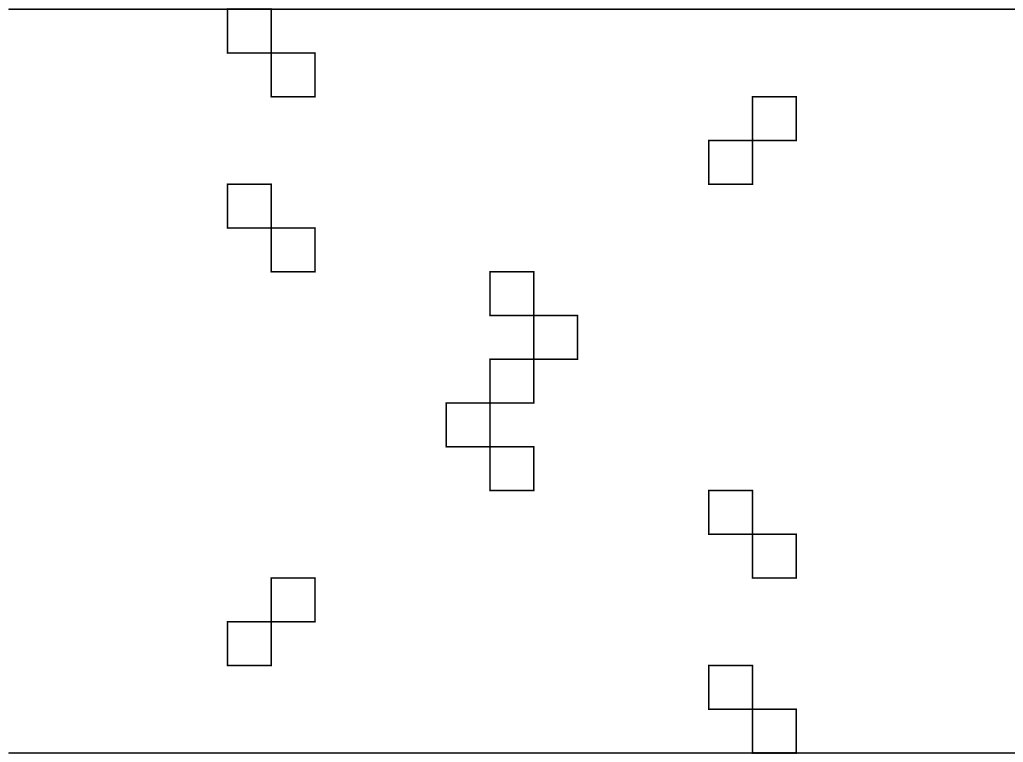
\caption{Time evolution diagram of the partition $E(F_2)$ of type $(2,5,1)$.}\label{fig-251}
\end{figure}
\end{example}

\begin{remark}
We will also need a degenerate case of the previous definitions.  We define the fundamental pattern $F_1$ to be the partition $(3,1|\emptyset|{-1})$. Its elongation $E(F_1) = (6,4|3,-2|{-4})$ still makes sense.  Observe that $E(F_1)$ is the Ulrich partition of Example \ref{ex-221}. To avoid discussing trivialities in the arguments that follow we generally focus on the $m\geq 2$ case and assure the reader that appropriate arguments work in the $m=1$ case.
\end{remark}

The main theorem in this section is the following.

\begin{theorem}\label{thm-2n1}
A partition $P$ of type $(2,n,1)$ is Ulrich if and only if there exists $k \geq 0$ and $m >0$ such that $n = 2mk + m -1$ and $P=E^k(F_{m})$.
\end{theorem}

The partitions $E^k(F_m)$ are clearly all distinct, so Theorem \ref{thm-2n1} implies Theorem \ref{thm-1n2intro}.  First we observe that the partitions $E^k(F_m)$ are indeed Ulrich.

\begin{lemma}\label{lem-2n1}
If $P$ is an Ulrich partition of type $(2,n,1)$ and $P$ has dimension $2y+m-1$ then $E(P)$ is an Ulrich partition of dimension $2y'+m-1$, where $y' = y+3m$.  In particular, the partition $E^k(F_m)$  is Ulrich of type $(2,2mk+m-1,1)$.\end{lemma}
\begin{proof}
The beginning and ending intersections in $E(P)$ all occur between $a$'s or $c$'s and the new contiguous blocks of $b$'s as follows.
\begin{itemize}
\item For $t\in [1,m]$, $a_2$ meets the left new $B$-block.
\item For $t\in (m,2m]$, $c$ meets the right new $B$-block.
\item For $t\in (2m,3m]$, $a_1$ meets the left new $B$-block.
\item For $t\in [2y+4m,2y+5m)$, $a_2$ meets the right new $B$-block.
\item For $t\in [2y+5m,2y+6m)$, $c$ meets the left new $B$-block.
\item For $t\in [2y+6m,2y+7m)$, $a_1$ meets the right new $B$-block.
\end{itemize}
Note that $P$ can be obtained from $E(P)$ by shifting to the time $3m$ position $(E(P))(3m)$ and throwing out the two new $B$ blocks.  Since $P$ is Ulrich of dimension $2y+m-1$, we conclude that there are unique intersections in $E(P)$ at all times $t\in (3m,2y+4m)$.  Clearly $\dim E(P) = \dim P + 6m=2y+7m-1$, so $E(P)$ is Ulrich.

For the second statement, it suffices to observe that the fundamental partition $F_m$ satisfies the equality $\dim F_m = 2y+m-1=3m-1$, which is clear.
\end{proof}

\begin{observation}\label{obs-2n1}
By Lemma \ref{lem-2n1}, if $P=(y+2m,y|b_1,\ldots,b_n|{-y})$ is of the form $E^k(F_m)$ then it satisfies the formula $$\dim P  = 2y+m-1.$$ For any $P$, we say that it satisfies the \emph{dimension formula} if the above equality holds.  Theorem \ref{thm-2n1} in particular claims that the dimension formula holds for any Ulrich partition of type $(2,n,1)$.

We also observe that Theorem \ref{thm-2n1} implies that for any Ulrich $P$ the sequence $b_1,\ldots,b_n$ begins with a contiguous block $y-1,\ldots,y-l$ of length exactly $l$, where $l$ is either $m$ or $m-1$ depending on whether $k>0$ or $k=0$ in the equality $P = E^k(F_m)$.
\end{observation}

The next lemma will form the base of an induction to prove Theorem \ref{thm-2n1}.

\begin{lemma}\label{lem-2n1base}
Let $P=(a_1,a_2|b_1,\ldots,b_n|c) = (y+2m,y|b_1,\ldots,b_n|{-y})$ be an Ulrich partition.  If $y\leq 2m$ then $P = F_m$.
\end{lemma}
\begin{proof}
The intersection $a_2c$ occurs before the intersection $a_1b_1$.  It follows that the partition  $(y|b_1,\ldots,b_n|{-y})$ is Ulrich of type $(1,n,1)$.  Recalling the classification of such partitions, the only possibility is that $n=y-1$, $a_1$ meets $c$ at time $2y$, and $(b_1,\ldots,b_n) = (y-1,\ldots,1)$.
\end{proof}

On the other hand, if $P$ is too large to be treated by Lemma \ref{lem-2n1base} then we show that it is an elongation of a smaller partition.  The next lemma completes the proof of Theorem \ref{thm-2n1}.

\begin{lemma}
Let $P = (a_1,a_2|b_1,\ldots,b_n|c) = (y+2m,y|b_1,\ldots,b_n|{-y})$ be an Ulrich partition.  If $y>2m$ then there is some Ulrich partition $P'$ of type $(2,n',1)$ with $E(P')=P$.
\end{lemma}
\begin{proof}
Inducting on $n$, by Lemma \ref{lem-2n1base} we may assume that any Ulrich partition of the form $$P'=(y'+2m,y'|b'_1,\ldots,b'_{n'}|{-y'})$$ with $n'<n$ is equal to $E^k(F_m)$ for some $k$.  In particular, $P'$ satisfies the dimension formula $$\dim P' = 2y' + m-1$$ and the $(b')$'s begin with a contiguous block $y'-1,\ldots,y'-l$ of length exactly $m$ or $m-1$.

\emph{Claim 1:} in the partition $P$ the time $t=1$\ intersection is $a_2b_1$, so $b_1 = y-1$.  Suppose this is not the case.  Then $b_n = -y+1$, and $a_1$ meets $b_n$ at time $2y+2m-1$.  By time $2y$ the $a_2$ and $c_1$ entries have already passed through the $B$-block, so all intersections for times $t\in [2y,2y+2m)$ must occur between $a_1$ and the $B$-block.  This gives that a contiguous block $B'=\{-y+2m,\ldots,-y+1\}\subset B$ of length $2m$ occurs in the $B$-block.  Since $a_2$ meets $B'$ for times in $[2y-2m,2y)$ and $c$ meets $B'$ for times in $[1,2m]$, it follows that the partition $$P'=(y+m,y-m|B\setminus B'|{-y+m})$$ is Ulrich.  By induction, $B\sm B'$ starts with a contiguous block $\{y-m-1,\ldots,y-m-l\}$ of length exactly $l \in \{m,m-1\}$.  In $P'$ the intersection at time $l+1$ must be between $c$ and some $b_0\in B\setminus B'$.  This is a contradiction, since in $P$ we find that $a_1$ meets $b_0$ at the same time as $a_2$ meets an entry of $B'$.  Therefore $b_1=y-1$.

Having established the claim, let $m_1\geq 1$ be the largest integer such that the contiguous block $B' = \{y-1,\ldots,y-m_1\}\subset B$.  Clearly $m_1\leq 2m$, since otherwise $a_1$ and $a_2$ would both intersect $B'$ at the same time.  An argument similar to the previous paragraph shows that in fact we must have $m_1<2m$.  At time $m_1+1$ the intersection must be $b_nc$; let $m_2\geq 1$ be the largest integer such that the contiguous block $B''=\{-y+m_1+m_2,\ldots,-y+m_1+1\}\subset B$.  Observe that $$\dim P = 2y+2m-m_1-1$$ since the last intersection is $a_1b_n$, so $P$ satisfies the dimension formula if and only if $m_1=m$; our eventual goal is to show that $m_1=m_2 = m$.

\emph{Claim 2: $m_1+m_2 = 2m$.} Let $t\in (m_1,2m]$.  When $a_1$ is at position $-y+t$, both $a_2$ and $c$ have finished intersecting the $B$-block.  Thus $-y+t\in B''$ for all $t\in (m_1,2m]$, and so $m_1+m_2\geq 2m$.  On the other hand, at time $t=2m+1$ we have an intersection $a_1b_1$, so $-y+2m+1\notin B''$.  Thus $m_1+m_2= 2m$.

\emph{Claim 3: $y> 2m+m_1$.} By assumption $y>2m$.  If $t\in (2m,2m+m_1]$ then $a_1$ meets $B'$ at time $t$, so it is not possible for the intersection $a_2c$ to happen at such a time.  Thus $y>2m+m_1$.

\emph{Claim 4: $m_1=m_2=m$.} Consider the partition $$P' = (y-m_1,y-2m-m_1|B\sm(B'\cup B'')|{-y+2m+m_1})$$ obtained by looking at the time $2m+m_1$ evolution $P(2m+m_1)$ and removing the contiguous blocks $B',B''$.  This makes sense since $y>2m+m_1$, and $P'$ is Ulrich because $P$ is Ulrich.  By induction, $P'$ satisfies the dimension formula $$\dim P' = 2(y-2m-m_1)+m-1.$$ On the other hand, the type of $P'$ is $(2,n-2m,1)$, so $$\dim P = \dim P'+6m = 2y+3m-2m_1-1.$$ Comparing this with our earlier expression for $\dim P$ gives $m_1=m$ and so $m_2=m$ as well.  This implies $P = E(P')$.
\end{proof}

\section{Ulrich partitions of type $(2,n,2)$}\label{sec-2n2}
In this section, we classify Ulrich partitions of type $(2,n,2)$. Throughout the section, we consider Ulrich partitions of the form $(a_1, a_2| b_1, \dots, b_n| c_1 , c_2)$, and normalize the evolution of the partition to subtract $1, 0, -1$ from the blocks. By symmetry, we may as well assume $a_1-a_2>c_1-c_2$.  We further normalize the positions
$$a_1 = y + s + r, \quad a_2 = y, \quad c_1 = - y, \quad \mbox{and} \quad c_2 = -y -s.$$
The intersection $a_2 c_1$ occurs at time $y$, the gap between the $a$-entries is $s+r$, and the gap between the $c$-entries is $s$.

\begin{definition}
Let $P_{u}$ be the partition of type $(2, 2u,2)$ given by
$$P_{u} := (6 u +5, 2u+1 | 2 u, 2 u -2, \dots, 4, 2, -1, -3, \dots, - 2 u +3, - 2u +1 | {- 2u -1}, - 6 u -3).$$ We observe that the subpartition $(a_2|B|c_1)$ is the Ulrich partition of type $(1,2u,1)$ corresponding to the subset of $[2u]$ consisting of even numbers in the proof of Theorem \ref{thm-1n1}.
\end{definition}

\begin{example}
We have $$P_1 = (11, 3|2, -1|{-3}, -9),$$
$$P_2 = (17, 5|4,2, -1, -3|{-5}, -15).$$ The time evolution diagram of $P_1$ was given in Example \ref{ex222}.  The time diagram for $P_2$ is Figure \ref{fig-242}.
\begin{figure}[htbp]
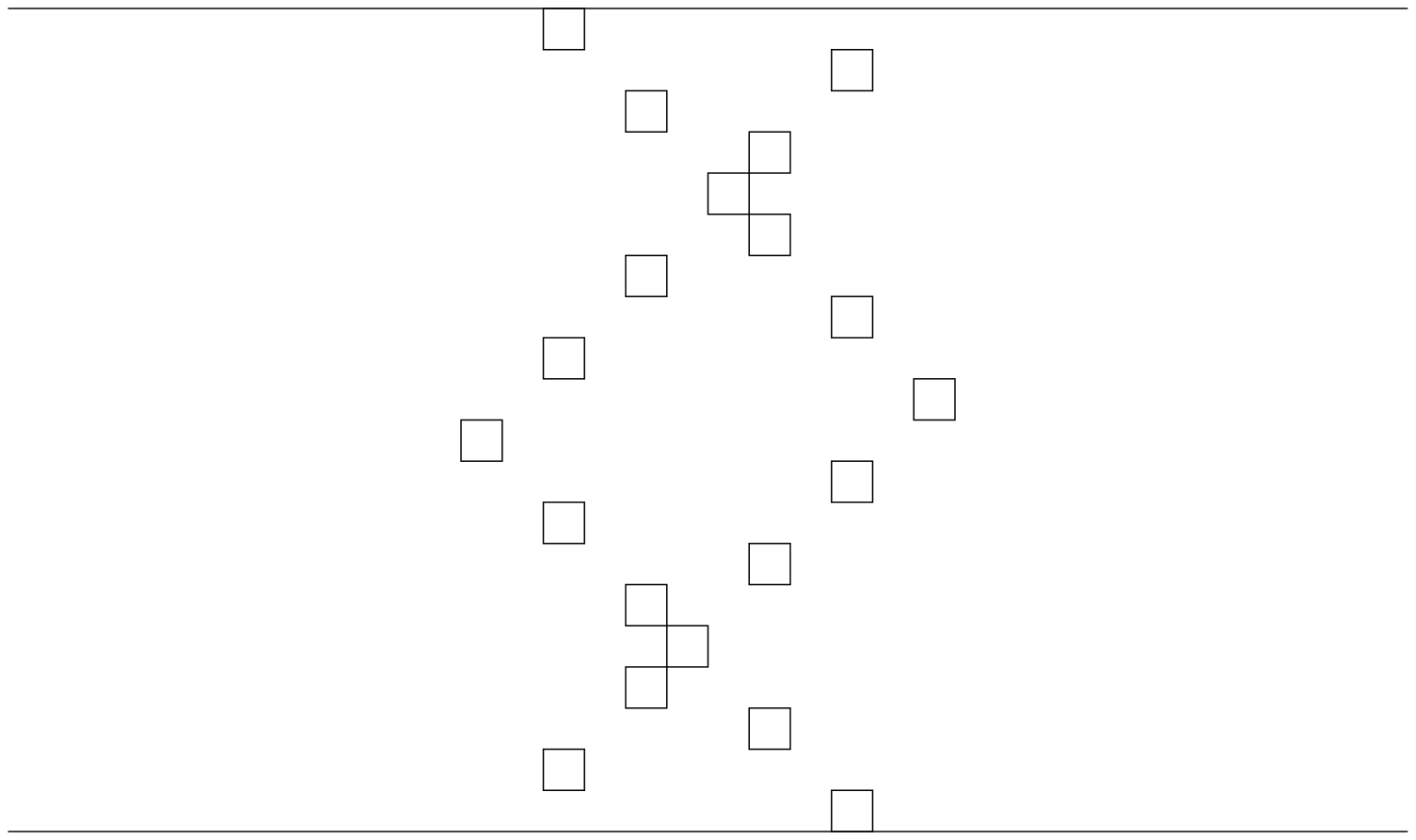
\caption{Time evolution diagram of the partition $P_2$ of type $(2,4,2)$.}\label{fig-242}
\end{figure}
\end{example}

The main theorem of this section asserts these are the only examples.

\begin{theorem}\label{thm-2n2}
If $P$ is an Ulrich partition of type $(2, n, 2)$, then $n= 2 u$ is even and up to symmetry $P= P_{u}$.
\end{theorem}

We first show that these examples are in fact Ulrich.

\begin{lemma}
The partition $P_{u}$ is Ulrich. In particular, every flag variety $F(2, 2n+2; 2n+4)$ admits a Schur bundle which is Ulrich for $\cO(1)$.
\end{lemma}
\begin{proof}
The dimension of $P_u$ is $N = 8u+4$.  For times $t\in[1,4u+1]$, there are intersections from the Ulrich subpartition $(a_2|B|c_1)$ of type $(1,2u,1)$.  At time $4u+2$ we have the intersection $a_2 c_2$.  As $P_u$ is its own symmetric dual, times in $[4u+3,8u+4]$ also have intersections.
\end{proof}

The plan of the proof of Theorem \ref{thm-2n2} is similar to the classification of partitions of type $(2,n,1)$ in the previous section.  We will first show that if $P$ is an Ulrich partition with $y\leq s$ then $P$ is a known example.  We next show that if instead $y>s$ then $P$ can be obtained from a shorter example by a process of elongation.  However, we will finally show that elongations of the known examples are never Ulrich; this final step is the primary difference from the strategy in the $(2,n,1)$ case, where such elongations were possible.

Before beginning the proof in earnest, we establish a couple of lemmas which are true for arbitrary Ulrich partitions of type $(2,n,2)$.  Let $P$ be an Ulrich partition of this type, normalized as in the first paragraph of the section.  (In particular, recall that $s+r = a_1-a_2>c_1-c_2 = s$.)

\begin{lemma}\label{lem-afirst}
In the Ulrich partition $P$, the intersection at time $t=1$ is $a_2  b_1$.
\end{lemma}

\begin{proof}
If not, then the intersection at time $t=1$ is of the form $b c_1$. Let $m\geq 1$ be the maximal number so that $\{-y+1,\ldots,-y+m\}\subset B$, so $-y+m+1\notin B$.  This implies that $y-i\notin B$ for $i\in [m]$.  Note that $m\leq s$.  At the time $t_0 = 2y+s+r-m-1$ when $a_1(t_0) = -y+m+1$ we have $c_2(t_0)=y-m-1+r\geq y-m$, so $c_2(t_0)\notin B$.  Furthermore, $c_1(t_0)=y+s+r-m-1>y$ and $a_1(t_0) = -y-s-r+m+1<-y$.  We conclude that there is no intersection at time $t_0$ even though $a_2$ has not yet passed through the $B$-block.
\end{proof}

Lemma \ref{lem-afirst} applied to the symmetric dual implies that the last intersection must be $a_1  b_n$.  From now on, we let $m\geq 1$ be the largest number such that $\{y-1,\ldots,y-m\}\subset B$.

\begin{lemma}\label{lem-boundm}
We have $m<\min\{r,y-1\}$.
\end{lemma}
\begin{proof}
First we show that $m<y-1$.  Clearly $m\leq y-1$ since the intersection $a_2c_1$ happens at time $y$.  If $m= y-1$ then $B$ is the contiguous block $\{y-1,\ldots,1\}$.  By Lemma \ref{lem-afirst} the last intersection in $P$ is between $a_1$ and $1\in B$, so $\dim P = s+r+y-1 = 4y$.  At time $2y$ we have $a_2(2y) = -y$ and $c_1(2y) = y$, so the intersection must be of the form $ac$; since $a_2-c_2< a_1-c_2$, it must be $a_2c_2$.  Thus $s=2y$, $r = y+1$, and $$P = (4y+1,y|y-1,\ldots,1|{-y},-3y).$$  But then $a_1$ and $c_2$ meet at position $(y+1)/2$, which is in $B$ if $y\geq 3$. Parity is violated if $y=2$, so we conclude $m<y-1$.

Next we show $m<r$.  If $m>r$ then the last intersection in $P$ is $b_1c_2$, contradicting Lemma \ref{lem-afirst}.  Suppose $m=r$.  At time $m+1$ the intersection is one of $bc_1$, $a_1b$, or $a_2c_1$.  If it is $a_2c_1$ then $m=y-1$ and we are done by the previous paragraph.  If the intersection is $bc_1$ then $-y+r+1\in B$ and $c_2$ meets $y-1\in B$ when $a_1$ meets $-y+r+1\in B$, a contradiction.  Finally, if the intersection is $a_1b$ then again the last intersection is of type $bc$, violating Lemma \ref{lem-afirst}.
\end{proof}

\begin{observation}\label{obs-dim2n2}
Combining Lemma \ref{lem-afirst} and $\ref{lem-boundm}$, we have $b_n = -y+m+1$ for any Ulrich partition $P$ of type $(2,n,2)$, and thus we have the dimension formula $$\dim P =2y+s+r-m-1.$$
\end{observation}

Next we establish the fact that if the intersection $a_2c_1$ happens before $a_1$ or $c_2$ meet the $B$-block then $P$ is a known example.

\begin{lemma}\label{lem-fundamental2n2}
Let $P$ be an Ulrich partition of type $(2,n,2)$ normalized as in this section.  If $y\leq s+m$ then $P = P_{\frac{s-2}{4}}$.
\end{lemma}
\begin{proof}
If $y \leq s+m$, then every intersection until time $y$ is of the form $a_2  b$ or $b c_1$. Hence, there must be a $b$ either at position $p$ or $-p$ for every $1 \leq p < y$. Consequently, the total number of $b$ entries is $y-1$ and for $t \in (y, 2y)$ every intersection is also of the form $a_2 b$ or $b c_1$. By the dimension formula, $4y = 2y+s+r-m-1$, so $2y = s+r-m-1$.

At time $t=2y$, $a_2$ and $c_1$ are at positions $-y$ and $y$, respectively. They cannot intersect a $b$. The intersection  $a_2 c_2$ happens before $a_1 c_1$, so at time $t = 2 y$ the intersection is $a_2 c_2$ and $s=2y$. At time $2y+1$, $c_2$ cannot intersect a $b$, hence the intersection must be  $a_1 c_1$. We conclude that $r=2$, so $m = 1$.

We now inductively determine the $B$-block.  We have $y-1,-y+2\in B$.  When $a_2$ is at position $-y+k$ with $3\leq k<y$ the entry $c_1$ is at position $y-k+2$, while $c_2$ and $a_1$ have already intersected all other entries.  By induction we may assume $y-k+2\in B$ iff $k$ is odd; therefore $-y+k\notin B$ iff $k$ is odd and $y-k\in B$ iff $k$ is odd.

Finally, note that $y$ is odd. Equivalently, $2\in B$. Otherwise, at time $t = 3y$, $c_2$ is at position $0$ and $a_2$ is at position $2$ and there is no intersection. Therefore $s\equiv 2 \pmod 4$, and $P = P_{\frac{s-2}{4}}$.
\end{proof}

We next argue that any Ulrich partition of type $(2,n,2)$ must be obtained from some $P_u$ by a process of elongation. Finally, we will see that $P_u$ cannot be elongated.

Given a partition $P$ at time $t=0$, we will view it as three sequences $A B C$, where $A$ is the sequence of $a$'s and blank spaces $a_1 \times \cdots \times a_2$, $C$ is the sequence of $c$'s and blank spaces $c_1 \times \cdots \times \ c_2$ and $B$ is the sequence of $b$'s and blank spaces in between $a_2$ and $c_1$. The partition $P$ is a concatenation of these three. Given a contiguous pattern $\Psi$ of $b$'s and blank spaces at positive positions, let $\Psi^c$ be the complementary pattern which has a $b$ at position $p$ if $\Psi$ does not have a $b$ at position $-p$ and vice versa. Let $\ell(\Psi)$ be the length of $\Psi$. Let $X_\ell$ denote a contiguous sequence of blank spaces of length $\ell$.

\begin{example}
If $\Psi = b \times b \ b \ \times$, then $\Psi^c = b \times \times  \ b \ \times $. Also, $X_3 = \times \times \times$.
\end{example}

\begin{definition}
Let $P=ABC$ be a partition and $\Psi$ a pattern of $b$'s and blanks. The \emph{$\Psi$-elongation of $P$} is the partition corresponding to the concatenation $$A \Psi X_{\ell(\Psi)} B X_{\ell(\Psi)} \Psi^c C$$ of patterns.
\end{definition}

\begin{example}
Continuing the previous example, if  $$P = a_1 \times \times \times \ a_2 \ b_1 \ b_2 \ \times \times \ b_3 \times \times \ c_1 \times c_2,$$ then the $\Psi$-elongation of $P$ is
$$ a_1 \times \times \times  a_2 \ b \times b\ b \times \times \times
\times \times \times \ b_1 \ b_2  \times \times \ b_3 \times \times \times \times \times \times \times  b \times \times \ b \times  c_1 \times c_2.$$
\end{example}

For integers $q,r>0$, we let $K(q, r,m)$ be the pattern of $b$'s consisting of $q$ iterations of a contiguous block of $m$ $b$'s followed by $r-m$ blanks.

\begin{example}
$K(2, 6, 2) = b \ b \times \times \times \times \ b \ b \times \times \times \times $.
\end{example}

Our final lemma easily implies Theorem \ref{thm-2n2}.

\begin{lemma}\label{lem-mainMeat}
Let $P$ be an Ulrich partition of type $(2,n,2)$, normalized as above. If $y > s+m$, then $r$ divides $s+m$; let the quotient be $q$.  Then $P$ is the $K(q, r,m)$-elongation of an Ulrich partition $P'$ of type $(2, n- s-m, 2)$.  Furthermore, the initial block of $b$'s in $P'$ also has length $m$.
\end{lemma}

\begin{proof}[Proof of Theorem \ref{thm-2n2} assuming Lemma \ref{lem-mainMeat}]
We have already classified the Ulrich partitions of type $(2,2,2)$ in \S \ref{sec-beta2}. By induction on $n$, suppose that for $n < n_0$ the only Ulrich partitions of type $(2, n, 2)$ are $P_{u}$ with $n=2u$. Let $P$ be an Ulrich partition of type $(2,n_0,2)$. By Lemma \ref{lem-fundamental2n2} we may assume $y>s+m$.  By Lemma \ref{lem-mainMeat}, $P$ is the $K(q, r,m)$-elongation of a smaller Ulrich partition $P'$. By induction and Lemma \ref{lem-mainMeat}, we must have $r=2$ and $m=1$. However, $s$ is even, so $r=2$ does not divide $s+m= s+1$.  This contradiction proves the theorem.
\end{proof}

\begin{proof}[Proof of Lemma \ref{lem-mainMeat}]
We can determine the pattern of intersections inductively until the time $s+m<y$ immediately before $c_2$ first meets the $B$-block.  Until this time, all intersections are of the form $a_1b$ or $bc_2$.

For the base of the induction, we first recall $\{y-1,\ldots,y-m\}\subset B$.  We know $-y+m+1\in B$ by Lemma \ref{lem-boundm}, and $m<r$.  Suppose $v\in (m,r]$.  The only possible intersection at time $t_0$ when $a_1(t_0) = -y+v$ is $a_1b$ since $a_2(t_0)=-y+v-r-s < -y$ and $c_2(t_0)=y-v+r\geq y$.  Therefore $-y+v\in B$ for $v\in (m,r]$.

We now continue by induction until we reach time $s+m$. The positions $y - (h-1)r -v$ have a $b$ for $1 \leq v \leq m$ by induction. Consequently, the positions $-y+ hr + v$ cannot have $b$'s for $1 \leq v \leq m$. Otherwise, when $c_2$ is at position $y- (h-1)r -v$, $a_1$ would be at position $-y + hr + v$ giving a coincident intersection. Thus there is a contiguous $B$-block of length $m$ at positions $y - hr - v$ for $1 \leq v \leq m$. Similarly, there are no $b$ entries in positions $y - (h-1)r - v$ for $m+1 \leq v \leq r$ by induction. When $c_2$ is at these positions $a_1$ is at the positions $-y + hr + v$. Since $a_1b$ is the only possible intersection, there must be a contiguous $B$-block of length $r-m$ at these positions.

\emph{Claim: $r|s+m$.}  Write $s+m = qr+j$ with $0\leq j<r$ as in the division algorithm.  What we have shown so far is that the pattern of $b$'s and blanks in the interval $B \cap (y,y-s-m]$ is the truncation of $K(q,r,m)$ to a sequence of length $s+m$.  Similarly, the pattern of $b$'s and blanks in $B\cap [-y+s+m,-y)$ is the truncation of $K(q,r,m)^c$ to a sequence of length $s+m$.  As $\ell(K(q,r,m)) = qr$, the claim is that no truncation actually takes place.  There are two cases to consider depending on the remainder $j$.

\emph{Case 1: $1\leq j\leq m$}.  In this case, $y-s-m\in B$.  When $c_2$ is at position $y-s-m$, we find that $c_1$ is at position $y-m\in B$, a contradiction.

\emph{Case 2: $m< j <r$}. Consider the time $t_0$ when $a_1(t_0) = -y+s+m+1$. Then $c_2(t_0) = y-s-m-1+r\notin B$, and $c_1(t_0)\geq y$, $a_2(t_0)\leq -y$ hold.  Furthermore, $-y+s+m+1\notin B$, since $c_1$ is at this position when $c_2$  is at $-y+m+1\in B$.  Thus there is no intersection at time $t_0$.

Therefore $r|s+m$.  Let us analyze the known intersections.  For times $t\in [1,s+m]$, the intersections are of type $a_2b$ or $bc_1$.  When $t\in (s+m,2s+2m]$, the intersections are all $a_1b$ or $bc_2$.   This gives $y-t,-y+t\notin B$ for all such times and $y>2s+2m$.  Dually, a symmetric description holds for the last $2s+2(r-m)$ times.  Evolving $P$ to time $2s+2m$ and throwing out all the $b$'s which have already met $a$'s and $c$'s, we arrive at an Ulrich partition $$P'=(y-s+r-2m,y-2s-2m|B'|{-y+2s+2m},-y+s+2m)$$  such that $P$ is the $K(q,r,m)$-elongation of $P'$; it is easy to see that $B'$ is nonempty.

Finally, we analyze the length $m'$ of the initial block of $b$'s in $P'$.  The dimension formula Observation \ref{obs-dim2n2} gives equalities \begin{align*}
 \dim P' &= 2(y-2s-2m)+r+s-m'-1\\
 \dim P &= 2y+r+s-m-1\\
 \dim P &= \dim P' + 4(s+m),
 \end{align*}
 from which $m=m'$ follows immediately.
\end{proof}

\bibliographystyle{plain}

\end{document}

%% file: 1n1Ex.pstex_t
\begin{picture}(0,0)%
\includegraphics{1n1Ex.eps}%
\end{picture}%
\setlength{\unitlength}{3552sp}%
\begingroup\makeatletter\ifx\SetFigFont\undefined%
\gdef\SetFigFont#1#2#3#4#5{%
  \reset@font\fontsize{#1}{#2pt}%
  \fontfamily{#3}\fontseries{#4}\fontshape{#5}%
  \selectfont}%
\fi\endgroup%
\begin{picture}(3777,3511)(1411,-3125)
\put(1426,-2761){\makebox(0,0)[b]{\smash{{\SetFigFont{11}{13.2}{\rmdefault}{\mddefault}{\updefault}{\color[rgb]{0,0,0}$9$}%
}}}}
\put(2026,-61){\makebox(0,0)[b]{\smash{{\SetFigFont{11}{13.2}{\rmdefault}{\mddefault}{\updefault}{\color[rgb]{0,0,0}$a$}%
}}}}
\put(2626,-61){\makebox(0,0)[b]{\smash{{\SetFigFont{11}{13.2}{\rmdefault}{\mddefault}{\updefault}{\color[rgb]{0,0,0}$b$}%
}}}}
\put(3226,-1261){\makebox(0,0)[b]{\smash{{\SetFigFont{11}{13.2}{\rmdefault}{\mddefault}{\updefault}{\color[rgb]{0,0,0}$a$}%
}}}}
\put(2626,-1861){\makebox(0,0)[b]{\smash{{\SetFigFont{11}{13.2}{\rmdefault}{\mddefault}{\updefault}{\color[rgb]{0,0,0}$b$}%
}}}}
\put(2626,-1561){\makebox(0,0)[b]{\smash{{\SetFigFont{11}{13.2}{\rmdefault}{\mddefault}{\updefault}{\color[rgb]{0,0,0}$b$}%
}}}}
\put(2626,-1261){\makebox(0,0)[b]{\smash{{\SetFigFont{11}{13.2}{\rmdefault}{\mddefault}{\updefault}{\color[rgb]{0,0,0}$b$}%
}}}}
\put(2626,-961){\makebox(0,0)[b]{\smash{{\SetFigFont{11}{13.2}{\rmdefault}{\mddefault}{\updefault}{\color[rgb]{0,0,0}$b$}%
}}}}
\put(2926,-961){\makebox(0,0)[b]{\smash{{\SetFigFont{11}{13.2}{\rmdefault}{\mddefault}{\updefault}{\color[rgb]{0,0,0}$a$}%
}}}}
\put(2626,-661){\makebox(0,0)[b]{\smash{{\SetFigFont{11}{13.2}{\rmdefault}{\mddefault}{\updefault}{\color[rgb]{0,0,0}$ab$}%
}}}}
\put(2626,-361){\makebox(0,0)[b]{\smash{{\SetFigFont{11}{13.2}{\rmdefault}{\mddefault}{\updefault}{\color[rgb]{0,0,0}$b$}%
}}}}
\put(2326,-361){\makebox(0,0)[b]{\smash{{\SetFigFont{11}{13.2}{\rmdefault}{\mddefault}{\updefault}{\color[rgb]{0,0,0}$a$}%
}}}}
\put(2626,-2761){\makebox(0,0)[b]{\smash{{\SetFigFont{11}{13.2}{\rmdefault}{\mddefault}{\updefault}{\color[rgb]{0,0,0}$b$}%
}}}}
\put(1426,-361){\makebox(0,0)[b]{\smash{{\SetFigFont{11}{13.2}{\rmdefault}{\mddefault}{\updefault}{\color[rgb]{0,0,0}$1$}%
}}}}
\put(1426,-661){\makebox(0,0)[b]{\smash{{\SetFigFont{11}{13.2}{\rmdefault}{\mddefault}{\updefault}{\color[rgb]{0,0,0}$2$}%
}}}}
\put(1426,-961){\makebox(0,0)[b]{\smash{{\SetFigFont{11}{13.2}{\rmdefault}{\mddefault}{\updefault}{\color[rgb]{0,0,0}$3$}%
}}}}
\put(1426,-1261){\makebox(0,0)[b]{\smash{{\SetFigFont{11}{13.2}{\rmdefault}{\mddefault}{\updefault}{\color[rgb]{0,0,0}$4$}%
}}}}
\put(1426,-1561){\makebox(0,0)[b]{\smash{{\SetFigFont{11}{13.2}{\rmdefault}{\mddefault}{\updefault}{\color[rgb]{0,0,0}$5$}%
}}}}
\put(1426,-1861){\makebox(0,0)[b]{\smash{{\SetFigFont{11}{13.2}{\rmdefault}{\mddefault}{\updefault}{\color[rgb]{0,0,0}$6$}%
}}}}
\put(1426,-2161){\makebox(0,0)[b]{\smash{{\SetFigFont{11}{13.2}{\rmdefault}{\mddefault}{\updefault}{\color[rgb]{0,0,0}$7$}%
}}}}
\put(1426,-2461){\makebox(0,0)[b]{\smash{{\SetFigFont{11}{13.2}{\rmdefault}{\mddefault}{\updefault}{\color[rgb]{0,0,0}$8$}%
}}}}
\put(1426,-61){\makebox(0,0)[b]{\smash{{\SetFigFont{11}{13.2}{\rmdefault}{\mddefault}{\updefault}{\color[rgb]{0,0,0}$t$}%
}}}}
\put(3826,-61){\makebox(0,0)[b]{\smash{{\SetFigFont{11}{13.2}{\rmdefault}{\mddefault}{\updefault}{\color[rgb]{0,0,0}$b$}%
}}}}
\put(4126,-61){\makebox(0,0)[b]{\smash{{\SetFigFont{11}{13.2}{\rmdefault}{\mddefault}{\updefault}{\color[rgb]{0,0,0}$b$}%
}}}}
\put(4726,-61){\makebox(0,0)[b]{\smash{{\SetFigFont{11}{13.2}{\rmdefault}{\mddefault}{\updefault}{\color[rgb]{0,0,0}$b$}%
}}}}
\put(5026,-61){\makebox(0,0)[b]{\smash{{\SetFigFont{11}{13.2}{\rmdefault}{\mddefault}{\updefault}{\color[rgb]{0,0,0}$c$}%
}}}}
\put(4726,-361){\makebox(0,0)[b]{\smash{{\SetFigFont{11}{13.2}{\rmdefault}{\mddefault}{\updefault}{\color[rgb]{0,0,0}$bc$}%
}}}}
\put(4126,-361){\makebox(0,0)[b]{\smash{{\SetFigFont{11}{13.2}{\rmdefault}{\mddefault}{\updefault}{\color[rgb]{0,0,0}$b$}%
}}}}
\put(3826,-361){\makebox(0,0)[b]{\smash{{\SetFigFont{11}{13.2}{\rmdefault}{\mddefault}{\updefault}{\color[rgb]{0,0,0}$b$}%
}}}}
\put(3826,-661){\makebox(0,0)[b]{\smash{{\SetFigFont{11}{13.2}{\rmdefault}{\mddefault}{\updefault}{\color[rgb]{0,0,0}$b$}%
}}}}
\put(4126,-661){\makebox(0,0)[b]{\smash{{\SetFigFont{11}{13.2}{\rmdefault}{\mddefault}{\updefault}{\color[rgb]{0,0,0}$b$}%
}}}}
\put(4426,-661){\makebox(0,0)[b]{\smash{{\SetFigFont{11}{13.2}{\rmdefault}{\mddefault}{\updefault}{\color[rgb]{0,0,0}$c$}%
}}}}
\put(4126,-961){\makebox(0,0)[b]{\smash{{\SetFigFont{11}{13.2}{\rmdefault}{\mddefault}{\updefault}{\color[rgb]{0,0,0}$bc$}%
}}}}
\put(3826,-961){\makebox(0,0)[b]{\smash{{\SetFigFont{11}{13.2}{\rmdefault}{\mddefault}{\updefault}{\color[rgb]{0,0,0}$b$}%
}}}}
\put(4126,-1261){\makebox(0,0)[b]{\smash{{\SetFigFont{11}{13.2}{\rmdefault}{\mddefault}{\updefault}{\color[rgb]{0,0,0}$b$}%
}}}}
\put(3826,-1261){\makebox(0,0)[b]{\smash{{\SetFigFont{11}{13.2}{\rmdefault}{\mddefault}{\updefault}{\color[rgb]{0,0,0}$bc$}%
}}}}
\put(4126,-1561){\makebox(0,0)[b]{\smash{{\SetFigFont{11}{13.2}{\rmdefault}{\mddefault}{\updefault}{\color[rgb]{0,0,0}$b$}%
}}}}
\put(3826,-2161){\makebox(0,0)[b]{\smash{{\SetFigFont{11}{13.2}{\rmdefault}{\mddefault}{\updefault}{\color[rgb]{0,0,0}$b$}%
}}}}
\put(3826,-2461){\makebox(0,0)[b]{\smash{{\SetFigFont{11}{13.2}{\rmdefault}{\mddefault}{\updefault}{\color[rgb]{0,0,0}$b$}%
}}}}
\put(4126,-2461){\makebox(0,0)[b]{\smash{{\SetFigFont{11}{13.2}{\rmdefault}{\mddefault}{\updefault}{\color[rgb]{0,0,0}$b$}%
}}}}
\put(4126,-2761){\makebox(0,0)[b]{\smash{{\SetFigFont{11}{13.2}{\rmdefault}{\mddefault}{\updefault}{\color[rgb]{0,0,0}$b$}%
}}}}
\put(3826,-2761){\makebox(0,0)[b]{\smash{{\SetFigFont{11}{13.2}{\rmdefault}{\mddefault}{\updefault}{\color[rgb]{0,0,0}$b$}%
}}}}
\put(4726,-661){\makebox(0,0)[b]{\smash{{\SetFigFont{11}{13.2}{\rmdefault}{\mddefault}{\updefault}{\color[rgb]{0,0,0}$b$}%
}}}}
\put(4726,-961){\makebox(0,0)[b]{\smash{{\SetFigFont{11}{13.2}{\rmdefault}{\mddefault}{\updefault}{\color[rgb]{0,0,0}$b$}%
}}}}
\put(4726,-1261){\makebox(0,0)[b]{\smash{{\SetFigFont{11}{13.2}{\rmdefault}{\mddefault}{\updefault}{\color[rgb]{0,0,0}$b$}%
}}}}
\put(4726,-1561){\makebox(0,0)[b]{\smash{{\SetFigFont{11}{13.2}{\rmdefault}{\mddefault}{\updefault}{\color[rgb]{0,0,0}$b$}%
}}}}
\put(4726,-1861){\makebox(0,0)[b]{\smash{{\SetFigFont{11}{13.2}{\rmdefault}{\mddefault}{\updefault}{\color[rgb]{0,0,0}$b$}%
}}}}
\put(4726,-2161){\makebox(0,0)[b]{\smash{{\SetFigFont{11}{13.2}{\rmdefault}{\mddefault}{\updefault}{\color[rgb]{0,0,0}$b$}%
}}}}
\put(3226,-1861){\makebox(0,0)[b]{\smash{{\SetFigFont{11}{13.2}{\rmdefault}{\mddefault}{\updefault}{\color[rgb]{0,0,0}$c$}%
}}}}
\put(2926,-2161){\makebox(0,0)[b]{\smash{{\SetFigFont{11}{13.2}{\rmdefault}{\mddefault}{\updefault}{\color[rgb]{0,0,0}$c$}%
}}}}
\put(2626,-2161){\makebox(0,0)[b]{\smash{{\SetFigFont{11}{13.2}{\rmdefault}{\mddefault}{\updefault}{\color[rgb]{0,0,0}$b$}%
}}}}
\put(2626,-2461){\makebox(0,0)[b]{\smash{{\SetFigFont{11}{13.2}{\rmdefault}{\mddefault}{\updefault}{\color[rgb]{0,0,0}$bc$}%
}}}}
\put(3526,-1561){\makebox(0,0)[b]{\smash{{\SetFigFont{11}{13.2}{\rmdefault}{\mddefault}{\updefault}{\color[rgb]{0,0,0}$ac$}%
}}}}
\put(4126,-1861){\makebox(0,0)[b]{\smash{{\SetFigFont{11}{13.2}{\rmdefault}{\mddefault}{\updefault}{\color[rgb]{0,0,0}$b$}%
}}}}
\put(3826,-1861){\makebox(0,0)[b]{\smash{{\SetFigFont{11}{13.2}{\rmdefault}{\mddefault}{\updefault}{\color[rgb]{0,0,0}$ab$}%
}}}}
\put(4126,-2161){\makebox(0,0)[b]{\smash{{\SetFigFont{11}{13.2}{\rmdefault}{\mddefault}{\updefault}{\color[rgb]{0,0,0}$ab$}%
}}}}
\put(3826,-1561){\makebox(0,0)[b]{\smash{{\SetFigFont{11}{13.2}{\rmdefault}{\mddefault}{\updefault}{\color[rgb]{0,0,0}$b$}%
}}}}
\put(2326,-2761){\makebox(0,0)[b]{\smash{{\SetFigFont{11}{13.2}{\rmdefault}{\mddefault}{\updefault}{\color[rgb]{0,0,0}$c$}%
}}}}
\put(2026,-3061){\makebox(0,0)[b]{\smash{{\SetFigFont{11}{13.2}{\rmdefault}{\mddefault}{\updefault}{\color[rgb]{0,0,0}$c$}%
}}}}
\put(4426,-2461){\makebox(0,0)[b]{\smash{{\SetFigFont{11}{13.2}{\rmdefault}{\mddefault}{\updefault}{\color[rgb]{0,0,0}$a$}%
}}}}
\put(4726,-2461){\makebox(0,0)[b]{\smash{{\SetFigFont{11}{13.2}{\rmdefault}{\mddefault}{\updefault}{\color[rgb]{0,0,0}$b$}%
}}}}
\put(4726,-2761){\makebox(0,0)[b]{\smash{{\SetFigFont{11}{13.2}{\rmdefault}{\mddefault}{\updefault}{\color[rgb]{0,0,0}$ab$}%
}}}}
\put(2626,-3061){\makebox(0,0)[b]{\smash{{\SetFigFont{11}{13.2}{\rmdefault}{\mddefault}{\updefault}{\color[rgb]{0,0,0}$b$}%
}}}}
\put(3826,-3061){\makebox(0,0)[b]{\smash{{\SetFigFont{11}{13.2}{\rmdefault}{\mddefault}{\updefault}{\color[rgb]{0,0,0}$b$}%
}}}}
\put(4126,-3061){\makebox(0,0)[b]{\smash{{\SetFigFont{11}{13.2}{\rmdefault}{\mddefault}{\updefault}{\color[rgb]{0,0,0}$b$}%
}}}}
\put(4726,-3061){\makebox(0,0)[b]{\smash{{\SetFigFont{11}{13.2}{\rmdefault}{\mddefault}{\updefault}{\color[rgb]{0,0,0}$b$}%
}}}}
\put(5026,-3061){\makebox(0,0)[b]{\smash{{\SetFigFont{11}{13.2}{\rmdefault}{\mddefault}{\updefault}{\color[rgb]{0,0,0}$a$}%
}}}}
\put(2026,239){\makebox(0,0)[b]{\smash{{\SetFigFont{11}{13.2}{\rmdefault}{\mddefault}{\updefault}{\color[rgb]{0,0,0}$5$}%
}}}}
\put(3526,239){\makebox(0,0)[b]{\smash{{\SetFigFont{11}{13.2}{\rmdefault}{\mddefault}{\updefault}{\color[rgb]{0,0,0}$0$}%
}}}}
\put(5026,239){\makebox(0,0)[b]{\smash{{\SetFigFont{11}{13.2}{\rmdefault}{\mddefault}{\updefault}{\color[rgb]{0,0,0}$-5$}%
}}}}
\end{picture}%

%% file: 132Ex.pstex_t
\begin{picture}(0,0)%
\includegraphics{132Ex.eps}%
\end{picture}%
\setlength{\unitlength}{2763sp}%
\begingroup\makeatletter\ifx\SetFigFont\undefined%
\gdef\SetFigFont#1#2#3#4#5{%
  \reset@font\fontsize{#1}{#2pt}%
  \fontfamily{#3}\fontseries{#4}\fontshape{#5}%
  \selectfont}%
\fi\endgroup%
\begin{picture}(6777,4111)(436,-3650)
\put(4351,314){\makebox(0,0)[b]{\smash{{\SetFigFont{8}{9.6}{\rmdefault}{\mddefault}{\updefault}{\color[rgb]{0,0,0}$0$}%
}}}}
\put(3751,-1486){\makebox(0,0)[b]{\smash{{\SetFigFont{8}{9.6}{\rmdefault}{\mddefault}{\updefault}{\color[rgb]{0,0,0}$ab$}%
}}}}
\put(4051,-1786){\makebox(0,0)[b]{\smash{{\SetFigFont{8}{9.6}{\rmdefault}{\mddefault}{\updefault}{\color[rgb]{0,0,0}$ab$}%
}}}}
\put(4351,-2086){\makebox(0,0)[b]{\smash{{\SetFigFont{8}{9.6}{\rmdefault}{\mddefault}{\updefault}{\color[rgb]{0,0,0}$ab$}%
}}}}
\put(4351,-286){\makebox(0,0)[b]{\smash{{\SetFigFont{8}{9.6}{\rmdefault}{\mddefault}{\updefault}{\color[rgb]{0,0,0}$bc$}%
}}}}
\put(4051,-586){\makebox(0,0)[b]{\smash{{\SetFigFont{8}{9.6}{\rmdefault}{\mddefault}{\updefault}{\color[rgb]{0,0,0}$bc$}%
}}}}
\put(3751,-886){\makebox(0,0)[b]{\smash{{\SetFigFont{8}{9.6}{\rmdefault}{\mddefault}{\updefault}{\color[rgb]{0,0,0}$bc$}%
}}}}
\put(4351,-2686){\makebox(0,0)[b]{\smash{{\SetFigFont{8}{9.6}{\rmdefault}{\mddefault}{\updefault}{\color[rgb]{0,0,0}$bc$}%
}}}}
\put(4051,-2986){\makebox(0,0)[b]{\smash{{\SetFigFont{8}{9.6}{\rmdefault}{\mddefault}{\updefault}{\color[rgb]{0,0,0}$bc$}%
}}}}
\put(3751,-3286){\makebox(0,0)[b]{\smash{{\SetFigFont{8}{9.6}{\rmdefault}{\mddefault}{\updefault}{\color[rgb]{0,0,0}$bc$}%
}}}}
\put(3451,-1186){\makebox(0,0)[b]{\smash{{\SetFigFont{8}{9.6}{\rmdefault}{\mddefault}{\updefault}{\color[rgb]{0,0,0}$ac$}%
}}}}
\put(4651,-2386){\makebox(0,0)[b]{\smash{{\SetFigFont{8}{9.6}{\rmdefault}{\mddefault}{\updefault}{\color[rgb]{0,0,0}$ac$}%
}}}}
\put(2251, 14){\makebox(0,0)[b]{\smash{{\SetFigFont{8}{9.6}{\rmdefault}{\mddefault}{\updefault}{\color[rgb]{0,0,0}$a$}%
}}}}
\put(4051, 14){\makebox(0,0)[b]{\smash{{\SetFigFont{8}{9.6}{\rmdefault}{\mddefault}{\updefault}{\color[rgb]{0,0,0}$b$}%
}}}}
\put(4351, 14){\makebox(0,0)[b]{\smash{{\SetFigFont{8}{9.6}{\rmdefault}{\mddefault}{\updefault}{\color[rgb]{0,0,0}$b$}%
}}}}
\put(2251,314){\makebox(0,0)[b]{\smash{{\SetFigFont{8}{9.6}{\rmdefault}{\mddefault}{\updefault}{\color[rgb]{0,0,0}$7$}%
}}}}
\put(3751,314){\makebox(0,0)[b]{\smash{{\SetFigFont{8}{9.6}{\rmdefault}{\mddefault}{\updefault}{\color[rgb]{0,0,0}$2$}%
}}}}
\put(4651,314){\makebox(0,0)[b]{\smash{{\SetFigFont{8}{9.6}{\rmdefault}{\mddefault}{\updefault}{\color[rgb]{0,0,0}$-1$}%
}}}}
\put(7051,314){\makebox(0,0)[b]{\smash{{\SetFigFont{8}{9.6}{\rmdefault}{\mddefault}{\updefault}{\color[rgb]{0,0,0}$-9$}%
}}}}
\put(3751, 14){\makebox(0,0)[b]{\smash{{\SetFigFont{8}{9.6}{\rmdefault}{\mddefault}{\updefault}{\color[rgb]{0,0,0}$b$}%
}}}}
\put(4651, 14){\makebox(0,0)[b]{\smash{{\SetFigFont{8}{9.6}{\rmdefault}{\mddefault}{\updefault}{\color[rgb]{0,0,0}$c$}%
}}}}
\put(7051, 14){\makebox(0,0)[b]{\smash{{\SetFigFont{8}{9.6}{\rmdefault}{\mddefault}{\updefault}{\color[rgb]{0,0,0}$c$}%
}}}}
\put(2551,-286){\makebox(0,0)[b]{\smash{{\SetFigFont{8}{9.6}{\rmdefault}{\mddefault}{\updefault}{\color[rgb]{0,0,0}$a$}%
}}}}
\put(2851,-586){\makebox(0,0)[b]{\smash{{\SetFigFont{8}{9.6}{\rmdefault}{\mddefault}{\updefault}{\color[rgb]{0,0,0}$a$}%
}}}}
\put(3151,-886){\makebox(0,0)[b]{\smash{{\SetFigFont{8}{9.6}{\rmdefault}{\mddefault}{\updefault}{\color[rgb]{0,0,0}$a$}%
}}}}
\put(4951,-2686){\makebox(0,0)[b]{\smash{{\SetFigFont{8}{9.6}{\rmdefault}{\mddefault}{\updefault}{\color[rgb]{0,0,0}$a$}%
}}}}
\put(5251,-2986){\makebox(0,0)[b]{\smash{{\SetFigFont{8}{9.6}{\rmdefault}{\mddefault}{\updefault}{\color[rgb]{0,0,0}$a$}%
}}}}
\put(5551,-3286){\makebox(0,0)[b]{\smash{{\SetFigFont{8}{9.6}{\rmdefault}{\mddefault}{\updefault}{\color[rgb]{0,0,0}$a$}%
}}}}
\put(5851,-3586){\makebox(0,0)[b]{\smash{{\SetFigFont{8}{9.6}{\rmdefault}{\mddefault}{\updefault}{\color[rgb]{0,0,0}$a$}%
}}}}
\put(6751,-286){\makebox(0,0)[b]{\smash{{\SetFigFont{8}{9.6}{\rmdefault}{\mddefault}{\updefault}{\color[rgb]{0,0,0}$c$}%
}}}}
\put(6451,-586){\makebox(0,0)[b]{\smash{{\SetFigFont{8}{9.6}{\rmdefault}{\mddefault}{\updefault}{\color[rgb]{0,0,0}$c$}%
}}}}
\put(6151,-886){\makebox(0,0)[b]{\smash{{\SetFigFont{8}{9.6}{\rmdefault}{\mddefault}{\updefault}{\color[rgb]{0,0,0}$c$}%
}}}}
\put(5851,-1186){\makebox(0,0)[b]{\smash{{\SetFigFont{8}{9.6}{\rmdefault}{\mddefault}{\updefault}{\color[rgb]{0,0,0}$c$}%
}}}}
\put(5551,-1486){\makebox(0,0)[b]{\smash{{\SetFigFont{8}{9.6}{\rmdefault}{\mddefault}{\updefault}{\color[rgb]{0,0,0}$c$}%
}}}}
\put(5251,-1786){\makebox(0,0)[b]{\smash{{\SetFigFont{8}{9.6}{\rmdefault}{\mddefault}{\updefault}{\color[rgb]{0,0,0}$c$}%
}}}}
\put(4951,-2086){\makebox(0,0)[b]{\smash{{\SetFigFont{8}{9.6}{\rmdefault}{\mddefault}{\updefault}{\color[rgb]{0,0,0}$c$}%
}}}}
\put(3151,-1486){\makebox(0,0)[b]{\smash{{\SetFigFont{8}{9.6}{\rmdefault}{\mddefault}{\updefault}{\color[rgb]{0,0,0}$c$}%
}}}}
\put(2851,-1786){\makebox(0,0)[b]{\smash{{\SetFigFont{8}{9.6}{\rmdefault}{\mddefault}{\updefault}{\color[rgb]{0,0,0}$c$}%
}}}}
\put(2551,-2086){\makebox(0,0)[b]{\smash{{\SetFigFont{8}{9.6}{\rmdefault}{\mddefault}{\updefault}{\color[rgb]{0,0,0}$c$}%
}}}}
\put(2251,-2386){\makebox(0,0)[b]{\smash{{\SetFigFont{8}{9.6}{\rmdefault}{\mddefault}{\updefault}{\color[rgb]{0,0,0}$c$}%
}}}}
\put(1951,-2686){\makebox(0,0)[b]{\smash{{\SetFigFont{8}{9.6}{\rmdefault}{\mddefault}{\updefault}{\color[rgb]{0,0,0}$c$}%
}}}}
\put(1651,-2986){\makebox(0,0)[b]{\smash{{\SetFigFont{8}{9.6}{\rmdefault}{\mddefault}{\updefault}{\color[rgb]{0,0,0}$c$}%
}}}}
\put(1351,-3286){\makebox(0,0)[b]{\smash{{\SetFigFont{8}{9.6}{\rmdefault}{\mddefault}{\updefault}{\color[rgb]{0,0,0}$c$}%
}}}}
\put(1051,-3586){\makebox(0,0)[b]{\smash{{\SetFigFont{8}{9.6}{\rmdefault}{\mddefault}{\updefault}{\color[rgb]{0,0,0}$c$}%
}}}}
\put(3451,-3586){\makebox(0,0)[b]{\smash{{\SetFigFont{8}{9.6}{\rmdefault}{\mddefault}{\updefault}{\color[rgb]{0,0,0}$c$}%
}}}}
\put(451,-2986){\makebox(0,0)[b]{\smash{{\SetFigFont{8}{9.6}{\rmdefault}{\mddefault}{\updefault}{\color[rgb]{0,0,0}$10$}%
}}}}
\put(451,-3286){\makebox(0,0)[b]{\smash{{\SetFigFont{8}{9.6}{\rmdefault}{\mddefault}{\updefault}{\color[rgb]{0,0,0}$11$}%
}}}}
\put(451,-286){\makebox(0,0)[b]{\smash{{\SetFigFont{8}{9.6}{\rmdefault}{\mddefault}{\updefault}{\color[rgb]{0,0,0}$1$}%
}}}}
\put(451,-586){\makebox(0,0)[b]{\smash{{\SetFigFont{8}{9.6}{\rmdefault}{\mddefault}{\updefault}{\color[rgb]{0,0,0}$2$}%
}}}}
\put(451,-886){\makebox(0,0)[b]{\smash{{\SetFigFont{8}{9.6}{\rmdefault}{\mddefault}{\updefault}{\color[rgb]{0,0,0}$3$}%
}}}}
\put(451,-1186){\makebox(0,0)[b]{\smash{{\SetFigFont{8}{9.6}{\rmdefault}{\mddefault}{\updefault}{\color[rgb]{0,0,0}$4$}%
}}}}
\put(451,-1486){\makebox(0,0)[b]{\smash{{\SetFigFont{8}{9.6}{\rmdefault}{\mddefault}{\updefault}{\color[rgb]{0,0,0}$5$}%
}}}}
\put(451,-1786){\makebox(0,0)[b]{\smash{{\SetFigFont{8}{9.6}{\rmdefault}{\mddefault}{\updefault}{\color[rgb]{0,0,0}$6$}%
}}}}
\put(451,-2086){\makebox(0,0)[b]{\smash{{\SetFigFont{8}{9.6}{\rmdefault}{\mddefault}{\updefault}{\color[rgb]{0,0,0}$7$}%
}}}}
\put(451,-2386){\makebox(0,0)[b]{\smash{{\SetFigFont{8}{9.6}{\rmdefault}{\mddefault}{\updefault}{\color[rgb]{0,0,0}$8$}%
}}}}
\put(451, 14){\makebox(0,0)[b]{\smash{{\SetFigFont{8}{9.6}{\rmdefault}{\mddefault}{\updefault}{\color[rgb]{0,0,0}$t$}%
}}}}
\put(451,-2686){\makebox(0,0)[b]{\smash{{\SetFigFont{8}{9.6}{\rmdefault}{\mddefault}{\updefault}{\color[rgb]{0,0,0}$9$}%
}}}}
\put(3751,-286){\makebox(0,0)[b]{\smash{{\SetFigFont{8}{9.6}{\rmdefault}{\mddefault}{\updefault}{\color[rgb]{0,0,0}$b$}%
}}}}
\put(4051,-286){\makebox(0,0)[b]{\smash{{\SetFigFont{8}{9.6}{\rmdefault}{\mddefault}{\updefault}{\color[rgb]{0,0,0}$b$}%
}}}}
\put(3751,-586){\makebox(0,0)[b]{\smash{{\SetFigFont{8}{9.6}{\rmdefault}{\mddefault}{\updefault}{\color[rgb]{0,0,0}$b$}%
}}}}
\put(4351,-586){\makebox(0,0)[b]{\smash{{\SetFigFont{8}{9.6}{\rmdefault}{\mddefault}{\updefault}{\color[rgb]{0,0,0}$b$}%
}}}}
\put(4351,-886){\makebox(0,0)[b]{\smash{{\SetFigFont{8}{9.6}{\rmdefault}{\mddefault}{\updefault}{\color[rgb]{0,0,0}$b$}%
}}}}
\put(4051,-886){\makebox(0,0)[b]{\smash{{\SetFigFont{8}{9.6}{\rmdefault}{\mddefault}{\updefault}{\color[rgb]{0,0,0}$b$}%
}}}}
\put(3751,-1186){\makebox(0,0)[b]{\smash{{\SetFigFont{8}{9.6}{\rmdefault}{\mddefault}{\updefault}{\color[rgb]{0,0,0}$b$}%
}}}}
\put(4051,-1186){\makebox(0,0)[b]{\smash{{\SetFigFont{8}{9.6}{\rmdefault}{\mddefault}{\updefault}{\color[rgb]{0,0,0}$b$}%
}}}}
\put(4351,-1186){\makebox(0,0)[b]{\smash{{\SetFigFont{8}{9.6}{\rmdefault}{\mddefault}{\updefault}{\color[rgb]{0,0,0}$b$}%
}}}}
\put(4051,-1486){\makebox(0,0)[b]{\smash{{\SetFigFont{8}{9.6}{\rmdefault}{\mddefault}{\updefault}{\color[rgb]{0,0,0}$b$}%
}}}}
\put(4351,-1486){\makebox(0,0)[b]{\smash{{\SetFigFont{8}{9.6}{\rmdefault}{\mddefault}{\updefault}{\color[rgb]{0,0,0}$b$}%
}}}}
\put(4351,-1786){\makebox(0,0)[b]{\smash{{\SetFigFont{8}{9.6}{\rmdefault}{\mddefault}{\updefault}{\color[rgb]{0,0,0}$b$}%
}}}}
\put(3751,-1786){\makebox(0,0)[b]{\smash{{\SetFigFont{8}{9.6}{\rmdefault}{\mddefault}{\updefault}{\color[rgb]{0,0,0}$b$}%
}}}}
\put(3751,-2086){\makebox(0,0)[b]{\smash{{\SetFigFont{8}{9.6}{\rmdefault}{\mddefault}{\updefault}{\color[rgb]{0,0,0}$b$}%
}}}}
\put(4051,-2086){\makebox(0,0)[b]{\smash{{\SetFigFont{8}{9.6}{\rmdefault}{\mddefault}{\updefault}{\color[rgb]{0,0,0}$b$}%
}}}}
\put(4351,-2386){\makebox(0,0)[b]{\smash{{\SetFigFont{8}{9.6}{\rmdefault}{\mddefault}{\updefault}{\color[rgb]{0,0,0}$b$}%
}}}}
\put(4051,-2386){\makebox(0,0)[b]{\smash{{\SetFigFont{8}{9.6}{\rmdefault}{\mddefault}{\updefault}{\color[rgb]{0,0,0}$b$}%
}}}}
\put(3751,-2386){\makebox(0,0)[b]{\smash{{\SetFigFont{8}{9.6}{\rmdefault}{\mddefault}{\updefault}{\color[rgb]{0,0,0}$b$}%
}}}}
\put(3751,-2686){\makebox(0,0)[b]{\smash{{\SetFigFont{8}{9.6}{\rmdefault}{\mddefault}{\updefault}{\color[rgb]{0,0,0}$b$}%
}}}}
\put(4051,-2686){\makebox(0,0)[b]{\smash{{\SetFigFont{8}{9.6}{\rmdefault}{\mddefault}{\updefault}{\color[rgb]{0,0,0}$b$}%
}}}}
\put(4351,-2986){\makebox(0,0)[b]{\smash{{\SetFigFont{8}{9.6}{\rmdefault}{\mddefault}{\updefault}{\color[rgb]{0,0,0}$b$}%
}}}}
\put(3751,-2986){\makebox(0,0)[b]{\smash{{\SetFigFont{8}{9.6}{\rmdefault}{\mddefault}{\updefault}{\color[rgb]{0,0,0}$b$}%
}}}}
\put(4051,-3286){\makebox(0,0)[b]{\smash{{\SetFigFont{8}{9.6}{\rmdefault}{\mddefault}{\updefault}{\color[rgb]{0,0,0}$b$}%
}}}}
\put(4351,-3286){\makebox(0,0)[b]{\smash{{\SetFigFont{8}{9.6}{\rmdefault}{\mddefault}{\updefault}{\color[rgb]{0,0,0}$b$}%
}}}}
\put(4351,-3586){\makebox(0,0)[b]{\smash{{\SetFigFont{8}{9.6}{\rmdefault}{\mddefault}{\updefault}{\color[rgb]{0,0,0}$b$}%
}}}}
\put(4051,-3586){\makebox(0,0)[b]{\smash{{\SetFigFont{8}{9.6}{\rmdefault}{\mddefault}{\updefault}{\color[rgb]{0,0,0}$b$}%
}}}}
\put(3751,-3586){\makebox(0,0)[b]{\smash{{\SetFigFont{8}{9.6}{\rmdefault}{\mddefault}{\updefault}{\color[rgb]{0,0,0}$b$}%
}}}}
\put(4051,314){\makebox(0,0)[b]{\smash{{\SetFigFont{8}{9.6}{\rmdefault}{\mddefault}{\updefault}{\color[rgb]{0,0,0}$1$}%
}}}}
\end{picture}%

%% file: 215Ex.pstex_t
\begin{picture}(0,0)%
\includegraphics{215Ex.eps}%
\end{picture}%
\setlength{\unitlength}{2763sp}%
\begingroup\makeatletter\ifx\SetFigFont\undefined%
\gdef\SetFigFont#1#2#3#4#5{%
  \reset@font\fontsize{#1}{#2pt}%
  \fontfamily{#3}\fontseries{#4}\fontshape{#5}%
  \selectfont}%
\fi\endgroup%
\begin{picture}(10977,5911)(136,-5150)
\put(5551,614){\makebox(0,0)[b]{\smash{{\SetFigFont{8}{9.6}{\rmdefault}{\mddefault}{\updefault}{\color[rgb]{0,0,0}$1$}%
}}}}
\put(3751,-2686){\makebox(0,0)[b]{\smash{{\SetFigFont{8}{9.6}{\rmdefault}{\mddefault}{\updefault}{\color[rgb]{0,0,0}$ac$}%
}}}}
\put(4351,-3286){\makebox(0,0)[b]{\smash{{\SetFigFont{8}{9.6}{\rmdefault}{\mddefault}{\updefault}{\color[rgb]{0,0,0}$ac$}%
}}}}
\put(4651,-3586){\makebox(0,0)[b]{\smash{{\SetFigFont{8}{9.6}{\rmdefault}{\mddefault}{\updefault}{\color[rgb]{0,0,0}$ac$}%
}}}}
\put(4951,-3886){\makebox(0,0)[b]{\smash{{\SetFigFont{8}{9.6}{\rmdefault}{\mddefault}{\updefault}{\color[rgb]{0,0,0}$ac$}%
}}}}
\put(5551,-4486){\makebox(0,0)[b]{\smash{{\SetFigFont{8}{9.6}{\rmdefault}{\mddefault}{\updefault}{\color[rgb]{0,0,0}$ac$}%
}}}}
\put(6151,-286){\makebox(0,0)[b]{\smash{{\SetFigFont{8}{9.6}{\rmdefault}{\mddefault}{\updefault}{\color[rgb]{0,0,0}$ac$}%
}}}}
\put(6751,-886){\makebox(0,0)[b]{\smash{{\SetFigFont{8}{9.6}{\rmdefault}{\mddefault}{\updefault}{\color[rgb]{0,0,0}$ac$}%
}}}}
\put(7051,-1186){\makebox(0,0)[b]{\smash{{\SetFigFont{8}{9.6}{\rmdefault}{\mddefault}{\updefault}{\color[rgb]{0,0,0}$ac$}%
}}}}
\put(7351,-1486){\makebox(0,0)[b]{\smash{{\SetFigFont{8}{9.6}{\rmdefault}{\mddefault}{\updefault}{\color[rgb]{0,0,0}$ac$}%
}}}}
\put(7951,-2086){\makebox(0,0)[b]{\smash{{\SetFigFont{8}{9.6}{\rmdefault}{\mddefault}{\updefault}{\color[rgb]{0,0,0}$ac$}%
}}}}
\put(5851, 14){\makebox(0,0)[b]{\smash{{\SetFigFont{8}{9.6}{\rmdefault}{\mddefault}{\updefault}{\color[rgb]{0,0,0}$ab$}%
}}}}
\put(5851,-4786){\makebox(0,0)[b]{\smash{{\SetFigFont{8}{9.6}{\rmdefault}{\mddefault}{\updefault}{\color[rgb]{0,0,0}$ab$}%
}}}}
\put(5851,-586){\makebox(0,0)[b]{\smash{{\SetFigFont{8}{9.6}{\rmdefault}{\mddefault}{\updefault}{\color[rgb]{0,0,0}$bc$}%
}}}}
\put(5851,-1786){\makebox(0,0)[b]{\smash{{\SetFigFont{8}{9.6}{\rmdefault}{\mddefault}{\updefault}{\color[rgb]{0,0,0}$bc$}%
}}}}
\put(5851,-2386){\makebox(0,0)[b]{\smash{{\SetFigFont{8}{9.6}{\rmdefault}{\mddefault}{\updefault}{\color[rgb]{0,0,0}$bc$}%
}}}}
\put(5851,-4186){\makebox(0,0)[b]{\smash{{\SetFigFont{8}{9.6}{\rmdefault}{\mddefault}{\updefault}{\color[rgb]{0,0,0}$bc$}%
}}}}
\put(5851,-2986){\makebox(0,0)[b]{\smash{{\SetFigFont{8}{9.6}{\rmdefault}{\mddefault}{\updefault}{\color[rgb]{0,0,0}$bc$}%
}}}}
\put(151,-2686){\makebox(0,0)[b]{\smash{{\SetFigFont{8}{9.6}{\rmdefault}{\mddefault}{\updefault}{\color[rgb]{0,0,0}$10$}%
}}}}
\put(151, 14){\makebox(0,0)[b]{\smash{{\SetFigFont{8}{9.6}{\rmdefault}{\mddefault}{\updefault}{\color[rgb]{0,0,0}$1$}%
}}}}
\put(151,-286){\makebox(0,0)[b]{\smash{{\SetFigFont{8}{9.6}{\rmdefault}{\mddefault}{\updefault}{\color[rgb]{0,0,0}$2$}%
}}}}
\put(151,-586){\makebox(0,0)[b]{\smash{{\SetFigFont{8}{9.6}{\rmdefault}{\mddefault}{\updefault}{\color[rgb]{0,0,0}$3$}%
}}}}
\put(151,-886){\makebox(0,0)[b]{\smash{{\SetFigFont{8}{9.6}{\rmdefault}{\mddefault}{\updefault}{\color[rgb]{0,0,0}$4$}%
}}}}
\put(151,-1186){\makebox(0,0)[b]{\smash{{\SetFigFont{8}{9.6}{\rmdefault}{\mddefault}{\updefault}{\color[rgb]{0,0,0}$5$}%
}}}}
\put(151,-1486){\makebox(0,0)[b]{\smash{{\SetFigFont{8}{9.6}{\rmdefault}{\mddefault}{\updefault}{\color[rgb]{0,0,0}$6$}%
}}}}
\put(151,-1786){\makebox(0,0)[b]{\smash{{\SetFigFont{8}{9.6}{\rmdefault}{\mddefault}{\updefault}{\color[rgb]{0,0,0}$7$}%
}}}}
\put(151,-2086){\makebox(0,0)[b]{\smash{{\SetFigFont{8}{9.6}{\rmdefault}{\mddefault}{\updefault}{\color[rgb]{0,0,0}$8$}%
}}}}
\put(151,314){\makebox(0,0)[b]{\smash{{\SetFigFont{8}{9.6}{\rmdefault}{\mddefault}{\updefault}{\color[rgb]{0,0,0}$t$}%
}}}}
\put(151,-2386){\makebox(0,0)[b]{\smash{{\SetFigFont{8}{9.6}{\rmdefault}{\mddefault}{\updefault}{\color[rgb]{0,0,0}$9$}%
}}}}
\put(751,614){\makebox(0,0)[b]{\smash{{\SetFigFont{8}{9.6}{\rmdefault}{\mddefault}{\updefault}{\color[rgb]{0,0,0}$17$}%
}}}}
\put(151,-3286){\makebox(0,0)[b]{\smash{{\SetFigFont{8}{9.6}{\rmdefault}{\mddefault}{\updefault}{\color[rgb]{0,0,0}$12$}%
}}}}
\put(151,-3586){\makebox(0,0)[b]{\smash{{\SetFigFont{8}{9.6}{\rmdefault}{\mddefault}{\updefault}{\color[rgb]{0,0,0}$13$}%
}}}}
\put(151,-3886){\makebox(0,0)[b]{\smash{{\SetFigFont{8}{9.6}{\rmdefault}{\mddefault}{\updefault}{\color[rgb]{0,0,0}$14$}%
}}}}
\put(151,-4186){\makebox(0,0)[b]{\smash{{\SetFigFont{8}{9.6}{\rmdefault}{\mddefault}{\updefault}{\color[rgb]{0,0,0}$15$}%
}}}}
\put(151,-4486){\makebox(0,0)[b]{\smash{{\SetFigFont{8}{9.6}{\rmdefault}{\mddefault}{\updefault}{\color[rgb]{0,0,0}$16$}%
}}}}
\put(151,-4786){\makebox(0,0)[b]{\smash{{\SetFigFont{8}{9.6}{\rmdefault}{\mddefault}{\updefault}{\color[rgb]{0,0,0}$17$}%
}}}}
\put(151,-2986){\makebox(0,0)[b]{\smash{{\SetFigFont{8}{9.6}{\rmdefault}{\mddefault}{\updefault}{\color[rgb]{0,0,0}$11$}%
}}}}
\put(5851,614){\makebox(0,0)[b]{\smash{{\SetFigFont{8}{9.6}{\rmdefault}{\mddefault}{\updefault}{\color[rgb]{0,0,0}$0$}%
}}}}
\put(6751,614){\makebox(0,0)[b]{\smash{{\SetFigFont{8}{9.6}{\rmdefault}{\mddefault}{\updefault}{\color[rgb]{0,0,0}$-3$}%
}}}}
\put(7951,614){\makebox(0,0)[b]{\smash{{\SetFigFont{8}{9.6}{\rmdefault}{\mddefault}{\updefault}{\color[rgb]{0,0,0}$-7$}%
}}}}
\put(8551,614){\makebox(0,0)[b]{\smash{{\SetFigFont{8}{9.6}{\rmdefault}{\mddefault}{\updefault}{\color[rgb]{0,0,0}$-9$}%
}}}}
\put(9151,614){\makebox(0,0)[b]{\smash{{\SetFigFont{8}{9.6}{\rmdefault}{\mddefault}{\updefault}{\color[rgb]{0,0,0}$-11$}%
}}}}
\put(10351,614){\makebox(0,0)[b]{\smash{{\SetFigFont{8}{9.6}{\rmdefault}{\mddefault}{\updefault}{\color[rgb]{0,0,0}$-15$}%
}}}}
\put(751,314){\makebox(0,0)[b]{\smash{{\SetFigFont{8}{9.6}{\rmdefault}{\mddefault}{\updefault}{\color[rgb]{0,0,0}$a$}%
}}}}
\put(5851,314){\makebox(0,0)[b]{\smash{{\SetFigFont{8}{9.6}{\rmdefault}{\mddefault}{\updefault}{\color[rgb]{0,0,0}$b$}%
}}}}
\put(5851,-286){\makebox(0,0)[b]{\smash{{\SetFigFont{8}{9.6}{\rmdefault}{\mddefault}{\updefault}{\color[rgb]{0,0,0}$b$}%
}}}}
\put(5851,-886){\makebox(0,0)[b]{\smash{{\SetFigFont{8}{9.6}{\rmdefault}{\mddefault}{\updefault}{\color[rgb]{0,0,0}$b$}%
}}}}
\put(5851,-1486){\makebox(0,0)[b]{\smash{{\SetFigFont{8}{9.6}{\rmdefault}{\mddefault}{\updefault}{\color[rgb]{0,0,0}$b$}%
}}}}
\put(5851,-2086){\makebox(0,0)[b]{\smash{{\SetFigFont{8}{9.6}{\rmdefault}{\mddefault}{\updefault}{\color[rgb]{0,0,0}$b$}%
}}}}
\put(5851,-2686){\makebox(0,0)[b]{\smash{{\SetFigFont{8}{9.6}{\rmdefault}{\mddefault}{\updefault}{\color[rgb]{0,0,0}$b$}%
}}}}
\put(5851,-3286){\makebox(0,0)[b]{\smash{{\SetFigFont{8}{9.6}{\rmdefault}{\mddefault}{\updefault}{\color[rgb]{0,0,0}$b$}%
}}}}
\put(5851,-3586){\makebox(0,0)[b]{\smash{{\SetFigFont{8}{9.6}{\rmdefault}{\mddefault}{\updefault}{\color[rgb]{0,0,0}$b$}%
}}}}
\put(5851,-3886){\makebox(0,0)[b]{\smash{{\SetFigFont{8}{9.6}{\rmdefault}{\mddefault}{\updefault}{\color[rgb]{0,0,0}$b$}%
}}}}
\put(5851,-4486){\makebox(0,0)[b]{\smash{{\SetFigFont{8}{9.6}{\rmdefault}{\mddefault}{\updefault}{\color[rgb]{0,0,0}$b$}%
}}}}
\put(5851,-5086){\makebox(0,0)[b]{\smash{{\SetFigFont{8}{9.6}{\rmdefault}{\mddefault}{\updefault}{\color[rgb]{0,0,0}$b$}%
}}}}
\put(5851,-1186){\makebox(0,0)[b]{\smash{{\SetFigFont{8}{9.6}{\rmdefault}{\mddefault}{\updefault}{\color[rgb]{0,0,0}$b$}%
}}}}
\put(6751,314){\makebox(0,0)[b]{\smash{{\SetFigFont{8}{9.6}{\rmdefault}{\mddefault}{\updefault}{\color[rgb]{0,0,0}$c$}%
}}}}
\put(6451, 14){\makebox(0,0)[b]{\smash{{\SetFigFont{8}{9.6}{\rmdefault}{\mddefault}{\updefault}{\color[rgb]{0,0,0}$c$}%
}}}}
\put(7951,314){\makebox(0,0)[b]{\smash{{\SetFigFont{8}{9.6}{\rmdefault}{\mddefault}{\updefault}{\color[rgb]{0,0,0}$c$}%
}}}}
\put(7651, 14){\makebox(0,0)[b]{\smash{{\SetFigFont{8}{9.6}{\rmdefault}{\mddefault}{\updefault}{\color[rgb]{0,0,0}$c$}%
}}}}
\put(7351,-286){\makebox(0,0)[b]{\smash{{\SetFigFont{8}{9.6}{\rmdefault}{\mddefault}{\updefault}{\color[rgb]{0,0,0}$c$}%
}}}}
\put(7051,-586){\makebox(0,0)[b]{\smash{{\SetFigFont{8}{9.6}{\rmdefault}{\mddefault}{\updefault}{\color[rgb]{0,0,0}$c$}%
}}}}
\put(6451,-1186){\makebox(0,0)[b]{\smash{{\SetFigFont{8}{9.6}{\rmdefault}{\mddefault}{\updefault}{\color[rgb]{0,0,0}$c$}%
}}}}
\put(6151,-1486){\makebox(0,0)[b]{\smash{{\SetFigFont{8}{9.6}{\rmdefault}{\mddefault}{\updefault}{\color[rgb]{0,0,0}$c$}%
}}}}
\put(5551,-3286){\makebox(0,0)[b]{\smash{{\SetFigFont{8}{9.6}{\rmdefault}{\mddefault}{\updefault}{\color[rgb]{0,0,0}$c$}%
}}}}
\put(5251,-3586){\makebox(0,0)[b]{\smash{{\SetFigFont{8}{9.6}{\rmdefault}{\mddefault}{\updefault}{\color[rgb]{0,0,0}$c$}%
}}}}
\put(4576,-4186){\makebox(0,0)[b]{\smash{{\SetFigFont{8}{9.6}{\rmdefault}{\mddefault}{\updefault}{\color[rgb]{0,0,0}$c$}%
}}}}
\put(4351,-4486){\makebox(0,0)[b]{\smash{{\SetFigFont{8}{9.6}{\rmdefault}{\mddefault}{\updefault}{\color[rgb]{0,0,0}$c$}%
}}}}
\put(3751,-5086){\makebox(0,0)[b]{\smash{{\SetFigFont{8}{9.6}{\rmdefault}{\mddefault}{\updefault}{\color[rgb]{0,0,0}$c$}%
}}}}
\put(4051,-4786){\makebox(0,0)[b]{\smash{{\SetFigFont{8}{9.6}{\rmdefault}{\mddefault}{\updefault}{\color[rgb]{0,0,0}$c$}%
}}}}
\put(5551,-886){\makebox(0,0)[b]{\smash{{\SetFigFont{8}{9.6}{\rmdefault}{\mddefault}{\updefault}{\color[rgb]{0,0,0}$c$}%
}}}}
\put(5251,-1186){\makebox(0,0)[b]{\smash{{\SetFigFont{8}{9.6}{\rmdefault}{\mddefault}{\updefault}{\color[rgb]{0,0,0}$c$}%
}}}}
\put(4951,-1486){\makebox(0,0)[b]{\smash{{\SetFigFont{8}{9.6}{\rmdefault}{\mddefault}{\updefault}{\color[rgb]{0,0,0}$c$}%
}}}}
\put(4576,-1786){\makebox(0,0)[b]{\smash{{\SetFigFont{8}{9.6}{\rmdefault}{\mddefault}{\updefault}{\color[rgb]{0,0,0}$c$}%
}}}}
\put(4351,-2086){\makebox(0,0)[b]{\smash{{\SetFigFont{8}{9.6}{\rmdefault}{\mddefault}{\updefault}{\color[rgb]{0,0,0}$c$}%
}}}}
\put(3451,-2986){\makebox(0,0)[b]{\smash{{\SetFigFont{8}{9.6}{\rmdefault}{\mddefault}{\updefault}{\color[rgb]{0,0,0}$c$}%
}}}}
\put(3151,-3286){\makebox(0,0)[b]{\smash{{\SetFigFont{8}{9.6}{\rmdefault}{\mddefault}{\updefault}{\color[rgb]{0,0,0}$c$}%
}}}}
\put(2851,-3586){\makebox(0,0)[b]{\smash{{\SetFigFont{8}{9.6}{\rmdefault}{\mddefault}{\updefault}{\color[rgb]{0,0,0}$c$}%
}}}}
\put(2551,-3886){\makebox(0,0)[b]{\smash{{\SetFigFont{8}{9.6}{\rmdefault}{\mddefault}{\updefault}{\color[rgb]{0,0,0}$c$}%
}}}}
\put(4051,-2386){\makebox(0,0)[b]{\smash{{\SetFigFont{8}{9.6}{\rmdefault}{\mddefault}{\updefault}{\color[rgb]{0,0,0}$c$}%
}}}}
\put(5551,-2086){\makebox(0,0)[b]{\smash{{\SetFigFont{8}{9.6}{\rmdefault}{\mddefault}{\updefault}{\color[rgb]{0,0,0}$c$}%
}}}}
\put(5251,-2386){\makebox(0,0)[b]{\smash{{\SetFigFont{8}{9.6}{\rmdefault}{\mddefault}{\updefault}{\color[rgb]{0,0,0}$c$}%
}}}}
\put(4951,-2686){\makebox(0,0)[b]{\smash{{\SetFigFont{8}{9.6}{\rmdefault}{\mddefault}{\updefault}{\color[rgb]{0,0,0}$c$}%
}}}}
\put(4576,-2986){\makebox(0,0)[b]{\smash{{\SetFigFont{8}{9.6}{\rmdefault}{\mddefault}{\updefault}{\color[rgb]{0,0,0}$c$}%
}}}}
\put(3751,-3886){\makebox(0,0)[b]{\smash{{\SetFigFont{8}{9.6}{\rmdefault}{\mddefault}{\updefault}{\color[rgb]{0,0,0}$c$}%
}}}}
\put(3451,-4186){\makebox(0,0)[b]{\smash{{\SetFigFont{8}{9.6}{\rmdefault}{\mddefault}{\updefault}{\color[rgb]{0,0,0}$c$}%
}}}}
\put(3151,-4486){\makebox(0,0)[b]{\smash{{\SetFigFont{8}{9.6}{\rmdefault}{\mddefault}{\updefault}{\color[rgb]{0,0,0}$c$}%
}}}}
\put(2851,-4786){\makebox(0,0)[b]{\smash{{\SetFigFont{8}{9.6}{\rmdefault}{\mddefault}{\updefault}{\color[rgb]{0,0,0}$c$}%
}}}}
\put(2551,-5086){\makebox(0,0)[b]{\smash{{\SetFigFont{8}{9.6}{\rmdefault}{\mddefault}{\updefault}{\color[rgb]{0,0,0}$c$}%
}}}}
\put(4051,-3586){\makebox(0,0)[b]{\smash{{\SetFigFont{8}{9.6}{\rmdefault}{\mddefault}{\updefault}{\color[rgb]{0,0,0}$c$}%
}}}}
\put(9151,314){\makebox(0,0)[b]{\smash{{\SetFigFont{8}{9.6}{\rmdefault}{\mddefault}{\updefault}{\color[rgb]{0,0,0}$c$}%
}}}}
\put(8851, 14){\makebox(0,0)[b]{\smash{{\SetFigFont{8}{9.6}{\rmdefault}{\mddefault}{\updefault}{\color[rgb]{0,0,0}$c$}%
}}}}
\put(8551,-286){\makebox(0,0)[b]{\smash{{\SetFigFont{8}{9.6}{\rmdefault}{\mddefault}{\updefault}{\color[rgb]{0,0,0}$c$}%
}}}}
\put(8176,-586){\makebox(0,0)[b]{\smash{{\SetFigFont{8}{9.6}{\rmdefault}{\mddefault}{\updefault}{\color[rgb]{0,0,0}$c$}%
}}}}
\put(7951,-886){\makebox(0,0)[b]{\smash{{\SetFigFont{8}{9.6}{\rmdefault}{\mddefault}{\updefault}{\color[rgb]{0,0,0}$c$}%
}}}}
\put(7051,-1786){\makebox(0,0)[b]{\smash{{\SetFigFont{8}{9.6}{\rmdefault}{\mddefault}{\updefault}{\color[rgb]{0,0,0}$c$}%
}}}}
\put(6751,-2086){\makebox(0,0)[b]{\smash{{\SetFigFont{8}{9.6}{\rmdefault}{\mddefault}{\updefault}{\color[rgb]{0,0,0}$c$}%
}}}}
\put(6451,-2386){\makebox(0,0)[b]{\smash{{\SetFigFont{8}{9.6}{\rmdefault}{\mddefault}{\updefault}{\color[rgb]{0,0,0}$c$}%
}}}}
\put(6151,-2686){\makebox(0,0)[b]{\smash{{\SetFigFont{8}{9.6}{\rmdefault}{\mddefault}{\updefault}{\color[rgb]{0,0,0}$c$}%
}}}}
\put(7651,-1186){\makebox(0,0)[b]{\smash{{\SetFigFont{8}{9.6}{\rmdefault}{\mddefault}{\updefault}{\color[rgb]{0,0,0}$c$}%
}}}}
\put(10351,314){\makebox(0,0)[b]{\smash{{\SetFigFont{8}{9.6}{\rmdefault}{\mddefault}{\updefault}{\color[rgb]{0,0,0}$c$}%
}}}}
\put(10051, 14){\makebox(0,0)[b]{\smash{{\SetFigFont{8}{9.6}{\rmdefault}{\mddefault}{\updefault}{\color[rgb]{0,0,0}$c$}%
}}}}
\put(9751,-286){\makebox(0,0)[b]{\smash{{\SetFigFont{8}{9.6}{\rmdefault}{\mddefault}{\updefault}{\color[rgb]{0,0,0}$c$}%
}}}}
\put(9376,-586){\makebox(0,0)[b]{\smash{{\SetFigFont{8}{9.6}{\rmdefault}{\mddefault}{\updefault}{\color[rgb]{0,0,0}$c$}%
}}}}
\put(9151,-886){\makebox(0,0)[b]{\smash{{\SetFigFont{8}{9.6}{\rmdefault}{\mddefault}{\updefault}{\color[rgb]{0,0,0}$c$}%
}}}}
\put(8551,-1486){\makebox(0,0)[b]{\smash{{\SetFigFont{8}{9.6}{\rmdefault}{\mddefault}{\updefault}{\color[rgb]{0,0,0}$c$}%
}}}}
\put(8251,-1786){\makebox(0,0)[b]{\smash{{\SetFigFont{8}{9.6}{\rmdefault}{\mddefault}{\updefault}{\color[rgb]{0,0,0}$c$}%
}}}}
\put(7651,-2386){\makebox(0,0)[b]{\smash{{\SetFigFont{8}{9.6}{\rmdefault}{\mddefault}{\updefault}{\color[rgb]{0,0,0}$c$}%
}}}}
\put(7351,-2686){\makebox(0,0)[b]{\smash{{\SetFigFont{8}{9.6}{\rmdefault}{\mddefault}{\updefault}{\color[rgb]{0,0,0}$c$}%
}}}}
\put(8851,-1186){\makebox(0,0)[b]{\smash{{\SetFigFont{8}{9.6}{\rmdefault}{\mddefault}{\updefault}{\color[rgb]{0,0,0}$c$}%
}}}}
\put(7051,-2986){\makebox(0,0)[b]{\smash{{\SetFigFont{8}{9.6}{\rmdefault}{\mddefault}{\updefault}{\color[rgb]{0,0,0}$c$}%
}}}}
\put(6751,-3286){\makebox(0,0)[b]{\smash{{\SetFigFont{8}{9.6}{\rmdefault}{\mddefault}{\updefault}{\color[rgb]{0,0,0}$c$}%
}}}}
\put(6451,-3586){\makebox(0,0)[b]{\smash{{\SetFigFont{8}{9.6}{\rmdefault}{\mddefault}{\updefault}{\color[rgb]{0,0,0}$c$}%
}}}}
\put(6151,-3886){\makebox(0,0)[b]{\smash{{\SetFigFont{8}{9.6}{\rmdefault}{\mddefault}{\updefault}{\color[rgb]{0,0,0}$c$}%
}}}}
\put(5251,-4786){\makebox(0,0)[b]{\smash{{\SetFigFont{8}{9.6}{\rmdefault}{\mddefault}{\updefault}{\color[rgb]{0,0,0}$c$}%
}}}}
\put(4951,-5086){\makebox(0,0)[b]{\smash{{\SetFigFont{8}{9.6}{\rmdefault}{\mddefault}{\updefault}{\color[rgb]{0,0,0}$c$}%
}}}}
\put(2251,-4186){\makebox(0,0)[b]{\smash{{\SetFigFont{8}{9.6}{\rmdefault}{\mddefault}{\updefault}{\color[rgb]{0,0,0}$c$}%
}}}}
\put(1951,-4486){\makebox(0,0)[b]{\smash{{\SetFigFont{8}{9.6}{\rmdefault}{\mddefault}{\updefault}{\color[rgb]{0,0,0}$c$}%
}}}}
\put(1651,-4786){\makebox(0,0)[b]{\smash{{\SetFigFont{8}{9.6}{\rmdefault}{\mddefault}{\updefault}{\color[rgb]{0,0,0}$c$}%
}}}}
\put(1351,-5086){\makebox(0,0)[b]{\smash{{\SetFigFont{8}{9.6}{\rmdefault}{\mddefault}{\updefault}{\color[rgb]{0,0,0}$c$}%
}}}}
\put(5551,314){\makebox(0,0)[b]{\smash{{\SetFigFont{8}{9.6}{\rmdefault}{\mddefault}{\updefault}{\color[rgb]{0,0,0}$a$}%
}}}}
\put(6451,-586){\makebox(0,0)[b]{\smash{{\SetFigFont{8}{9.6}{\rmdefault}{\mddefault}{\updefault}{\color[rgb]{0,0,0}$a$}%
}}}}
\put(7651,-1786){\makebox(0,0)[b]{\smash{{\SetFigFont{8}{9.6}{\rmdefault}{\mddefault}{\updefault}{\color[rgb]{0,0,0}$a$}%
}}}}
\put(8251,-2386){\makebox(0,0)[b]{\smash{{\SetFigFont{8}{9.6}{\rmdefault}{\mddefault}{\updefault}{\color[rgb]{0,0,0}$a$}%
}}}}
\put(8551,-2686){\makebox(0,0)[b]{\smash{{\SetFigFont{8}{9.6}{\rmdefault}{\mddefault}{\updefault}{\color[rgb]{0,0,0}$a$}%
}}}}
\put(8851,-2986){\makebox(0,0)[b]{\smash{{\SetFigFont{8}{9.6}{\rmdefault}{\mddefault}{\updefault}{\color[rgb]{0,0,0}$a$}%
}}}}
\put(9226,-3286){\makebox(0,0)[b]{\smash{{\SetFigFont{8}{9.6}{\rmdefault}{\mddefault}{\updefault}{\color[rgb]{0,0,0}$a$}%
}}}}
\put(9451,-3586){\makebox(0,0)[b]{\smash{{\SetFigFont{8}{9.6}{\rmdefault}{\mddefault}{\updefault}{\color[rgb]{0,0,0}$a$}%
}}}}
\put(9751,-3886){\makebox(0,0)[b]{\smash{{\SetFigFont{8}{9.6}{\rmdefault}{\mddefault}{\updefault}{\color[rgb]{0,0,0}$a$}%
}}}}
\put(10051,-4186){\makebox(0,0)[b]{\smash{{\SetFigFont{8}{9.6}{\rmdefault}{\mddefault}{\updefault}{\color[rgb]{0,0,0}$a$}%
}}}}
\put(10351,-4486){\makebox(0,0)[b]{\smash{{\SetFigFont{8}{9.6}{\rmdefault}{\mddefault}{\updefault}{\color[rgb]{0,0,0}$a$}%
}}}}
\put(10651,-4786){\makebox(0,0)[b]{\smash{{\SetFigFont{8}{9.6}{\rmdefault}{\mddefault}{\updefault}{\color[rgb]{0,0,0}$a$}%
}}}}
\put(10951,-5086){\makebox(0,0)[b]{\smash{{\SetFigFont{8}{9.6}{\rmdefault}{\mddefault}{\updefault}{\color[rgb]{0,0,0}$a$}%
}}}}
\put(1051, 14){\makebox(0,0)[b]{\smash{{\SetFigFont{8}{9.6}{\rmdefault}{\mddefault}{\updefault}{\color[rgb]{0,0,0}$a$}%
}}}}
\put(1351,-286){\makebox(0,0)[b]{\smash{{\SetFigFont{8}{9.6}{\rmdefault}{\mddefault}{\updefault}{\color[rgb]{0,0,0}$a$}%
}}}}
\put(1651,-586){\makebox(0,0)[b]{\smash{{\SetFigFont{8}{9.6}{\rmdefault}{\mddefault}{\updefault}{\color[rgb]{0,0,0}$a$}%
}}}}
\put(1951,-886){\makebox(0,0)[b]{\smash{{\SetFigFont{8}{9.6}{\rmdefault}{\mddefault}{\updefault}{\color[rgb]{0,0,0}$a$}%
}}}}
\put(2251,-1186){\makebox(0,0)[b]{\smash{{\SetFigFont{8}{9.6}{\rmdefault}{\mddefault}{\updefault}{\color[rgb]{0,0,0}$a$}%
}}}}
\put(2551,-1486){\makebox(0,0)[b]{\smash{{\SetFigFont{8}{9.6}{\rmdefault}{\mddefault}{\updefault}{\color[rgb]{0,0,0}$a$}%
}}}}
\put(2851,-1786){\makebox(0,0)[b]{\smash{{\SetFigFont{8}{9.6}{\rmdefault}{\mddefault}{\updefault}{\color[rgb]{0,0,0}$a$}%
}}}}
\put(3151,-2086){\makebox(0,0)[b]{\smash{{\SetFigFont{8}{9.6}{\rmdefault}{\mddefault}{\updefault}{\color[rgb]{0,0,0}$a$}%
}}}}
\put(3451,-2386){\makebox(0,0)[b]{\smash{{\SetFigFont{8}{9.6}{\rmdefault}{\mddefault}{\updefault}{\color[rgb]{0,0,0}$a$}%
}}}}
\put(4051,-2986){\makebox(0,0)[b]{\smash{{\SetFigFont{8}{9.6}{\rmdefault}{\mddefault}{\updefault}{\color[rgb]{0,0,0}$a$}%
}}}}
\put(5251,-4186){\makebox(0,0)[b]{\smash{{\SetFigFont{8}{9.6}{\rmdefault}{\mddefault}{\updefault}{\color[rgb]{0,0,0}$a$}%
}}}}
\put(6151,-5086){\makebox(0,0)[b]{\smash{{\SetFigFont{8}{9.6}{\rmdefault}{\mddefault}{\updefault}{\color[rgb]{0,0,0}$a$}%
}}}}
\put(8551,314){\makebox(0,0)[b]{\smash{{\SetFigFont{8}{9.6}{\rmdefault}{\mddefault}{\updefault}{\color[rgb]{0,0,0}$c$}%
}}}}
\put(8251, 14){\makebox(0,0)[b]{\smash{{\SetFigFont{8}{9.6}{\rmdefault}{\mddefault}{\updefault}{\color[rgb]{0,0,0}$c$}%
}}}}
\put(7951,-286){\makebox(0,0)[b]{\smash{{\SetFigFont{8}{9.6}{\rmdefault}{\mddefault}{\updefault}{\color[rgb]{0,0,0}$c$}%
}}}}
\put(7651,-586){\makebox(0,0)[b]{\smash{{\SetFigFont{8}{9.6}{\rmdefault}{\mddefault}{\updefault}{\color[rgb]{0,0,0}$c$}%
}}}}
\put(7351,-886){\makebox(0,0)[b]{\smash{{\SetFigFont{8}{9.6}{\rmdefault}{\mddefault}{\updefault}{\color[rgb]{0,0,0}$c$}%
}}}}
\put(6751,-1486){\makebox(0,0)[b]{\smash{{\SetFigFont{8}{9.6}{\rmdefault}{\mddefault}{\updefault}{\color[rgb]{0,0,0}$c$}%
}}}}
\put(6451,-1786){\makebox(0,0)[b]{\smash{{\SetFigFont{8}{9.6}{\rmdefault}{\mddefault}{\updefault}{\color[rgb]{0,0,0}$c$}%
}}}}
\put(6151,-2086){\makebox(0,0)[b]{\smash{{\SetFigFont{8}{9.6}{\rmdefault}{\mddefault}{\updefault}{\color[rgb]{0,0,0}$c$}%
}}}}
\put(5551,-2686){\makebox(0,0)[b]{\smash{{\SetFigFont{8}{9.6}{\rmdefault}{\mddefault}{\updefault}{\color[rgb]{0,0,0}$c$}%
}}}}
\put(5251,-2986){\makebox(0,0)[b]{\smash{{\SetFigFont{8}{9.6}{\rmdefault}{\mddefault}{\updefault}{\color[rgb]{0,0,0}$c$}%
}}}}
\put(4951,-3286){\makebox(0,0)[b]{\smash{{\SetFigFont{8}{9.6}{\rmdefault}{\mddefault}{\updefault}{\color[rgb]{0,0,0}$c$}%
}}}}
\put(4351,-3886){\makebox(0,0)[b]{\smash{{\SetFigFont{8}{9.6}{\rmdefault}{\mddefault}{\updefault}{\color[rgb]{0,0,0}$c$}%
}}}}
\put(4051,-4186){\makebox(0,0)[b]{\smash{{\SetFigFont{8}{9.6}{\rmdefault}{\mddefault}{\updefault}{\color[rgb]{0,0,0}$c$}%
}}}}
\put(3451,-4786){\makebox(0,0)[b]{\smash{{\SetFigFont{8}{9.6}{\rmdefault}{\mddefault}{\updefault}{\color[rgb]{0,0,0}$c$}%
}}}}
\put(3151,-5086){\makebox(0,0)[b]{\smash{{\SetFigFont{8}{9.6}{\rmdefault}{\mddefault}{\updefault}{\color[rgb]{0,0,0}$c$}%
}}}}
\put(3751,-4486){\makebox(0,0)[b]{\smash{{\SetFigFont{8}{9.6}{\rmdefault}{\mddefault}{\updefault}{\color[rgb]{0,0,0}$c$}%
}}}}
\end{picture}%

%% file: trapezoid.pstex_t
\begin{picture}(0,0)%
\includegraphics{trapezoid.eps}%
\end{picture}%
\setlength{\unitlength}{2763sp}%
\begingroup\makeatletter\ifx\SetFigFont\undefined%
\gdef\SetFigFont#1#2#3#4#5{%
  \reset@font\fontsize{#1}{#2pt}%
  \fontfamily{#3}\fontseries{#4}\fontshape{#5}%
  \selectfont}%
\fi\endgroup%
\begin{picture}(6177,4134)(136,-3373)
\put(1351,-2986){\makebox(0,0)[b]{\smash{{\SetFigFont{8}{9.6}{\rmdefault}{\mddefault}{\updefault}{\color[rgb]{0,0,0}$a^*$}%
}}}}
\put(151,314){\makebox(0,0)[b]{\smash{{\SetFigFont{8}{9.6}{\rmdefault}{\mddefault}{\updefault}{\color[rgb]{0,0,0}$0$}%
}}}}
\put(151,-3286){\makebox(0,0)[b]{\smash{{\SetFigFont{8}{9.6}{\rmdefault}{\mddefault}{\updefault}{\color[rgb]{0,0,0}$N+1$}%
}}}}
\put(2251,314){\makebox(0,0)[b]{\smash{{\SetFigFont{8}{9.6}{\rmdefault}{\mddefault}{\updefault}{\color[rgb]{0,0,0}$a$}%
}}}}
\put(1051,-3286){\makebox(0,0)[b]{\smash{{\SetFigFont{8}{9.6}{\rmdefault}{\mddefault}{\updefault}{\color[rgb]{0,0,0}$a^*$}%
}}}}
\put(2551, 14){\makebox(0,0)[b]{\smash{{\SetFigFont{8}{9.6}{\rmdefault}{\mddefault}{\updefault}{\color[rgb]{0,0,0}$a$}%
}}}}
\put(2851,-286){\makebox(0,0)[b]{\smash{{\SetFigFont{8}{9.6}{\rmdefault}{\mddefault}{\updefault}{\color[rgb]{0,0,0}$a$}%
}}}}
\put(3151,-586){\makebox(0,0)[b]{\smash{{\SetFigFont{8}{9.6}{\rmdefault}{\mddefault}{\updefault}{\color[rgb]{0,0,0}$a$}%
}}}}
\put(1651,-2686){\makebox(0,0)[b]{\smash{{\SetFigFont{8}{9.6}{\rmdefault}{\mddefault}{\updefault}{\color[rgb]{0,0,0}$a^*$}%
}}}}
\put(1951,-2386){\makebox(0,0)[b]{\smash{{\SetFigFont{8}{9.6}{\rmdefault}{\mddefault}{\updefault}{\color[rgb]{0,0,0}$a^*$}%
}}}}
\put(2251,-2086){\makebox(0,0)[b]{\smash{{\SetFigFont{8}{9.6}{\rmdefault}{\mddefault}{\updefault}{\color[rgb]{0,0,0}$a^*$}%
}}}}
\put(2551,-1786){\makebox(0,0)[b]{\smash{{\SetFigFont{8}{9.6}{\rmdefault}{\mddefault}{\updefault}{\color[rgb]{0,0,0}$a^*$}%
}}}}
\put(2851,-1486){\makebox(0,0)[b]{\smash{{\SetFigFont{8}{9.6}{\rmdefault}{\mddefault}{\updefault}{\color[rgb]{0,0,0}$a^*$}%
}}}}
\put(3151,-1186){\makebox(0,0)[b]{\smash{{\SetFigFont{8}{9.6}{\rmdefault}{\mddefault}{\updefault}{\color[rgb]{0,0,0}$a^*$}%
}}}}
\put(151,614){\makebox(0,0)[b]{\smash{{\SetFigFont{8}{9.6}{\rmdefault}{\mddefault}{\updefault}{\color[rgb]{0,0,0}$t$}%
}}}}
\put(3751,-1186){\makebox(0,0)[b]{\smash{{\SetFigFont{8}{9.6}{\rmdefault}{\mddefault}{\updefault}{\color[rgb]{0,0,0}$c^*$}%
}}}}
\put(4051,-1486){\makebox(0,0)[b]{\smash{{\SetFigFont{8}{9.6}{\rmdefault}{\mddefault}{\updefault}{\color[rgb]{0,0,0}$c^*$}%
}}}}
\put(4351,-1786){\makebox(0,0)[b]{\smash{{\SetFigFont{8}{9.6}{\rmdefault}{\mddefault}{\updefault}{\color[rgb]{0,0,0}$c^*$}%
}}}}
\put(4651,-2086){\makebox(0,0)[b]{\smash{{\SetFigFont{8}{9.6}{\rmdefault}{\mddefault}{\updefault}{\color[rgb]{0,0,0}$c^*$}%
}}}}
\put(4951,-2386){\makebox(0,0)[b]{\smash{{\SetFigFont{8}{9.6}{\rmdefault}{\mddefault}{\updefault}{\color[rgb]{0,0,0}$c^*$}%
}}}}
\put(5251,-2686){\makebox(0,0)[b]{\smash{{\SetFigFont{8}{9.6}{\rmdefault}{\mddefault}{\updefault}{\color[rgb]{0,0,0}$c^*$}%
}}}}
\put(5551,-2986){\makebox(0,0)[b]{\smash{{\SetFigFont{8}{9.6}{\rmdefault}{\mddefault}{\updefault}{\color[rgb]{0,0,0}$c^*$}%
}}}}
\put(5851,-3286){\makebox(0,0)[b]{\smash{{\SetFigFont{8}{9.6}{\rmdefault}{\mddefault}{\updefault}{\color[rgb]{0,0,0}$c^*$}%
}}}}
\put(4651,314){\makebox(0,0)[b]{\smash{{\SetFigFont{8}{9.6}{\rmdefault}{\mddefault}{\updefault}{\color[rgb]{0,0,0}$c$}%
}}}}
\put(4351, 14){\makebox(0,0)[b]{\smash{{\SetFigFont{8}{9.6}{\rmdefault}{\mddefault}{\updefault}{\color[rgb]{0,0,0}$c$}%
}}}}
\put(4051,-286){\makebox(0,0)[b]{\smash{{\SetFigFont{8}{9.6}{\rmdefault}{\mddefault}{\updefault}{\color[rgb]{0,0,0}$c$}%
}}}}
\put(3751,-586){\makebox(0,0)[b]{\smash{{\SetFigFont{8}{9.6}{\rmdefault}{\mddefault}{\updefault}{\color[rgb]{0,0,0}$c$}%
}}}}
\put(3451,-886){\makebox(0,0)[b]{\smash{{\SetFigFont{8}{9.6}{\rmdefault}{\mddefault}{\updefault}{\color[rgb]{0,0,0}$ac$}%
}}}}
\end{picture}%

%% file: 221Ex.pstex_t
\begin{picture}(0,0)%
\includegraphics{221Ex.eps}%
\end{picture}%
\setlength{\unitlength}{2763sp}%
\begingroup\makeatletter\ifx\SetFigFont\undefined%
\gdef\SetFigFont#1#2#3#4#5{%
  \reset@font\fontsize{#1}{#2pt}%
  \fontfamily{#3}\fontseries{#4}\fontshape{#5}%
  \selectfont}%
\fi\endgroup%
\begin{picture}(4077,3211)(136,-2450)
\put(3151, 14){\makebox(0,0)[b]{\smash{{\SetFigFont{8}{9.6}{\rmdefault}{\mddefault}{\updefault}{\color[rgb]{0,0,0}$b$}%
}}}}
\put(1651, 14){\makebox(0,0)[b]{\smash{{\SetFigFont{8}{9.6}{\rmdefault}{\mddefault}{\updefault}{\color[rgb]{0,0,0}$ab$}%
}}}}
\put(1651,-586){\makebox(0,0)[b]{\smash{{\SetFigFont{8}{9.6}{\rmdefault}{\mddefault}{\updefault}{\color[rgb]{0,0,0}$ab$}%
}}}}
\put(1651,-1786){\makebox(0,0)[b]{\smash{{\SetFigFont{8}{9.6}{\rmdefault}{\mddefault}{\updefault}{\color[rgb]{0,0,0}$bc$}%
}}}}
\put(3151,-2086){\makebox(0,0)[b]{\smash{{\SetFigFont{8}{9.6}{\rmdefault}{\mddefault}{\updefault}{\color[rgb]{0,0,0}$ab$}%
}}}}
\put(3151,-1486){\makebox(0,0)[b]{\smash{{\SetFigFont{8}{9.6}{\rmdefault}{\mddefault}{\updefault}{\color[rgb]{0,0,0}$ab$}%
}}}}
\put(3151,-286){\makebox(0,0)[b]{\smash{{\SetFigFont{8}{9.6}{\rmdefault}{\mddefault}{\updefault}{\color[rgb]{0,0,0}$bc$}%
}}}}
\put(2551,-886){\makebox(0,0)[b]{\smash{{\SetFigFont{8}{9.6}{\rmdefault}{\mddefault}{\updefault}{\color[rgb]{0,0,0}$ac$}%
}}}}
\put(2251,-1186){\makebox(0,0)[b]{\smash{{\SetFigFont{8}{9.6}{\rmdefault}{\mddefault}{\updefault}{\color[rgb]{0,0,0}$ac$}%
}}}}
\put(151, 14){\makebox(0,0)[b]{\smash{{\SetFigFont{8}{9.6}{\rmdefault}{\mddefault}{\updefault}{\color[rgb]{0,0,0}$1$}%
}}}}
\put(151,-286){\makebox(0,0)[b]{\smash{{\SetFigFont{8}{9.6}{\rmdefault}{\mddefault}{\updefault}{\color[rgb]{0,0,0}$2$}%
}}}}
\put(151,-586){\makebox(0,0)[b]{\smash{{\SetFigFont{8}{9.6}{\rmdefault}{\mddefault}{\updefault}{\color[rgb]{0,0,0}$3$}%
}}}}
\put(151,-886){\makebox(0,0)[b]{\smash{{\SetFigFont{8}{9.6}{\rmdefault}{\mddefault}{\updefault}{\color[rgb]{0,0,0}$4$}%
}}}}
\put(151,-1186){\makebox(0,0)[b]{\smash{{\SetFigFont{8}{9.6}{\rmdefault}{\mddefault}{\updefault}{\color[rgb]{0,0,0}$5$}%
}}}}
\put(151,-1486){\makebox(0,0)[b]{\smash{{\SetFigFont{8}{9.6}{\rmdefault}{\mddefault}{\updefault}{\color[rgb]{0,0,0}$6$}%
}}}}
\put(151,-1786){\makebox(0,0)[b]{\smash{{\SetFigFont{8}{9.6}{\rmdefault}{\mddefault}{\updefault}{\color[rgb]{0,0,0}$7$}%
}}}}
\put(151,-2086){\makebox(0,0)[b]{\smash{{\SetFigFont{8}{9.6}{\rmdefault}{\mddefault}{\updefault}{\color[rgb]{0,0,0}$8$}%
}}}}
\put(151,314){\makebox(0,0)[b]{\smash{{\SetFigFont{8}{9.6}{\rmdefault}{\mddefault}{\updefault}{\color[rgb]{0,0,0}$t$}%
}}}}
\put(751,614){\makebox(0,0)[b]{\smash{{\SetFigFont{8}{9.6}{\rmdefault}{\mddefault}{\updefault}{\color[rgb]{0,0,0}$8$}%
}}}}
\put(751,314){\makebox(0,0)[b]{\smash{{\SetFigFont{8}{9.6}{\rmdefault}{\mddefault}{\updefault}{\color[rgb]{0,0,0}$a$}%
}}}}
\put(1351,614){\makebox(0,0)[b]{\smash{{\SetFigFont{8}{9.6}{\rmdefault}{\mddefault}{\updefault}{\color[rgb]{0,0,0}$6$}%
}}}}
\put(1651,614){\makebox(0,0)[b]{\smash{{\SetFigFont{8}{9.6}{\rmdefault}{\mddefault}{\updefault}{\color[rgb]{0,0,0}$5$}%
}}}}
\put(3151,614){\makebox(0,0)[b]{\smash{{\SetFigFont{8}{9.6}{\rmdefault}{\mddefault}{\updefault}{\color[rgb]{0,0,0}$0$}%
}}}}
\put(3751,614){\makebox(0,0)[b]{\smash{{\SetFigFont{8}{9.6}{\rmdefault}{\mddefault}{\updefault}{\color[rgb]{0,0,0}$-2$}%
}}}}
\put(1351,314){\makebox(0,0)[b]{\smash{{\SetFigFont{8}{9.6}{\rmdefault}{\mddefault}{\updefault}{\color[rgb]{0,0,0}$a$}%
}}}}
\put(3151,314){\makebox(0,0)[b]{\smash{{\SetFigFont{8}{9.6}{\rmdefault}{\mddefault}{\updefault}{\color[rgb]{0,0,0}$b$}%
}}}}
\put(1651,314){\makebox(0,0)[b]{\smash{{\SetFigFont{8}{9.6}{\rmdefault}{\mddefault}{\updefault}{\color[rgb]{0,0,0}$b$}%
}}}}
\put(3751,314){\makebox(0,0)[b]{\smash{{\SetFigFont{8}{9.6}{\rmdefault}{\mddefault}{\updefault}{\color[rgb]{0,0,0}$c$}%
}}}}
\put(3451, 14){\makebox(0,0)[b]{\smash{{\SetFigFont{8}{9.6}{\rmdefault}{\mddefault}{\updefault}{\color[rgb]{0,0,0}$c$}%
}}}}
\put(2851,-586){\makebox(0,0)[b]{\smash{{\SetFigFont{8}{9.6}{\rmdefault}{\mddefault}{\updefault}{\color[rgb]{0,0,0}$c$}%
}}}}
\put(1951,-1486){\makebox(0,0)[b]{\smash{{\SetFigFont{8}{9.6}{\rmdefault}{\mddefault}{\updefault}{\color[rgb]{0,0,0}$c$}%
}}}}
\put(1351,-2086){\makebox(0,0)[b]{\smash{{\SetFigFont{8}{9.6}{\rmdefault}{\mddefault}{\updefault}{\color[rgb]{0,0,0}$c$}%
}}}}
\put(1051,-2386){\makebox(0,0)[b]{\smash{{\SetFigFont{8}{9.6}{\rmdefault}{\mddefault}{\updefault}{\color[rgb]{0,0,0}$c$}%
}}}}
\put(1051, 14){\makebox(0,0)[b]{\smash{{\SetFigFont{8}{9.6}{\rmdefault}{\mddefault}{\updefault}{\color[rgb]{0,0,0}$a$}%
}}}}
\put(1351,-286){\makebox(0,0)[b]{\smash{{\SetFigFont{8}{9.6}{\rmdefault}{\mddefault}{\updefault}{\color[rgb]{0,0,0}$a$}%
}}}}
\put(1951,-886){\makebox(0,0)[b]{\smash{{\SetFigFont{8}{9.6}{\rmdefault}{\mddefault}{\updefault}{\color[rgb]{0,0,0}$a$}%
}}}}
\put(2551,-1486){\makebox(0,0)[b]{\smash{{\SetFigFont{8}{9.6}{\rmdefault}{\mddefault}{\updefault}{\color[rgb]{0,0,0}$a$}%
}}}}
\put(2851,-1786){\makebox(0,0)[b]{\smash{{\SetFigFont{8}{9.6}{\rmdefault}{\mddefault}{\updefault}{\color[rgb]{0,0,0}$a$}%
}}}}
\put(3451,-2386){\makebox(0,0)[b]{\smash{{\SetFigFont{8}{9.6}{\rmdefault}{\mddefault}{\updefault}{\color[rgb]{0,0,0}$a$}%
}}}}
\put(1951,-286){\makebox(0,0)[b]{\smash{{\SetFigFont{8}{9.6}{\rmdefault}{\mddefault}{\updefault}{\color[rgb]{0,0,0}$a$}%
}}}}
\put(2251,-586){\makebox(0,0)[b]{\smash{{\SetFigFont{8}{9.6}{\rmdefault}{\mddefault}{\updefault}{\color[rgb]{0,0,0}$a$}%
}}}}
\put(2851,-1186){\makebox(0,0)[b]{\smash{{\SetFigFont{8}{9.6}{\rmdefault}{\mddefault}{\updefault}{\color[rgb]{0,0,0}$a$}%
}}}}
\put(3451,-1786){\makebox(0,0)[b]{\smash{{\SetFigFont{8}{9.6}{\rmdefault}{\mddefault}{\updefault}{\color[rgb]{0,0,0}$a$}%
}}}}
\put(3751,-2086){\makebox(0,0)[b]{\smash{{\SetFigFont{8}{9.6}{\rmdefault}{\mddefault}{\updefault}{\color[rgb]{0,0,0}$a$}%
}}}}
\put(4051,-2386){\makebox(0,0)[b]{\smash{{\SetFigFont{8}{9.6}{\rmdefault}{\mddefault}{\updefault}{\color[rgb]{0,0,0}$a$}%
}}}}
\put(1651,-286){\makebox(0,0)[b]{\smash{{\SetFigFont{8}{9.6}{\rmdefault}{\mddefault}{\updefault}{\color[rgb]{0,0,0}$b$}%
}}}}
\put(1651,-886){\makebox(0,0)[b]{\smash{{\SetFigFont{8}{9.6}{\rmdefault}{\mddefault}{\updefault}{\color[rgb]{0,0,0}$b$}%
}}}}
\put(1651,-1186){\makebox(0,0)[b]{\smash{{\SetFigFont{8}{9.6}{\rmdefault}{\mddefault}{\updefault}{\color[rgb]{0,0,0}$b$}%
}}}}
\put(1651,-1486){\makebox(0,0)[b]{\smash{{\SetFigFont{8}{9.6}{\rmdefault}{\mddefault}{\updefault}{\color[rgb]{0,0,0}$b$}%
}}}}
\put(1651,-2086){\makebox(0,0)[b]{\smash{{\SetFigFont{8}{9.6}{\rmdefault}{\mddefault}{\updefault}{\color[rgb]{0,0,0}$b$}%
}}}}
\put(1651,-2386){\makebox(0,0)[b]{\smash{{\SetFigFont{8}{9.6}{\rmdefault}{\mddefault}{\updefault}{\color[rgb]{0,0,0}$b$}%
}}}}
\put(3151,-2386){\makebox(0,0)[b]{\smash{{\SetFigFont{8}{9.6}{\rmdefault}{\mddefault}{\updefault}{\color[rgb]{0,0,0}$b$}%
}}}}
\put(3151,-1786){\makebox(0,0)[b]{\smash{{\SetFigFont{8}{9.6}{\rmdefault}{\mddefault}{\updefault}{\color[rgb]{0,0,0}$b$}%
}}}}
\put(3151,-1186){\makebox(0,0)[b]{\smash{{\SetFigFont{8}{9.6}{\rmdefault}{\mddefault}{\updefault}{\color[rgb]{0,0,0}$b$}%
}}}}
\put(3151,-886){\makebox(0,0)[b]{\smash{{\SetFigFont{8}{9.6}{\rmdefault}{\mddefault}{\updefault}{\color[rgb]{0,0,0}$b$}%
}}}}
\put(3151,-586){\makebox(0,0)[b]{\smash{{\SetFigFont{8}{9.6}{\rmdefault}{\mddefault}{\updefault}{\color[rgb]{0,0,0}$b$}%
}}}}
\end{picture}%

%% file: 222Ex.pstex_t
\begin{picture}(0,0)%
\includegraphics{222Ex.eps}%
\end{picture}%
\setlength{\unitlength}{2763sp}%
\begingroup\makeatletter\ifx\SetFigFont\undefined%
\gdef\SetFigFont#1#2#3#4#5{%
  \reset@font\fontsize{#1}{#2pt}%
  \fontfamily{#3}\fontseries{#4}\fontshape{#5}%
  \selectfont}%
\fi\endgroup%
\begin{picture}(7077,4411)(136,-3650)
\put(4351, 14){\makebox(0,0)[b]{\smash{{\SetFigFont{8}{9.6}{\rmdefault}{\mddefault}{\updefault}{\color[rgb]{0,0,0}$b$}%
}}}}
\put(3451, 14){\makebox(0,0)[b]{\smash{{\SetFigFont{8}{9.6}{\rmdefault}{\mddefault}{\updefault}{\color[rgb]{0,0,0}$ab$}%
}}}}
\put(4351,-886){\makebox(0,0)[b]{\smash{{\SetFigFont{8}{9.6}{\rmdefault}{\mddefault}{\updefault}{\color[rgb]{0,0,0}$ab$}%
}}}}
\put(4351,-286){\makebox(0,0)[b]{\smash{{\SetFigFont{8}{9.6}{\rmdefault}{\mddefault}{\updefault}{\color[rgb]{0,0,0}$bc$}%
}}}}
\put(3451,-1186){\makebox(0,0)[b]{\smash{{\SetFigFont{8}{9.6}{\rmdefault}{\mddefault}{\updefault}{\color[rgb]{0,0,0}$bc$}%
}}}}
\put(4051,-586){\makebox(0,0)[b]{\smash{{\SetFigFont{8}{9.6}{\rmdefault}{\mddefault}{\updefault}{\color[rgb]{0,0,0}$ac$}%
}}}}
\put(3451,-2386){\makebox(0,0)[b]{\smash{{\SetFigFont{8}{9.6}{\rmdefault}{\mddefault}{\updefault}{\color[rgb]{0,0,0}$ab$}%
}}}}
\put(4351,-3286){\makebox(0,0)[b]{\smash{{\SetFigFont{8}{9.6}{\rmdefault}{\mddefault}{\updefault}{\color[rgb]{0,0,0}$ab$}%
}}}}
\put(2851,-1786){\makebox(0,0)[b]{\smash{{\SetFigFont{8}{9.6}{\rmdefault}{\mddefault}{\updefault}{\color[rgb]{0,0,0}$ac$}%
}}}}
\put(3751,-2686){\makebox(0,0)[b]{\smash{{\SetFigFont{8}{9.6}{\rmdefault}{\mddefault}{\updefault}{\color[rgb]{0,0,0}$ac$}%
}}}}
\put(3451,-2986){\makebox(0,0)[b]{\smash{{\SetFigFont{8}{9.6}{\rmdefault}{\mddefault}{\updefault}{\color[rgb]{0,0,0}$bc$}%
}}}}
\put(4351,-2086){\makebox(0,0)[b]{\smash{{\SetFigFont{8}{9.6}{\rmdefault}{\mddefault}{\updefault}{\color[rgb]{0,0,0}$bc$}%
}}}}
\put(4951,-1486){\makebox(0,0)[b]{\smash{{\SetFigFont{8}{9.6}{\rmdefault}{\mddefault}{\updefault}{\color[rgb]{0,0,0}$ac$}%
}}}}
\put(151,-2686){\makebox(0,0)[b]{\smash{{\SetFigFont{8}{9.6}{\rmdefault}{\mddefault}{\updefault}{\color[rgb]{0,0,0}$10$}%
}}}}
\put(151, 14){\makebox(0,0)[b]{\smash{{\SetFigFont{8}{9.6}{\rmdefault}{\mddefault}{\updefault}{\color[rgb]{0,0,0}$1$}%
}}}}
\put(151,-286){\makebox(0,0)[b]{\smash{{\SetFigFont{8}{9.6}{\rmdefault}{\mddefault}{\updefault}{\color[rgb]{0,0,0}$2$}%
}}}}
\put(151,-586){\makebox(0,0)[b]{\smash{{\SetFigFont{8}{9.6}{\rmdefault}{\mddefault}{\updefault}{\color[rgb]{0,0,0}$3$}%
}}}}
\put(151,-886){\makebox(0,0)[b]{\smash{{\SetFigFont{8}{9.6}{\rmdefault}{\mddefault}{\updefault}{\color[rgb]{0,0,0}$4$}%
}}}}
\put(151,-1186){\makebox(0,0)[b]{\smash{{\SetFigFont{8}{9.6}{\rmdefault}{\mddefault}{\updefault}{\color[rgb]{0,0,0}$5$}%
}}}}
\put(151,-1486){\makebox(0,0)[b]{\smash{{\SetFigFont{8}{9.6}{\rmdefault}{\mddefault}{\updefault}{\color[rgb]{0,0,0}$6$}%
}}}}
\put(151,-1786){\makebox(0,0)[b]{\smash{{\SetFigFont{8}{9.6}{\rmdefault}{\mddefault}{\updefault}{\color[rgb]{0,0,0}$7$}%
}}}}
\put(151,-2086){\makebox(0,0)[b]{\smash{{\SetFigFont{8}{9.6}{\rmdefault}{\mddefault}{\updefault}{\color[rgb]{0,0,0}$8$}%
}}}}
\put(151,314){\makebox(0,0)[b]{\smash{{\SetFigFont{8}{9.6}{\rmdefault}{\mddefault}{\updefault}{\color[rgb]{0,0,0}$t$}%
}}}}
\put(151,-2386){\makebox(0,0)[b]{\smash{{\SetFigFont{8}{9.6}{\rmdefault}{\mddefault}{\updefault}{\color[rgb]{0,0,0}$9$}%
}}}}
\put(751,614){\makebox(0,0)[b]{\smash{{\SetFigFont{8}{9.6}{\rmdefault}{\mddefault}{\updefault}{\color[rgb]{0,0,0}$12$}%
}}}}
\put(151,-3286){\makebox(0,0)[b]{\smash{{\SetFigFont{8}{9.6}{\rmdefault}{\mddefault}{\updefault}{\color[rgb]{0,0,0}$12$}%
}}}}
\put(151,-2986){\makebox(0,0)[b]{\smash{{\SetFigFont{8}{9.6}{\rmdefault}{\mddefault}{\updefault}{\color[rgb]{0,0,0}$11$}%
}}}}
\put(751,314){\makebox(0,0)[b]{\smash{{\SetFigFont{8}{9.6}{\rmdefault}{\mddefault}{\updefault}{\color[rgb]{0,0,0}$a$}%
}}}}
\put(3151,614){\makebox(0,0)[b]{\smash{{\SetFigFont{8}{9.6}{\rmdefault}{\mddefault}{\updefault}{\color[rgb]{0,0,0}$4$}%
}}}}
\put(3451,614){\makebox(0,0)[b]{\smash{{\SetFigFont{8}{9.6}{\rmdefault}{\mddefault}{\updefault}{\color[rgb]{0,0,0}$3$}%
}}}}
\put(4351,614){\makebox(0,0)[b]{\smash{{\SetFigFont{8}{9.6}{\rmdefault}{\mddefault}{\updefault}{\color[rgb]{0,0,0}$0$}%
}}}}
\put(4951,614){\makebox(0,0)[b]{\smash{{\SetFigFont{8}{9.6}{\rmdefault}{\mddefault}{\updefault}{\color[rgb]{0,0,0}$-2$}%
}}}}
\put(6751,614){\makebox(0,0)[b]{\smash{{\SetFigFont{8}{9.6}{\rmdefault}{\mddefault}{\updefault}{\color[rgb]{0,0,0}$-8$}%
}}}}
\put(3151,314){\makebox(0,0)[b]{\smash{{\SetFigFont{8}{9.6}{\rmdefault}{\mddefault}{\updefault}{\color[rgb]{0,0,0}$a$}%
}}}}
\put(3451,314){\makebox(0,0)[b]{\smash{{\SetFigFont{8}{9.6}{\rmdefault}{\mddefault}{\updefault}{\color[rgb]{0,0,0}$b$}%
}}}}
\put(4351,314){\makebox(0,0)[b]{\smash{{\SetFigFont{8}{9.6}{\rmdefault}{\mddefault}{\updefault}{\color[rgb]{0,0,0}$b$}%
}}}}
\put(4951,314){\makebox(0,0)[b]{\smash{{\SetFigFont{8}{9.6}{\rmdefault}{\mddefault}{\updefault}{\color[rgb]{0,0,0}$c$}%
}}}}
\put(6751,314){\makebox(0,0)[b]{\smash{{\SetFigFont{8}{9.6}{\rmdefault}{\mddefault}{\updefault}{\color[rgb]{0,0,0}$c$}%
}}}}
\put(3751,-286){\makebox(0,0)[b]{\smash{{\SetFigFont{8}{9.6}{\rmdefault}{\mddefault}{\updefault}{\color[rgb]{0,0,0}$a$}%
}}}}
\put(4651,-1186){\makebox(0,0)[b]{\smash{{\SetFigFont{8}{9.6}{\rmdefault}{\mddefault}{\updefault}{\color[rgb]{0,0,0}$a$}%
}}}}
\put(5251,-1786){\makebox(0,0)[b]{\smash{{\SetFigFont{8}{9.6}{\rmdefault}{\mddefault}{\updefault}{\color[rgb]{0,0,0}$a$}%
}}}}
\put(5551,-2086){\makebox(0,0)[b]{\smash{{\SetFigFont{8}{9.6}{\rmdefault}{\mddefault}{\updefault}{\color[rgb]{0,0,0}$a$}%
}}}}
\put(5851,-2386){\makebox(0,0)[b]{\smash{{\SetFigFont{8}{9.6}{\rmdefault}{\mddefault}{\updefault}{\color[rgb]{0,0,0}$a$}%
}}}}
\put(6151,-2686){\makebox(0,0)[b]{\smash{{\SetFigFont{8}{9.6}{\rmdefault}{\mddefault}{\updefault}{\color[rgb]{0,0,0}$a$}%
}}}}
\put(6451,-2986){\makebox(0,0)[b]{\smash{{\SetFigFont{8}{9.6}{\rmdefault}{\mddefault}{\updefault}{\color[rgb]{0,0,0}$a$}%
}}}}
\put(6751,-3286){\makebox(0,0)[b]{\smash{{\SetFigFont{8}{9.6}{\rmdefault}{\mddefault}{\updefault}{\color[rgb]{0,0,0}$a$}%
}}}}
\put(7051,-3586){\makebox(0,0)[b]{\smash{{\SetFigFont{8}{9.6}{\rmdefault}{\mddefault}{\updefault}{\color[rgb]{0,0,0}$a$}%
}}}}
\put(1051, 14){\makebox(0,0)[b]{\smash{{\SetFigFont{8}{9.6}{\rmdefault}{\mddefault}{\updefault}{\color[rgb]{0,0,0}$a$}%
}}}}
\put(1351,-286){\makebox(0,0)[b]{\smash{{\SetFigFont{8}{9.6}{\rmdefault}{\mddefault}{\updefault}{\color[rgb]{0,0,0}$a$}%
}}}}
\put(1651,-586){\makebox(0,0)[b]{\smash{{\SetFigFont{8}{9.6}{\rmdefault}{\mddefault}{\updefault}{\color[rgb]{0,0,0}$a$}%
}}}}
\put(1951,-886){\makebox(0,0)[b]{\smash{{\SetFigFont{8}{9.6}{\rmdefault}{\mddefault}{\updefault}{\color[rgb]{0,0,0}$a$}%
}}}}
\put(2251,-1186){\makebox(0,0)[b]{\smash{{\SetFigFont{8}{9.6}{\rmdefault}{\mddefault}{\updefault}{\color[rgb]{0,0,0}$a$}%
}}}}
\put(2551,-1486){\makebox(0,0)[b]{\smash{{\SetFigFont{8}{9.6}{\rmdefault}{\mddefault}{\updefault}{\color[rgb]{0,0,0}$a$}%
}}}}
\put(4051,-2986){\makebox(0,0)[b]{\smash{{\SetFigFont{8}{9.6}{\rmdefault}{\mddefault}{\updefault}{\color[rgb]{0,0,0}$a$}%
}}}}
\put(4651,-3586){\makebox(0,0)[b]{\smash{{\SetFigFont{8}{9.6}{\rmdefault}{\mddefault}{\updefault}{\color[rgb]{0,0,0}$a$}%
}}}}
\put(3151,-2086){\makebox(0,0)[b]{\smash{{\SetFigFont{8}{9.6}{\rmdefault}{\mddefault}{\updefault}{\color[rgb]{0,0,0}$a$}%
}}}}
\put(6451, 14){\makebox(0,0)[b]{\smash{{\SetFigFont{8}{9.6}{\rmdefault}{\mddefault}{\updefault}{\color[rgb]{0,0,0}$c$}%
}}}}
\put(6151,-286){\makebox(0,0)[b]{\smash{{\SetFigFont{8}{9.6}{\rmdefault}{\mddefault}{\updefault}{\color[rgb]{0,0,0}$c$}%
}}}}
\put(5851,-586){\makebox(0,0)[b]{\smash{{\SetFigFont{8}{9.6}{\rmdefault}{\mddefault}{\updefault}{\color[rgb]{0,0,0}$c$}%
}}}}
\put(5551,-886){\makebox(0,0)[b]{\smash{{\SetFigFont{8}{9.6}{\rmdefault}{\mddefault}{\updefault}{\color[rgb]{0,0,0}$c$}%
}}}}
\put(5251,-1186){\makebox(0,0)[b]{\smash{{\SetFigFont{8}{9.6}{\rmdefault}{\mddefault}{\updefault}{\color[rgb]{0,0,0}$c$}%
}}}}
\put(4651,-1786){\makebox(0,0)[b]{\smash{{\SetFigFont{8}{9.6}{\rmdefault}{\mddefault}{\updefault}{\color[rgb]{0,0,0}$c$}%
}}}}
\put(3151,-3286){\makebox(0,0)[b]{\smash{{\SetFigFont{8}{9.6}{\rmdefault}{\mddefault}{\updefault}{\color[rgb]{0,0,0}$c$}%
}}}}
\put(2851,-3586){\makebox(0,0)[b]{\smash{{\SetFigFont{8}{9.6}{\rmdefault}{\mddefault}{\updefault}{\color[rgb]{0,0,0}$c$}%
}}}}
\put(4051,-2386){\makebox(0,0)[b]{\smash{{\SetFigFont{8}{9.6}{\rmdefault}{\mddefault}{\updefault}{\color[rgb]{0,0,0}$c$}%
}}}}
\put(4651, 14){\makebox(0,0)[b]{\smash{{\SetFigFont{8}{9.6}{\rmdefault}{\mddefault}{\updefault}{\color[rgb]{0,0,0}$c$}%
}}}}
\put(3751,-886){\makebox(0,0)[b]{\smash{{\SetFigFont{8}{9.6}{\rmdefault}{\mddefault}{\updefault}{\color[rgb]{0,0,0}$c$}%
}}}}
\put(3151,-1486){\makebox(0,0)[b]{\smash{{\SetFigFont{8}{9.6}{\rmdefault}{\mddefault}{\updefault}{\color[rgb]{0,0,0}$c$}%
}}}}
\put(2551,-2086){\makebox(0,0)[b]{\smash{{\SetFigFont{8}{9.6}{\rmdefault}{\mddefault}{\updefault}{\color[rgb]{0,0,0}$c$}%
}}}}
\put(2251,-2386){\makebox(0,0)[b]{\smash{{\SetFigFont{8}{9.6}{\rmdefault}{\mddefault}{\updefault}{\color[rgb]{0,0,0}$c$}%
}}}}
\put(1951,-2686){\makebox(0,0)[b]{\smash{{\SetFigFont{8}{9.6}{\rmdefault}{\mddefault}{\updefault}{\color[rgb]{0,0,0}$c$}%
}}}}
\put(1651,-2986){\makebox(0,0)[b]{\smash{{\SetFigFont{8}{9.6}{\rmdefault}{\mddefault}{\updefault}{\color[rgb]{0,0,0}$c$}%
}}}}
\put(1351,-3286){\makebox(0,0)[b]{\smash{{\SetFigFont{8}{9.6}{\rmdefault}{\mddefault}{\updefault}{\color[rgb]{0,0,0}$c$}%
}}}}
\put(1051,-3586){\makebox(0,0)[b]{\smash{{\SetFigFont{8}{9.6}{\rmdefault}{\mddefault}{\updefault}{\color[rgb]{0,0,0}$c$}%
}}}}
\put(3451,-286){\makebox(0,0)[b]{\smash{{\SetFigFont{8}{9.6}{\rmdefault}{\mddefault}{\updefault}{\color[rgb]{0,0,0}$b$}%
}}}}
\put(3451,-586){\makebox(0,0)[b]{\smash{{\SetFigFont{8}{9.6}{\rmdefault}{\mddefault}{\updefault}{\color[rgb]{0,0,0}$b$}%
}}}}
\put(3451,-886){\makebox(0,0)[b]{\smash{{\SetFigFont{8}{9.6}{\rmdefault}{\mddefault}{\updefault}{\color[rgb]{0,0,0}$b$}%
}}}}
\put(3451,-1486){\makebox(0,0)[b]{\smash{{\SetFigFont{8}{9.6}{\rmdefault}{\mddefault}{\updefault}{\color[rgb]{0,0,0}$b$}%
}}}}
\put(3451,-1786){\makebox(0,0)[b]{\smash{{\SetFigFont{8}{9.6}{\rmdefault}{\mddefault}{\updefault}{\color[rgb]{0,0,0}$b$}%
}}}}
\put(3451,-2086){\makebox(0,0)[b]{\smash{{\SetFigFont{8}{9.6}{\rmdefault}{\mddefault}{\updefault}{\color[rgb]{0,0,0}$b$}%
}}}}
\put(3451,-2686){\makebox(0,0)[b]{\smash{{\SetFigFont{8}{9.6}{\rmdefault}{\mddefault}{\updefault}{\color[rgb]{0,0,0}$b$}%
}}}}
\put(3451,-3286){\makebox(0,0)[b]{\smash{{\SetFigFont{8}{9.6}{\rmdefault}{\mddefault}{\updefault}{\color[rgb]{0,0,0}$b$}%
}}}}
\put(3451,-3586){\makebox(0,0)[b]{\smash{{\SetFigFont{8}{9.6}{\rmdefault}{\mddefault}{\updefault}{\color[rgb]{0,0,0}$b$}%
}}}}
\put(4351,-3586){\makebox(0,0)[b]{\smash{{\SetFigFont{8}{9.6}{\rmdefault}{\mddefault}{\updefault}{\color[rgb]{0,0,0}$b$}%
}}}}
\put(4351,-2986){\makebox(0,0)[b]{\smash{{\SetFigFont{8}{9.6}{\rmdefault}{\mddefault}{\updefault}{\color[rgb]{0,0,0}$b$}%
}}}}
\put(4351,-2686){\makebox(0,0)[b]{\smash{{\SetFigFont{8}{9.6}{\rmdefault}{\mddefault}{\updefault}{\color[rgb]{0,0,0}$b$}%
}}}}
\put(4351,-2386){\makebox(0,0)[b]{\smash{{\SetFigFont{8}{9.6}{\rmdefault}{\mddefault}{\updefault}{\color[rgb]{0,0,0}$b$}%
}}}}
\put(4351,-1786){\makebox(0,0)[b]{\smash{{\SetFigFont{8}{9.6}{\rmdefault}{\mddefault}{\updefault}{\color[rgb]{0,0,0}$b$}%
}}}}
\put(4351,-1486){\makebox(0,0)[b]{\smash{{\SetFigFont{8}{9.6}{\rmdefault}{\mddefault}{\updefault}{\color[rgb]{0,0,0}$b$}%
}}}}
\put(4351,-1186){\makebox(0,0)[b]{\smash{{\SetFigFont{8}{9.6}{\rmdefault}{\mddefault}{\updefault}{\color[rgb]{0,0,0}$b$}%
}}}}
\put(4351,-586){\makebox(0,0)[b]{\smash{{\SetFigFont{8}{9.6}{\rmdefault}{\mddefault}{\updefault}{\color[rgb]{0,0,0}$b$}%
}}}}
\end{picture}%

%% file: 322Ex.pstex_t
\begin{picture}(0,0)%
\includegraphics{322Ex.eps}%
\end{picture}%
\setlength{\unitlength}{2763sp}%
\begingroup\makeatletter\ifx\SetFigFont\undefined%
\gdef\SetFigFont#1#2#3#4#5{%
  \reset@font\fontsize{#1}{#2pt}%
  \fontfamily{#3}\fontseries{#4}\fontshape{#5}%
  \selectfont}%
\fi\endgroup%
\begin{picture}(9477,5611)(136,-4850)
\put(4651,-4786){\makebox(0,0)[b]{\smash{{\SetFigFont{8}{9.6}{\rmdefault}{\mddefault}{\updefault}{\color[rgb]{0,0,0}$b$}%
}}}}
\put(3451,-2386){\makebox(0,0)[b]{\smash{{\SetFigFont{8}{9.6}{\rmdefault}{\mddefault}{\updefault}{\color[rgb]{0,0,0}$ac$}%
}}}}
\put(4351,-1486){\makebox(0,0)[b]{\smash{{\SetFigFont{8}{9.6}{\rmdefault}{\mddefault}{\updefault}{\color[rgb]{0,0,0}$ac$}%
}}}}
\put(6751,-2086){\makebox(0,0)[b]{\smash{{\SetFigFont{8}{9.6}{\rmdefault}{\mddefault}{\updefault}{\color[rgb]{0,0,0}$ac$}%
}}}}
\put(5851,-2986){\makebox(0,0)[b]{\smash{{\SetFigFont{8}{9.6}{\rmdefault}{\mddefault}{\updefault}{\color[rgb]{0,0,0}$ac$}%
}}}}
\put(4951,-3886){\makebox(0,0)[b]{\smash{{\SetFigFont{8}{9.6}{\rmdefault}{\mddefault}{\updefault}{\color[rgb]{0,0,0}$ac$}%
}}}}
\put(5251,-586){\makebox(0,0)[b]{\smash{{\SetFigFont{8}{9.6}{\rmdefault}{\mddefault}{\updefault}{\color[rgb]{0,0,0}$ac$}%
}}}}
\put(4651,-1786){\makebox(0,0)[b]{\smash{{\SetFigFont{8}{9.6}{\rmdefault}{\mddefault}{\updefault}{\color[rgb]{0,0,0}$ab$}%
}}}}
\put(4651,-3586){\makebox(0,0)[b]{\smash{{\SetFigFont{8}{9.6}{\rmdefault}{\mddefault}{\updefault}{\color[rgb]{0,0,0}$ab$}%
}}}}
\put(4651, 14){\makebox(0,0)[b]{\smash{{\SetFigFont{8}{9.6}{\rmdefault}{\mddefault}{\updefault}{\color[rgb]{0,0,0}$ab$}%
}}}}
\put(5551,-886){\makebox(0,0)[b]{\smash{{\SetFigFont{8}{9.6}{\rmdefault}{\mddefault}{\updefault}{\color[rgb]{0,0,0}$ab$}%
}}}}
\put(5551,-2686){\makebox(0,0)[b]{\smash{{\SetFigFont{8}{9.6}{\rmdefault}{\mddefault}{\updefault}{\color[rgb]{0,0,0}$ab$}%
}}}}
\put(5551,-4486){\makebox(0,0)[b]{\smash{{\SetFigFont{8}{9.6}{\rmdefault}{\mddefault}{\updefault}{\color[rgb]{0,0,0}$ab$}%
}}}}
\put(5551,-286){\makebox(0,0)[b]{\smash{{\SetFigFont{8}{9.6}{\rmdefault}{\mddefault}{\updefault}{\color[rgb]{0,0,0}$bc$}%
}}}}
\put(5551,-3286){\makebox(0,0)[b]{\smash{{\SetFigFont{8}{9.6}{\rmdefault}{\mddefault}{\updefault}{\color[rgb]{0,0,0}$bc$}%
}}}}
\put(4651,-1186){\makebox(0,0)[b]{\smash{{\SetFigFont{8}{9.6}{\rmdefault}{\mddefault}{\updefault}{\color[rgb]{0,0,0}$bc$}%
}}}}
\put(4651,-4186){\makebox(0,0)[b]{\smash{{\SetFigFont{8}{9.6}{\rmdefault}{\mddefault}{\updefault}{\color[rgb]{0,0,0}$bc$}%
}}}}
\put(151,-2686){\makebox(0,0)[b]{\smash{{\SetFigFont{8}{9.6}{\rmdefault}{\mddefault}{\updefault}{\color[rgb]{0,0,0}$10$}%
}}}}
\put(151, 14){\makebox(0,0)[b]{\smash{{\SetFigFont{8}{9.6}{\rmdefault}{\mddefault}{\updefault}{\color[rgb]{0,0,0}$1$}%
}}}}
\put(151,-286){\makebox(0,0)[b]{\smash{{\SetFigFont{8}{9.6}{\rmdefault}{\mddefault}{\updefault}{\color[rgb]{0,0,0}$2$}%
}}}}
\put(151,-586){\makebox(0,0)[b]{\smash{{\SetFigFont{8}{9.6}{\rmdefault}{\mddefault}{\updefault}{\color[rgb]{0,0,0}$3$}%
}}}}
\put(151,-886){\makebox(0,0)[b]{\smash{{\SetFigFont{8}{9.6}{\rmdefault}{\mddefault}{\updefault}{\color[rgb]{0,0,0}$4$}%
}}}}
\put(151,-1186){\makebox(0,0)[b]{\smash{{\SetFigFont{8}{9.6}{\rmdefault}{\mddefault}{\updefault}{\color[rgb]{0,0,0}$5$}%
}}}}
\put(151,-1486){\makebox(0,0)[b]{\smash{{\SetFigFont{8}{9.6}{\rmdefault}{\mddefault}{\updefault}{\color[rgb]{0,0,0}$6$}%
}}}}
\put(151,-1786){\makebox(0,0)[b]{\smash{{\SetFigFont{8}{9.6}{\rmdefault}{\mddefault}{\updefault}{\color[rgb]{0,0,0}$7$}%
}}}}
\put(151,-2086){\makebox(0,0)[b]{\smash{{\SetFigFont{8}{9.6}{\rmdefault}{\mddefault}{\updefault}{\color[rgb]{0,0,0}$8$}%
}}}}
\put(151,314){\makebox(0,0)[b]{\smash{{\SetFigFont{8}{9.6}{\rmdefault}{\mddefault}{\updefault}{\color[rgb]{0,0,0}$t$}%
}}}}
\put(151,-2386){\makebox(0,0)[b]{\smash{{\SetFigFont{8}{9.6}{\rmdefault}{\mddefault}{\updefault}{\color[rgb]{0,0,0}$9$}%
}}}}
\put(751,614){\makebox(0,0)[b]{\smash{{\SetFigFont{8}{9.6}{\rmdefault}{\mddefault}{\updefault}{\color[rgb]{0,0,0}$16$}%
}}}}
\put(151,-3286){\makebox(0,0)[b]{\smash{{\SetFigFont{8}{9.6}{\rmdefault}{\mddefault}{\updefault}{\color[rgb]{0,0,0}$12$}%
}}}}
\put(151,-3586){\makebox(0,0)[b]{\smash{{\SetFigFont{8}{9.6}{\rmdefault}{\mddefault}{\updefault}{\color[rgb]{0,0,0}$13$}%
}}}}
\put(151,-3886){\makebox(0,0)[b]{\smash{{\SetFigFont{8}{9.6}{\rmdefault}{\mddefault}{\updefault}{\color[rgb]{0,0,0}$14$}%
}}}}
\put(151,-4186){\makebox(0,0)[b]{\smash{{\SetFigFont{8}{9.6}{\rmdefault}{\mddefault}{\updefault}{\color[rgb]{0,0,0}$15$}%
}}}}
\put(151,-4486){\makebox(0,0)[b]{\smash{{\SetFigFont{8}{9.6}{\rmdefault}{\mddefault}{\updefault}{\color[rgb]{0,0,0}$16$}%
}}}}
\put(151,-2986){\makebox(0,0)[b]{\smash{{\SetFigFont{8}{9.6}{\rmdefault}{\mddefault}{\updefault}{\color[rgb]{0,0,0}$11$}%
}}}}
\put(751,314){\makebox(0,0)[b]{\smash{{\SetFigFont{8}{9.6}{\rmdefault}{\mddefault}{\updefault}{\color[rgb]{0,0,0}$a$}%
}}}}
\put(2551,614){\makebox(0,0)[b]{\smash{{\SetFigFont{8}{9.6}{\rmdefault}{\mddefault}{\updefault}{\color[rgb]{0,0,0}$10$}%
}}}}
\put(4351,614){\makebox(0,0)[b]{\smash{{\SetFigFont{8}{9.6}{\rmdefault}{\mddefault}{\updefault}{\color[rgb]{0,0,0}$4$}%
}}}}
\put(4651,614){\makebox(0,0)[b]{\smash{{\SetFigFont{8}{9.6}{\rmdefault}{\mddefault}{\updefault}{\color[rgb]{0,0,0}$3$}%
}}}}
\put(5551,614){\makebox(0,0)[b]{\smash{{\SetFigFont{8}{9.6}{\rmdefault}{\mddefault}{\updefault}{\color[rgb]{0,0,0}$0$}%
}}}}
\put(6151,614){\makebox(0,0)[b]{\smash{{\SetFigFont{8}{9.6}{\rmdefault}{\mddefault}{\updefault}{\color[rgb]{0,0,0}$-2$}%
}}}}
\put(9151,614){\makebox(0,0)[b]{\smash{{\SetFigFont{8}{9.6}{\rmdefault}{\mddefault}{\updefault}{\color[rgb]{0,0,0}$-12$}%
}}}}
\put(2551,314){\makebox(0,0)[b]{\smash{{\SetFigFont{8}{9.6}{\rmdefault}{\mddefault}{\updefault}{\color[rgb]{0,0,0}$a$}%
}}}}
\put(4351,314){\makebox(0,0)[b]{\smash{{\SetFigFont{8}{9.6}{\rmdefault}{\mddefault}{\updefault}{\color[rgb]{0,0,0}$a$}%
}}}}
\put(4651,314){\makebox(0,0)[b]{\smash{{\SetFigFont{8}{9.6}{\rmdefault}{\mddefault}{\updefault}{\color[rgb]{0,0,0}$b$}%
}}}}
\put(5551,314){\makebox(0,0)[b]{\smash{{\SetFigFont{8}{9.6}{\rmdefault}{\mddefault}{\updefault}{\color[rgb]{0,0,0}$b$}%
}}}}
\put(6151,314){\makebox(0,0)[b]{\smash{{\SetFigFont{8}{9.6}{\rmdefault}{\mddefault}{\updefault}{\color[rgb]{0,0,0}$c$}%
}}}}
\put(9151,314){\makebox(0,0)[b]{\smash{{\SetFigFont{8}{9.6}{\rmdefault}{\mddefault}{\updefault}{\color[rgb]{0,0,0}$c$}%
}}}}
\put(1051, 14){\makebox(0,0)[b]{\smash{{\SetFigFont{8}{9.6}{\rmdefault}{\mddefault}{\updefault}{\color[rgb]{0,0,0}$a$}%
}}}}
\put(1351,-286){\makebox(0,0)[b]{\smash{{\SetFigFont{8}{9.6}{\rmdefault}{\mddefault}{\updefault}{\color[rgb]{0,0,0}$a$}%
}}}}
\put(1651,-586){\makebox(0,0)[b]{\smash{{\SetFigFont{8}{9.6}{\rmdefault}{\mddefault}{\updefault}{\color[rgb]{0,0,0}$a$}%
}}}}
\put(1951,-886){\makebox(0,0)[b]{\smash{{\SetFigFont{8}{9.6}{\rmdefault}{\mddefault}{\updefault}{\color[rgb]{0,0,0}$a$}%
}}}}
\put(2251,-1186){\makebox(0,0)[b]{\smash{{\SetFigFont{8}{9.6}{\rmdefault}{\mddefault}{\updefault}{\color[rgb]{0,0,0}$a$}%
}}}}
\put(2551,-1486){\makebox(0,0)[b]{\smash{{\SetFigFont{8}{9.6}{\rmdefault}{\mddefault}{\updefault}{\color[rgb]{0,0,0}$a$}%
}}}}
\put(2851,-1786){\makebox(0,0)[b]{\smash{{\SetFigFont{8}{9.6}{\rmdefault}{\mddefault}{\updefault}{\color[rgb]{0,0,0}$a$}%
}}}}
\put(3151,-2086){\makebox(0,0)[b]{\smash{{\SetFigFont{8}{9.6}{\rmdefault}{\mddefault}{\updefault}{\color[rgb]{0,0,0}$a$}%
}}}}
\put(3751,-2686){\makebox(0,0)[b]{\smash{{\SetFigFont{8}{9.6}{\rmdefault}{\mddefault}{\updefault}{\color[rgb]{0,0,0}$a$}%
}}}}
\put(4051,-2986){\makebox(0,0)[b]{\smash{{\SetFigFont{8}{9.6}{\rmdefault}{\mddefault}{\updefault}{\color[rgb]{0,0,0}$a$}%
}}}}
\put(4351,-3286){\makebox(0,0)[b]{\smash{{\SetFigFont{8}{9.6}{\rmdefault}{\mddefault}{\updefault}{\color[rgb]{0,0,0}$a$}%
}}}}
\put(5251,-4186){\makebox(0,0)[b]{\smash{{\SetFigFont{8}{9.6}{\rmdefault}{\mddefault}{\updefault}{\color[rgb]{0,0,0}$a$}%
}}}}
\put(5851,-4786){\makebox(0,0)[b]{\smash{{\SetFigFont{8}{9.6}{\rmdefault}{\mddefault}{\updefault}{\color[rgb]{0,0,0}$a$}%
}}}}
\put(2851, 14){\makebox(0,0)[b]{\smash{{\SetFigFont{8}{9.6}{\rmdefault}{\mddefault}{\updefault}{\color[rgb]{0,0,0}$a$}%
}}}}
\put(3151,-286){\makebox(0,0)[b]{\smash{{\SetFigFont{8}{9.6}{\rmdefault}{\mddefault}{\updefault}{\color[rgb]{0,0,0}$a$}%
}}}}
\put(3451,-586){\makebox(0,0)[b]{\smash{{\SetFigFont{8}{9.6}{\rmdefault}{\mddefault}{\updefault}{\color[rgb]{0,0,0}$a$}%
}}}}
\put(3751,-886){\makebox(0,0)[b]{\smash{{\SetFigFont{8}{9.6}{\rmdefault}{\mddefault}{\updefault}{\color[rgb]{0,0,0}$a$}%
}}}}
\put(4051,-1186){\makebox(0,0)[b]{\smash{{\SetFigFont{8}{9.6}{\rmdefault}{\mddefault}{\updefault}{\color[rgb]{0,0,0}$a$}%
}}}}
\put(4951,-2086){\makebox(0,0)[b]{\smash{{\SetFigFont{8}{9.6}{\rmdefault}{\mddefault}{\updefault}{\color[rgb]{0,0,0}$a$}%
}}}}
\put(5251,-2386){\makebox(0,0)[b]{\smash{{\SetFigFont{8}{9.6}{\rmdefault}{\mddefault}{\updefault}{\color[rgb]{0,0,0}$a$}%
}}}}
\put(6151,-3286){\makebox(0,0)[b]{\smash{{\SetFigFont{8}{9.6}{\rmdefault}{\mddefault}{\updefault}{\color[rgb]{0,0,0}$a$}%
}}}}
\put(6451,-3586){\makebox(0,0)[b]{\smash{{\SetFigFont{8}{9.6}{\rmdefault}{\mddefault}{\updefault}{\color[rgb]{0,0,0}$a$}%
}}}}
\put(6751,-3886){\makebox(0,0)[b]{\smash{{\SetFigFont{8}{9.6}{\rmdefault}{\mddefault}{\updefault}{\color[rgb]{0,0,0}$a$}%
}}}}
\put(7051,-4186){\makebox(0,0)[b]{\smash{{\SetFigFont{8}{9.6}{\rmdefault}{\mddefault}{\updefault}{\color[rgb]{0,0,0}$a$}%
}}}}
\put(7351,-4486){\makebox(0,0)[b]{\smash{{\SetFigFont{8}{9.6}{\rmdefault}{\mddefault}{\updefault}{\color[rgb]{0,0,0}$a$}%
}}}}
\put(7651,-4786){\makebox(0,0)[b]{\smash{{\SetFigFont{8}{9.6}{\rmdefault}{\mddefault}{\updefault}{\color[rgb]{0,0,0}$a$}%
}}}}
\put(4951,-286){\makebox(0,0)[b]{\smash{{\SetFigFont{8}{9.6}{\rmdefault}{\mddefault}{\updefault}{\color[rgb]{0,0,0}$a$}%
}}}}
\put(5851,-1186){\makebox(0,0)[b]{\smash{{\SetFigFont{8}{9.6}{\rmdefault}{\mddefault}{\updefault}{\color[rgb]{0,0,0}$a$}%
}}}}
\put(6151,-1486){\makebox(0,0)[b]{\smash{{\SetFigFont{8}{9.6}{\rmdefault}{\mddefault}{\updefault}{\color[rgb]{0,0,0}$a$}%
}}}}
\put(6451,-1786){\makebox(0,0)[b]{\smash{{\SetFigFont{8}{9.6}{\rmdefault}{\mddefault}{\updefault}{\color[rgb]{0,0,0}$a$}%
}}}}
\put(7051,-2386){\makebox(0,0)[b]{\smash{{\SetFigFont{8}{9.6}{\rmdefault}{\mddefault}{\updefault}{\color[rgb]{0,0,0}$a$}%
}}}}
\put(7351,-2686){\makebox(0,0)[b]{\smash{{\SetFigFont{8}{9.6}{\rmdefault}{\mddefault}{\updefault}{\color[rgb]{0,0,0}$a$}%
}}}}
\put(7651,-2986){\makebox(0,0)[b]{\smash{{\SetFigFont{8}{9.6}{\rmdefault}{\mddefault}{\updefault}{\color[rgb]{0,0,0}$a$}%
}}}}
\put(7951,-3286){\makebox(0,0)[b]{\smash{{\SetFigFont{8}{9.6}{\rmdefault}{\mddefault}{\updefault}{\color[rgb]{0,0,0}$a$}%
}}}}
\put(8251,-3586){\makebox(0,0)[b]{\smash{{\SetFigFont{8}{9.6}{\rmdefault}{\mddefault}{\updefault}{\color[rgb]{0,0,0}$a$}%
}}}}
\put(8551,-3886){\makebox(0,0)[b]{\smash{{\SetFigFont{8}{9.6}{\rmdefault}{\mddefault}{\updefault}{\color[rgb]{0,0,0}$a$}%
}}}}
\put(8851,-4186){\makebox(0,0)[b]{\smash{{\SetFigFont{8}{9.6}{\rmdefault}{\mddefault}{\updefault}{\color[rgb]{0,0,0}$a$}%
}}}}
\put(9151,-4486){\makebox(0,0)[b]{\smash{{\SetFigFont{8}{9.6}{\rmdefault}{\mddefault}{\updefault}{\color[rgb]{0,0,0}$a$}%
}}}}
\put(9451,-4786){\makebox(0,0)[b]{\smash{{\SetFigFont{8}{9.6}{\rmdefault}{\mddefault}{\updefault}{\color[rgb]{0,0,0}$a$}%
}}}}
\put(8851, 14){\makebox(0,0)[b]{\smash{{\SetFigFont{8}{9.6}{\rmdefault}{\mddefault}{\updefault}{\color[rgb]{0,0,0}$c$}%
}}}}
\put(8551,-286){\makebox(0,0)[b]{\smash{{\SetFigFont{8}{9.6}{\rmdefault}{\mddefault}{\updefault}{\color[rgb]{0,0,0}$c$}%
}}}}
\put(8251,-586){\makebox(0,0)[b]{\smash{{\SetFigFont{8}{9.6}{\rmdefault}{\mddefault}{\updefault}{\color[rgb]{0,0,0}$c$}%
}}}}
\put(7951,-886){\makebox(0,0)[b]{\smash{{\SetFigFont{8}{9.6}{\rmdefault}{\mddefault}{\updefault}{\color[rgb]{0,0,0}$c$}%
}}}}
\put(7651,-1186){\makebox(0,0)[b]{\smash{{\SetFigFont{8}{9.6}{\rmdefault}{\mddefault}{\updefault}{\color[rgb]{0,0,0}$c$}%
}}}}
\put(7351,-1486){\makebox(0,0)[b]{\smash{{\SetFigFont{8}{9.6}{\rmdefault}{\mddefault}{\updefault}{\color[rgb]{0,0,0}$c$}%
}}}}
\put(6451,-2386){\makebox(0,0)[b]{\smash{{\SetFigFont{8}{9.6}{\rmdefault}{\mddefault}{\updefault}{\color[rgb]{0,0,0}$c$}%
}}}}
\put(6151,-2686){\makebox(0,0)[b]{\smash{{\SetFigFont{8}{9.6}{\rmdefault}{\mddefault}{\updefault}{\color[rgb]{0,0,0}$c$}%
}}}}
\put(5251,-3586){\makebox(0,0)[b]{\smash{{\SetFigFont{8}{9.6}{\rmdefault}{\mddefault}{\updefault}{\color[rgb]{0,0,0}$c$}%
}}}}
\put(4351,-4486){\makebox(0,0)[b]{\smash{{\SetFigFont{8}{9.6}{\rmdefault}{\mddefault}{\updefault}{\color[rgb]{0,0,0}$c$}%
}}}}
\put(4051,-4786){\makebox(0,0)[b]{\smash{{\SetFigFont{8}{9.6}{\rmdefault}{\mddefault}{\updefault}{\color[rgb]{0,0,0}$c$}%
}}}}
\put(7051,-1786){\makebox(0,0)[b]{\smash{{\SetFigFont{8}{9.6}{\rmdefault}{\mddefault}{\updefault}{\color[rgb]{0,0,0}$c$}%
}}}}
\put(5851, 14){\makebox(0,0)[b]{\smash{{\SetFigFont{8}{9.6}{\rmdefault}{\mddefault}{\updefault}{\color[rgb]{0,0,0}$c$}%
}}}}
\put(4951,-886){\makebox(0,0)[b]{\smash{{\SetFigFont{8}{9.6}{\rmdefault}{\mddefault}{\updefault}{\color[rgb]{0,0,0}$c$}%
}}}}
\put(3751,-2086){\makebox(0,0)[b]{\smash{{\SetFigFont{8}{9.6}{\rmdefault}{\mddefault}{\updefault}{\color[rgb]{0,0,0}$c$}%
}}}}
\put(3151,-2686){\makebox(0,0)[b]{\smash{{\SetFigFont{8}{9.6}{\rmdefault}{\mddefault}{\updefault}{\color[rgb]{0,0,0}$c$}%
}}}}
\put(2851,-2986){\makebox(0,0)[b]{\smash{{\SetFigFont{8}{9.6}{\rmdefault}{\mddefault}{\updefault}{\color[rgb]{0,0,0}$c$}%
}}}}
\put(2551,-3286){\makebox(0,0)[b]{\smash{{\SetFigFont{8}{9.6}{\rmdefault}{\mddefault}{\updefault}{\color[rgb]{0,0,0}$c$}%
}}}}
\put(2251,-3586){\makebox(0,0)[b]{\smash{{\SetFigFont{8}{9.6}{\rmdefault}{\mddefault}{\updefault}{\color[rgb]{0,0,0}$c$}%
}}}}
\put(1951,-3886){\makebox(0,0)[b]{\smash{{\SetFigFont{8}{9.6}{\rmdefault}{\mddefault}{\updefault}{\color[rgb]{0,0,0}$c$}%
}}}}
\put(1651,-4186){\makebox(0,0)[b]{\smash{{\SetFigFont{8}{9.6}{\rmdefault}{\mddefault}{\updefault}{\color[rgb]{0,0,0}$c$}%
}}}}
\put(1351,-4486){\makebox(0,0)[b]{\smash{{\SetFigFont{8}{9.6}{\rmdefault}{\mddefault}{\updefault}{\color[rgb]{0,0,0}$c$}%
}}}}
\put(1051,-4786){\makebox(0,0)[b]{\smash{{\SetFigFont{8}{9.6}{\rmdefault}{\mddefault}{\updefault}{\color[rgb]{0,0,0}$c$}%
}}}}
\put(4051,-1786){\makebox(0,0)[b]{\smash{{\SetFigFont{8}{9.6}{\rmdefault}{\mddefault}{\updefault}{\color[rgb]{0,0,0}$c$}%
}}}}
\put(5551, 14){\makebox(0,0)[b]{\smash{{\SetFigFont{8}{9.6}{\rmdefault}{\mddefault}{\updefault}{\color[rgb]{0,0,0}$b$}%
}}}}
\put(5551,-586){\makebox(0,0)[b]{\smash{{\SetFigFont{8}{9.6}{\rmdefault}{\mddefault}{\updefault}{\color[rgb]{0,0,0}$b$}%
}}}}
\put(5551,-1186){\makebox(0,0)[b]{\smash{{\SetFigFont{8}{9.6}{\rmdefault}{\mddefault}{\updefault}{\color[rgb]{0,0,0}$b$}%
}}}}
\put(5551,-1486){\makebox(0,0)[b]{\smash{{\SetFigFont{8}{9.6}{\rmdefault}{\mddefault}{\updefault}{\color[rgb]{0,0,0}$b$}%
}}}}
\put(5551,-1786){\makebox(0,0)[b]{\smash{{\SetFigFont{8}{9.6}{\rmdefault}{\mddefault}{\updefault}{\color[rgb]{0,0,0}$b$}%
}}}}
\put(5551,-2086){\makebox(0,0)[b]{\smash{{\SetFigFont{8}{9.6}{\rmdefault}{\mddefault}{\updefault}{\color[rgb]{0,0,0}$b$}%
}}}}
\put(5551,-2386){\makebox(0,0)[b]{\smash{{\SetFigFont{8}{9.6}{\rmdefault}{\mddefault}{\updefault}{\color[rgb]{0,0,0}$b$}%
}}}}
\put(5551,-2986){\makebox(0,0)[b]{\smash{{\SetFigFont{8}{9.6}{\rmdefault}{\mddefault}{\updefault}{\color[rgb]{0,0,0}$b$}%
}}}}
\put(5551,-3586){\makebox(0,0)[b]{\smash{{\SetFigFont{8}{9.6}{\rmdefault}{\mddefault}{\updefault}{\color[rgb]{0,0,0}$b$}%
}}}}
\put(5551,-3886){\makebox(0,0)[b]{\smash{{\SetFigFont{8}{9.6}{\rmdefault}{\mddefault}{\updefault}{\color[rgb]{0,0,0}$b$}%
}}}}
\put(5551,-4186){\makebox(0,0)[b]{\smash{{\SetFigFont{8}{9.6}{\rmdefault}{\mddefault}{\updefault}{\color[rgb]{0,0,0}$b$}%
}}}}
\put(5551,-4786){\makebox(0,0)[b]{\smash{{\SetFigFont{8}{9.6}{\rmdefault}{\mddefault}{\updefault}{\color[rgb]{0,0,0}$b$}%
}}}}
\put(4651,-286){\makebox(0,0)[b]{\smash{{\SetFigFont{8}{9.6}{\rmdefault}{\mddefault}{\updefault}{\color[rgb]{0,0,0}$b$}%
}}}}
\put(4651,-586){\makebox(0,0)[b]{\smash{{\SetFigFont{8}{9.6}{\rmdefault}{\mddefault}{\updefault}{\color[rgb]{0,0,0}$b$}%
}}}}
\put(4651,-886){\makebox(0,0)[b]{\smash{{\SetFigFont{8}{9.6}{\rmdefault}{\mddefault}{\updefault}{\color[rgb]{0,0,0}$b$}%
}}}}
\put(4651,-1486){\makebox(0,0)[b]{\smash{{\SetFigFont{8}{9.6}{\rmdefault}{\mddefault}{\updefault}{\color[rgb]{0,0,0}$b$}%
}}}}
\put(4651,-2086){\makebox(0,0)[b]{\smash{{\SetFigFont{8}{9.6}{\rmdefault}{\mddefault}{\updefault}{\color[rgb]{0,0,0}$b$}%
}}}}
\put(4651,-2386){\makebox(0,0)[b]{\smash{{\SetFigFont{8}{9.6}{\rmdefault}{\mddefault}{\updefault}{\color[rgb]{0,0,0}$b$}%
}}}}
\put(4651,-2686){\makebox(0,0)[b]{\smash{{\SetFigFont{8}{9.6}{\rmdefault}{\mddefault}{\updefault}{\color[rgb]{0,0,0}$b$}%
}}}}
\put(4651,-2986){\makebox(0,0)[b]{\smash{{\SetFigFont{8}{9.6}{\rmdefault}{\mddefault}{\updefault}{\color[rgb]{0,0,0}$b$}%
}}}}
\put(4651,-3286){\makebox(0,0)[b]{\smash{{\SetFigFont{8}{9.6}{\rmdefault}{\mddefault}{\updefault}{\color[rgb]{0,0,0}$b$}%
}}}}
\put(4651,-3886){\makebox(0,0)[b]{\smash{{\SetFigFont{8}{9.6}{\rmdefault}{\mddefault}{\updefault}{\color[rgb]{0,0,0}$b$}%
}}}}
\put(4651,-4486){\makebox(0,0)[b]{\smash{{\SetFigFont{8}{9.6}{\rmdefault}{\mddefault}{\updefault}{\color[rgb]{0,0,0}$b$}%
}}}}
\end{picture}%

%% file: 125Ex.pstex_t
\begin{picture}(0,0)%
\includegraphics{125Ex.eps}%
\end{picture}%
\setlength{\unitlength}{2763sp}%
\begingroup\makeatletter\ifx\SetFigFont\undefined%
\gdef\SetFigFont#1#2#3#4#5{%
  \reset@font\fontsize{#1}{#2pt}%
  \fontfamily{#3}\fontseries{#4}\fontshape{#5}%
  \selectfont}%
\fi\endgroup%
\begin{picture}(9777,5911)(136,-5150)
\put(4951,-5086){\makebox(0,0)[b]{\smash{{\SetFigFont{8}{9.6}{\rmdefault}{\mddefault}{\updefault}{\color[rgb]{0,0,0}$b$}%
}}}}
\put(5251,-586){\makebox(0,0)[b]{\smash{{\SetFigFont{8}{9.6}{\rmdefault}{\mddefault}{\updefault}{\color[rgb]{0,0,0}$ac$}%
}}}}
\put(6151,-1486){\makebox(0,0)[b]{\smash{{\SetFigFont{8}{9.6}{\rmdefault}{\mddefault}{\updefault}{\color[rgb]{0,0,0}$ac$}%
}}}}
\put(6451,-1786){\makebox(0,0)[b]{\smash{{\SetFigFont{8}{9.6}{\rmdefault}{\mddefault}{\updefault}{\color[rgb]{0,0,0}$ac$}%
}}}}
\put(6751,-2086){\makebox(0,0)[b]{\smash{{\SetFigFont{8}{9.6}{\rmdefault}{\mddefault}{\updefault}{\color[rgb]{0,0,0}$ac$}%
}}}}
\put(7051,-2386){\makebox(0,0)[b]{\smash{{\SetFigFont{8}{9.6}{\rmdefault}{\mddefault}{\updefault}{\color[rgb]{0,0,0}$ac$}%
}}}}
\put(4651, 14){\makebox(0,0)[b]{\smash{{\SetFigFont{8}{9.6}{\rmdefault}{\mddefault}{\updefault}{\color[rgb]{0,0,0}$ab$}%
}}}}
\put(4951,-286){\makebox(0,0)[b]{\smash{{\SetFigFont{8}{9.6}{\rmdefault}{\mddefault}{\updefault}{\color[rgb]{0,0,0}$ab$}%
}}}}
\put(4951,-886){\makebox(0,0)[b]{\smash{{\SetFigFont{8}{9.6}{\rmdefault}{\mddefault}{\updefault}{\color[rgb]{0,0,0}$bc$}%
}}}}
\put(4651,-1186){\makebox(0,0)[b]{\smash{{\SetFigFont{8}{9.6}{\rmdefault}{\mddefault}{\updefault}{\color[rgb]{0,0,0}$bc$}%
}}}}
\put(4951,-2686){\makebox(0,0)[b]{\smash{{\SetFigFont{8}{9.6}{\rmdefault}{\mddefault}{\updefault}{\color[rgb]{0,0,0}$bc$}%
}}}}
\put(4651,-2986){\makebox(0,0)[b]{\smash{{\SetFigFont{8}{9.6}{\rmdefault}{\mddefault}{\updefault}{\color[rgb]{0,0,0}$bc$}%
}}}}
\put(4951,-3286){\makebox(0,0)[b]{\smash{{\SetFigFont{8}{9.6}{\rmdefault}{\mddefault}{\updefault}{\color[rgb]{0,0,0}$bc$}%
}}}}
\put(4651,-3586){\makebox(0,0)[b]{\smash{{\SetFigFont{8}{9.6}{\rmdefault}{\mddefault}{\updefault}{\color[rgb]{0,0,0}$bc$}%
}}}}
\put(4951,-3886){\makebox(0,0)[b]{\smash{{\SetFigFont{8}{9.6}{\rmdefault}{\mddefault}{\updefault}{\color[rgb]{0,0,0}$bc$}%
}}}}
\put(4651,-4186){\makebox(0,0)[b]{\smash{{\SetFigFont{8}{9.6}{\rmdefault}{\mddefault}{\updefault}{\color[rgb]{0,0,0}$bc$}%
}}}}
\put(4951,-4486){\makebox(0,0)[b]{\smash{{\SetFigFont{8}{9.6}{\rmdefault}{\mddefault}{\updefault}{\color[rgb]{0,0,0}$bc$}%
}}}}
\put(4651,-4786){\makebox(0,0)[b]{\smash{{\SetFigFont{8}{9.6}{\rmdefault}{\mddefault}{\updefault}{\color[rgb]{0,0,0}$bc$}%
}}}}
\put(151,-2686){\makebox(0,0)[b]{\smash{{\SetFigFont{8}{9.6}{\rmdefault}{\mddefault}{\updefault}{\color[rgb]{0,0,0}$10$}%
}}}}
\put(151, 14){\makebox(0,0)[b]{\smash{{\SetFigFont{8}{9.6}{\rmdefault}{\mddefault}{\updefault}{\color[rgb]{0,0,0}$1$}%
}}}}
\put(151,-286){\makebox(0,0)[b]{\smash{{\SetFigFont{8}{9.6}{\rmdefault}{\mddefault}{\updefault}{\color[rgb]{0,0,0}$2$}%
}}}}
\put(151,-586){\makebox(0,0)[b]{\smash{{\SetFigFont{8}{9.6}{\rmdefault}{\mddefault}{\updefault}{\color[rgb]{0,0,0}$3$}%
}}}}
\put(151,-886){\makebox(0,0)[b]{\smash{{\SetFigFont{8}{9.6}{\rmdefault}{\mddefault}{\updefault}{\color[rgb]{0,0,0}$4$}%
}}}}
\put(151,-1186){\makebox(0,0)[b]{\smash{{\SetFigFont{8}{9.6}{\rmdefault}{\mddefault}{\updefault}{\color[rgb]{0,0,0}$5$}%
}}}}
\put(151,-1486){\makebox(0,0)[b]{\smash{{\SetFigFont{8}{9.6}{\rmdefault}{\mddefault}{\updefault}{\color[rgb]{0,0,0}$6$}%
}}}}
\put(151,-1786){\makebox(0,0)[b]{\smash{{\SetFigFont{8}{9.6}{\rmdefault}{\mddefault}{\updefault}{\color[rgb]{0,0,0}$7$}%
}}}}
\put(151,-2086){\makebox(0,0)[b]{\smash{{\SetFigFont{8}{9.6}{\rmdefault}{\mddefault}{\updefault}{\color[rgb]{0,0,0}$8$}%
}}}}
\put(151,314){\makebox(0,0)[b]{\smash{{\SetFigFont{8}{9.6}{\rmdefault}{\mddefault}{\updefault}{\color[rgb]{0,0,0}$t$}%
}}}}
\put(151,-2386){\makebox(0,0)[b]{\smash{{\SetFigFont{8}{9.6}{\rmdefault}{\mddefault}{\updefault}{\color[rgb]{0,0,0}$9$}%
}}}}
\put(751,614){\makebox(0,0)[b]{\smash{{\SetFigFont{8}{9.6}{\rmdefault}{\mddefault}{\updefault}{\color[rgb]{0,0,0}$14$}%
}}}}
\put(151,-3286){\makebox(0,0)[b]{\smash{{\SetFigFont{8}{9.6}{\rmdefault}{\mddefault}{\updefault}{\color[rgb]{0,0,0}$12$}%
}}}}
\put(151,-3586){\makebox(0,0)[b]{\smash{{\SetFigFont{8}{9.6}{\rmdefault}{\mddefault}{\updefault}{\color[rgb]{0,0,0}$13$}%
}}}}
\put(151,-3886){\makebox(0,0)[b]{\smash{{\SetFigFont{8}{9.6}{\rmdefault}{\mddefault}{\updefault}{\color[rgb]{0,0,0}$14$}%
}}}}
\put(151,-4186){\makebox(0,0)[b]{\smash{{\SetFigFont{8}{9.6}{\rmdefault}{\mddefault}{\updefault}{\color[rgb]{0,0,0}$15$}%
}}}}
\put(151,-4486){\makebox(0,0)[b]{\smash{{\SetFigFont{8}{9.6}{\rmdefault}{\mddefault}{\updefault}{\color[rgb]{0,0,0}$16$}%
}}}}
\put(151,-4786){\makebox(0,0)[b]{\smash{{\SetFigFont{8}{9.6}{\rmdefault}{\mddefault}{\updefault}{\color[rgb]{0,0,0}$17$}%
}}}}
\put(151,-2986){\makebox(0,0)[b]{\smash{{\SetFigFont{8}{9.6}{\rmdefault}{\mddefault}{\updefault}{\color[rgb]{0,0,0}$11$}%
}}}}
\put(4351,614){\makebox(0,0)[b]{\smash{{\SetFigFont{8}{9.6}{\rmdefault}{\mddefault}{\updefault}{\color[rgb]{0,0,0}$2$}%
}}}}
\put(4651,614){\makebox(0,0)[b]{\smash{{\SetFigFont{8}{9.6}{\rmdefault}{\mddefault}{\updefault}{\color[rgb]{0,0,0}$1$}%
}}}}
\put(4951,614){\makebox(0,0)[b]{\smash{{\SetFigFont{8}{9.6}{\rmdefault}{\mddefault}{\updefault}{\color[rgb]{0,0,0}$0$}%
}}}}
\put(6151,614){\makebox(0,0)[b]{\smash{{\SetFigFont{8}{9.6}{\rmdefault}{\mddefault}{\updefault}{\color[rgb]{0,0,0}$-4$}%
}}}}
\put(7951,614){\makebox(0,0)[b]{\smash{{\SetFigFont{8}{9.6}{\rmdefault}{\mddefault}{\updefault}{\color[rgb]{0,0,0}$-10$}%
}}}}
\put(8551,614){\makebox(0,0)[b]{\smash{{\SetFigFont{8}{9.6}{\rmdefault}{\mddefault}{\updefault}{\color[rgb]{0,0,0}$-12$}%
}}}}
\put(9151,614){\makebox(0,0)[b]{\smash{{\SetFigFont{8}{9.6}{\rmdefault}{\mddefault}{\updefault}{\color[rgb]{0,0,0}$-14$}%
}}}}
\put(9751,614){\makebox(0,0)[b]{\smash{{\SetFigFont{8}{9.6}{\rmdefault}{\mddefault}{\updefault}{\color[rgb]{0,0,0}$-16$}%
}}}}
\put(4351,314){\makebox(0,0)[b]{\smash{{\SetFigFont{8}{9.6}{\rmdefault}{\mddefault}{\updefault}{\color[rgb]{0,0,0}$a$}%
}}}}
\put(4651,314){\makebox(0,0)[b]{\smash{{\SetFigFont{8}{9.6}{\rmdefault}{\mddefault}{\updefault}{\color[rgb]{0,0,0}$b$}%
}}}}
\put(4951,314){\makebox(0,0)[b]{\smash{{\SetFigFont{8}{9.6}{\rmdefault}{\mddefault}{\updefault}{\color[rgb]{0,0,0}$b$}%
}}}}
\put(6151,314){\makebox(0,0)[b]{\smash{{\SetFigFont{8}{9.6}{\rmdefault}{\mddefault}{\updefault}{\color[rgb]{0,0,0}$c$}%
}}}}
\put(7951,314){\makebox(0,0)[b]{\smash{{\SetFigFont{8}{9.6}{\rmdefault}{\mddefault}{\updefault}{\color[rgb]{0,0,0}$c$}%
}}}}
\put(8551,314){\makebox(0,0)[b]{\smash{{\SetFigFont{8}{9.6}{\rmdefault}{\mddefault}{\updefault}{\color[rgb]{0,0,0}$c$}%
}}}}
\put(9151,314){\makebox(0,0)[b]{\smash{{\SetFigFont{8}{9.6}{\rmdefault}{\mddefault}{\updefault}{\color[rgb]{0,0,0}$c$}%
}}}}
\put(9751,314){\makebox(0,0)[b]{\smash{{\SetFigFont{8}{9.6}{\rmdefault}{\mddefault}{\updefault}{\color[rgb]{0,0,0}$c$}%
}}}}
\put(5551,-886){\makebox(0,0)[b]{\smash{{\SetFigFont{8}{9.6}{\rmdefault}{\mddefault}{\updefault}{\color[rgb]{0,0,0}$a$}%
}}}}
\put(5851,-1186){\makebox(0,0)[b]{\smash{{\SetFigFont{8}{9.6}{\rmdefault}{\mddefault}{\updefault}{\color[rgb]{0,0,0}$a$}%
}}}}
\put(7351,-2686){\makebox(0,0)[b]{\smash{{\SetFigFont{8}{9.6}{\rmdefault}{\mddefault}{\updefault}{\color[rgb]{0,0,0}$a$}%
}}}}
\put(7651,-2986){\makebox(0,0)[b]{\smash{{\SetFigFont{8}{9.6}{\rmdefault}{\mddefault}{\updefault}{\color[rgb]{0,0,0}$a$}%
}}}}
\put(7951,-3286){\makebox(0,0)[b]{\smash{{\SetFigFont{8}{9.6}{\rmdefault}{\mddefault}{\updefault}{\color[rgb]{0,0,0}$a$}%
}}}}
\put(8251,-3586){\makebox(0,0)[b]{\smash{{\SetFigFont{8}{9.6}{\rmdefault}{\mddefault}{\updefault}{\color[rgb]{0,0,0}$a$}%
}}}}
\put(8551,-3886){\makebox(0,0)[b]{\smash{{\SetFigFont{8}{9.6}{\rmdefault}{\mddefault}{\updefault}{\color[rgb]{0,0,0}$a$}%
}}}}
\put(8851,-4186){\makebox(0,0)[b]{\smash{{\SetFigFont{8}{9.6}{\rmdefault}{\mddefault}{\updefault}{\color[rgb]{0,0,0}$a$}%
}}}}
\put(9151,-4486){\makebox(0,0)[b]{\smash{{\SetFigFont{8}{9.6}{\rmdefault}{\mddefault}{\updefault}{\color[rgb]{0,0,0}$a$}%
}}}}
\put(9451,-4786){\makebox(0,0)[b]{\smash{{\SetFigFont{8}{9.6}{\rmdefault}{\mddefault}{\updefault}{\color[rgb]{0,0,0}$a$}%
}}}}
\put(9751,-5086){\makebox(0,0)[b]{\smash{{\SetFigFont{8}{9.6}{\rmdefault}{\mddefault}{\updefault}{\color[rgb]{0,0,0}$a$}%
}}}}
\put(7651, 14){\makebox(0,0)[b]{\smash{{\SetFigFont{8}{9.6}{\rmdefault}{\mddefault}{\updefault}{\color[rgb]{0,0,0}$c$}%
}}}}
\put(7351,-286){\makebox(0,0)[b]{\smash{{\SetFigFont{8}{9.6}{\rmdefault}{\mddefault}{\updefault}{\color[rgb]{0,0,0}$c$}%
}}}}
\put(6751,-886){\makebox(0,0)[b]{\smash{{\SetFigFont{8}{9.6}{\rmdefault}{\mddefault}{\updefault}{\color[rgb]{0,0,0}$c$}%
}}}}
\put(7051,-586){\makebox(0,0)[b]{\smash{{\SetFigFont{8}{9.6}{\rmdefault}{\mddefault}{\updefault}{\color[rgb]{0,0,0}$c$}%
}}}}
\put(6451,-1186){\makebox(0,0)[b]{\smash{{\SetFigFont{8}{9.6}{\rmdefault}{\mddefault}{\updefault}{\color[rgb]{0,0,0}$c$}%
}}}}
\put(5851,-1786){\makebox(0,0)[b]{\smash{{\SetFigFont{8}{9.6}{\rmdefault}{\mddefault}{\updefault}{\color[rgb]{0,0,0}$c$}%
}}}}
\put(5551,-2086){\makebox(0,0)[b]{\smash{{\SetFigFont{8}{9.6}{\rmdefault}{\mddefault}{\updefault}{\color[rgb]{0,0,0}$c$}%
}}}}
\put(5251,-2386){\makebox(0,0)[b]{\smash{{\SetFigFont{8}{9.6}{\rmdefault}{\mddefault}{\updefault}{\color[rgb]{0,0,0}$c$}%
}}}}
\put(4351,-3286){\makebox(0,0)[b]{\smash{{\SetFigFont{8}{9.6}{\rmdefault}{\mddefault}{\updefault}{\color[rgb]{0,0,0}$c$}%
}}}}
\put(4051,-3586){\makebox(0,0)[b]{\smash{{\SetFigFont{8}{9.6}{\rmdefault}{\mddefault}{\updefault}{\color[rgb]{0,0,0}$c$}%
}}}}
\put(3751,-3886){\makebox(0,0)[b]{\smash{{\SetFigFont{8}{9.6}{\rmdefault}{\mddefault}{\updefault}{\color[rgb]{0,0,0}$c$}%
}}}}
\put(3451,-4186){\makebox(0,0)[b]{\smash{{\SetFigFont{8}{9.6}{\rmdefault}{\mddefault}{\updefault}{\color[rgb]{0,0,0}$c$}%
}}}}
\put(3151,-4486){\makebox(0,0)[b]{\smash{{\SetFigFont{8}{9.6}{\rmdefault}{\mddefault}{\updefault}{\color[rgb]{0,0,0}$c$}%
}}}}
\put(2851,-4786){\makebox(0,0)[b]{\smash{{\SetFigFont{8}{9.6}{\rmdefault}{\mddefault}{\updefault}{\color[rgb]{0,0,0}$c$}%
}}}}
\put(2551,-5086){\makebox(0,0)[b]{\smash{{\SetFigFont{8}{9.6}{\rmdefault}{\mddefault}{\updefault}{\color[rgb]{0,0,0}$c$}%
}}}}
\put(8251, 14){\makebox(0,0)[b]{\smash{{\SetFigFont{8}{9.6}{\rmdefault}{\mddefault}{\updefault}{\color[rgb]{0,0,0}$c$}%
}}}}
\put(7951,-286){\makebox(0,0)[b]{\smash{{\SetFigFont{8}{9.6}{\rmdefault}{\mddefault}{\updefault}{\color[rgb]{0,0,0}$c$}%
}}}}
\put(7351,-886){\makebox(0,0)[b]{\smash{{\SetFigFont{8}{9.6}{\rmdefault}{\mddefault}{\updefault}{\color[rgb]{0,0,0}$c$}%
}}}}
\put(7651,-586){\makebox(0,0)[b]{\smash{{\SetFigFont{8}{9.6}{\rmdefault}{\mddefault}{\updefault}{\color[rgb]{0,0,0}$c$}%
}}}}
\put(7051,-1186){\makebox(0,0)[b]{\smash{{\SetFigFont{8}{9.6}{\rmdefault}{\mddefault}{\updefault}{\color[rgb]{0,0,0}$c$}%
}}}}
\put(6751,-1486){\makebox(0,0)[b]{\smash{{\SetFigFont{8}{9.6}{\rmdefault}{\mddefault}{\updefault}{\color[rgb]{0,0,0}$c$}%
}}}}
\put(6151,-2086){\makebox(0,0)[b]{\smash{{\SetFigFont{8}{9.6}{\rmdefault}{\mddefault}{\updefault}{\color[rgb]{0,0,0}$c$}%
}}}}
\put(5851,-2386){\makebox(0,0)[b]{\smash{{\SetFigFont{8}{9.6}{\rmdefault}{\mddefault}{\updefault}{\color[rgb]{0,0,0}$c$}%
}}}}
\put(5551,-2686){\makebox(0,0)[b]{\smash{{\SetFigFont{8}{9.6}{\rmdefault}{\mddefault}{\updefault}{\color[rgb]{0,0,0}$c$}%
}}}}
\put(5251,-2986){\makebox(0,0)[b]{\smash{{\SetFigFont{8}{9.6}{\rmdefault}{\mddefault}{\updefault}{\color[rgb]{0,0,0}$c$}%
}}}}
\put(4351,-3886){\makebox(0,0)[b]{\smash{{\SetFigFont{8}{9.6}{\rmdefault}{\mddefault}{\updefault}{\color[rgb]{0,0,0}$c$}%
}}}}
\put(4051,-4186){\makebox(0,0)[b]{\smash{{\SetFigFont{8}{9.6}{\rmdefault}{\mddefault}{\updefault}{\color[rgb]{0,0,0}$c$}%
}}}}
\put(3751,-4486){\makebox(0,0)[b]{\smash{{\SetFigFont{8}{9.6}{\rmdefault}{\mddefault}{\updefault}{\color[rgb]{0,0,0}$c$}%
}}}}
\put(3451,-4786){\makebox(0,0)[b]{\smash{{\SetFigFont{8}{9.6}{\rmdefault}{\mddefault}{\updefault}{\color[rgb]{0,0,0}$c$}%
}}}}
\put(3151,-5086){\makebox(0,0)[b]{\smash{{\SetFigFont{8}{9.6}{\rmdefault}{\mddefault}{\updefault}{\color[rgb]{0,0,0}$c$}%
}}}}
\put(8851, 14){\makebox(0,0)[b]{\smash{{\SetFigFont{8}{9.6}{\rmdefault}{\mddefault}{\updefault}{\color[rgb]{0,0,0}$c$}%
}}}}
\put(8551,-286){\makebox(0,0)[b]{\smash{{\SetFigFont{8}{9.6}{\rmdefault}{\mddefault}{\updefault}{\color[rgb]{0,0,0}$c$}%
}}}}
\put(7951,-886){\makebox(0,0)[b]{\smash{{\SetFigFont{8}{9.6}{\rmdefault}{\mddefault}{\updefault}{\color[rgb]{0,0,0}$c$}%
}}}}
\put(8251,-586){\makebox(0,0)[b]{\smash{{\SetFigFont{8}{9.6}{\rmdefault}{\mddefault}{\updefault}{\color[rgb]{0,0,0}$c$}%
}}}}
\put(7651,-1186){\makebox(0,0)[b]{\smash{{\SetFigFont{8}{9.6}{\rmdefault}{\mddefault}{\updefault}{\color[rgb]{0,0,0}$c$}%
}}}}
\put(7351,-1486){\makebox(0,0)[b]{\smash{{\SetFigFont{8}{9.6}{\rmdefault}{\mddefault}{\updefault}{\color[rgb]{0,0,0}$c$}%
}}}}
\put(7051,-1786){\makebox(0,0)[b]{\smash{{\SetFigFont{8}{9.6}{\rmdefault}{\mddefault}{\updefault}{\color[rgb]{0,0,0}$c$}%
}}}}
\put(6451,-2386){\makebox(0,0)[b]{\smash{{\SetFigFont{8}{9.6}{\rmdefault}{\mddefault}{\updefault}{\color[rgb]{0,0,0}$c$}%
}}}}
\put(6151,-2686){\makebox(0,0)[b]{\smash{{\SetFigFont{8}{9.6}{\rmdefault}{\mddefault}{\updefault}{\color[rgb]{0,0,0}$c$}%
}}}}
\put(5851,-2986){\makebox(0,0)[b]{\smash{{\SetFigFont{8}{9.6}{\rmdefault}{\mddefault}{\updefault}{\color[rgb]{0,0,0}$c$}%
}}}}
\put(5551,-3286){\makebox(0,0)[b]{\smash{{\SetFigFont{8}{9.6}{\rmdefault}{\mddefault}{\updefault}{\color[rgb]{0,0,0}$c$}%
}}}}
\put(5251,-3586){\makebox(0,0)[b]{\smash{{\SetFigFont{8}{9.6}{\rmdefault}{\mddefault}{\updefault}{\color[rgb]{0,0,0}$c$}%
}}}}
\put(4351,-4486){\makebox(0,0)[b]{\smash{{\SetFigFont{8}{9.6}{\rmdefault}{\mddefault}{\updefault}{\color[rgb]{0,0,0}$c$}%
}}}}
\put(4051,-4786){\makebox(0,0)[b]{\smash{{\SetFigFont{8}{9.6}{\rmdefault}{\mddefault}{\updefault}{\color[rgb]{0,0,0}$c$}%
}}}}
\put(3751,-5086){\makebox(0,0)[b]{\smash{{\SetFigFont{8}{9.6}{\rmdefault}{\mddefault}{\updefault}{\color[rgb]{0,0,0}$c$}%
}}}}
\put(9451, 14){\makebox(0,0)[b]{\smash{{\SetFigFont{8}{9.6}{\rmdefault}{\mddefault}{\updefault}{\color[rgb]{0,0,0}$c$}%
}}}}
\put(9151,-286){\makebox(0,0)[b]{\smash{{\SetFigFont{8}{9.6}{\rmdefault}{\mddefault}{\updefault}{\color[rgb]{0,0,0}$c$}%
}}}}
\put(8551,-886){\makebox(0,0)[b]{\smash{{\SetFigFont{8}{9.6}{\rmdefault}{\mddefault}{\updefault}{\color[rgb]{0,0,0}$c$}%
}}}}
\put(8851,-586){\makebox(0,0)[b]{\smash{{\SetFigFont{8}{9.6}{\rmdefault}{\mddefault}{\updefault}{\color[rgb]{0,0,0}$c$}%
}}}}
\put(8251,-1186){\makebox(0,0)[b]{\smash{{\SetFigFont{8}{9.6}{\rmdefault}{\mddefault}{\updefault}{\color[rgb]{0,0,0}$c$}%
}}}}
\put(7951,-1486){\makebox(0,0)[b]{\smash{{\SetFigFont{8}{9.6}{\rmdefault}{\mddefault}{\updefault}{\color[rgb]{0,0,0}$c$}%
}}}}
\put(7651,-1786){\makebox(0,0)[b]{\smash{{\SetFigFont{8}{9.6}{\rmdefault}{\mddefault}{\updefault}{\color[rgb]{0,0,0}$c$}%
}}}}
\put(7351,-2086){\makebox(0,0)[b]{\smash{{\SetFigFont{8}{9.6}{\rmdefault}{\mddefault}{\updefault}{\color[rgb]{0,0,0}$c$}%
}}}}
\put(6751,-2686){\makebox(0,0)[b]{\smash{{\SetFigFont{8}{9.6}{\rmdefault}{\mddefault}{\updefault}{\color[rgb]{0,0,0}$c$}%
}}}}
\put(6451,-2986){\makebox(0,0)[b]{\smash{{\SetFigFont{8}{9.6}{\rmdefault}{\mddefault}{\updefault}{\color[rgb]{0,0,0}$c$}%
}}}}
\put(6151,-3286){\makebox(0,0)[b]{\smash{{\SetFigFont{8}{9.6}{\rmdefault}{\mddefault}{\updefault}{\color[rgb]{0,0,0}$c$}%
}}}}
\put(5851,-3586){\makebox(0,0)[b]{\smash{{\SetFigFont{8}{9.6}{\rmdefault}{\mddefault}{\updefault}{\color[rgb]{0,0,0}$c$}%
}}}}
\put(5551,-3886){\makebox(0,0)[b]{\smash{{\SetFigFont{8}{9.6}{\rmdefault}{\mddefault}{\updefault}{\color[rgb]{0,0,0}$c$}%
}}}}
\put(5251,-4186){\makebox(0,0)[b]{\smash{{\SetFigFont{8}{9.6}{\rmdefault}{\mddefault}{\updefault}{\color[rgb]{0,0,0}$c$}%
}}}}
\put(4351,-5086){\makebox(0,0)[b]{\smash{{\SetFigFont{8}{9.6}{\rmdefault}{\mddefault}{\updefault}{\color[rgb]{0,0,0}$c$}%
}}}}
\put(5851, 14){\makebox(0,0)[b]{\smash{{\SetFigFont{8}{9.6}{\rmdefault}{\mddefault}{\updefault}{\color[rgb]{0,0,0}$c$}%
}}}}
\put(5551,-286){\makebox(0,0)[b]{\smash{{\SetFigFont{8}{9.6}{\rmdefault}{\mddefault}{\updefault}{\color[rgb]{0,0,0}$c$}%
}}}}
\put(4351,-1486){\makebox(0,0)[b]{\smash{{\SetFigFont{8}{9.6}{\rmdefault}{\mddefault}{\updefault}{\color[rgb]{0,0,0}$c$}%
}}}}
\put(4051,-1786){\makebox(0,0)[b]{\smash{{\SetFigFont{8}{9.6}{\rmdefault}{\mddefault}{\updefault}{\color[rgb]{0,0,0}$c$}%
}}}}
\put(3751,-2086){\makebox(0,0)[b]{\smash{{\SetFigFont{8}{9.6}{\rmdefault}{\mddefault}{\updefault}{\color[rgb]{0,0,0}$c$}%
}}}}
\put(3451,-2386){\makebox(0,0)[b]{\smash{{\SetFigFont{8}{9.6}{\rmdefault}{\mddefault}{\updefault}{\color[rgb]{0,0,0}$c$}%
}}}}
\put(3151,-2686){\makebox(0,0)[b]{\smash{{\SetFigFont{8}{9.6}{\rmdefault}{\mddefault}{\updefault}{\color[rgb]{0,0,0}$c$}%
}}}}
\put(2851,-2986){\makebox(0,0)[b]{\smash{{\SetFigFont{8}{9.6}{\rmdefault}{\mddefault}{\updefault}{\color[rgb]{0,0,0}$c$}%
}}}}
\put(2551,-3286){\makebox(0,0)[b]{\smash{{\SetFigFont{8}{9.6}{\rmdefault}{\mddefault}{\updefault}{\color[rgb]{0,0,0}$c$}%
}}}}
\put(2251,-3586){\makebox(0,0)[b]{\smash{{\SetFigFont{8}{9.6}{\rmdefault}{\mddefault}{\updefault}{\color[rgb]{0,0,0}$c$}%
}}}}
\put(1951,-3886){\makebox(0,0)[b]{\smash{{\SetFigFont{8}{9.6}{\rmdefault}{\mddefault}{\updefault}{\color[rgb]{0,0,0}$c$}%
}}}}
\put(1651,-4186){\makebox(0,0)[b]{\smash{{\SetFigFont{8}{9.6}{\rmdefault}{\mddefault}{\updefault}{\color[rgb]{0,0,0}$c$}%
}}}}
\put(1351,-4486){\makebox(0,0)[b]{\smash{{\SetFigFont{8}{9.6}{\rmdefault}{\mddefault}{\updefault}{\color[rgb]{0,0,0}$c$}%
}}}}
\put(1051,-4786){\makebox(0,0)[b]{\smash{{\SetFigFont{8}{9.6}{\rmdefault}{\mddefault}{\updefault}{\color[rgb]{0,0,0}$c$}%
}}}}
\put(751,-5086){\makebox(0,0)[b]{\smash{{\SetFigFont{8}{9.6}{\rmdefault}{\mddefault}{\updefault}{\color[rgb]{0,0,0}$c$}%
}}}}
\put(4651,-286){\makebox(0,0)[b]{\smash{{\SetFigFont{8}{9.6}{\rmdefault}{\mddefault}{\updefault}{\color[rgb]{0,0,0}$b$}%
}}}}
\put(4951, 14){\makebox(0,0)[b]{\smash{{\SetFigFont{8}{9.6}{\rmdefault}{\mddefault}{\updefault}{\color[rgb]{0,0,0}$b$}%
}}}}
\put(4951,-586){\makebox(0,0)[b]{\smash{{\SetFigFont{8}{9.6}{\rmdefault}{\mddefault}{\updefault}{\color[rgb]{0,0,0}$b$}%
}}}}
\put(4651,-586){\makebox(0,0)[b]{\smash{{\SetFigFont{8}{9.6}{\rmdefault}{\mddefault}{\updefault}{\color[rgb]{0,0,0}$b$}%
}}}}
\put(4651,-886){\makebox(0,0)[b]{\smash{{\SetFigFont{8}{9.6}{\rmdefault}{\mddefault}{\updefault}{\color[rgb]{0,0,0}$b$}%
}}}}
\put(4951,-1186){\makebox(0,0)[b]{\smash{{\SetFigFont{8}{9.6}{\rmdefault}{\mddefault}{\updefault}{\color[rgb]{0,0,0}$b$}%
}}}}
\put(4951,-1486){\makebox(0,0)[b]{\smash{{\SetFigFont{8}{9.6}{\rmdefault}{\mddefault}{\updefault}{\color[rgb]{0,0,0}$b$}%
}}}}
\put(4651,-1486){\makebox(0,0)[b]{\smash{{\SetFigFont{8}{9.6}{\rmdefault}{\mddefault}{\updefault}{\color[rgb]{0,0,0}$b$}%
}}}}
\put(4651,-1786){\makebox(0,0)[b]{\smash{{\SetFigFont{8}{9.6}{\rmdefault}{\mddefault}{\updefault}{\color[rgb]{0,0,0}$b$}%
}}}}
\put(4951,-1786){\makebox(0,0)[b]{\smash{{\SetFigFont{8}{9.6}{\rmdefault}{\mddefault}{\updefault}{\color[rgb]{0,0,0}$b$}%
}}}}
\put(4951,-2086){\makebox(0,0)[b]{\smash{{\SetFigFont{8}{9.6}{\rmdefault}{\mddefault}{\updefault}{\color[rgb]{0,0,0}$b$}%
}}}}
\put(4651,-2086){\makebox(0,0)[b]{\smash{{\SetFigFont{8}{9.6}{\rmdefault}{\mddefault}{\updefault}{\color[rgb]{0,0,0}$b$}%
}}}}
\put(4651,-2386){\makebox(0,0)[b]{\smash{{\SetFigFont{8}{9.6}{\rmdefault}{\mddefault}{\updefault}{\color[rgb]{0,0,0}$b$}%
}}}}
\put(4951,-2386){\makebox(0,0)[b]{\smash{{\SetFigFont{8}{9.6}{\rmdefault}{\mddefault}{\updefault}{\color[rgb]{0,0,0}$b$}%
}}}}
\put(4651,-2686){\makebox(0,0)[b]{\smash{{\SetFigFont{8}{9.6}{\rmdefault}{\mddefault}{\updefault}{\color[rgb]{0,0,0}$b$}%
}}}}
\put(4951,-2986){\makebox(0,0)[b]{\smash{{\SetFigFont{8}{9.6}{\rmdefault}{\mddefault}{\updefault}{\color[rgb]{0,0,0}$b$}%
}}}}
\put(4651,-3286){\makebox(0,0)[b]{\smash{{\SetFigFont{8}{9.6}{\rmdefault}{\mddefault}{\updefault}{\color[rgb]{0,0,0}$b$}%
}}}}
\put(4951,-3586){\makebox(0,0)[b]{\smash{{\SetFigFont{8}{9.6}{\rmdefault}{\mddefault}{\updefault}{\color[rgb]{0,0,0}$b$}%
}}}}
\put(4651,-3886){\makebox(0,0)[b]{\smash{{\SetFigFont{8}{9.6}{\rmdefault}{\mddefault}{\updefault}{\color[rgb]{0,0,0}$b$}%
}}}}
\put(4951,-4186){\makebox(0,0)[b]{\smash{{\SetFigFont{8}{9.6}{\rmdefault}{\mddefault}{\updefault}{\color[rgb]{0,0,0}$b$}%
}}}}
\put(4651,-4486){\makebox(0,0)[b]{\smash{{\SetFigFont{8}{9.6}{\rmdefault}{\mddefault}{\updefault}{\color[rgb]{0,0,0}$b$}%
}}}}
\put(4951,-4786){\makebox(0,0)[b]{\smash{{\SetFigFont{8}{9.6}{\rmdefault}{\mddefault}{\updefault}{\color[rgb]{0,0,0}$b$}%
}}}}
\put(4651,-5086){\makebox(0,0)[b]{\smash{{\SetFigFont{8}{9.6}{\rmdefault}{\mddefault}{\updefault}{\color[rgb]{0,0,0}$b$}%
}}}}
\end{picture}%

%% file: 621Ex.pstex_t
\begin{picture}(0,0)%
\includegraphics{621Ex.eps}%
\end{picture}%
\setlength{\unitlength}{2368sp}%
\begingroup\makeatletter\ifx\SetFigFont\undefined%
\gdef\SetFigFont#1#2#3#4#5{%
  \reset@font\fontsize{#1}{#2pt}%
  \fontfamily{#3}\fontseries{#4}\fontshape{#5}%
  \selectfont}%
\fi\endgroup%
\begin{picture}(12477,6811)(136,-6050)
\put(2551,614){\makebox(0,0)[b]{\smash{{\SetFigFont{7}{8.4}{\rmdefault}{\mddefault}{\updefault}{\color[rgb]{0,0,0}$14$}%
}}}}
\put(4351,-3286){\makebox(0,0)[b]{\smash{{\SetFigFont{7}{8.4}{\rmdefault}{\mddefault}{\updefault}{\color[rgb]{0,0,0}$ac$}%
}}}}
\put(4651,-2986){\makebox(0,0)[b]{\smash{{\SetFigFont{7}{8.4}{\rmdefault}{\mddefault}{\updefault}{\color[rgb]{0,0,0}$ac$}%
}}}}
\put(4951,-2686){\makebox(0,0)[b]{\smash{{\SetFigFont{7}{8.4}{\rmdefault}{\mddefault}{\updefault}{\color[rgb]{0,0,0}$ac$}%
}}}}
\put(5251,-2386){\makebox(0,0)[b]{\smash{{\SetFigFont{7}{8.4}{\rmdefault}{\mddefault}{\updefault}{\color[rgb]{0,0,0}$ac$}%
}}}}
\put(6151,-1486){\makebox(0,0)[b]{\smash{{\SetFigFont{7}{8.4}{\rmdefault}{\mddefault}{\updefault}{\color[rgb]{0,0,0}$ac$}%
}}}}
\put(7051,-586){\makebox(0,0)[b]{\smash{{\SetFigFont{7}{8.4}{\rmdefault}{\mddefault}{\updefault}{\color[rgb]{0,0,0}$ac$}%
}}}}
\put(6751,-886){\makebox(0,0)[b]{\smash{{\SetFigFont{7}{8.4}{\rmdefault}{\mddefault}{\updefault}{\color[rgb]{0,0,0}$bc$}%
}}}}
\put(6451,-1186){\makebox(0,0)[b]{\smash{{\SetFigFont{7}{8.4}{\rmdefault}{\mddefault}{\updefault}{\color[rgb]{0,0,0}$bc$}%
}}}}
\put(6451, 14){\makebox(0,0)[b]{\smash{{\SetFigFont{7}{8.4}{\rmdefault}{\mddefault}{\updefault}{\color[rgb]{0,0,0}$ab$}%
}}}}
\put(6751,-286){\makebox(0,0)[b]{\smash{{\SetFigFont{7}{8.4}{\rmdefault}{\mddefault}{\updefault}{\color[rgb]{0,0,0}$ab$}%
}}}}
\put(6451,-1786){\makebox(0,0)[b]{\smash{{\SetFigFont{7}{8.4}{\rmdefault}{\mddefault}{\updefault}{\color[rgb]{0,0,0}$ab$}%
}}}}
\put(6751,-2086){\makebox(0,0)[b]{\smash{{\SetFigFont{7}{8.4}{\rmdefault}{\mddefault}{\updefault}{\color[rgb]{0,0,0}$ab$}%
}}}}
\put(6451,-3586){\makebox(0,0)[b]{\smash{{\SetFigFont{7}{8.4}{\rmdefault}{\mddefault}{\updefault}{\color[rgb]{0,0,0}$ab$}%
}}}}
\put(6751,-3886){\makebox(0,0)[b]{\smash{{\SetFigFont{7}{8.4}{\rmdefault}{\mddefault}{\updefault}{\color[rgb]{0,0,0}$ab$}%
}}}}
\put(6451,-4186){\makebox(0,0)[b]{\smash{{\SetFigFont{7}{8.4}{\rmdefault}{\mddefault}{\updefault}{\color[rgb]{0,0,0}$ab$}%
}}}}
\put(6751,-4486){\makebox(0,0)[b]{\smash{{\SetFigFont{7}{8.4}{\rmdefault}{\mddefault}{\updefault}{\color[rgb]{0,0,0}$ab$}%
}}}}
\put(6451,-4786){\makebox(0,0)[b]{\smash{{\SetFigFont{7}{8.4}{\rmdefault}{\mddefault}{\updefault}{\color[rgb]{0,0,0}$ab$}%
}}}}
\put(6751,-5086){\makebox(0,0)[b]{\smash{{\SetFigFont{7}{8.4}{\rmdefault}{\mddefault}{\updefault}{\color[rgb]{0,0,0}$ab$}%
}}}}
\put(6451,-5386){\makebox(0,0)[b]{\smash{{\SetFigFont{7}{8.4}{\rmdefault}{\mddefault}{\updefault}{\color[rgb]{0,0,0}$ab$}%
}}}}
\put(6751,-5686){\makebox(0,0)[b]{\smash{{\SetFigFont{7}{8.4}{\rmdefault}{\mddefault}{\updefault}{\color[rgb]{0,0,0}$ab$}%
}}}}
\put(151,-2686){\makebox(0,0)[b]{\smash{{\SetFigFont{7}{8.4}{\rmdefault}{\mddefault}{\updefault}{\color[rgb]{0,0,0}$10$}%
}}}}
\put(151, 14){\makebox(0,0)[b]{\smash{{\SetFigFont{7}{8.4}{\rmdefault}{\mddefault}{\updefault}{\color[rgb]{0,0,0}$1$}%
}}}}
\put(151,-286){\makebox(0,0)[b]{\smash{{\SetFigFont{7}{8.4}{\rmdefault}{\mddefault}{\updefault}{\color[rgb]{0,0,0}$2$}%
}}}}
\put(151,-586){\makebox(0,0)[b]{\smash{{\SetFigFont{7}{8.4}{\rmdefault}{\mddefault}{\updefault}{\color[rgb]{0,0,0}$3$}%
}}}}
\put(151,-886){\makebox(0,0)[b]{\smash{{\SetFigFont{7}{8.4}{\rmdefault}{\mddefault}{\updefault}{\color[rgb]{0,0,0}$4$}%
}}}}
\put(151,-1186){\makebox(0,0)[b]{\smash{{\SetFigFont{7}{8.4}{\rmdefault}{\mddefault}{\updefault}{\color[rgb]{0,0,0}$5$}%
}}}}
\put(151,-1486){\makebox(0,0)[b]{\smash{{\SetFigFont{7}{8.4}{\rmdefault}{\mddefault}{\updefault}{\color[rgb]{0,0,0}$6$}%
}}}}
\put(151,-1786){\makebox(0,0)[b]{\smash{{\SetFigFont{7}{8.4}{\rmdefault}{\mddefault}{\updefault}{\color[rgb]{0,0,0}$7$}%
}}}}
\put(151,-2086){\makebox(0,0)[b]{\smash{{\SetFigFont{7}{8.4}{\rmdefault}{\mddefault}{\updefault}{\color[rgb]{0,0,0}$8$}%
}}}}
\put(151,314){\makebox(0,0)[b]{\smash{{\SetFigFont{7}{8.4}{\rmdefault}{\mddefault}{\updefault}{\color[rgb]{0,0,0}$t$}%
}}}}
\put(151,-2386){\makebox(0,0)[b]{\smash{{\SetFigFont{7}{8.4}{\rmdefault}{\mddefault}{\updefault}{\color[rgb]{0,0,0}$9$}%
}}}}
\put(751,614){\makebox(0,0)[b]{\smash{{\SetFigFont{7}{8.4}{\rmdefault}{\mddefault}{\updefault}{\color[rgb]{0,0,0}$20$}%
}}}}
\put(151,-3286){\makebox(0,0)[b]{\smash{{\SetFigFont{7}{8.4}{\rmdefault}{\mddefault}{\updefault}{\color[rgb]{0,0,0}$12$}%
}}}}
\put(151,-3586){\makebox(0,0)[b]{\smash{{\SetFigFont{7}{8.4}{\rmdefault}{\mddefault}{\updefault}{\color[rgb]{0,0,0}$13$}%
}}}}
\put(151,-3886){\makebox(0,0)[b]{\smash{{\SetFigFont{7}{8.4}{\rmdefault}{\mddefault}{\updefault}{\color[rgb]{0,0,0}$14$}%
}}}}
\put(151,-4186){\makebox(0,0)[b]{\smash{{\SetFigFont{7}{8.4}{\rmdefault}{\mddefault}{\updefault}{\color[rgb]{0,0,0}$15$}%
}}}}
\put(151,-4486){\makebox(0,0)[b]{\smash{{\SetFigFont{7}{8.4}{\rmdefault}{\mddefault}{\updefault}{\color[rgb]{0,0,0}$16$}%
}}}}
\put(151,-4786){\makebox(0,0)[b]{\smash{{\SetFigFont{7}{8.4}{\rmdefault}{\mddefault}{\updefault}{\color[rgb]{0,0,0}$17$}%
}}}}
\put(151,-2986){\makebox(0,0)[b]{\smash{{\SetFigFont{7}{8.4}{\rmdefault}{\mddefault}{\updefault}{\color[rgb]{0,0,0}$11$}%
}}}}
\put(751,314){\makebox(0,0)[b]{\smash{{\SetFigFont{7}{8.4}{\rmdefault}{\mddefault}{\updefault}{\color[rgb]{0,0,0}$a$}%
}}}}
\put(151,-5086){\makebox(0,0)[b]{\smash{{\SetFigFont{7}{8.4}{\rmdefault}{\mddefault}{\updefault}{\color[rgb]{0,0,0}$18$}%
}}}}
\put(151,-5386){\makebox(0,0)[b]{\smash{{\SetFigFont{7}{8.4}{\rmdefault}{\mddefault}{\updefault}{\color[rgb]{0,0,0}$19$}%
}}}}
\put(151,-5686){\makebox(0,0)[b]{\smash{{\SetFigFont{7}{8.4}{\rmdefault}{\mddefault}{\updefault}{\color[rgb]{0,0,0}$20$}%
}}}}
\put(4351,614){\makebox(0,0)[b]{\smash{{\SetFigFont{7}{8.4}{\rmdefault}{\mddefault}{\updefault}{\color[rgb]{0,0,0}$8$}%
}}}}
\put(6151,614){\makebox(0,0)[b]{\smash{{\SetFigFont{7}{8.4}{\rmdefault}{\mddefault}{\updefault}{\color[rgb]{0,0,0}$2$}%
}}}}
\put(6451,614){\makebox(0,0)[b]{\smash{{\SetFigFont{7}{8.4}{\rmdefault}{\mddefault}{\updefault}{\color[rgb]{0,0,0}$1$}%
}}}}
\put(6751,614){\makebox(0,0)[b]{\smash{{\SetFigFont{7}{8.4}{\rmdefault}{\mddefault}{\updefault}{\color[rgb]{0,0,0}$0$}%
}}}}
\put(7951,614){\makebox(0,0)[b]{\smash{{\SetFigFont{7}{8.4}{\rmdefault}{\mddefault}{\updefault}{\color[rgb]{0,0,0}$-4$}%
}}}}
\put(12451,614){\makebox(0,0)[b]{\smash{{\SetFigFont{7}{8.4}{\rmdefault}{\mddefault}{\updefault}{\color[rgb]{0,0,0}$-19$}%
}}}}
\put(4351,314){\makebox(0,0)[b]{\smash{{\SetFigFont{7}{8.4}{\rmdefault}{\mddefault}{\updefault}{\color[rgb]{0,0,0}$a$}%
}}}}
\put(6151,314){\makebox(0,0)[b]{\smash{{\SetFigFont{7}{8.4}{\rmdefault}{\mddefault}{\updefault}{\color[rgb]{0,0,0}$a$}%
}}}}
\put(6451,314){\makebox(0,0)[b]{\smash{{\SetFigFont{7}{8.4}{\rmdefault}{\mddefault}{\updefault}{\color[rgb]{0,0,0}$b$}%
}}}}
\put(6751,314){\makebox(0,0)[b]{\smash{{\SetFigFont{7}{8.4}{\rmdefault}{\mddefault}{\updefault}{\color[rgb]{0,0,0}$b$}%
}}}}
\put(7951,314){\makebox(0,0)[b]{\smash{{\SetFigFont{7}{8.4}{\rmdefault}{\mddefault}{\updefault}{\color[rgb]{0,0,0}$c$}%
}}}}
\put(7651, 14){\makebox(0,0)[b]{\smash{{\SetFigFont{7}{8.4}{\rmdefault}{\mddefault}{\updefault}{\color[rgb]{0,0,0}$c$}%
}}}}
\put(7351,-286){\makebox(0,0)[b]{\smash{{\SetFigFont{7}{8.4}{\rmdefault}{\mddefault}{\updefault}{\color[rgb]{0,0,0}$c$}%
}}}}
\put(4051,-3586){\makebox(0,0)[b]{\smash{{\SetFigFont{7}{8.4}{\rmdefault}{\mddefault}{\updefault}{\color[rgb]{0,0,0}$c$}%
}}}}
\put(3751,-3886){\makebox(0,0)[b]{\smash{{\SetFigFont{7}{8.4}{\rmdefault}{\mddefault}{\updefault}{\color[rgb]{0,0,0}$c$}%
}}}}
\put(3451,-4186){\makebox(0,0)[b]{\smash{{\SetFigFont{7}{8.4}{\rmdefault}{\mddefault}{\updefault}{\color[rgb]{0,0,0}$c$}%
}}}}
\put(3151,-4486){\makebox(0,0)[b]{\smash{{\SetFigFont{7}{8.4}{\rmdefault}{\mddefault}{\updefault}{\color[rgb]{0,0,0}$c$}%
}}}}
\put(2851,-4786){\makebox(0,0)[b]{\smash{{\SetFigFont{7}{8.4}{\rmdefault}{\mddefault}{\updefault}{\color[rgb]{0,0,0}$c$}%
}}}}
\put(2551,-5086){\makebox(0,0)[b]{\smash{{\SetFigFont{7}{8.4}{\rmdefault}{\mddefault}{\updefault}{\color[rgb]{0,0,0}$c$}%
}}}}
\put(2251,-5386){\makebox(0,0)[b]{\smash{{\SetFigFont{7}{8.4}{\rmdefault}{\mddefault}{\updefault}{\color[rgb]{0,0,0}$c$}%
}}}}
\put(1951,-5686){\makebox(0,0)[b]{\smash{{\SetFigFont{7}{8.4}{\rmdefault}{\mddefault}{\updefault}{\color[rgb]{0,0,0}$c$}%
}}}}
\put(1651,-5986){\makebox(0,0)[b]{\smash{{\SetFigFont{7}{8.4}{\rmdefault}{\mddefault}{\updefault}{\color[rgb]{0,0,0}$c$}%
}}}}
\put(5551,-2086){\makebox(0,0)[b]{\smash{{\SetFigFont{7}{8.4}{\rmdefault}{\mddefault}{\updefault}{\color[rgb]{0,0,0}$c$}%
}}}}
\put(5851,-1786){\makebox(0,0)[b]{\smash{{\SetFigFont{7}{8.4}{\rmdefault}{\mddefault}{\updefault}{\color[rgb]{0,0,0}$c$}%
}}}}
\put(1051, 14){\makebox(0,0)[b]{\smash{{\SetFigFont{7}{8.4}{\rmdefault}{\mddefault}{\updefault}{\color[rgb]{0,0,0}$a$}%
}}}}
\put(1351,-286){\makebox(0,0)[b]{\smash{{\SetFigFont{7}{8.4}{\rmdefault}{\mddefault}{\updefault}{\color[rgb]{0,0,0}$a$}%
}}}}
\put(1651,-586){\makebox(0,0)[b]{\smash{{\SetFigFont{7}{8.4}{\rmdefault}{\mddefault}{\updefault}{\color[rgb]{0,0,0}$a$}%
}}}}
\put(1951,-886){\makebox(0,0)[b]{\smash{{\SetFigFont{7}{8.4}{\rmdefault}{\mddefault}{\updefault}{\color[rgb]{0,0,0}$a$}%
}}}}
\put(2251,-1186){\makebox(0,0)[b]{\smash{{\SetFigFont{7}{8.4}{\rmdefault}{\mddefault}{\updefault}{\color[rgb]{0,0,0}$a$}%
}}}}
\put(2551,-1486){\makebox(0,0)[b]{\smash{{\SetFigFont{7}{8.4}{\rmdefault}{\mddefault}{\updefault}{\color[rgb]{0,0,0}$a$}%
}}}}
\put(2851,-1786){\makebox(0,0)[b]{\smash{{\SetFigFont{7}{8.4}{\rmdefault}{\mddefault}{\updefault}{\color[rgb]{0,0,0}$a$}%
}}}}
\put(3151,-2086){\makebox(0,0)[b]{\smash{{\SetFigFont{7}{8.4}{\rmdefault}{\mddefault}{\updefault}{\color[rgb]{0,0,0}$a$}%
}}}}
\put(3451,-2386){\makebox(0,0)[b]{\smash{{\SetFigFont{7}{8.4}{\rmdefault}{\mddefault}{\updefault}{\color[rgb]{0,0,0}$a$}%
}}}}
\put(3751,-2686){\makebox(0,0)[b]{\smash{{\SetFigFont{7}{8.4}{\rmdefault}{\mddefault}{\updefault}{\color[rgb]{0,0,0}$a$}%
}}}}
\put(4051,-2986){\makebox(0,0)[b]{\smash{{\SetFigFont{7}{8.4}{\rmdefault}{\mddefault}{\updefault}{\color[rgb]{0,0,0}$a$}%
}}}}
\put(4651,-3586){\makebox(0,0)[b]{\smash{{\SetFigFont{7}{8.4}{\rmdefault}{\mddefault}{\updefault}{\color[rgb]{0,0,0}$a$}%
}}}}
\put(4951,-3886){\makebox(0,0)[b]{\smash{{\SetFigFont{7}{8.4}{\rmdefault}{\mddefault}{\updefault}{\color[rgb]{0,0,0}$a$}%
}}}}
\put(5251,-4186){\makebox(0,0)[b]{\smash{{\SetFigFont{7}{8.4}{\rmdefault}{\mddefault}{\updefault}{\color[rgb]{0,0,0}$a$}%
}}}}
\put(5551,-4486){\makebox(0,0)[b]{\smash{{\SetFigFont{7}{8.4}{\rmdefault}{\mddefault}{\updefault}{\color[rgb]{0,0,0}$a$}%
}}}}
\put(5851,-4786){\makebox(0,0)[b]{\smash{{\SetFigFont{7}{8.4}{\rmdefault}{\mddefault}{\updefault}{\color[rgb]{0,0,0}$a$}%
}}}}
\put(6151,-5086){\makebox(0,0)[b]{\smash{{\SetFigFont{7}{8.4}{\rmdefault}{\mddefault}{\updefault}{\color[rgb]{0,0,0}$a$}%
}}}}
\put(7051,-5986){\makebox(0,0)[b]{\smash{{\SetFigFont{7}{8.4}{\rmdefault}{\mddefault}{\updefault}{\color[rgb]{0,0,0}$a$}%
}}}}
\put(4651, 14){\makebox(0,0)[b]{\smash{{\SetFigFont{7}{8.4}{\rmdefault}{\mddefault}{\updefault}{\color[rgb]{0,0,0}$a$}%
}}}}
\put(4951,-286){\makebox(0,0)[b]{\smash{{\SetFigFont{7}{8.4}{\rmdefault}{\mddefault}{\updefault}{\color[rgb]{0,0,0}$a$}%
}}}}
\put(5251,-586){\makebox(0,0)[b]{\smash{{\SetFigFont{7}{8.4}{\rmdefault}{\mddefault}{\updefault}{\color[rgb]{0,0,0}$a$}%
}}}}
\put(5551,-886){\makebox(0,0)[b]{\smash{{\SetFigFont{7}{8.4}{\rmdefault}{\mddefault}{\updefault}{\color[rgb]{0,0,0}$a$}%
}}}}
\put(5851,-1186){\makebox(0,0)[b]{\smash{{\SetFigFont{7}{8.4}{\rmdefault}{\mddefault}{\updefault}{\color[rgb]{0,0,0}$a$}%
}}}}
\put(7051,-2386){\makebox(0,0)[b]{\smash{{\SetFigFont{7}{8.4}{\rmdefault}{\mddefault}{\updefault}{\color[rgb]{0,0,0}$a$}%
}}}}
\put(7351,-2686){\makebox(0,0)[b]{\smash{{\SetFigFont{7}{8.4}{\rmdefault}{\mddefault}{\updefault}{\color[rgb]{0,0,0}$a$}%
}}}}
\put(7651,-2986){\makebox(0,0)[b]{\smash{{\SetFigFont{7}{8.4}{\rmdefault}{\mddefault}{\updefault}{\color[rgb]{0,0,0}$a$}%
}}}}
\put(7951,-3286){\makebox(0,0)[b]{\smash{{\SetFigFont{7}{8.4}{\rmdefault}{\mddefault}{\updefault}{\color[rgb]{0,0,0}$a$}%
}}}}
\put(8251,-3586){\makebox(0,0)[b]{\smash{{\SetFigFont{7}{8.4}{\rmdefault}{\mddefault}{\updefault}{\color[rgb]{0,0,0}$a$}%
}}}}
\put(8551,-3886){\makebox(0,0)[b]{\smash{{\SetFigFont{7}{8.4}{\rmdefault}{\mddefault}{\updefault}{\color[rgb]{0,0,0}$a$}%
}}}}
\put(8851,-4186){\makebox(0,0)[b]{\smash{{\SetFigFont{7}{8.4}{\rmdefault}{\mddefault}{\updefault}{\color[rgb]{0,0,0}$a$}%
}}}}
\put(9151,-4486){\makebox(0,0)[b]{\smash{{\SetFigFont{7}{8.4}{\rmdefault}{\mddefault}{\updefault}{\color[rgb]{0,0,0}$a$}%
}}}}
\put(9451,-4786){\makebox(0,0)[b]{\smash{{\SetFigFont{7}{8.4}{\rmdefault}{\mddefault}{\updefault}{\color[rgb]{0,0,0}$a$}%
}}}}
\put(9751,-5086){\makebox(0,0)[b]{\smash{{\SetFigFont{7}{8.4}{\rmdefault}{\mddefault}{\updefault}{\color[rgb]{0,0,0}$a$}%
}}}}
\put(10051,-5386){\makebox(0,0)[b]{\smash{{\SetFigFont{7}{8.4}{\rmdefault}{\mddefault}{\updefault}{\color[rgb]{0,0,0}$a$}%
}}}}
\put(10351,-5686){\makebox(0,0)[b]{\smash{{\SetFigFont{7}{8.4}{\rmdefault}{\mddefault}{\updefault}{\color[rgb]{0,0,0}$a$}%
}}}}
\put(10651,-5986){\makebox(0,0)[b]{\smash{{\SetFigFont{7}{8.4}{\rmdefault}{\mddefault}{\updefault}{\color[rgb]{0,0,0}$a$}%
}}}}
\put(7351,-886){\makebox(0,0)[b]{\smash{{\SetFigFont{7}{8.4}{\rmdefault}{\mddefault}{\updefault}{\color[rgb]{0,0,0}$a$}%
}}}}
\put(7651,-1186){\makebox(0,0)[b]{\smash{{\SetFigFont{7}{8.4}{\rmdefault}{\mddefault}{\updefault}{\color[rgb]{0,0,0}$a$}%
}}}}
\put(7951,-1486){\makebox(0,0)[b]{\smash{{\SetFigFont{7}{8.4}{\rmdefault}{\mddefault}{\updefault}{\color[rgb]{0,0,0}$a$}%
}}}}
\put(8251,-1786){\makebox(0,0)[b]{\smash{{\SetFigFont{7}{8.4}{\rmdefault}{\mddefault}{\updefault}{\color[rgb]{0,0,0}$a$}%
}}}}
\put(8551,-2086){\makebox(0,0)[b]{\smash{{\SetFigFont{7}{8.4}{\rmdefault}{\mddefault}{\updefault}{\color[rgb]{0,0,0}$a$}%
}}}}
\put(8851,-2386){\makebox(0,0)[b]{\smash{{\SetFigFont{7}{8.4}{\rmdefault}{\mddefault}{\updefault}{\color[rgb]{0,0,0}$a$}%
}}}}
\put(9151,-2686){\makebox(0,0)[b]{\smash{{\SetFigFont{7}{8.4}{\rmdefault}{\mddefault}{\updefault}{\color[rgb]{0,0,0}$a$}%
}}}}
\put(9451,-2986){\makebox(0,0)[b]{\smash{{\SetFigFont{7}{8.4}{\rmdefault}{\mddefault}{\updefault}{\color[rgb]{0,0,0}$a$}%
}}}}
\put(9751,-3286){\makebox(0,0)[b]{\smash{{\SetFigFont{7}{8.4}{\rmdefault}{\mddefault}{\updefault}{\color[rgb]{0,0,0}$a$}%
}}}}
\put(10051,-3586){\makebox(0,0)[b]{\smash{{\SetFigFont{7}{8.4}{\rmdefault}{\mddefault}{\updefault}{\color[rgb]{0,0,0}$a$}%
}}}}
\put(10351,-3886){\makebox(0,0)[b]{\smash{{\SetFigFont{7}{8.4}{\rmdefault}{\mddefault}{\updefault}{\color[rgb]{0,0,0}$a$}%
}}}}
\put(10651,-4186){\makebox(0,0)[b]{\smash{{\SetFigFont{7}{8.4}{\rmdefault}{\mddefault}{\updefault}{\color[rgb]{0,0,0}$a$}%
}}}}
\put(10951,-4486){\makebox(0,0)[b]{\smash{{\SetFigFont{7}{8.4}{\rmdefault}{\mddefault}{\updefault}{\color[rgb]{0,0,0}$a$}%
}}}}
\put(11251,-4786){\makebox(0,0)[b]{\smash{{\SetFigFont{7}{8.4}{\rmdefault}{\mddefault}{\updefault}{\color[rgb]{0,0,0}$a$}%
}}}}
\put(11551,-5086){\makebox(0,0)[b]{\smash{{\SetFigFont{7}{8.4}{\rmdefault}{\mddefault}{\updefault}{\color[rgb]{0,0,0}$a$}%
}}}}
\put(11851,-5386){\makebox(0,0)[b]{\smash{{\SetFigFont{7}{8.4}{\rmdefault}{\mddefault}{\updefault}{\color[rgb]{0,0,0}$a$}%
}}}}
\put(12151,-5686){\makebox(0,0)[b]{\smash{{\SetFigFont{7}{8.4}{\rmdefault}{\mddefault}{\updefault}{\color[rgb]{0,0,0}$a$}%
}}}}
\put(12451,-5986){\makebox(0,0)[b]{\smash{{\SetFigFont{7}{8.4}{\rmdefault}{\mddefault}{\updefault}{\color[rgb]{0,0,0}$a$}%
}}}}
\put(2251, 14){\makebox(0,0)[b]{\smash{{\SetFigFont{7}{8.4}{\rmdefault}{\mddefault}{\updefault}{\color[rgb]{0,0,0}$a$}%
}}}}
\put(2551,-286){\makebox(0,0)[b]{\smash{{\SetFigFont{7}{8.4}{\rmdefault}{\mddefault}{\updefault}{\color[rgb]{0,0,0}$a$}%
}}}}
\put(2851,-586){\makebox(0,0)[b]{\smash{{\SetFigFont{7}{8.4}{\rmdefault}{\mddefault}{\updefault}{\color[rgb]{0,0,0}$a$}%
}}}}
\put(3151,-886){\makebox(0,0)[b]{\smash{{\SetFigFont{7}{8.4}{\rmdefault}{\mddefault}{\updefault}{\color[rgb]{0,0,0}$a$}%
}}}}
\put(3451,-1186){\makebox(0,0)[b]{\smash{{\SetFigFont{7}{8.4}{\rmdefault}{\mddefault}{\updefault}{\color[rgb]{0,0,0}$a$}%
}}}}
\put(3751,-1486){\makebox(0,0)[b]{\smash{{\SetFigFont{7}{8.4}{\rmdefault}{\mddefault}{\updefault}{\color[rgb]{0,0,0}$a$}%
}}}}
\put(4051,-1786){\makebox(0,0)[b]{\smash{{\SetFigFont{7}{8.4}{\rmdefault}{\mddefault}{\updefault}{\color[rgb]{0,0,0}$a$}%
}}}}
\put(4351,-2086){\makebox(0,0)[b]{\smash{{\SetFigFont{7}{8.4}{\rmdefault}{\mddefault}{\updefault}{\color[rgb]{0,0,0}$a$}%
}}}}
\put(4651,-2386){\makebox(0,0)[b]{\smash{{\SetFigFont{7}{8.4}{\rmdefault}{\mddefault}{\updefault}{\color[rgb]{0,0,0}$a$}%
}}}}
\put(5251,-2986){\makebox(0,0)[b]{\smash{{\SetFigFont{7}{8.4}{\rmdefault}{\mddefault}{\updefault}{\color[rgb]{0,0,0}$a$}%
}}}}
\put(5551,-3286){\makebox(0,0)[b]{\smash{{\SetFigFont{7}{8.4}{\rmdefault}{\mddefault}{\updefault}{\color[rgb]{0,0,0}$a$}%
}}}}
\put(5851,-3586){\makebox(0,0)[b]{\smash{{\SetFigFont{7}{8.4}{\rmdefault}{\mddefault}{\updefault}{\color[rgb]{0,0,0}$a$}%
}}}}
\put(6151,-3886){\makebox(0,0)[b]{\smash{{\SetFigFont{7}{8.4}{\rmdefault}{\mddefault}{\updefault}{\color[rgb]{0,0,0}$a$}%
}}}}
\put(7051,-4786){\makebox(0,0)[b]{\smash{{\SetFigFont{7}{8.4}{\rmdefault}{\mddefault}{\updefault}{\color[rgb]{0,0,0}$a$}%
}}}}
\put(7351,-5086){\makebox(0,0)[b]{\smash{{\SetFigFont{7}{8.4}{\rmdefault}{\mddefault}{\updefault}{\color[rgb]{0,0,0}$a$}%
}}}}
\put(7651,-5386){\makebox(0,0)[b]{\smash{{\SetFigFont{7}{8.4}{\rmdefault}{\mddefault}{\updefault}{\color[rgb]{0,0,0}$a$}%
}}}}
\put(7951,-5686){\makebox(0,0)[b]{\smash{{\SetFigFont{7}{8.4}{\rmdefault}{\mddefault}{\updefault}{\color[rgb]{0,0,0}$a$}%
}}}}
\put(8251,-5986){\makebox(0,0)[b]{\smash{{\SetFigFont{7}{8.4}{\rmdefault}{\mddefault}{\updefault}{\color[rgb]{0,0,0}$a$}%
}}}}
\put(2851, 14){\makebox(0,0)[b]{\smash{{\SetFigFont{7}{8.4}{\rmdefault}{\mddefault}{\updefault}{\color[rgb]{0,0,0}$a$}%
}}}}
\put(3151,-286){\makebox(0,0)[b]{\smash{{\SetFigFont{7}{8.4}{\rmdefault}{\mddefault}{\updefault}{\color[rgb]{0,0,0}$a$}%
}}}}
\put(3451,-586){\makebox(0,0)[b]{\smash{{\SetFigFont{7}{8.4}{\rmdefault}{\mddefault}{\updefault}{\color[rgb]{0,0,0}$a$}%
}}}}
\put(3751,-886){\makebox(0,0)[b]{\smash{{\SetFigFont{7}{8.4}{\rmdefault}{\mddefault}{\updefault}{\color[rgb]{0,0,0}$a$}%
}}}}
\put(4051,-1186){\makebox(0,0)[b]{\smash{{\SetFigFont{7}{8.4}{\rmdefault}{\mddefault}{\updefault}{\color[rgb]{0,0,0}$a$}%
}}}}
\put(4351,-1486){\makebox(0,0)[b]{\smash{{\SetFigFont{7}{8.4}{\rmdefault}{\mddefault}{\updefault}{\color[rgb]{0,0,0}$a$}%
}}}}
\put(4651,-1786){\makebox(0,0)[b]{\smash{{\SetFigFont{7}{8.4}{\rmdefault}{\mddefault}{\updefault}{\color[rgb]{0,0,0}$a$}%
}}}}
\put(4951,-2086){\makebox(0,0)[b]{\smash{{\SetFigFont{7}{8.4}{\rmdefault}{\mddefault}{\updefault}{\color[rgb]{0,0,0}$a$}%
}}}}
\put(5551,-2686){\makebox(0,0)[b]{\smash{{\SetFigFont{7}{8.4}{\rmdefault}{\mddefault}{\updefault}{\color[rgb]{0,0,0}$a$}%
}}}}
\put(5851,-2986){\makebox(0,0)[b]{\smash{{\SetFigFont{7}{8.4}{\rmdefault}{\mddefault}{\updefault}{\color[rgb]{0,0,0}$a$}%
}}}}
\put(6151,-3286){\makebox(0,0)[b]{\smash{{\SetFigFont{7}{8.4}{\rmdefault}{\mddefault}{\updefault}{\color[rgb]{0,0,0}$a$}%
}}}}
\put(7051,-4186){\makebox(0,0)[b]{\smash{{\SetFigFont{7}{8.4}{\rmdefault}{\mddefault}{\updefault}{\color[rgb]{0,0,0}$a$}%
}}}}
\put(7351,-4486){\makebox(0,0)[b]{\smash{{\SetFigFont{7}{8.4}{\rmdefault}{\mddefault}{\updefault}{\color[rgb]{0,0,0}$a$}%
}}}}
\put(7651,-4786){\makebox(0,0)[b]{\smash{{\SetFigFont{7}{8.4}{\rmdefault}{\mddefault}{\updefault}{\color[rgb]{0,0,0}$a$}%
}}}}
\put(7951,-5086){\makebox(0,0)[b]{\smash{{\SetFigFont{7}{8.4}{\rmdefault}{\mddefault}{\updefault}{\color[rgb]{0,0,0}$a$}%
}}}}
\put(8251,-5386){\makebox(0,0)[b]{\smash{{\SetFigFont{7}{8.4}{\rmdefault}{\mddefault}{\updefault}{\color[rgb]{0,0,0}$a$}%
}}}}
\put(8551,-5686){\makebox(0,0)[b]{\smash{{\SetFigFont{7}{8.4}{\rmdefault}{\mddefault}{\updefault}{\color[rgb]{0,0,0}$a$}%
}}}}
\put(8851,-5986){\makebox(0,0)[b]{\smash{{\SetFigFont{7}{8.4}{\rmdefault}{\mddefault}{\updefault}{\color[rgb]{0,0,0}$a$}%
}}}}
\put(1651, 14){\makebox(0,0)[b]{\smash{{\SetFigFont{7}{8.4}{\rmdefault}{\mddefault}{\updefault}{\color[rgb]{0,0,0}$a$}%
}}}}
\put(1951,-286){\makebox(0,0)[b]{\smash{{\SetFigFont{7}{8.4}{\rmdefault}{\mddefault}{\updefault}{\color[rgb]{0,0,0}$a$}%
}}}}
\put(2251,-586){\makebox(0,0)[b]{\smash{{\SetFigFont{7}{8.4}{\rmdefault}{\mddefault}{\updefault}{\color[rgb]{0,0,0}$a$}%
}}}}
\put(2551,-886){\makebox(0,0)[b]{\smash{{\SetFigFont{7}{8.4}{\rmdefault}{\mddefault}{\updefault}{\color[rgb]{0,0,0}$a$}%
}}}}
\put(2851,-1186){\makebox(0,0)[b]{\smash{{\SetFigFont{7}{8.4}{\rmdefault}{\mddefault}{\updefault}{\color[rgb]{0,0,0}$a$}%
}}}}
\put(3151,-1486){\makebox(0,0)[b]{\smash{{\SetFigFont{7}{8.4}{\rmdefault}{\mddefault}{\updefault}{\color[rgb]{0,0,0}$a$}%
}}}}
\put(3451,-1786){\makebox(0,0)[b]{\smash{{\SetFigFont{7}{8.4}{\rmdefault}{\mddefault}{\updefault}{\color[rgb]{0,0,0}$a$}%
}}}}
\put(3751,-2086){\makebox(0,0)[b]{\smash{{\SetFigFont{7}{8.4}{\rmdefault}{\mddefault}{\updefault}{\color[rgb]{0,0,0}$a$}%
}}}}
\put(4051,-2386){\makebox(0,0)[b]{\smash{{\SetFigFont{7}{8.4}{\rmdefault}{\mddefault}{\updefault}{\color[rgb]{0,0,0}$a$}%
}}}}
\put(4351,-2686){\makebox(0,0)[b]{\smash{{\SetFigFont{7}{8.4}{\rmdefault}{\mddefault}{\updefault}{\color[rgb]{0,0,0}$a$}%
}}}}
\put(4951,-3286){\makebox(0,0)[b]{\smash{{\SetFigFont{7}{8.4}{\rmdefault}{\mddefault}{\updefault}{\color[rgb]{0,0,0}$a$}%
}}}}
\put(5251,-3586){\makebox(0,0)[b]{\smash{{\SetFigFont{7}{8.4}{\rmdefault}{\mddefault}{\updefault}{\color[rgb]{0,0,0}$a$}%
}}}}
\put(5551,-3886){\makebox(0,0)[b]{\smash{{\SetFigFont{7}{8.4}{\rmdefault}{\mddefault}{\updefault}{\color[rgb]{0,0,0}$a$}%
}}}}
\put(5851,-4186){\makebox(0,0)[b]{\smash{{\SetFigFont{7}{8.4}{\rmdefault}{\mddefault}{\updefault}{\color[rgb]{0,0,0}$a$}%
}}}}
\put(6151,-4486){\makebox(0,0)[b]{\smash{{\SetFigFont{7}{8.4}{\rmdefault}{\mddefault}{\updefault}{\color[rgb]{0,0,0}$a$}%
}}}}
\put(7051,-5386){\makebox(0,0)[b]{\smash{{\SetFigFont{7}{8.4}{\rmdefault}{\mddefault}{\updefault}{\color[rgb]{0,0,0}$a$}%
}}}}
\put(7351,-5686){\makebox(0,0)[b]{\smash{{\SetFigFont{7}{8.4}{\rmdefault}{\mddefault}{\updefault}{\color[rgb]{0,0,0}$a$}%
}}}}
\put(7651,-5986){\makebox(0,0)[b]{\smash{{\SetFigFont{7}{8.4}{\rmdefault}{\mddefault}{\updefault}{\color[rgb]{0,0,0}$a$}%
}}}}
\put(6451,-286){\makebox(0,0)[b]{\smash{{\SetFigFont{7}{8.4}{\rmdefault}{\mddefault}{\updefault}{\color[rgb]{0,0,0}$b$}%
}}}}
\put(6451,-586){\makebox(0,0)[b]{\smash{{\SetFigFont{7}{8.4}{\rmdefault}{\mddefault}{\updefault}{\color[rgb]{0,0,0}$b$}%
}}}}
\put(6451,-886){\makebox(0,0)[b]{\smash{{\SetFigFont{7}{8.4}{\rmdefault}{\mddefault}{\updefault}{\color[rgb]{0,0,0}$b$}%
}}}}
\put(6451,-1486){\makebox(0,0)[b]{\smash{{\SetFigFont{7}{8.4}{\rmdefault}{\mddefault}{\updefault}{\color[rgb]{0,0,0}$b$}%
}}}}
\put(6451,-2086){\makebox(0,0)[b]{\smash{{\SetFigFont{7}{8.4}{\rmdefault}{\mddefault}{\updefault}{\color[rgb]{0,0,0}$b$}%
}}}}
\put(6451,-2386){\makebox(0,0)[b]{\smash{{\SetFigFont{7}{8.4}{\rmdefault}{\mddefault}{\updefault}{\color[rgb]{0,0,0}$b$}%
}}}}
\put(6451,-2686){\makebox(0,0)[b]{\smash{{\SetFigFont{7}{8.4}{\rmdefault}{\mddefault}{\updefault}{\color[rgb]{0,0,0}$b$}%
}}}}
\put(6451,-2986){\makebox(0,0)[b]{\smash{{\SetFigFont{7}{8.4}{\rmdefault}{\mddefault}{\updefault}{\color[rgb]{0,0,0}$b$}%
}}}}
\put(6451,-3286){\makebox(0,0)[b]{\smash{{\SetFigFont{7}{8.4}{\rmdefault}{\mddefault}{\updefault}{\color[rgb]{0,0,0}$b$}%
}}}}
\put(6451,-3886){\makebox(0,0)[b]{\smash{{\SetFigFont{7}{8.4}{\rmdefault}{\mddefault}{\updefault}{\color[rgb]{0,0,0}$b$}%
}}}}
\put(6451,-5086){\makebox(0,0)[b]{\smash{{\SetFigFont{7}{8.4}{\rmdefault}{\mddefault}{\updefault}{\color[rgb]{0,0,0}$b$}%
}}}}
\put(6451,-5686){\makebox(0,0)[b]{\smash{{\SetFigFont{7}{8.4}{\rmdefault}{\mddefault}{\updefault}{\color[rgb]{0,0,0}$b$}%
}}}}
\put(6451,-5986){\makebox(0,0)[b]{\smash{{\SetFigFont{7}{8.4}{\rmdefault}{\mddefault}{\updefault}{\color[rgb]{0,0,0}$b$}%
}}}}
\put(6451,-4486){\makebox(0,0)[b]{\smash{{\SetFigFont{7}{8.4}{\rmdefault}{\mddefault}{\updefault}{\color[rgb]{0,0,0}$b$}%
}}}}
\put(6751, 14){\makebox(0,0)[b]{\smash{{\SetFigFont{7}{8.4}{\rmdefault}{\mddefault}{\updefault}{\color[rgb]{0,0,0}$b$}%
}}}}
\put(6751,-586){\makebox(0,0)[b]{\smash{{\SetFigFont{7}{8.4}{\rmdefault}{\mddefault}{\updefault}{\color[rgb]{0,0,0}$b$}%
}}}}
\put(6751,-1186){\makebox(0,0)[b]{\smash{{\SetFigFont{7}{8.4}{\rmdefault}{\mddefault}{\updefault}{\color[rgb]{0,0,0}$b$}%
}}}}
\put(6751,-1486){\makebox(0,0)[b]{\smash{{\SetFigFont{7}{8.4}{\rmdefault}{\mddefault}{\updefault}{\color[rgb]{0,0,0}$b$}%
}}}}
\put(6751,-1786){\makebox(0,0)[b]{\smash{{\SetFigFont{7}{8.4}{\rmdefault}{\mddefault}{\updefault}{\color[rgb]{0,0,0}$b$}%
}}}}
\put(6751,-2386){\makebox(0,0)[b]{\smash{{\SetFigFont{7}{8.4}{\rmdefault}{\mddefault}{\updefault}{\color[rgb]{0,0,0}$b$}%
}}}}
\put(6751,-2686){\makebox(0,0)[b]{\smash{{\SetFigFont{7}{8.4}{\rmdefault}{\mddefault}{\updefault}{\color[rgb]{0,0,0}$b$}%
}}}}
\put(6751,-2986){\makebox(0,0)[b]{\smash{{\SetFigFont{7}{8.4}{\rmdefault}{\mddefault}{\updefault}{\color[rgb]{0,0,0}$b$}%
}}}}
\put(6751,-3286){\makebox(0,0)[b]{\smash{{\SetFigFont{7}{8.4}{\rmdefault}{\mddefault}{\updefault}{\color[rgb]{0,0,0}$b$}%
}}}}
\put(6751,-3586){\makebox(0,0)[b]{\smash{{\SetFigFont{7}{8.4}{\rmdefault}{\mddefault}{\updefault}{\color[rgb]{0,0,0}$b$}%
}}}}
\put(6751,-4786){\makebox(0,0)[b]{\smash{{\SetFigFont{7}{8.4}{\rmdefault}{\mddefault}{\updefault}{\color[rgb]{0,0,0}$b$}%
}}}}
\put(6751,-5386){\makebox(0,0)[b]{\smash{{\SetFigFont{7}{8.4}{\rmdefault}{\mddefault}{\updefault}{\color[rgb]{0,0,0}$b$}%
}}}}
\put(6751,-5986){\makebox(0,0)[b]{\smash{{\SetFigFont{7}{8.4}{\rmdefault}{\mddefault}{\updefault}{\color[rgb]{0,0,0}$b$}%
}}}}
\put(6751,-4186){\makebox(0,0)[b]{\smash{{\SetFigFont{7}{8.4}{\rmdefault}{\mddefault}{\updefault}{\color[rgb]{0,0,0}$b$}%
}}}}
\put(1351,314){\makebox(0,0)[b]{\smash{{\SetFigFont{7}{8.4}{\rmdefault}{\mddefault}{\updefault}{\color[rgb]{0,0,0}$a$}%
}}}}
\put(1951,314){\makebox(0,0)[b]{\smash{{\SetFigFont{7}{8.4}{\rmdefault}{\mddefault}{\updefault}{\color[rgb]{0,0,0}$a$}%
}}}}
\put(2551,314){\makebox(0,0)[b]{\smash{{\SetFigFont{7}{8.4}{\rmdefault}{\mddefault}{\updefault}{\color[rgb]{0,0,0}$a$}%
}}}}
\put(1351,614){\makebox(0,0)[b]{\smash{{\SetFigFont{7}{8.4}{\rmdefault}{\mddefault}{\updefault}{\color[rgb]{0,0,0}$18$}%
}}}}
\put(1951,614){\makebox(0,0)[b]{\smash{{\SetFigFont{7}{8.4}{\rmdefault}{\mddefault}{\updefault}{\color[rgb]{0,0,0}$16$}%
}}}}
\end{picture}%

%% file: 251Ex.pstex_t
\begin{picture}(0,0)%
\includegraphics{251Ex.eps}%
\end{picture}%
\setlength{\unitlength}{2763sp}%
\begingroup\makeatletter\ifx\SetFigFont\undefined%
\gdef\SetFigFont#1#2#3#4#5{%
  \reset@font\fontsize{#1}{#2pt}%
  \fontfamily{#3}\fontseries{#4}\fontshape{#5}%
  \selectfont}%
\fi\endgroup%
\begin{picture}(7377,5911)(136,-5150)
\put(5851,-1486){\makebox(0,0)[b]{\smash{{\SetFigFont{8}{9.6}{\rmdefault}{\mddefault}{\updefault}{\color[rgb]{0,0,0}$b$}%
}}}}
\put(2251,-1186){\makebox(0,0)[b]{\smash{{\SetFigFont{8}{9.6}{\rmdefault}{\mddefault}{\updefault}{\color[rgb]{0,0,0}$ab$}%
}}}}
\put(2551,-1486){\makebox(0,0)[b]{\smash{{\SetFigFont{8}{9.6}{\rmdefault}{\mddefault}{\updefault}{\color[rgb]{0,0,0}$ab$}%
}}}}
\put(2251, 14){\makebox(0,0)[b]{\smash{{\SetFigFont{8}{9.6}{\rmdefault}{\mddefault}{\updefault}{\color[rgb]{0,0,0}$ab$}%
}}}}
\put(2551,-286){\makebox(0,0)[b]{\smash{{\SetFigFont{8}{9.6}{\rmdefault}{\mddefault}{\updefault}{\color[rgb]{0,0,0}$ab$}%
}}}}
\put(4051,-1786){\makebox(0,0)[b]{\smash{{\SetFigFont{8}{9.6}{\rmdefault}{\mddefault}{\updefault}{\color[rgb]{0,0,0}$ab$}%
}}}}
\put(4051,-2986){\makebox(0,0)[b]{\smash{{\SetFigFont{8}{9.6}{\rmdefault}{\mddefault}{\updefault}{\color[rgb]{0,0,0}$ab$}%
}}}}
\put(5551,-3286){\makebox(0,0)[b]{\smash{{\SetFigFont{8}{9.6}{\rmdefault}{\mddefault}{\updefault}{\color[rgb]{0,0,0}$ab$}%
}}}}
\put(5851,-3586){\makebox(0,0)[b]{\smash{{\SetFigFont{8}{9.6}{\rmdefault}{\mddefault}{\updefault}{\color[rgb]{0,0,0}$ab$}%
}}}}
\put(5551,-4486){\makebox(0,0)[b]{\smash{{\SetFigFont{8}{9.6}{\rmdefault}{\mddefault}{\updefault}{\color[rgb]{0,0,0}$ab$}%
}}}}
\put(5851,-4786){\makebox(0,0)[b]{\smash{{\SetFigFont{8}{9.6}{\rmdefault}{\mddefault}{\updefault}{\color[rgb]{0,0,0}$ab$}%
}}}}
\put(2251,-4186){\makebox(0,0)[b]{\smash{{\SetFigFont{8}{9.6}{\rmdefault}{\mddefault}{\updefault}{\color[rgb]{0,0,0}$bc$}%
}}}}
\put(2551,-3886){\makebox(0,0)[b]{\smash{{\SetFigFont{8}{9.6}{\rmdefault}{\mddefault}{\updefault}{\color[rgb]{0,0,0}$bc$}%
}}}}
\put(4051,-2386){\makebox(0,0)[b]{\smash{{\SetFigFont{8}{9.6}{\rmdefault}{\mddefault}{\updefault}{\color[rgb]{0,0,0}$bc$}%
}}}}
\put(5551,-886){\makebox(0,0)[b]{\smash{{\SetFigFont{8}{9.6}{\rmdefault}{\mddefault}{\updefault}{\color[rgb]{0,0,0}$bc$}%
}}}}
\put(5851,-586){\makebox(0,0)[b]{\smash{{\SetFigFont{8}{9.6}{\rmdefault}{\mddefault}{\updefault}{\color[rgb]{0,0,0}$bc$}%
}}}}
\put(3751,-2686){\makebox(0,0)[b]{\smash{{\SetFigFont{8}{9.6}{\rmdefault}{\mddefault}{\updefault}{\color[rgb]{0,0,0}$ac$}%
}}}}
\put(4351,-2086){\makebox(0,0)[b]{\smash{{\SetFigFont{8}{9.6}{\rmdefault}{\mddefault}{\updefault}{\color[rgb]{0,0,0}$ac$}%
}}}}
\put(151,-2686){\makebox(0,0)[b]{\smash{{\SetFigFont{8}{9.6}{\rmdefault}{\mddefault}{\updefault}{\color[rgb]{0,0,0}$10$}%
}}}}
\put(151, 14){\makebox(0,0)[b]{\smash{{\SetFigFont{8}{9.6}{\rmdefault}{\mddefault}{\updefault}{\color[rgb]{0,0,0}$1$}%
}}}}
\put(151,-286){\makebox(0,0)[b]{\smash{{\SetFigFont{8}{9.6}{\rmdefault}{\mddefault}{\updefault}{\color[rgb]{0,0,0}$2$}%
}}}}
\put(151,-586){\makebox(0,0)[b]{\smash{{\SetFigFont{8}{9.6}{\rmdefault}{\mddefault}{\updefault}{\color[rgb]{0,0,0}$3$}%
}}}}
\put(151,-886){\makebox(0,0)[b]{\smash{{\SetFigFont{8}{9.6}{\rmdefault}{\mddefault}{\updefault}{\color[rgb]{0,0,0}$4$}%
}}}}
\put(151,-1186){\makebox(0,0)[b]{\smash{{\SetFigFont{8}{9.6}{\rmdefault}{\mddefault}{\updefault}{\color[rgb]{0,0,0}$5$}%
}}}}
\put(151,-1486){\makebox(0,0)[b]{\smash{{\SetFigFont{8}{9.6}{\rmdefault}{\mddefault}{\updefault}{\color[rgb]{0,0,0}$6$}%
}}}}
\put(151,-1786){\makebox(0,0)[b]{\smash{{\SetFigFont{8}{9.6}{\rmdefault}{\mddefault}{\updefault}{\color[rgb]{0,0,0}$7$}%
}}}}
\put(151,-2086){\makebox(0,0)[b]{\smash{{\SetFigFont{8}{9.6}{\rmdefault}{\mddefault}{\updefault}{\color[rgb]{0,0,0}$8$}%
}}}}
\put(151,314){\makebox(0,0)[b]{\smash{{\SetFigFont{8}{9.6}{\rmdefault}{\mddefault}{\updefault}{\color[rgb]{0,0,0}$t$}%
}}}}
\put(151,-2386){\makebox(0,0)[b]{\smash{{\SetFigFont{8}{9.6}{\rmdefault}{\mddefault}{\updefault}{\color[rgb]{0,0,0}$9$}%
}}}}
\put(751,614){\makebox(0,0)[b]{\smash{{\SetFigFont{8}{9.6}{\rmdefault}{\mddefault}{\updefault}{\color[rgb]{0,0,0}$12$}%
}}}}
\put(151,-3286){\makebox(0,0)[b]{\smash{{\SetFigFont{8}{9.6}{\rmdefault}{\mddefault}{\updefault}{\color[rgb]{0,0,0}$12$}%
}}}}
\put(151,-3586){\makebox(0,0)[b]{\smash{{\SetFigFont{8}{9.6}{\rmdefault}{\mddefault}{\updefault}{\color[rgb]{0,0,0}$13$}%
}}}}
\put(151,-3886){\makebox(0,0)[b]{\smash{{\SetFigFont{8}{9.6}{\rmdefault}{\mddefault}{\updefault}{\color[rgb]{0,0,0}$14$}%
}}}}
\put(151,-4186){\makebox(0,0)[b]{\smash{{\SetFigFont{8}{9.6}{\rmdefault}{\mddefault}{\updefault}{\color[rgb]{0,0,0}$15$}%
}}}}
\put(151,-4486){\makebox(0,0)[b]{\smash{{\SetFigFont{8}{9.6}{\rmdefault}{\mddefault}{\updefault}{\color[rgb]{0,0,0}$16$}%
}}}}
\put(151,-4786){\makebox(0,0)[b]{\smash{{\SetFigFont{8}{9.6}{\rmdefault}{\mddefault}{\updefault}{\color[rgb]{0,0,0}$17$}%
}}}}
\put(151,-2986){\makebox(0,0)[b]{\smash{{\SetFigFont{8}{9.6}{\rmdefault}{\mddefault}{\updefault}{\color[rgb]{0,0,0}$11$}%
}}}}
\put(751,314){\makebox(0,0)[b]{\smash{{\SetFigFont{8}{9.6}{\rmdefault}{\mddefault}{\updefault}{\color[rgb]{0,0,0}$a$}%
}}}}
\put(1951,614){\makebox(0,0)[b]{\smash{{\SetFigFont{8}{9.6}{\rmdefault}{\mddefault}{\updefault}{\color[rgb]{0,0,0}$8$}%
}}}}
\put(2251,614){\makebox(0,0)[b]{\smash{{\SetFigFont{8}{9.6}{\rmdefault}{\mddefault}{\updefault}{\color[rgb]{0,0,0}$7$}%
}}}}
\put(2551,614){\makebox(0,0)[b]{\smash{{\SetFigFont{8}{9.6}{\rmdefault}{\mddefault}{\updefault}{\color[rgb]{0,0,0}$6$}%
}}}}
\put(4051,614){\makebox(0,0)[b]{\smash{{\SetFigFont{8}{9.6}{\rmdefault}{\mddefault}{\updefault}{\color[rgb]{0,0,0}$1$}%
}}}}
\put(5551,614){\makebox(0,0)[b]{\smash{{\SetFigFont{8}{9.6}{\rmdefault}{\mddefault}{\updefault}{\color[rgb]{0,0,0}$-4$}%
}}}}
\put(5851,614){\makebox(0,0)[b]{\smash{{\SetFigFont{8}{9.6}{\rmdefault}{\mddefault}{\updefault}{\color[rgb]{0,0,0}$-5$}%
}}}}
\put(6751,614){\makebox(0,0)[b]{\smash{{\SetFigFont{8}{9.6}{\rmdefault}{\mddefault}{\updefault}{\color[rgb]{0,0,0}$-8$}%
}}}}
\put(1951,314){\makebox(0,0)[b]{\smash{{\SetFigFont{8}{9.6}{\rmdefault}{\mddefault}{\updefault}{\color[rgb]{0,0,0}$a$}%
}}}}
\put(2251,314){\makebox(0,0)[b]{\smash{{\SetFigFont{8}{9.6}{\rmdefault}{\mddefault}{\updefault}{\color[rgb]{0,0,0}$b$}%
}}}}
\put(2551,314){\makebox(0,0)[b]{\smash{{\SetFigFont{8}{9.6}{\rmdefault}{\mddefault}{\updefault}{\color[rgb]{0,0,0}$b$}%
}}}}
\put(4051,314){\makebox(0,0)[b]{\smash{{\SetFigFont{8}{9.6}{\rmdefault}{\mddefault}{\updefault}{\color[rgb]{0,0,0}$b$}%
}}}}
\put(5551,314){\makebox(0,0)[b]{\smash{{\SetFigFont{8}{9.6}{\rmdefault}{\mddefault}{\updefault}{\color[rgb]{0,0,0}$b$}%
}}}}
\put(5851,314){\makebox(0,0)[b]{\smash{{\SetFigFont{8}{9.6}{\rmdefault}{\mddefault}{\updefault}{\color[rgb]{0,0,0}$b$}%
}}}}
\put(6751,314){\makebox(0,0)[b]{\smash{{\SetFigFont{8}{9.6}{\rmdefault}{\mddefault}{\updefault}{\color[rgb]{0,0,0}$c$}%
}}}}
\put(6451, 14){\makebox(0,0)[b]{\smash{{\SetFigFont{8}{9.6}{\rmdefault}{\mddefault}{\updefault}{\color[rgb]{0,0,0}$c$}%
}}}}
\put(6151,-286){\makebox(0,0)[b]{\smash{{\SetFigFont{8}{9.6}{\rmdefault}{\mddefault}{\updefault}{\color[rgb]{0,0,0}$c$}%
}}}}
\put(5251,-1186){\makebox(0,0)[b]{\smash{{\SetFigFont{8}{9.6}{\rmdefault}{\mddefault}{\updefault}{\color[rgb]{0,0,0}$c$}%
}}}}
\put(4951,-1486){\makebox(0,0)[b]{\smash{{\SetFigFont{8}{9.6}{\rmdefault}{\mddefault}{\updefault}{\color[rgb]{0,0,0}$c$}%
}}}}
\put(4651,-1786){\makebox(0,0)[b]{\smash{{\SetFigFont{8}{9.6}{\rmdefault}{\mddefault}{\updefault}{\color[rgb]{0,0,0}$c$}%
}}}}
\put(3451,-2986){\makebox(0,0)[b]{\smash{{\SetFigFont{8}{9.6}{\rmdefault}{\mddefault}{\updefault}{\color[rgb]{0,0,0}$c$}%
}}}}
\put(3151,-3286){\makebox(0,0)[b]{\smash{{\SetFigFont{8}{9.6}{\rmdefault}{\mddefault}{\updefault}{\color[rgb]{0,0,0}$c$}%
}}}}
\put(2851,-3586){\makebox(0,0)[b]{\smash{{\SetFigFont{8}{9.6}{\rmdefault}{\mddefault}{\updefault}{\color[rgb]{0,0,0}$c$}%
}}}}
\put(1951,-4486){\makebox(0,0)[b]{\smash{{\SetFigFont{8}{9.6}{\rmdefault}{\mddefault}{\updefault}{\color[rgb]{0,0,0}$c$}%
}}}}
\put(1651,-4786){\makebox(0,0)[b]{\smash{{\SetFigFont{8}{9.6}{\rmdefault}{\mddefault}{\updefault}{\color[rgb]{0,0,0}$c$}%
}}}}
\put(1351,-5086){\makebox(0,0)[b]{\smash{{\SetFigFont{8}{9.6}{\rmdefault}{\mddefault}{\updefault}{\color[rgb]{0,0,0}$c$}%
}}}}
\put(1051, 14){\makebox(0,0)[b]{\smash{{\SetFigFont{8}{9.6}{\rmdefault}{\mddefault}{\updefault}{\color[rgb]{0,0,0}$a$}%
}}}}
\put(1351,-286){\makebox(0,0)[b]{\smash{{\SetFigFont{8}{9.6}{\rmdefault}{\mddefault}{\updefault}{\color[rgb]{0,0,0}$a$}%
}}}}
\put(1651,-586){\makebox(0,0)[b]{\smash{{\SetFigFont{8}{9.6}{\rmdefault}{\mddefault}{\updefault}{\color[rgb]{0,0,0}$a$}%
}}}}
\put(1951,-886){\makebox(0,0)[b]{\smash{{\SetFigFont{8}{9.6}{\rmdefault}{\mddefault}{\updefault}{\color[rgb]{0,0,0}$a$}%
}}}}
\put(2851,-1786){\makebox(0,0)[b]{\smash{{\SetFigFont{8}{9.6}{\rmdefault}{\mddefault}{\updefault}{\color[rgb]{0,0,0}$a$}%
}}}}
\put(3151,-2086){\makebox(0,0)[b]{\smash{{\SetFigFont{8}{9.6}{\rmdefault}{\mddefault}{\updefault}{\color[rgb]{0,0,0}$a$}%
}}}}
\put(3451,-2386){\makebox(0,0)[b]{\smash{{\SetFigFont{8}{9.6}{\rmdefault}{\mddefault}{\updefault}{\color[rgb]{0,0,0}$a$}%
}}}}
\put(4351,-3286){\makebox(0,0)[b]{\smash{{\SetFigFont{8}{9.6}{\rmdefault}{\mddefault}{\updefault}{\color[rgb]{0,0,0}$a$}%
}}}}
\put(4651,-3586){\makebox(0,0)[b]{\smash{{\SetFigFont{8}{9.6}{\rmdefault}{\mddefault}{\updefault}{\color[rgb]{0,0,0}$a$}%
}}}}
\put(4951,-3886){\makebox(0,0)[b]{\smash{{\SetFigFont{8}{9.6}{\rmdefault}{\mddefault}{\updefault}{\color[rgb]{0,0,0}$a$}%
}}}}
\put(6151,-5086){\makebox(0,0)[b]{\smash{{\SetFigFont{8}{9.6}{\rmdefault}{\mddefault}{\updefault}{\color[rgb]{0,0,0}$a$}%
}}}}
\put(2851,-586){\makebox(0,0)[b]{\smash{{\SetFigFont{8}{9.6}{\rmdefault}{\mddefault}{\updefault}{\color[rgb]{0,0,0}$a$}%
}}}}
\put(3451,-1186){\makebox(0,0)[b]{\smash{{\SetFigFont{8}{9.6}{\rmdefault}{\mddefault}{\updefault}{\color[rgb]{0,0,0}$a$}%
}}}}
\put(3751,-1486){\makebox(0,0)[b]{\smash{{\SetFigFont{8}{9.6}{\rmdefault}{\mddefault}{\updefault}{\color[rgb]{0,0,0}$a$}%
}}}}
\put(4651,-2386){\makebox(0,0)[b]{\smash{{\SetFigFont{8}{9.6}{\rmdefault}{\mddefault}{\updefault}{\color[rgb]{0,0,0}$a$}%
}}}}
\put(4951,-2686){\makebox(0,0)[b]{\smash{{\SetFigFont{8}{9.6}{\rmdefault}{\mddefault}{\updefault}{\color[rgb]{0,0,0}$a$}%
}}}}
\put(5251,-2986){\makebox(0,0)[b]{\smash{{\SetFigFont{8}{9.6}{\rmdefault}{\mddefault}{\updefault}{\color[rgb]{0,0,0}$a$}%
}}}}
\put(5251,-4186){\makebox(0,0)[b]{\smash{{\SetFigFont{8}{9.6}{\rmdefault}{\mddefault}{\updefault}{\color[rgb]{0,0,0}$a$}%
}}}}
\put(6151,-3886){\makebox(0,0)[b]{\smash{{\SetFigFont{8}{9.6}{\rmdefault}{\mddefault}{\updefault}{\color[rgb]{0,0,0}$a$}%
}}}}
\put(6451,-4186){\makebox(0,0)[b]{\smash{{\SetFigFont{8}{9.6}{\rmdefault}{\mddefault}{\updefault}{\color[rgb]{0,0,0}$a$}%
}}}}
\put(6751,-4486){\makebox(0,0)[b]{\smash{{\SetFigFont{8}{9.6}{\rmdefault}{\mddefault}{\updefault}{\color[rgb]{0,0,0}$a$}%
}}}}
\put(7051,-4786){\makebox(0,0)[b]{\smash{{\SetFigFont{8}{9.6}{\rmdefault}{\mddefault}{\updefault}{\color[rgb]{0,0,0}$a$}%
}}}}
\put(7351,-5086){\makebox(0,0)[b]{\smash{{\SetFigFont{8}{9.6}{\rmdefault}{\mddefault}{\updefault}{\color[rgb]{0,0,0}$a$}%
}}}}
\put(3151,-886){\makebox(0,0)[b]{\smash{{\SetFigFont{8}{9.6}{\rmdefault}{\mddefault}{\updefault}{\color[rgb]{0,0,0}$a$}%
}}}}
\put(2251,-286){\makebox(0,0)[b]{\smash{{\SetFigFont{8}{9.6}{\rmdefault}{\mddefault}{\updefault}{\color[rgb]{0,0,0}$b$}%
}}}}
\put(2251,-586){\makebox(0,0)[b]{\smash{{\SetFigFont{8}{9.6}{\rmdefault}{\mddefault}{\updefault}{\color[rgb]{0,0,0}$b$}%
}}}}
\put(2251,-886){\makebox(0,0)[b]{\smash{{\SetFigFont{8}{9.6}{\rmdefault}{\mddefault}{\updefault}{\color[rgb]{0,0,0}$b$}%
}}}}
\put(2251,-1786){\makebox(0,0)[b]{\smash{{\SetFigFont{8}{9.6}{\rmdefault}{\mddefault}{\updefault}{\color[rgb]{0,0,0}$b$}%
}}}}
\put(2251,-2086){\makebox(0,0)[b]{\smash{{\SetFigFont{8}{9.6}{\rmdefault}{\mddefault}{\updefault}{\color[rgb]{0,0,0}$b$}%
}}}}
\put(2251,-2386){\makebox(0,0)[b]{\smash{{\SetFigFont{8}{9.6}{\rmdefault}{\mddefault}{\updefault}{\color[rgb]{0,0,0}$b$}%
}}}}
\put(2251,-2986){\makebox(0,0)[b]{\smash{{\SetFigFont{8}{9.6}{\rmdefault}{\mddefault}{\updefault}{\color[rgb]{0,0,0}$b$}%
}}}}
\put(2251,-3286){\makebox(0,0)[b]{\smash{{\SetFigFont{8}{9.6}{\rmdefault}{\mddefault}{\updefault}{\color[rgb]{0,0,0}$b$}%
}}}}
\put(2251,-3886){\makebox(0,0)[b]{\smash{{\SetFigFont{8}{9.6}{\rmdefault}{\mddefault}{\updefault}{\color[rgb]{0,0,0}$b$}%
}}}}
\put(2251,-4486){\makebox(0,0)[b]{\smash{{\SetFigFont{8}{9.6}{\rmdefault}{\mddefault}{\updefault}{\color[rgb]{0,0,0}$b$}%
}}}}
\put(2251,-4786){\makebox(0,0)[b]{\smash{{\SetFigFont{8}{9.6}{\rmdefault}{\mddefault}{\updefault}{\color[rgb]{0,0,0}$b$}%
}}}}
\put(2251,-5086){\makebox(0,0)[b]{\smash{{\SetFigFont{8}{9.6}{\rmdefault}{\mddefault}{\updefault}{\color[rgb]{0,0,0}$b$}%
}}}}
\put(2251,-2686){\makebox(0,0)[b]{\smash{{\SetFigFont{8}{9.6}{\rmdefault}{\mddefault}{\updefault}{\color[rgb]{0,0,0}$b$}%
}}}}
\put(2251,-3586){\makebox(0,0)[b]{\smash{{\SetFigFont{8}{9.6}{\rmdefault}{\mddefault}{\updefault}{\color[rgb]{0,0,0}$b$}%
}}}}
\put(2251,-1486){\makebox(0,0)[b]{\smash{{\SetFigFont{8}{9.6}{\rmdefault}{\mddefault}{\updefault}{\color[rgb]{0,0,0}$b$}%
}}}}
\put(2551, 14){\makebox(0,0)[b]{\smash{{\SetFigFont{8}{9.6}{\rmdefault}{\mddefault}{\updefault}{\color[rgb]{0,0,0}$b$}%
}}}}
\put(2551,-586){\makebox(0,0)[b]{\smash{{\SetFigFont{8}{9.6}{\rmdefault}{\mddefault}{\updefault}{\color[rgb]{0,0,0}$b$}%
}}}}
\put(2551,-886){\makebox(0,0)[b]{\smash{{\SetFigFont{8}{9.6}{\rmdefault}{\mddefault}{\updefault}{\color[rgb]{0,0,0}$b$}%
}}}}
\put(2551,-1186){\makebox(0,0)[b]{\smash{{\SetFigFont{8}{9.6}{\rmdefault}{\mddefault}{\updefault}{\color[rgb]{0,0,0}$b$}%
}}}}
\put(2551,-1786){\makebox(0,0)[b]{\smash{{\SetFigFont{8}{9.6}{\rmdefault}{\mddefault}{\updefault}{\color[rgb]{0,0,0}$b$}%
}}}}
\put(2551,-2086){\makebox(0,0)[b]{\smash{{\SetFigFont{8}{9.6}{\rmdefault}{\mddefault}{\updefault}{\color[rgb]{0,0,0}$b$}%
}}}}
\put(2551,-2386){\makebox(0,0)[b]{\smash{{\SetFigFont{8}{9.6}{\rmdefault}{\mddefault}{\updefault}{\color[rgb]{0,0,0}$b$}%
}}}}
\put(2551,-2986){\makebox(0,0)[b]{\smash{{\SetFigFont{8}{9.6}{\rmdefault}{\mddefault}{\updefault}{\color[rgb]{0,0,0}$b$}%
}}}}
\put(2551,-3286){\makebox(0,0)[b]{\smash{{\SetFigFont{8}{9.6}{\rmdefault}{\mddefault}{\updefault}{\color[rgb]{0,0,0}$b$}%
}}}}
\put(2551,-4186){\makebox(0,0)[b]{\smash{{\SetFigFont{8}{9.6}{\rmdefault}{\mddefault}{\updefault}{\color[rgb]{0,0,0}$b$}%
}}}}
\put(2551,-4486){\makebox(0,0)[b]{\smash{{\SetFigFont{8}{9.6}{\rmdefault}{\mddefault}{\updefault}{\color[rgb]{0,0,0}$b$}%
}}}}
\put(2551,-4786){\makebox(0,0)[b]{\smash{{\SetFigFont{8}{9.6}{\rmdefault}{\mddefault}{\updefault}{\color[rgb]{0,0,0}$b$}%
}}}}
\put(2551,-5086){\makebox(0,0)[b]{\smash{{\SetFigFont{8}{9.6}{\rmdefault}{\mddefault}{\updefault}{\color[rgb]{0,0,0}$b$}%
}}}}
\put(2551,-2686){\makebox(0,0)[b]{\smash{{\SetFigFont{8}{9.6}{\rmdefault}{\mddefault}{\updefault}{\color[rgb]{0,0,0}$b$}%
}}}}
\put(2551,-3586){\makebox(0,0)[b]{\smash{{\SetFigFont{8}{9.6}{\rmdefault}{\mddefault}{\updefault}{\color[rgb]{0,0,0}$b$}%
}}}}
\put(4051, 14){\makebox(0,0)[b]{\smash{{\SetFigFont{8}{9.6}{\rmdefault}{\mddefault}{\updefault}{\color[rgb]{0,0,0}$b$}%
}}}}
\put(4051,-286){\makebox(0,0)[b]{\smash{{\SetFigFont{8}{9.6}{\rmdefault}{\mddefault}{\updefault}{\color[rgb]{0,0,0}$b$}%
}}}}
\put(4051,-586){\makebox(0,0)[b]{\smash{{\SetFigFont{8}{9.6}{\rmdefault}{\mddefault}{\updefault}{\color[rgb]{0,0,0}$b$}%
}}}}
\put(4051,-886){\makebox(0,0)[b]{\smash{{\SetFigFont{8}{9.6}{\rmdefault}{\mddefault}{\updefault}{\color[rgb]{0,0,0}$b$}%
}}}}
\put(4051,-1186){\makebox(0,0)[b]{\smash{{\SetFigFont{8}{9.6}{\rmdefault}{\mddefault}{\updefault}{\color[rgb]{0,0,0}$b$}%
}}}}
\put(4051,-2086){\makebox(0,0)[b]{\smash{{\SetFigFont{8}{9.6}{\rmdefault}{\mddefault}{\updefault}{\color[rgb]{0,0,0}$b$}%
}}}}
\put(4051,-3286){\makebox(0,0)[b]{\smash{{\SetFigFont{8}{9.6}{\rmdefault}{\mddefault}{\updefault}{\color[rgb]{0,0,0}$b$}%
}}}}
\put(4051,-3886){\makebox(0,0)[b]{\smash{{\SetFigFont{8}{9.6}{\rmdefault}{\mddefault}{\updefault}{\color[rgb]{0,0,0}$b$}%
}}}}
\put(4051,-4186){\makebox(0,0)[b]{\smash{{\SetFigFont{8}{9.6}{\rmdefault}{\mddefault}{\updefault}{\color[rgb]{0,0,0}$b$}%
}}}}
\put(4051,-4486){\makebox(0,0)[b]{\smash{{\SetFigFont{8}{9.6}{\rmdefault}{\mddefault}{\updefault}{\color[rgb]{0,0,0}$b$}%
}}}}
\put(4051,-4786){\makebox(0,0)[b]{\smash{{\SetFigFont{8}{9.6}{\rmdefault}{\mddefault}{\updefault}{\color[rgb]{0,0,0}$b$}%
}}}}
\put(4051,-5086){\makebox(0,0)[b]{\smash{{\SetFigFont{8}{9.6}{\rmdefault}{\mddefault}{\updefault}{\color[rgb]{0,0,0}$b$}%
}}}}
\put(4051,-2686){\makebox(0,0)[b]{\smash{{\SetFigFont{8}{9.6}{\rmdefault}{\mddefault}{\updefault}{\color[rgb]{0,0,0}$b$}%
}}}}
\put(4051,-3586){\makebox(0,0)[b]{\smash{{\SetFigFont{8}{9.6}{\rmdefault}{\mddefault}{\updefault}{\color[rgb]{0,0,0}$b$}%
}}}}
\put(4051,-1486){\makebox(0,0)[b]{\smash{{\SetFigFont{8}{9.6}{\rmdefault}{\mddefault}{\updefault}{\color[rgb]{0,0,0}$b$}%
}}}}
\put(5551, 14){\makebox(0,0)[b]{\smash{{\SetFigFont{8}{9.6}{\rmdefault}{\mddefault}{\updefault}{\color[rgb]{0,0,0}$b$}%
}}}}
\put(5551,-286){\makebox(0,0)[b]{\smash{{\SetFigFont{8}{9.6}{\rmdefault}{\mddefault}{\updefault}{\color[rgb]{0,0,0}$b$}%
}}}}
\put(5551,-586){\makebox(0,0)[b]{\smash{{\SetFigFont{8}{9.6}{\rmdefault}{\mddefault}{\updefault}{\color[rgb]{0,0,0}$b$}%
}}}}
\put(5551,-1186){\makebox(0,0)[b]{\smash{{\SetFigFont{8}{9.6}{\rmdefault}{\mddefault}{\updefault}{\color[rgb]{0,0,0}$b$}%
}}}}
\put(5551,-1786){\makebox(0,0)[b]{\smash{{\SetFigFont{8}{9.6}{\rmdefault}{\mddefault}{\updefault}{\color[rgb]{0,0,0}$b$}%
}}}}
\put(5551,-2086){\makebox(0,0)[b]{\smash{{\SetFigFont{8}{9.6}{\rmdefault}{\mddefault}{\updefault}{\color[rgb]{0,0,0}$b$}%
}}}}
\put(5551,-2386){\makebox(0,0)[b]{\smash{{\SetFigFont{8}{9.6}{\rmdefault}{\mddefault}{\updefault}{\color[rgb]{0,0,0}$b$}%
}}}}
\put(5551,-2986){\makebox(0,0)[b]{\smash{{\SetFigFont{8}{9.6}{\rmdefault}{\mddefault}{\updefault}{\color[rgb]{0,0,0}$b$}%
}}}}
\put(5551,-3886){\makebox(0,0)[b]{\smash{{\SetFigFont{8}{9.6}{\rmdefault}{\mddefault}{\updefault}{\color[rgb]{0,0,0}$b$}%
}}}}
\put(5551,-4186){\makebox(0,0)[b]{\smash{{\SetFigFont{8}{9.6}{\rmdefault}{\mddefault}{\updefault}{\color[rgb]{0,0,0}$b$}%
}}}}
\put(5551,-4786){\makebox(0,0)[b]{\smash{{\SetFigFont{8}{9.6}{\rmdefault}{\mddefault}{\updefault}{\color[rgb]{0,0,0}$b$}%
}}}}
\put(5551,-5086){\makebox(0,0)[b]{\smash{{\SetFigFont{8}{9.6}{\rmdefault}{\mddefault}{\updefault}{\color[rgb]{0,0,0}$b$}%
}}}}
\put(5551,-2686){\makebox(0,0)[b]{\smash{{\SetFigFont{8}{9.6}{\rmdefault}{\mddefault}{\updefault}{\color[rgb]{0,0,0}$b$}%
}}}}
\put(5551,-3586){\makebox(0,0)[b]{\smash{{\SetFigFont{8}{9.6}{\rmdefault}{\mddefault}{\updefault}{\color[rgb]{0,0,0}$b$}%
}}}}
\put(5551,-1486){\makebox(0,0)[b]{\smash{{\SetFigFont{8}{9.6}{\rmdefault}{\mddefault}{\updefault}{\color[rgb]{0,0,0}$b$}%
}}}}
\put(5851, 14){\makebox(0,0)[b]{\smash{{\SetFigFont{8}{9.6}{\rmdefault}{\mddefault}{\updefault}{\color[rgb]{0,0,0}$b$}%
}}}}
\put(5851,-286){\makebox(0,0)[b]{\smash{{\SetFigFont{8}{9.6}{\rmdefault}{\mddefault}{\updefault}{\color[rgb]{0,0,0}$b$}%
}}}}
\put(5851,-886){\makebox(0,0)[b]{\smash{{\SetFigFont{8}{9.6}{\rmdefault}{\mddefault}{\updefault}{\color[rgb]{0,0,0}$b$}%
}}}}
\put(5851,-1186){\makebox(0,0)[b]{\smash{{\SetFigFont{8}{9.6}{\rmdefault}{\mddefault}{\updefault}{\color[rgb]{0,0,0}$b$}%
}}}}
\put(5851,-1786){\makebox(0,0)[b]{\smash{{\SetFigFont{8}{9.6}{\rmdefault}{\mddefault}{\updefault}{\color[rgb]{0,0,0}$b$}%
}}}}
\put(5851,-2086){\makebox(0,0)[b]{\smash{{\SetFigFont{8}{9.6}{\rmdefault}{\mddefault}{\updefault}{\color[rgb]{0,0,0}$b$}%
}}}}
\put(5851,-2386){\makebox(0,0)[b]{\smash{{\SetFigFont{8}{9.6}{\rmdefault}{\mddefault}{\updefault}{\color[rgb]{0,0,0}$b$}%
}}}}
\put(5851,-2986){\makebox(0,0)[b]{\smash{{\SetFigFont{8}{9.6}{\rmdefault}{\mddefault}{\updefault}{\color[rgb]{0,0,0}$b$}%
}}}}
\put(5851,-3286){\makebox(0,0)[b]{\smash{{\SetFigFont{8}{9.6}{\rmdefault}{\mddefault}{\updefault}{\color[rgb]{0,0,0}$b$}%
}}}}
\put(5851,-3886){\makebox(0,0)[b]{\smash{{\SetFigFont{8}{9.6}{\rmdefault}{\mddefault}{\updefault}{\color[rgb]{0,0,0}$b$}%
}}}}
\put(5851,-4186){\makebox(0,0)[b]{\smash{{\SetFigFont{8}{9.6}{\rmdefault}{\mddefault}{\updefault}{\color[rgb]{0,0,0}$b$}%
}}}}
\put(5851,-4486){\makebox(0,0)[b]{\smash{{\SetFigFont{8}{9.6}{\rmdefault}{\mddefault}{\updefault}{\color[rgb]{0,0,0}$b$}%
}}}}
\put(5851,-5086){\makebox(0,0)[b]{\smash{{\SetFigFont{8}{9.6}{\rmdefault}{\mddefault}{\updefault}{\color[rgb]{0,0,0}$b$}%
}}}}
\put(5851,-2686){\makebox(0,0)[b]{\smash{{\SetFigFont{8}{9.6}{\rmdefault}{\mddefault}{\updefault}{\color[rgb]{0,0,0}$b$}%
}}}}
\end{picture}%

%% file: 242Ex.pstex_t
\begin{picture}(0,0)%
\includegraphics{242Ex.eps}%
\end{picture}%
\setlength{\unitlength}{2763sp}%
\begingroup\makeatletter\ifx\SetFigFont\undefined%
\gdef\SetFigFont#1#2#3#4#5{%
  \reset@font\fontsize{#1}{#2pt}%
  \fontfamily{#3}\fontseries{#4}\fontshape{#5}%
  \selectfont}%
\fi\endgroup%
\begin{picture}(10677,6811)(136,-6050)
\put(4951,-3886){\makebox(0,0)[b]{\smash{{\SetFigFont{8}{9.6}{\rmdefault}{\mddefault}{\updefault}{\color[rgb]{0,0,0}$a$}%
}}}}
\put(4051,-2986){\makebox(0,0)[b]{\smash{{\SetFigFont{8}{9.6}{\rmdefault}{\mddefault}{\updefault}{\color[rgb]{0,0,0}$ac$}%
}}}}
\put(5851,-1186){\makebox(0,0)[b]{\smash{{\SetFigFont{8}{9.6}{\rmdefault}{\mddefault}{\updefault}{\color[rgb]{0,0,0}$ac$}%
}}}}
\put(7351,-2686){\makebox(0,0)[b]{\smash{{\SetFigFont{8}{9.6}{\rmdefault}{\mddefault}{\updefault}{\color[rgb]{0,0,0}$ac$}%
}}}}
\put(4651, 14){\makebox(0,0)[b]{\smash{{\SetFigFont{8}{9.6}{\rmdefault}{\mddefault}{\updefault}{\color[rgb]{0,0,0}$ab$}%
}}}}
\put(5251,-586){\makebox(0,0)[b]{\smash{{\SetFigFont{8}{9.6}{\rmdefault}{\mddefault}{\updefault}{\color[rgb]{0,0,0}$ab$}%
}}}}
\put(6151,-1486){\makebox(0,0)[b]{\smash{{\SetFigFont{8}{9.6}{\rmdefault}{\mddefault}{\updefault}{\color[rgb]{0,0,0}$ab$}%
}}}}
\put(6751,-2086){\makebox(0,0)[b]{\smash{{\SetFigFont{8}{9.6}{\rmdefault}{\mddefault}{\updefault}{\color[rgb]{0,0,0}$ab$}%
}}}}
\put(4651,-3586){\makebox(0,0)[b]{\smash{{\SetFigFont{8}{9.6}{\rmdefault}{\mddefault}{\updefault}{\color[rgb]{0,0,0}$ab$}%
}}}}
\put(5251,-4186){\makebox(0,0)[b]{\smash{{\SetFigFont{8}{9.6}{\rmdefault}{\mddefault}{\updefault}{\color[rgb]{0,0,0}$ab$}%
}}}}
\put(6151,-5086){\makebox(0,0)[b]{\smash{{\SetFigFont{8}{9.6}{\rmdefault}{\mddefault}{\updefault}{\color[rgb]{0,0,0}$ab$}%
}}}}
\put(6751,-5686){\makebox(0,0)[b]{\smash{{\SetFigFont{8}{9.6}{\rmdefault}{\mddefault}{\updefault}{\color[rgb]{0,0,0}$ab$}%
}}}}
\put(6751,-286){\makebox(0,0)[b]{\smash{{\SetFigFont{8}{9.6}{\rmdefault}{\mddefault}{\updefault}{\color[rgb]{0,0,0}$bc$}%
}}}}
\put(6151,-886){\makebox(0,0)[b]{\smash{{\SetFigFont{8}{9.6}{\rmdefault}{\mddefault}{\updefault}{\color[rgb]{0,0,0}$bc$}%
}}}}
\put(5251,-1786){\makebox(0,0)[b]{\smash{{\SetFigFont{8}{9.6}{\rmdefault}{\mddefault}{\updefault}{\color[rgb]{0,0,0}$bc$}%
}}}}
\put(4651,-2386){\makebox(0,0)[b]{\smash{{\SetFigFont{8}{9.6}{\rmdefault}{\mddefault}{\updefault}{\color[rgb]{0,0,0}$bc$}%
}}}}
\put(6751,-3286){\makebox(0,0)[b]{\smash{{\SetFigFont{8}{9.6}{\rmdefault}{\mddefault}{\updefault}{\color[rgb]{0,0,0}$bc$}%
}}}}
\put(6151,-3886){\makebox(0,0)[b]{\smash{{\SetFigFont{8}{9.6}{\rmdefault}{\mddefault}{\updefault}{\color[rgb]{0,0,0}$bc$}%
}}}}
\put(4651,-5386){\makebox(0,0)[b]{\smash{{\SetFigFont{8}{9.6}{\rmdefault}{\mddefault}{\updefault}{\color[rgb]{0,0,0}$bc$}%
}}}}
\put(5551,-4486){\makebox(0,0)[b]{\smash{{\SetFigFont{8}{9.6}{\rmdefault}{\mddefault}{\updefault}{\color[rgb]{0,0,0}$ac$}%
}}}}
\put(5251,-4786){\makebox(0,0)[b]{\smash{{\SetFigFont{8}{9.6}{\rmdefault}{\mddefault}{\updefault}{\color[rgb]{0,0,0}$bc$}%
}}}}
\put(151,-2686){\makebox(0,0)[b]{\smash{{\SetFigFont{8}{9.6}{\rmdefault}{\mddefault}{\updefault}{\color[rgb]{0,0,0}$10$}%
}}}}
\put(151, 14){\makebox(0,0)[b]{\smash{{\SetFigFont{8}{9.6}{\rmdefault}{\mddefault}{\updefault}{\color[rgb]{0,0,0}$1$}%
}}}}
\put(151,-286){\makebox(0,0)[b]{\smash{{\SetFigFont{8}{9.6}{\rmdefault}{\mddefault}{\updefault}{\color[rgb]{0,0,0}$2$}%
}}}}
\put(151,-586){\makebox(0,0)[b]{\smash{{\SetFigFont{8}{9.6}{\rmdefault}{\mddefault}{\updefault}{\color[rgb]{0,0,0}$3$}%
}}}}
\put(151,-886){\makebox(0,0)[b]{\smash{{\SetFigFont{8}{9.6}{\rmdefault}{\mddefault}{\updefault}{\color[rgb]{0,0,0}$4$}%
}}}}
\put(151,-1186){\makebox(0,0)[b]{\smash{{\SetFigFont{8}{9.6}{\rmdefault}{\mddefault}{\updefault}{\color[rgb]{0,0,0}$5$}%
}}}}
\put(151,-1486){\makebox(0,0)[b]{\smash{{\SetFigFont{8}{9.6}{\rmdefault}{\mddefault}{\updefault}{\color[rgb]{0,0,0}$6$}%
}}}}
\put(151,-1786){\makebox(0,0)[b]{\smash{{\SetFigFont{8}{9.6}{\rmdefault}{\mddefault}{\updefault}{\color[rgb]{0,0,0}$7$}%
}}}}
\put(151,-2086){\makebox(0,0)[b]{\smash{{\SetFigFont{8}{9.6}{\rmdefault}{\mddefault}{\updefault}{\color[rgb]{0,0,0}$8$}%
}}}}
\put(151,314){\makebox(0,0)[b]{\smash{{\SetFigFont{8}{9.6}{\rmdefault}{\mddefault}{\updefault}{\color[rgb]{0,0,0}$t$}%
}}}}
\put(151,-2386){\makebox(0,0)[b]{\smash{{\SetFigFont{8}{9.6}{\rmdefault}{\mddefault}{\updefault}{\color[rgb]{0,0,0}$9$}%
}}}}
\put(751,614){\makebox(0,0)[b]{\smash{{\SetFigFont{8}{9.6}{\rmdefault}{\mddefault}{\updefault}{\color[rgb]{0,0,0}$17$}%
}}}}
\put(151,-3286){\makebox(0,0)[b]{\smash{{\SetFigFont{8}{9.6}{\rmdefault}{\mddefault}{\updefault}{\color[rgb]{0,0,0}$12$}%
}}}}
\put(151,-3586){\makebox(0,0)[b]{\smash{{\SetFigFont{8}{9.6}{\rmdefault}{\mddefault}{\updefault}{\color[rgb]{0,0,0}$13$}%
}}}}
\put(151,-3886){\makebox(0,0)[b]{\smash{{\SetFigFont{8}{9.6}{\rmdefault}{\mddefault}{\updefault}{\color[rgb]{0,0,0}$14$}%
}}}}
\put(151,-4186){\makebox(0,0)[b]{\smash{{\SetFigFont{8}{9.6}{\rmdefault}{\mddefault}{\updefault}{\color[rgb]{0,0,0}$15$}%
}}}}
\put(151,-4486){\makebox(0,0)[b]{\smash{{\SetFigFont{8}{9.6}{\rmdefault}{\mddefault}{\updefault}{\color[rgb]{0,0,0}$16$}%
}}}}
\put(151,-4786){\makebox(0,0)[b]{\smash{{\SetFigFont{8}{9.6}{\rmdefault}{\mddefault}{\updefault}{\color[rgb]{0,0,0}$17$}%
}}}}
\put(151,-2986){\makebox(0,0)[b]{\smash{{\SetFigFont{8}{9.6}{\rmdefault}{\mddefault}{\updefault}{\color[rgb]{0,0,0}$11$}%
}}}}
\put(151,-5086){\makebox(0,0)[b]{\smash{{\SetFigFont{8}{9.6}{\rmdefault}{\mddefault}{\updefault}{\color[rgb]{0,0,0}$18$}%
}}}}
\put(151,-5386){\makebox(0,0)[b]{\smash{{\SetFigFont{8}{9.6}{\rmdefault}{\mddefault}{\updefault}{\color[rgb]{0,0,0}$19$}%
}}}}
\put(151,-5686){\makebox(0,0)[b]{\smash{{\SetFigFont{8}{9.6}{\rmdefault}{\mddefault}{\updefault}{\color[rgb]{0,0,0}$20$}%
}}}}
\put(4351,614){\makebox(0,0)[b]{\smash{{\SetFigFont{8}{9.6}{\rmdefault}{\mddefault}{\updefault}{\color[rgb]{0,0,0}$5$}%
}}}}
\put(4651,614){\makebox(0,0)[b]{\smash{{\SetFigFont{8}{9.6}{\rmdefault}{\mddefault}{\updefault}{\color[rgb]{0,0,0}$4$}%
}}}}
\put(5251,614){\makebox(0,0)[b]{\smash{{\SetFigFont{8}{9.6}{\rmdefault}{\mddefault}{\updefault}{\color[rgb]{0,0,0}$2$}%
}}}}
\put(6151,614){\makebox(0,0)[b]{\smash{{\SetFigFont{8}{9.6}{\rmdefault}{\mddefault}{\updefault}{\color[rgb]{0,0,0}$-1$}%
}}}}
\put(6751,614){\makebox(0,0)[b]{\smash{{\SetFigFont{8}{9.6}{\rmdefault}{\mddefault}{\updefault}{\color[rgb]{0,0,0}$-3$}%
}}}}
\put(7351,614){\makebox(0,0)[b]{\smash{{\SetFigFont{8}{9.6}{\rmdefault}{\mddefault}{\updefault}{\color[rgb]{0,0,0}$-5$}%
}}}}
\put(10351,614){\makebox(0,0)[b]{\smash{{\SetFigFont{8}{9.6}{\rmdefault}{\mddefault}{\updefault}{\color[rgb]{0,0,0}$-15$}%
}}}}
\put(4351,314){\makebox(0,0)[b]{\smash{{\SetFigFont{8}{9.6}{\rmdefault}{\mddefault}{\updefault}{\color[rgb]{0,0,0}$a$}%
}}}}
\put(4951,-286){\makebox(0,0)[b]{\smash{{\SetFigFont{8}{9.6}{\rmdefault}{\mddefault}{\updefault}{\color[rgb]{0,0,0}$a$}%
}}}}
\put(5551,-886){\makebox(0,0)[b]{\smash{{\SetFigFont{8}{9.6}{\rmdefault}{\mddefault}{\updefault}{\color[rgb]{0,0,0}$a$}%
}}}}
\put(6451,-1786){\makebox(0,0)[b]{\smash{{\SetFigFont{8}{9.6}{\rmdefault}{\mddefault}{\updefault}{\color[rgb]{0,0,0}$a$}%
}}}}
\put(7051,-2386){\makebox(0,0)[b]{\smash{{\SetFigFont{8}{9.6}{\rmdefault}{\mddefault}{\updefault}{\color[rgb]{0,0,0}$a$}%
}}}}
\put(7651,-2986){\makebox(0,0)[b]{\smash{{\SetFigFont{8}{9.6}{\rmdefault}{\mddefault}{\updefault}{\color[rgb]{0,0,0}$a$}%
}}}}
\put(7951,-3286){\makebox(0,0)[b]{\smash{{\SetFigFont{8}{9.6}{\rmdefault}{\mddefault}{\updefault}{\color[rgb]{0,0,0}$a$}%
}}}}
\put(8251,-3586){\makebox(0,0)[b]{\smash{{\SetFigFont{8}{9.6}{\rmdefault}{\mddefault}{\updefault}{\color[rgb]{0,0,0}$a$}%
}}}}
\put(8551,-3886){\makebox(0,0)[b]{\smash{{\SetFigFont{8}{9.6}{\rmdefault}{\mddefault}{\updefault}{\color[rgb]{0,0,0}$a$}%
}}}}
\put(8851,-4186){\makebox(0,0)[b]{\smash{{\SetFigFont{8}{9.6}{\rmdefault}{\mddefault}{\updefault}{\color[rgb]{0,0,0}$a$}%
}}}}
\put(9151,-4486){\makebox(0,0)[b]{\smash{{\SetFigFont{8}{9.6}{\rmdefault}{\mddefault}{\updefault}{\color[rgb]{0,0,0}$a$}%
}}}}
\put(9451,-4786){\makebox(0,0)[b]{\smash{{\SetFigFont{8}{9.6}{\rmdefault}{\mddefault}{\updefault}{\color[rgb]{0,0,0}$a$}%
}}}}
\put(9751,-5086){\makebox(0,0)[b]{\smash{{\SetFigFont{8}{9.6}{\rmdefault}{\mddefault}{\updefault}{\color[rgb]{0,0,0}$a$}%
}}}}
\put(10051,-5386){\makebox(0,0)[b]{\smash{{\SetFigFont{8}{9.6}{\rmdefault}{\mddefault}{\updefault}{\color[rgb]{0,0,0}$a$}%
}}}}
\put(10351,-5686){\makebox(0,0)[b]{\smash{{\SetFigFont{8}{9.6}{\rmdefault}{\mddefault}{\updefault}{\color[rgb]{0,0,0}$a$}%
}}}}
\put(10651,-5986){\makebox(0,0)[b]{\smash{{\SetFigFont{8}{9.6}{\rmdefault}{\mddefault}{\updefault}{\color[rgb]{0,0,0}$a$}%
}}}}
\put(9151,-886){\makebox(0,0)[b]{\smash{{\SetFigFont{8}{9.6}{\rmdefault}{\mddefault}{\updefault}{\color[rgb]{0,0,0}$c$}%
}}}}
\put(9451,-586){\makebox(0,0)[b]{\smash{{\SetFigFont{8}{9.6}{\rmdefault}{\mddefault}{\updefault}{\color[rgb]{0,0,0}$c$}%
}}}}
\put(9751,-286){\makebox(0,0)[b]{\smash{{\SetFigFont{8}{9.6}{\rmdefault}{\mddefault}{\updefault}{\color[rgb]{0,0,0}$c$}%
}}}}
\put(10351,314){\makebox(0,0)[b]{\smash{{\SetFigFont{8}{9.6}{\rmdefault}{\mddefault}{\updefault}{\color[rgb]{0,0,0}$c$}%
}}}}
\put(8851,-1186){\makebox(0,0)[b]{\smash{{\SetFigFont{8}{9.6}{\rmdefault}{\mddefault}{\updefault}{\color[rgb]{0,0,0}$c$}%
}}}}
\put(8551,-1486){\makebox(0,0)[b]{\smash{{\SetFigFont{8}{9.6}{\rmdefault}{\mddefault}{\updefault}{\color[rgb]{0,0,0}$c$}%
}}}}
\put(8251,-1786){\makebox(0,0)[b]{\smash{{\SetFigFont{8}{9.6}{\rmdefault}{\mddefault}{\updefault}{\color[rgb]{0,0,0}$c$}%
}}}}
\put(7951,-2086){\makebox(0,0)[b]{\smash{{\SetFigFont{8}{9.6}{\rmdefault}{\mddefault}{\updefault}{\color[rgb]{0,0,0}$c$}%
}}}}
\put(7651,-2386){\makebox(0,0)[b]{\smash{{\SetFigFont{8}{9.6}{\rmdefault}{\mddefault}{\updefault}{\color[rgb]{0,0,0}$c$}%
}}}}
\put(7051,-2986){\makebox(0,0)[b]{\smash{{\SetFigFont{8}{9.6}{\rmdefault}{\mddefault}{\updefault}{\color[rgb]{0,0,0}$c$}%
}}}}
\put(6451,-3586){\makebox(0,0)[b]{\smash{{\SetFigFont{8}{9.6}{\rmdefault}{\mddefault}{\updefault}{\color[rgb]{0,0,0}$c$}%
}}}}
\put(5851,-4186){\makebox(0,0)[b]{\smash{{\SetFigFont{8}{9.6}{\rmdefault}{\mddefault}{\updefault}{\color[rgb]{0,0,0}$c$}%
}}}}
\put(4951,-5086){\makebox(0,0)[b]{\smash{{\SetFigFont{8}{9.6}{\rmdefault}{\mddefault}{\updefault}{\color[rgb]{0,0,0}$c$}%
}}}}
\put(4351,-5686){\makebox(0,0)[b]{\smash{{\SetFigFont{8}{9.6}{\rmdefault}{\mddefault}{\updefault}{\color[rgb]{0,0,0}$c$}%
}}}}
\put(4051,-5986){\makebox(0,0)[b]{\smash{{\SetFigFont{8}{9.6}{\rmdefault}{\mddefault}{\updefault}{\color[rgb]{0,0,0}$c$}%
}}}}
\put(10051, 14){\makebox(0,0)[b]{\smash{{\SetFigFont{8}{9.6}{\rmdefault}{\mddefault}{\updefault}{\color[rgb]{0,0,0}$c$}%
}}}}
\put(6451,-586){\makebox(0,0)[b]{\smash{{\SetFigFont{8}{9.6}{\rmdefault}{\mddefault}{\updefault}{\color[rgb]{0,0,0}$c$}%
}}}}
\put(7351,314){\makebox(0,0)[b]{\smash{{\SetFigFont{8}{9.6}{\rmdefault}{\mddefault}{\updefault}{\color[rgb]{0,0,0}$c$}%
}}}}
\put(5551,-1486){\makebox(0,0)[b]{\smash{{\SetFigFont{8}{9.6}{\rmdefault}{\mddefault}{\updefault}{\color[rgb]{0,0,0}$c$}%
}}}}
\put(4951,-2086){\makebox(0,0)[b]{\smash{{\SetFigFont{8}{9.6}{\rmdefault}{\mddefault}{\updefault}{\color[rgb]{0,0,0}$c$}%
}}}}
\put(4351,-2686){\makebox(0,0)[b]{\smash{{\SetFigFont{8}{9.6}{\rmdefault}{\mddefault}{\updefault}{\color[rgb]{0,0,0}$c$}%
}}}}
\put(3751,-3286){\makebox(0,0)[b]{\smash{{\SetFigFont{8}{9.6}{\rmdefault}{\mddefault}{\updefault}{\color[rgb]{0,0,0}$c$}%
}}}}
\put(3451,-3586){\makebox(0,0)[b]{\smash{{\SetFigFont{8}{9.6}{\rmdefault}{\mddefault}{\updefault}{\color[rgb]{0,0,0}$c$}%
}}}}
\put(3151,-3886){\makebox(0,0)[b]{\smash{{\SetFigFont{8}{9.6}{\rmdefault}{\mddefault}{\updefault}{\color[rgb]{0,0,0}$c$}%
}}}}
\put(2851,-4186){\makebox(0,0)[b]{\smash{{\SetFigFont{8}{9.6}{\rmdefault}{\mddefault}{\updefault}{\color[rgb]{0,0,0}$c$}%
}}}}
\put(2551,-4486){\makebox(0,0)[b]{\smash{{\SetFigFont{8}{9.6}{\rmdefault}{\mddefault}{\updefault}{\color[rgb]{0,0,0}$c$}%
}}}}
\put(2251,-4786){\makebox(0,0)[b]{\smash{{\SetFigFont{8}{9.6}{\rmdefault}{\mddefault}{\updefault}{\color[rgb]{0,0,0}$c$}%
}}}}
\put(1951,-5086){\makebox(0,0)[b]{\smash{{\SetFigFont{8}{9.6}{\rmdefault}{\mddefault}{\updefault}{\color[rgb]{0,0,0}$c$}%
}}}}
\put(1651,-5386){\makebox(0,0)[b]{\smash{{\SetFigFont{8}{9.6}{\rmdefault}{\mddefault}{\updefault}{\color[rgb]{0,0,0}$c$}%
}}}}
\put(1351,-5686){\makebox(0,0)[b]{\smash{{\SetFigFont{8}{9.6}{\rmdefault}{\mddefault}{\updefault}{\color[rgb]{0,0,0}$c$}%
}}}}
\put(1051,-5986){\makebox(0,0)[b]{\smash{{\SetFigFont{8}{9.6}{\rmdefault}{\mddefault}{\updefault}{\color[rgb]{0,0,0}$c$}%
}}}}
\put(7051, 14){\makebox(0,0)[b]{\smash{{\SetFigFont{8}{9.6}{\rmdefault}{\mddefault}{\updefault}{\color[rgb]{0,0,0}$c$}%
}}}}
\put(751,314){\makebox(0,0)[b]{\smash{{\SetFigFont{8}{9.6}{\rmdefault}{\mddefault}{\updefault}{\color[rgb]{0,0,0}$a$}%
}}}}
\put(1051, 14){\makebox(0,0)[b]{\smash{{\SetFigFont{8}{9.6}{\rmdefault}{\mddefault}{\updefault}{\color[rgb]{0,0,0}$a$}%
}}}}
\put(1351,-286){\makebox(0,0)[b]{\smash{{\SetFigFont{8}{9.6}{\rmdefault}{\mddefault}{\updefault}{\color[rgb]{0,0,0}$a$}%
}}}}
\put(1651,-586){\makebox(0,0)[b]{\smash{{\SetFigFont{8}{9.6}{\rmdefault}{\mddefault}{\updefault}{\color[rgb]{0,0,0}$a$}%
}}}}
\put(1951,-886){\makebox(0,0)[b]{\smash{{\SetFigFont{8}{9.6}{\rmdefault}{\mddefault}{\updefault}{\color[rgb]{0,0,0}$a$}%
}}}}
\put(2251,-1186){\makebox(0,0)[b]{\smash{{\SetFigFont{8}{9.6}{\rmdefault}{\mddefault}{\updefault}{\color[rgb]{0,0,0}$a$}%
}}}}
\put(2551,-1486){\makebox(0,0)[b]{\smash{{\SetFigFont{8}{9.6}{\rmdefault}{\mddefault}{\updefault}{\color[rgb]{0,0,0}$a$}%
}}}}
\put(2851,-1786){\makebox(0,0)[b]{\smash{{\SetFigFont{8}{9.6}{\rmdefault}{\mddefault}{\updefault}{\color[rgb]{0,0,0}$a$}%
}}}}
\put(3151,-2086){\makebox(0,0)[b]{\smash{{\SetFigFont{8}{9.6}{\rmdefault}{\mddefault}{\updefault}{\color[rgb]{0,0,0}$a$}%
}}}}
\put(3451,-2386){\makebox(0,0)[b]{\smash{{\SetFigFont{8}{9.6}{\rmdefault}{\mddefault}{\updefault}{\color[rgb]{0,0,0}$a$}%
}}}}
\put(3751,-2686){\makebox(0,0)[b]{\smash{{\SetFigFont{8}{9.6}{\rmdefault}{\mddefault}{\updefault}{\color[rgb]{0,0,0}$a$}%
}}}}
\put(4351,-3286){\makebox(0,0)[b]{\smash{{\SetFigFont{8}{9.6}{\rmdefault}{\mddefault}{\updefault}{\color[rgb]{0,0,0}$a$}%
}}}}
\put(5851,-4786){\makebox(0,0)[b]{\smash{{\SetFigFont{8}{9.6}{\rmdefault}{\mddefault}{\updefault}{\color[rgb]{0,0,0}$a$}%
}}}}
\put(6451,-5386){\makebox(0,0)[b]{\smash{{\SetFigFont{8}{9.6}{\rmdefault}{\mddefault}{\updefault}{\color[rgb]{0,0,0}$a$}%
}}}}
\put(7051,-5986){\makebox(0,0)[b]{\smash{{\SetFigFont{8}{9.6}{\rmdefault}{\mddefault}{\updefault}{\color[rgb]{0,0,0}$a$}%
}}}}
\put(6751,314){\makebox(0,0)[b]{\smash{{\SetFigFont{8}{9.6}{\rmdefault}{\mddefault}{\updefault}{\color[rgb]{0,0,0}$b$}%
}}}}
\put(6751, 14){\makebox(0,0)[b]{\smash{{\SetFigFont{8}{9.6}{\rmdefault}{\mddefault}{\updefault}{\color[rgb]{0,0,0}$b$}%
}}}}
\put(6751,-586){\makebox(0,0)[b]{\smash{{\SetFigFont{8}{9.6}{\rmdefault}{\mddefault}{\updefault}{\color[rgb]{0,0,0}$b$}%
}}}}
\put(6751,-886){\makebox(0,0)[b]{\smash{{\SetFigFont{8}{9.6}{\rmdefault}{\mddefault}{\updefault}{\color[rgb]{0,0,0}$b$}%
}}}}
\put(6751,-1186){\makebox(0,0)[b]{\smash{{\SetFigFont{8}{9.6}{\rmdefault}{\mddefault}{\updefault}{\color[rgb]{0,0,0}$b$}%
}}}}
\put(6751,-1486){\makebox(0,0)[b]{\smash{{\SetFigFont{8}{9.6}{\rmdefault}{\mddefault}{\updefault}{\color[rgb]{0,0,0}$b$}%
}}}}
\put(6751,-1786){\makebox(0,0)[b]{\smash{{\SetFigFont{8}{9.6}{\rmdefault}{\mddefault}{\updefault}{\color[rgb]{0,0,0}$b$}%
}}}}
\put(6751,-2386){\makebox(0,0)[b]{\smash{{\SetFigFont{8}{9.6}{\rmdefault}{\mddefault}{\updefault}{\color[rgb]{0,0,0}$b$}%
}}}}
\put(6751,-2686){\makebox(0,0)[b]{\smash{{\SetFigFont{8}{9.6}{\rmdefault}{\mddefault}{\updefault}{\color[rgb]{0,0,0}$b$}%
}}}}
\put(6751,-2986){\makebox(0,0)[b]{\smash{{\SetFigFont{8}{9.6}{\rmdefault}{\mddefault}{\updefault}{\color[rgb]{0,0,0}$b$}%
}}}}
\put(6751,-3586){\makebox(0,0)[b]{\smash{{\SetFigFont{8}{9.6}{\rmdefault}{\mddefault}{\updefault}{\color[rgb]{0,0,0}$b$}%
}}}}
\put(6751,-3886){\makebox(0,0)[b]{\smash{{\SetFigFont{8}{9.6}{\rmdefault}{\mddefault}{\updefault}{\color[rgb]{0,0,0}$b$}%
}}}}
\put(6751,-4186){\makebox(0,0)[b]{\smash{{\SetFigFont{8}{9.6}{\rmdefault}{\mddefault}{\updefault}{\color[rgb]{0,0,0}$b$}%
}}}}
\put(6751,-4486){\makebox(0,0)[b]{\smash{{\SetFigFont{8}{9.6}{\rmdefault}{\mddefault}{\updefault}{\color[rgb]{0,0,0}$b$}%
}}}}
\put(6751,-4786){\makebox(0,0)[b]{\smash{{\SetFigFont{8}{9.6}{\rmdefault}{\mddefault}{\updefault}{\color[rgb]{0,0,0}$b$}%
}}}}
\put(6751,-5086){\makebox(0,0)[b]{\smash{{\SetFigFont{8}{9.6}{\rmdefault}{\mddefault}{\updefault}{\color[rgb]{0,0,0}$b$}%
}}}}
\put(6751,-5386){\makebox(0,0)[b]{\smash{{\SetFigFont{8}{9.6}{\rmdefault}{\mddefault}{\updefault}{\color[rgb]{0,0,0}$b$}%
}}}}
\put(6751,-5986){\makebox(0,0)[b]{\smash{{\SetFigFont{8}{9.6}{\rmdefault}{\mddefault}{\updefault}{\color[rgb]{0,0,0}$b$}%
}}}}
\put(6151,314){\makebox(0,0)[b]{\smash{{\SetFigFont{8}{9.6}{\rmdefault}{\mddefault}{\updefault}{\color[rgb]{0,0,0}$b$}%
}}}}
\put(6151, 14){\makebox(0,0)[b]{\smash{{\SetFigFont{8}{9.6}{\rmdefault}{\mddefault}{\updefault}{\color[rgb]{0,0,0}$b$}%
}}}}
\put(6151,-286){\makebox(0,0)[b]{\smash{{\SetFigFont{8}{9.6}{\rmdefault}{\mddefault}{\updefault}{\color[rgb]{0,0,0}$b$}%
}}}}
\put(6151,-586){\makebox(0,0)[b]{\smash{{\SetFigFont{8}{9.6}{\rmdefault}{\mddefault}{\updefault}{\color[rgb]{0,0,0}$b$}%
}}}}
\put(6151,-1186){\makebox(0,0)[b]{\smash{{\SetFigFont{8}{9.6}{\rmdefault}{\mddefault}{\updefault}{\color[rgb]{0,0,0}$b$}%
}}}}
\put(6151,-1786){\makebox(0,0)[b]{\smash{{\SetFigFont{8}{9.6}{\rmdefault}{\mddefault}{\updefault}{\color[rgb]{0,0,0}$b$}%
}}}}
\put(6151,-2086){\makebox(0,0)[b]{\smash{{\SetFigFont{8}{9.6}{\rmdefault}{\mddefault}{\updefault}{\color[rgb]{0,0,0}$b$}%
}}}}
\put(6151,-2386){\makebox(0,0)[b]{\smash{{\SetFigFont{8}{9.6}{\rmdefault}{\mddefault}{\updefault}{\color[rgb]{0,0,0}$b$}%
}}}}
\put(6151,-2686){\makebox(0,0)[b]{\smash{{\SetFigFont{8}{9.6}{\rmdefault}{\mddefault}{\updefault}{\color[rgb]{0,0,0}$b$}%
}}}}
\put(6151,-2986){\makebox(0,0)[b]{\smash{{\SetFigFont{8}{9.6}{\rmdefault}{\mddefault}{\updefault}{\color[rgb]{0,0,0}$b$}%
}}}}
\put(6151,-3286){\makebox(0,0)[b]{\smash{{\SetFigFont{8}{9.6}{\rmdefault}{\mddefault}{\updefault}{\color[rgb]{0,0,0}$b$}%
}}}}
\put(6151,-3586){\makebox(0,0)[b]{\smash{{\SetFigFont{8}{9.6}{\rmdefault}{\mddefault}{\updefault}{\color[rgb]{0,0,0}$b$}%
}}}}
\put(6151,-4186){\makebox(0,0)[b]{\smash{{\SetFigFont{8}{9.6}{\rmdefault}{\mddefault}{\updefault}{\color[rgb]{0,0,0}$b$}%
}}}}
\put(6151,-4486){\makebox(0,0)[b]{\smash{{\SetFigFont{8}{9.6}{\rmdefault}{\mddefault}{\updefault}{\color[rgb]{0,0,0}$b$}%
}}}}
\put(6151,-4786){\makebox(0,0)[b]{\smash{{\SetFigFont{8}{9.6}{\rmdefault}{\mddefault}{\updefault}{\color[rgb]{0,0,0}$b$}%
}}}}
\put(6151,-5386){\makebox(0,0)[b]{\smash{{\SetFigFont{8}{9.6}{\rmdefault}{\mddefault}{\updefault}{\color[rgb]{0,0,0}$b$}%
}}}}
\put(6151,-5686){\makebox(0,0)[b]{\smash{{\SetFigFont{8}{9.6}{\rmdefault}{\mddefault}{\updefault}{\color[rgb]{0,0,0}$b$}%
}}}}
\put(6151,-5986){\makebox(0,0)[b]{\smash{{\SetFigFont{8}{9.6}{\rmdefault}{\mddefault}{\updefault}{\color[rgb]{0,0,0}$b$}%
}}}}
\put(5251,314){\makebox(0,0)[b]{\smash{{\SetFigFont{8}{9.6}{\rmdefault}{\mddefault}{\updefault}{\color[rgb]{0,0,0}$b$}%
}}}}
\put(5251, 14){\makebox(0,0)[b]{\smash{{\SetFigFont{8}{9.6}{\rmdefault}{\mddefault}{\updefault}{\color[rgb]{0,0,0}$b$}%
}}}}
\put(5251,-286){\makebox(0,0)[b]{\smash{{\SetFigFont{8}{9.6}{\rmdefault}{\mddefault}{\updefault}{\color[rgb]{0,0,0}$b$}%
}}}}
\put(5251,-886){\makebox(0,0)[b]{\smash{{\SetFigFont{8}{9.6}{\rmdefault}{\mddefault}{\updefault}{\color[rgb]{0,0,0}$b$}%
}}}}
\put(5251,-1186){\makebox(0,0)[b]{\smash{{\SetFigFont{8}{9.6}{\rmdefault}{\mddefault}{\updefault}{\color[rgb]{0,0,0}$b$}%
}}}}
\put(5251,-1486){\makebox(0,0)[b]{\smash{{\SetFigFont{8}{9.6}{\rmdefault}{\mddefault}{\updefault}{\color[rgb]{0,0,0}$b$}%
}}}}
\put(5251,-2086){\makebox(0,0)[b]{\smash{{\SetFigFont{8}{9.6}{\rmdefault}{\mddefault}{\updefault}{\color[rgb]{0,0,0}$b$}%
}}}}
\put(5251,-2386){\makebox(0,0)[b]{\smash{{\SetFigFont{8}{9.6}{\rmdefault}{\mddefault}{\updefault}{\color[rgb]{0,0,0}$b$}%
}}}}
\put(5251,-2686){\makebox(0,0)[b]{\smash{{\SetFigFont{8}{9.6}{\rmdefault}{\mddefault}{\updefault}{\color[rgb]{0,0,0}$b$}%
}}}}
\put(5251,-2986){\makebox(0,0)[b]{\smash{{\SetFigFont{8}{9.6}{\rmdefault}{\mddefault}{\updefault}{\color[rgb]{0,0,0}$b$}%
}}}}
\put(5251,-3286){\makebox(0,0)[b]{\smash{{\SetFigFont{8}{9.6}{\rmdefault}{\mddefault}{\updefault}{\color[rgb]{0,0,0}$b$}%
}}}}
\put(5251,-3586){\makebox(0,0)[b]{\smash{{\SetFigFont{8}{9.6}{\rmdefault}{\mddefault}{\updefault}{\color[rgb]{0,0,0}$b$}%
}}}}
\put(5251,-3886){\makebox(0,0)[b]{\smash{{\SetFigFont{8}{9.6}{\rmdefault}{\mddefault}{\updefault}{\color[rgb]{0,0,0}$b$}%
}}}}
\put(5251,-4486){\makebox(0,0)[b]{\smash{{\SetFigFont{8}{9.6}{\rmdefault}{\mddefault}{\updefault}{\color[rgb]{0,0,0}$b$}%
}}}}
\put(5251,-5086){\makebox(0,0)[b]{\smash{{\SetFigFont{8}{9.6}{\rmdefault}{\mddefault}{\updefault}{\color[rgb]{0,0,0}$b$}%
}}}}
\put(5251,-5386){\makebox(0,0)[b]{\smash{{\SetFigFont{8}{9.6}{\rmdefault}{\mddefault}{\updefault}{\color[rgb]{0,0,0}$b$}%
}}}}
\put(5251,-5686){\makebox(0,0)[b]{\smash{{\SetFigFont{8}{9.6}{\rmdefault}{\mddefault}{\updefault}{\color[rgb]{0,0,0}$b$}%
}}}}
\put(5251,-5986){\makebox(0,0)[b]{\smash{{\SetFigFont{8}{9.6}{\rmdefault}{\mddefault}{\updefault}{\color[rgb]{0,0,0}$b$}%
}}}}
\put(4651,314){\makebox(0,0)[b]{\smash{{\SetFigFont{8}{9.6}{\rmdefault}{\mddefault}{\updefault}{\color[rgb]{0,0,0}$b$}%
}}}}
\put(4651,-286){\makebox(0,0)[b]{\smash{{\SetFigFont{8}{9.6}{\rmdefault}{\mddefault}{\updefault}{\color[rgb]{0,0,0}$b$}%
}}}}
\put(4651,-586){\makebox(0,0)[b]{\smash{{\SetFigFont{8}{9.6}{\rmdefault}{\mddefault}{\updefault}{\color[rgb]{0,0,0}$b$}%
}}}}
\put(4651,-886){\makebox(0,0)[b]{\smash{{\SetFigFont{8}{9.6}{\rmdefault}{\mddefault}{\updefault}{\color[rgb]{0,0,0}$b$}%
}}}}
\put(4651,-1186){\makebox(0,0)[b]{\smash{{\SetFigFont{8}{9.6}{\rmdefault}{\mddefault}{\updefault}{\color[rgb]{0,0,0}$b$}%
}}}}
\put(4651,-1486){\makebox(0,0)[b]{\smash{{\SetFigFont{8}{9.6}{\rmdefault}{\mddefault}{\updefault}{\color[rgb]{0,0,0}$b$}%
}}}}
\put(4651,-1786){\makebox(0,0)[b]{\smash{{\SetFigFont{8}{9.6}{\rmdefault}{\mddefault}{\updefault}{\color[rgb]{0,0,0}$b$}%
}}}}
\put(4651,-2086){\makebox(0,0)[b]{\smash{{\SetFigFont{8}{9.6}{\rmdefault}{\mddefault}{\updefault}{\color[rgb]{0,0,0}$b$}%
}}}}
\put(4651,-2686){\makebox(0,0)[b]{\smash{{\SetFigFont{8}{9.6}{\rmdefault}{\mddefault}{\updefault}{\color[rgb]{0,0,0}$b$}%
}}}}
\put(4651,-2986){\makebox(0,0)[b]{\smash{{\SetFigFont{8}{9.6}{\rmdefault}{\mddefault}{\updefault}{\color[rgb]{0,0,0}$b$}%
}}}}
\put(4651,-3286){\makebox(0,0)[b]{\smash{{\SetFigFont{8}{9.6}{\rmdefault}{\mddefault}{\updefault}{\color[rgb]{0,0,0}$b$}%
}}}}
\put(4651,-3886){\makebox(0,0)[b]{\smash{{\SetFigFont{8}{9.6}{\rmdefault}{\mddefault}{\updefault}{\color[rgb]{0,0,0}$b$}%
}}}}
\put(4651,-4186){\makebox(0,0)[b]{\smash{{\SetFigFont{8}{9.6}{\rmdefault}{\mddefault}{\updefault}{\color[rgb]{0,0,0}$b$}%
}}}}
\put(4651,-4486){\makebox(0,0)[b]{\smash{{\SetFigFont{8}{9.6}{\rmdefault}{\mddefault}{\updefault}{\color[rgb]{0,0,0}$b$}%
}}}}
\put(4651,-4786){\makebox(0,0)[b]{\smash{{\SetFigFont{8}{9.6}{\rmdefault}{\mddefault}{\updefault}{\color[rgb]{0,0,0}$b$}%
}}}}
\put(4651,-5086){\makebox(0,0)[b]{\smash{{\SetFigFont{8}{9.6}{\rmdefault}{\mddefault}{\updefault}{\color[rgb]{0,0,0}$b$}%
}}}}
\put(4651,-5686){\makebox(0,0)[b]{\smash{{\SetFigFont{8}{9.6}{\rmdefault}{\mddefault}{\updefault}{\color[rgb]{0,0,0}$b$}%
}}}}
\put(4651,-5986){\makebox(0,0)[b]{\smash{{\SetFigFont{8}{9.6}{\rmdefault}{\mddefault}{\updefault}{\color[rgb]{0,0,0}$b$}%
}}}}
\end{picture}%